\documentclass[10pt, a4paper]{amsart}

\pdfoutput=1

\usepackage[T1]{fontenc}

\usepackage{mathrsfs, amsmath, amsthm, amssymb, stmaryrd, enumerate, bbm}%
\usepackage{tikz}%
\usetikzlibrary{knots}
\usepackage[linktocpage]{hyperref}%
\usepackage[all]{xy}%
\xyoption{2cell}%
\UseAllTwocells%

\newdir{t>}{{\UseTips\dir{>}}}
\newdir{ >}{{}*!/-5pt/@{>}}

\DeclareMathOperator{\id}{id}

\DeclareMathOperator{\op}{op}

\DeclareMathOperator{\Mod}{\mathbf{Mod}}

\DeclareMathOperator{\Spec}{Spec}

\DeclareMathOperator{\Sym}{Sym}

\DeclareMathOperator{\SymMonCat}{\mathbf{SymMonCat}}

\DeclareMathOperator{\VB}{\mathbf{VB}}

\DeclareMathOperator{\fpqc}{\mathit{fpqc}}

\DeclareMathOperator{\Cat}{\mathbf{Cat}}

\DeclareMathOperator{\Gpd}{\mathbf{Gpd}}

\DeclareMathOperator{\Set}{\mathbf{Set}}
\DeclareMathOperator{\Ab}{\mathbf{Ab}}

\DeclareMathOperator{\CAlg}{\mathbf{CAlg}}
\DeclareMathOperator{\CRing}{\mathbf{CRing}}
\DeclareMathOperator{\Grp}{\mathbf{Grp}}

\newcommand{\ca}[1]{\mathscr{#1}}

\newcommand{\Proj}{\mathbf{P}}
\newcommand{\Alt}{\mathbf{A}}

\newcommand{\ten}[1]{\mathop{{\otimes}_{#1}}}

\newcommand{\pb}[1]{\mathop{{\times}_{#1}}}

\newcommand{\defl}{\mathrel{\mathop:}=}

\theoremstyle{plain}
\newtheorem{thm}{Theorem}[section]
\newtheorem*{thm*}{Theorem}
\newtheorem{prop}[thm]{Proposition}
\newtheorem{lemma}[thm]{Lemma}
\newtheorem{cor}[thm]{Corollary}

\theoremstyle{definition}
\newtheorem{example}[thm]{Example}
\newtheorem{rmk}[thm]{Remark}
\newtheorem{dfn}[thm]{Definition}
\newtheorem{notation}[thm]{Notation}

\newtheoremstyle{citing}{}{}{\itshape}{}{\bfseries}{.}{ }{\thmnote{#3}}
\theoremstyle{citing}
\newtheorem{cit}{}

\newtheoremstyle{citingdfn}{}{}{}{}{\bfseries}{.}{ }{\thmnote{#3}}
\theoremstyle{citingdfn}

\newtheorem*{question*}{}

\numberwithin{equation}{section}



\author{Daniel Sch\"appi}
\thanks{This research was supported by the DFG grant: SFB 1085 ``Higher invariants.''}
\address{Fakult{\"a}t f{\"u}r Mathematik,
Universit{\"a}t Regensburg,
93040 Regensburg,
Germany}
\email{daniel.schaeppi@ur.de}

\title{Symplectic K-theory and a problem of Murthy}

\begin{document}

\begin{abstract}

 We compute low-dimensional $K$-groups of certain rings associated with the study of the Hermite ring conjecture. This includes a monoid ring whose low-dimensional $K$-groups were recently computed by Krishna and Sarwar in the case where the base ring is a regular ring containing the rationals. We are able to extend their result to an arbitrary regular base ring, thereby completing an answer to a question of Gubeladze. Our computation only relies on certain conveniently chosen analytic patching diagrams.
 
 These patching diagrams also allow us to investigate stably free modules appearing in a problem posed by Murthy. They make it possible to relate Murthy's problem to a stable question (about the relationship between symplectic and ordinary $K$-theory). More precisely, we show that Murthy's problem has a solution if the stable question has an affirmative answer, while a negative answer to the stable question implies that there are counterexamples to the Hermite ring conjecture.
 
 Since the module appearing in Murthy's problem has rank 2, the above mentioned argument relies on some theorems about patching with 2-by-2 matrices. These use so-called pseudoelementary 2-by-2 matrices, a notion we introduce which makes it possible to extend certain results valid for larger matrices to the case of 2-by-2 matrices (if we use pseudoelementary instead of elementary matrices). For example, we prove that the analogue of Vorst's theorem about matrices over polynomial extensions over a regular ring containing a field holds for 2-by-2 matrices.
 
\end{abstract}

\maketitle

\tableofcontents

\section{Introduction}\label{section:introduction}

Given a functor $F \colon \CRing \rightarrow \Set$ from commutative rings to sets, it is often an interesting (and difficult) problem to understand the value $F(R[t])$, where $R[t]$ denotes the polynomial ring in one variable over $R$. For example, if $F=P$ is the functor which sends $R$ to the set of isomorphism classes of finitely generated projective $R$-modules, then the Bass--Quillen conjecture asserts that the function
\[
P(R) \rightarrow P(R[t])
\]
induced by the inclusion of the constant polynomials is a bijection for all \emph{regular} rings $R$.

 Recall that a finitely generated $R$-module $Q$ is called \emph{stably free} if $Q \oplus R^n$ is free for $n$ sufficiently large. Following Lam \cite{LAM}, we call $R$ a \emph{Hermite ring} if all stably free $R$-modules are free. The Hermite ring conjecture is the conjecture that the following question has an affirmative answer.
 
 \begin{question*}[Question~1]
 If $R$ is Hermite, does it follow that $R[t]$ is Hermite?
 \end{question*}

 If we write $F \colon \CRing \rightarrow \Set$ for the functor which sends $R$ to the set of isomorphism classes of stably free modules, then the Hermite ring conjecture turns out to be equivalent to the assertion that
 \[
 F(R) \rightarrow F(R[t])
 \]
 is a bijection for \emph{all} commutative rings $R$ (because of Quillen's local-global principle). We refer the reader to \cite[\S V.3]{LAM} for several other equivalent statements.
 
 In \cite{SCHAEPPI_HERMITE}, we introduced a so called \emph{test set} $\ca{H}$ for the Hermite ring conjecture. This is a countable set with the property that the Hermite ring conjecture holds if
 \[
 F(R) \rightarrow F(R[t])
 \]
 is a bijection for all $R \in \ca{H}$. The rings in $\ca{H}$ are also rather ``nice,'' for example, they are all unique factorization domains and they are overrings of polynomial rings in the sense of \cite{BHATWADEKAR_ROY_OVERRINGS}. The simplest kind of ring in $\ca{H}$ is
 \[
 R=\mathbb{Z}[x_i,y_i] \slash (\textstyle \sum_{i=1}^n x_i y_i)
 \]
 for $n \geq 3$.
 
 The Hermite ring conjecture is a rather strong statement---it implies the Bass--Quillen conjecture, for example---and it does not have much supporting evidence. As Roitman pointed out in \cite{ROITMAN}, it would imply that all unipotent unimodular rows (that is, unimodular rows whose first entry is unipotent modulo the remaining entries) are completable to an invertible matrix. One of the simplest cases of this is Murthy's $(a,b,c)$-problem (see \cite[Problem~VIII.5.7]{LAM}):
 
 \begin{question*}[Question~2]
 Let $R=\mathbb{F}_2[a,x_1,y_1,x_2,y_2] \slash (x_1 y_1+x_2 y_2 - a^{\ell})$ for some $\ell \geq 2$. Is the kernel of the epimorphism
 \[
 (1+a,x_1,x_2) \colon R^3 \rightarrow R
 \]
 a free module?
 \end{question*}

 Both questions naturally lead to the study of overrings of polynomial rings of the form $A[x,y] \slash (xy-\alpha)$ for some $\alpha \in A$. A natural first step is the study of the $K$-groups of such rings. Swan has shown that, in the case $A=k[t_1, \ldots, t_n]$, we have $NK_j\bigl(A[x,y]\slash (xy-\alpha)\bigr) \neq 0$ if $NK_j(A \slash \alpha) \neq 0$ (see \cite[Theorem~3.6]{WEIBEL_ANALYTIC}). The following theorem gives a partial converse to this (for $j=1$).

\begin{thm}\label{thm:converse_to_Swan's_thm}
 Let $k$ be a field and let $\alpha \in k[t_1, \ldots, t_n]$ be a homogeneous polynomial of degree $\geq 2$. Assume that for all field extensions $k \subseteq k^{\prime}$, the algebra $k^{\prime}[t_1, \ldots, t_n] \slash \alpha$ is a $K_1$-regular domain. Then for all field extensions $k \subseteq k^{\prime}$, the algebra
 \[
 A \defl k^{\prime}[t_1, \ldots, t_n, x,y] \slash (xy-\alpha)
 \]
 is a $K_1$-regular unique factorization domain.
 \end{thm}

 Here a ring is called $F$-regular for some functor $F$ if $F(R) \rightarrow F(R[t_1, \ldots, t_n])$ (induced by the inclusion of the constant polynomials) is an isomorphism for all $n \in \mathbb{N}$. For functors taking values in abelian groups (or even just pointed sets), the functor $NF$ is given by the kernel $NF(R)$ of $F(R[t]) \rightarrow F(R)$ induced by evaluation at $0$. If $NF(R) \neq 0$, then it is clear that $R$ is not $F$-regular.
 
 Similar arguments used in the proof of the above theorem can be used to establish the following result.

 \begin{thm}\label{thm:K_theory_of_simple_Hermite_conjecture_rings}
 If $R$ is a regular ring and $n \geq 2$, then the ring $R[x_i,y_i] \slash ( \sum_{i=1}^n x_i y_i )$ is $K_1$-regular. In particular, the induced morphism
 \[
 K_1(R) \rightarrow K_1\bigl( R[x_i,y_i] \slash ( \textstyle\sum_{i=1}^n x_i y_i ) \bigr)
 \]
 is an isomorphism.
 \end{thm}

 Since $K_1$-regularity implies $K_0$-regularity, it follows in particular that we also have $K_0\bigl(\mathbb{Z}[x_i,y_i] \slash (\sum_{i=1}^n x_i y_i)\bigr) \cong \mathbb{Z}$. Thus the Hermite ring conjecture implies that \emph{all} finitely generated projective modules over $\mathbb{Z}[x_i,y_i] \slash (\sum_{i=1}^n x_i y_i)$ are free.
 
 The above result was first shown by Krishna and Sarwar \cite{KRISHNA_SARWAR} in the special case $n=2$ and when $R$ is a $\mathbb{Q}$-algebra. We note that, in this special case, the higher $K$-groups are also computed in \cite{KRISHNA_SARWAR} for $R=k$ a field that is algebraic over $\mathbb{Q}$; this is beyond the reach of our methods.
 
 Since $R[x_1,y_1,x_2,y_2] \slash (x_1 y_1+x_2 y_2)$ is isomorphic to the monoid algebra
 \[
 R[ux,uy,vx,vy] \subseteq R[u,v,x,y] \smash{\rlap{,}}
 \]
 we can now give a complete answer to a question of Gubeladze \cite[p.~170]{GUBELADZE_NONTRIVIALITY}: there exists a non-free, finitely generated, cancellative, and torsion free monoid $M$ such that $K_1(R) \rightarrow K_1(R[M])$ is an isomorphism for all regular rings $R$.

 The proof of these two results relies on a careful analysis of (analytic and Milnor) patching diagrams, so it in particular never leaves the realm of affine schemes (in contrast to the methods employed in \cite{KRISHNA_SARWAR}).

 Since there are many functors besides $K$-groups which are compatible with patching diagrams---for example, the functor $P$ sending $R$ to the set of isomorphism classes of finitely generated projective modules---these patching arguments can also be applied in other contexts. To illustrate this, we use the same patching diagrams appearing in the proofs of Theorems~\ref{thm:converse_to_Swan's_thm} and \ref{thm:K_theory_of_simple_Hermite_conjecture_rings} in order to study stably free modules over the rings
 \[
 k[a,x_1,y_1,x_2,y_2] \slash (x_1 y_1 + x_2 y_2 - a^{\ell})
 \]
 where $k$ is any field and $\ell \geq 2$. This leads to a question that lies between Questions~1 and 2. Recall that the first symplectic $K$-group of a commutative ring $R$ is given by
 \[
 K_1 \mathrm{Sp}(R)=\mathrm{Sp}(R) \slash \mathrm{Ep}(R) =\mathrm{Sp}(R) \slash [\mathrm{Sp}(R),\mathrm{Sp}(R)] \smash{\rlap{,}}
 \]
 where $\mathrm{Sp}(R)=\mathrm{colim}_r \mathrm{Sp}_{2r}(R)$ is the stabilization of the symplectic groups and $[\mathrm{Sp}(R),\mathrm{Sp}(R)]$ denotes the commutator subgroup (which coincides with the elementary symplectic group $\mathrm{Ep}(R)$). This group comes with a natural forgetful homomorphism $K_1 \mathrm{Sp}(R) \rightarrow K_1(R)$.
 
\begin{question*}[Question~3]
Is the homomorphism
\[
K_1 \mathrm{Sp} \big(k[a,x,y] \slash (xy-a^{\ell})\bigr) \rightarrow K_1 \big(k[a,x,y] \slash (xy-a^{\ell})\bigr)
\]
injective for all fields $k$ and all natural numbers $\ell \geq 2$?
\end{question*}

 The following theorem shows how these three questions are related.
 
  \begin{thm}\label{thm:intermediate_stable_question}
 If the Hermite ring conjecture holds, then the kernel of
 \[
 K_1 \mathrm{Sp}\bigl(k[a,x,y] \slash (xy-a^{\ell}) \bigr) \rightarrow  K_1 \bigl(k[a,x,y] \slash (xy-a^{\ell}) \bigr)
 \]
 vanishes for all fields $k$ and all natural numbers $\ell \geq 2$.
 
 On the other hand, if this kernel vanishes for $k$ and $k \langle t \rangle$ as ground field, then all stably free modules over the ring
 \[
 k[a,x_1, y_1, x_2, y_2] \slash (x_1 y_1 + x_2 y_2 - a^{\ell})
 \]
 are free for all $\ell \geq 2$.
 \end{thm}
 
 Here $k \langle t \rangle=k(t)$ denotes the field of fractions of $k[t]$. The theorem shows in particular that Question~2 has an affirmative answer if Question~3 does. On the other hand, if Question~3 has a negative answer for some field $k$, then there are counterexamples to the Hermite ring conjecture (Question~1). In the companion article \cite{SCHAEPPI_COUNTEREXAMPLE} we will exploit this connection to construct a counterexample to the Hermite ring conjecture by showing that Question~3 has a negative answer.
 
 We will see in the course of the proof of Theorem~\ref{thm:intermediate_stable_question} that all stably free modules over the ring $ k[a,x_1,y_1,x_2,y_2] \slash (x_1 y_1 + x_2 y_2 - a^{\ell})$ are free if $\mathrm{char}(k) \neq 2$. Thus Question~3 is most interesting in the case where the field has characteristic $2$. We conclude with the following theorem, which shows that Question~3 has an affirmative answer for \emph{perfect} fields of characteristic $2$.
 
 \begin{thm}\label{thm:injectivity_for_perfect_field}
 If $k$ is a perfect field of characteristic $2$ and $\ell \geq 2$, then
 \[
 K_1 \mathrm{Sp}\bigl(k[a,x,y] \slash (xy-a^{\ell}) \bigr) \rightarrow  K_1 \bigl(k[a,x,y] \slash (xy-a^{\ell}) \bigr)
 \]
 is injective
\end{thm}

 While the proofs of Theorems~\ref{thm:converse_to_Swan's_thm} and \ref{thm:K_theory_of_simple_Hermite_conjecture_rings} mostly rely on classical patching techniques, the proofs of Theorem~\ref{thm:intermediate_stable_question} and \ref{thm:injectivity_for_perfect_field} require some additional theory. Namely, patching data for rank $2$ projective modules is given by invertible $2 \times 2$-matrices. It is well-known that many results that hold for elementary $r \times r$-matrices, $r \geq 3$, do not extend to the case $r=2$.
 
 To circumvent this problem, we introduce the notion of \emph{pseudoelementary} $2 \times 2$-matrices: a matrix $\sigma \in \mathrm{SL}_2(R)$ is called pseudoelementary if there exists a ring homomorphism
 \[
 \varphi \colon \mathbb{Z}[t_1, \ldots, t_n] \rightarrow R
 \]
 and a matrix $\sigma^{\prime} \in \mathrm{SL}_2(\mathbb{Z}[t_i])$ such that $\sigma=\varphi(\sigma^{\prime})$. We denote the group of pseudoelementary matrices by $\widetilde{\mathrm{E}}_2(R)$. With this notion, we can extend many of the results for $r \times r$-matrices, $r \geq 3$, to the case $r=2$. The following theorem, for example, extends Vorst's theorem (see \cite{VORST_GLN}) to the case $r=2$.

  \begin{thm}\label{thm:Vorsts_Theorem_for_2x2_matrices}
 Let $R$ be a commutative ring whose local rings are unramified regular local rings. For each $\sigma(t_1, \ldots, t_n) \in \mathrm{SL}_2(R[t_1, \ldots, t_n])$, there exists a pseudoelementary matrix $\varepsilon(t_1, \ldots, t_n) \in \widetilde{\mathrm{E}}_2(R[t_1, \ldots, t_n])$ such that the equation
 \[
 \sigma(t_1, \ldots, t_n)=\sigma(0, \ldots, 0) \cdot \varepsilon(t_1, \ldots, t_n)
 \]
 holds.
 \end{thm}

 We now give an overview of the contents of the various sections and we indicate briefly how the above results are proved. Of central importance is the notion of an analytic isomorphism $\varphi \colon A \rightarrow B$ along a multiplicatively closed set $S \subseteq A$. This is a ring homomorphism such that the induced homomorphism $A \slash s \rightarrow B \slash \varphi(s)$ is an isomorphism for all $s \in S$ and the diagram
 \[
 \xymatrix{A \ar[r] \ar[d]_{\varphi} & A_S \ar[d]^{\varphi_S} \\ B \ar[r] & B_{\varphi(S)}}
 \]
 is a pullback diagram. We note that sometimes more stringent assumptions are made (for example, that $\varphi$ is injective or that $S$ and $\varphi(S)$ both consist of nonzerodivisors).
 
 In \S \ref{section:analytic}, we show that many functors $F \colon \CRing \rightarrow \Set$ that we care about satisfy \emph{weak analytic excision}: they send diagrams as above to \emph{weak} pullback diagrams of sets. This section is largely a collection of well-known results in the literature, though these are often stated in the more restrictive cases mentioned above. We give detailed proofs (or in some cases pointers to the literature) which show that these restrictions are usually unnecessary.
 
 In \S \ref{section:excision} we turn to the study of the value $F(A[t_1, \ldots, t_n])$ of $F$ on a polynomial extension and in particular the class of $F$-regular rings. There are three important principles a functor $F \colon \CRing \rightarrow \Set$ can satisfy. Quillen's local-global principle $(\mathrm{Q})$ for $F$ asserts that $\sigma \in F(A[t])$ is extended from $F(A)$ if the localization $\sigma_{\mathfrak{m}} \in F(A_{\mathfrak{m}}[t])$ is extended for all maximal ideals $\mathfrak{m} \subseteq A$. For $F=P$, this first appeared in \cite{QUILLEN}. Roitman's localization principle $(\mathrm{R})$ states that $F$-regular rings are closed under localization (this is proved for $F=P$ in \cite{ROITMAN_POLYNOMIAL}). Finally, Lindel \cite{LINDEL} and Popescu \cite{POPESCU_DESINGULARIZATION} have shown that $P$-regularity of fields (respectively of discrete valuation rings) implies $P$-regularity for all regular rings containing a field (respectively all unramified regular rings).
 
 In \S \ref{section:excision}, we show that these three properties are largely consequences of weak analytic excision. We write $\Set_{\ast}$ for the category of pointed sets and we call $A$ \emph{$F$-contractible} if $F(A) \cong \ast$. For the definition of a natural transitive group action we refer to Definition~\ref{dfn:F_contractible_and_transitive_group_action} and simply note that any group-valued functor has such an action. Finally, we recall that a functor is called \emph{finitary} if it commutes with filtered colimits.
 
 \begin{thm}\label{thm:excision_implies_three_principles}
Let $F \colon \CRing \rightarrow \Set_{\ast}$ be a finitary functor which satisfies weak analytic excision. Then $F$ satisfies Quillen's principle $(\mathrm{Q})$. Assume that one of the following conditions holds:
 \begin{enumerate}
 \item[(A)] all local rings are $F$-contractible;
 \item[(B)] the functor $F$ admits a natural transitive group action.
 \end{enumerate}
 Then $F$ also satisfies Roitman's principle $(\mathrm{R})$. Moreover, the following hold:
 \begin{enumerate}
 \item[(i)] if all fields are $F$ regular, then all regular rings containing a field are $F$-regular;
 \item[(ii)] if, in addition, all discrete valuation rings are $F$-regular, then all unramified regular local rings are $F$-regular.
\end{enumerate} 
\end{thm}

 Here the first two implications are built on a proof of Vorst in the case where $F=K_i$ is given by the higher $K$-group functors (see \cite{VORST_POLYNOMIAL}). The novelty here is that we extend Vorst's result to functors which are not necessarily group-valued, thereby recovering Quillen's original result as a special case. The final part of the above theorem is simply an axiomatic treatment of Lindel's and Popescu's results based on the formalism of {\'e}tale neighbourhoods.
 
 There is a fourth important principle that $F$ can satisfy, the \emph{monic inversion principle} $(\mathrm{H})$ due to Horrocks in the case $F=P$ (see \cite{HORROCKS}). In \S \ref{section:monic}, we provide pointers to the literature which show that many of the functors discussed in \S \ref{section:analytic} satisfy this principle. We also discuss the strong $F$-extension property for a ring $A$ (see Definition~\ref{dfn:strong_extension_property}) which---in the presence of the monic inversion principle---not only implies that $A$ is $F$-regular, but also a corresponding extension property for Laurent polynomial rings (see Theorem~\ref{thm:strong_F_extension_for_Laurent}).
 
 In \S \ref{section:henselian_pairs}, we discuss some general facts about henselian pairs. The purpose of this section is twofold: first, we prove a generalized form of Hensel's lemma which only requires one of the polynomials to be monic (this is well-known for henselian local rings, see \cite[\href{https://stacks.math.columbia.edu/tag/04GG}{Lemma 04GG}]{stacks-project}, but does not appear to be in the literature for general henselian pairs). This result is later used to analyze various patching diagrams associated to overrings of polynomial rings.
 
 The second purpose of \S \ref{section:henselian_pairs} is a detailed description of {\'e}tale neighbourhoods and henselizations in terms of some simple classes of homomorphisms (so called \emph{standard} and \emph{basic} Nisnevich homomorphisms). These are important for the study of weak excision properties of pseudoelementary matrices (more precisely, of the functor sending $R$ to $\mathrm{SL}_2(R) \slash \widetilde{\mathrm{E}}_2(R)$) in \S \ref{section:pseudoelementary}.
 
 In \S \ref{section:overrings}, we study patching diagrams associated to rings of the form
 \[
 B=A[x,y] \slash (xy-\alpha)
 \]
 for some $\alpha \in A$. We provide some methods to check that a functor takes trivial values on the top right and bottom left corner of the analytic patching diagram
 \[
 \xymatrix{B \ar[r] \ar[d] & A[\frac{1}{\alpha}][x,x^{-1}] \ar[d] \\ A^h_{(\alpha)}[x,y]\slash (xy-\alpha) \ar[r] & A^h_{(\alpha)}[\frac{1}{\alpha}][x,x^{-1}] }
 \]
 arising from the henselization $A \rightarrow A^h_{(\alpha)}$ at $(\alpha)$. We then use these general results to prove Theorems~\ref{thm:converse_to_Swan's_thm} and \ref{thm:K_theory_of_simple_Hermite_conjecture_rings}.
 
  In \S \ref{section:pseudoelementary}, we study the group $\widetilde{\mathrm{E}}_2(R) \subseteq \mathrm{SL}_2(R)$ of pseudoelementary matrices and the associated functor $\widetilde{SK}_{1,2}(R) \defl \mathrm{SL}_2(R) \slash \widetilde{\mathrm{E}}_2(R)$ (taking values in pointed sets). We deduce the three principles $(\mathrm{Q})$, $(\mathrm{R})$, and $(\mathrm{H})$ for $\widetilde{SK}_{1,2}$ from the corresponding facts for the functor
  \[
  P_2 \colon \CRing \rightarrow \Set_{\ast}
  \]
 which sends $R$ to the set of isomorphism classes of rank $2$ projective modules, by using $\sigma \in \mathrm{SL}_2(R)$ as patching data for such modules.
 
 We show that $\widetilde{SK}_{1,2}$ satisfies weak Zariski excision (see Theorem~\ref{thm:pseudoelementary_weak_Zariski_excision}), but it is not clear if it satisfies weak analytic excision. We show that weak excision does hold in a number of cases, namely for patching diagrams which admit ``good lifts.'' This includes the patching diagrams appearing in the axiomatic form of Lindel's argument (see Theorem~\ref{thm:excision_implies_three_principles}), which makes it possible to extend Vorst's theorem to the case of $2 \times 2$-matrices, that is, to prove Theorem~\ref{thm:Vorsts_Theorem_for_2x2_matrices}.
 
 The theory of pseudoelementary matrices can also be used to prove the following fact about projective modules. Namely, if $R$ is a principal ideal domain such that $\mathrm{SL}_2(R)=\mathrm{E}_2(R)$ (for example, a euclidean domain), then all projective modules over the ring
 \[
 R[\mathrm{SL}_2]=R[x,y,u,v] \slash (xv-yu-1)
 \]
 are free (see Corollary~\ref{cor:projectives_free_over_coordinate_ring_of_SL2}). If $R$ is a field, this fact is due to Murthy (see \cite[Theorem~6.2]{SWAN_BUNDLES}).
 
 In \S \ref{section:symplectic}, we study the rings
 \[
 R=k[a,x_1,y_1,x_2,y_2] \slash (x_1 y_1 + x_2 y_2 - a^{\ell})
 \]
 for $k$ a field and $\ell \geq 2$ a natural number. Murthy's $(a,b,c)$-problem asks if the kernel of
 \[
 (1+a,x_1,x_2) \colon R^3 \rightarrow R
 \]
 is a free module if $k=\mathbb{F}_2$ (see Question~2). It is well-known that this kernel is free if $\mathrm{char}(k) \neq 2$ (see for example \cite[Proposition~6]{ROITMAN}). We show that, more generally, \emph{all} stably free modules over $R$ are free if $\mathrm{char}(k) \neq 2$ (see Proposition~\ref{prop:stably_free_trivial_if_char_neq_2}).
 
 In the remaining case $\mathrm{char}(k)=2$, we establish a link between the above strengthening of Question~2 and the question whether
 \[
 K_1 \mathrm{Sp}\bigl(k[a,x,y] \slash (xy-a^{\ell})\bigr) \rightarrow K_1\bigl(k[a,x,y] \slash (xy-a^{\ell})\bigr)
 \]
 is injective (see Question~3 and Theorem~\ref{thm:intermediate_stable_question}). This connection is obtained by studying if certain $1$-stably elementary $2 \times 2$-matrices (that is, matrices $\sigma \in \mathrm{SL}_2(R)$ such that $\bigl(\begin{smallmatrix} \sigma & 0 \\ 0 & 1 \end{smallmatrix} \bigr)$ is elementary) are pseudoelementary. The link with symplectic $K$-theory is established through a theorem of Vaserstein (see \cite[Th{\'e}or{\`e}me~4]{BASS_LIBERATION}).
 
 The second half of \S \ref{section:symplectic} is taken up by the proof that the above homomorphism is injective if $k$ is a perfect field of characteristic $2$ (see Theorem~\ref{thm:injectivity_for_perfect_field}). To get the relevant long exact sequences that make this work, we use the theory of higher Grothendieck--Witt groups developed in \cite{CALMES_ET_AL_I, CALMES_ET_AL_II}, together with the relationship between these and classical symplectic $K$-groups established in \cite{HEBESTREIT_STEIMLE}. Besides this, the proof of Theorem~\ref{thm:injectivity_for_perfect_field} relies on a computation in \cite{MORITA_REHMANN}.
 
 We conclude with some notation that is used throughout the article. For an element $a \in A$ of a commutative ring, we write $A \slash a$ for the quotient $A \slash Aa$ of $A$ modulo the ideal generated by $a$. For $S \subseteq A$ multiplicative, we write $A_S$ for the localization of $A$ at $S$ and we denote the localization homomorphism $A \rightarrow A_S$, $a \mapsto \frac{a}{1}$ by $\lambda_S$. If $S=\{1,f,f^2, \ldots\}$ is generated by a single element $f$, we write $A_f$ for $A_S$. We sometimes write $A[t_1^{\pm}, \ldots, t_n^{\pm}]$ for the Laurent polynomial ring
 \[
A[t_1,t_1^{-1}, \ldots, t_n,t_n^{-1}] 
 \]
 in $n$ variables.
 
 Given a commutative ring $R$, we write $\CAlg_R$ for the category of commutative $R$-algebras (and $\CRing$ for the category $\CAlg_{\mathbb{Z}}$ of commutative rings). We write $\Set$, $\Set_{\ast}$, $\Grp$, $\Ab$ for the categories of sets, pointed sets, groups, and abelian groups. Finally, we write $\Proj(R)$ for the category of finitely generated projective $R$-modules and $\Mod_R$ for the category of all $R$-modules.
 
  \subsection*{Acknowledgments}
 I am very grateful to Niko Naumann for carefully reading certain technical aspects of this article and for his suggestions for improvement.
\section{Analytic isomorphisms}\label{section:analytic}

 We say that a functor $F \colon \CAlg_R \rightarrow \Set$ satisfies \emph{weak excision} for a class of commutative squares $\ca{D}$ in $\CAlg_R$ if for each commutative diagram on the left
 \[
\vcenter{\xymatrix{A \ar[r] \ar[d] & B \ar[d] \\ C \ar[r] & D }} \quad \quad \quad
\vcenter{\xymatrix{FA \ar[r] \ar[d] & FB \ar[d] \\ FC \ar[r] & FD }}
 \]
 in $\ca{D}$, the diagram on the right above is a weak pullback diagram. Throughout the later sections we are mostly interested in the case where $\ca{D}$ is the class of analytic patching diagrams (defined below). In this section, we show that many functors of interest satisfy weak excision for this class of diagrams. In most cases, this is well-known. However, the notion of analytic patching diagram we consider here is slightly more general than the classical one (where certain elements are assumed to be nonzerodivisors). The aim of this section is to show that weak excision also holds for this weaker notion of analytic patching diagrams in the examples we care about.
 
 \begin{dfn}\label{dfn:analytic_patching_diagram}
Let $\varphi \colon A \rightarrow B$ be a homomorphism of $R$-algebras and let $S \subseteq A$ be a multiplicative subset. We say that $\varphi$ is an \emph{analytic isomorphism along $S$} if the induced homomorphism $\bar{\varphi} \colon A \slash s \rightarrow B \slash \varphi(s)$ is an isomorphism for all $s \in S$ and the diagram
\[
\vcenter{\xymatrix{A \ar[r]^-{\lambda_S} \ar[d]_{\varphi} & A_S \ar[d]^{\varphi_S} \\ B \ar[r]_-{\lambda_{\varphi(S)}} & B_{\varphi(S)} }}
\]
 is a pullback diagram of $R$-algebras. In this case, we call the above diagram an \emph{analytic patching diagram}. We say that $F \colon \CAlg_R \rightarrow \Set$ satisfies \emph{weak analytic excision} if it sends all analytic patching diagrams to weak pullback diagrams.
 \end{dfn}
 
 The following lemma allows us to check the condition on $\bar{\varphi} \colon A \slash s \rightarrow B \slash \varphi(s)$ merely for the generators of $S$. This is of course particularly useful is $S$ is generated by a single element, that is, $S =\{1,f,f^2, \ldots, f^n, \ldots\}$ for some $f \in A$.
 
 \begin{lemma}\label{lemma:analytic_iso_generators}
 Let $\varphi \colon A \rightarrow B$ be a homomorphism of $R$-algebras, $S \subseteq A$  a multiplicative set, and let $f, g \in S$. Then the following hold:
\begin{enumerate}
 \item[(1)] If $\bar{\varphi} \colon A \slash f \rightarrow B \slash \varphi(f)$ and $\bar{\varphi} \colon A \slash g \rightarrow B \slash \varphi(g)$ are surjective, then the induced homomorphism $\bar{\varphi} \colon A \slash fg \rightarrow B \slash \varphi(fg)$ is also surjective;
 
 \item[(2)] Let $I=\ker(\lambda_S \colon A \rightarrow A_S )$ and let $J=\ker(\lambda_{\varphi(S)} \colon B \rightarrow B_{\varphi(S)})$ and assume that
 \[
 \varphi \vert_{I} \colon I \rightarrow J
 \]
 is a bijection. If $\bar{\varphi} \colon A \slash f \rightarrow B \slash \varphi(f)$ and $\bar{\varphi} \colon A \slash g \rightarrow B \slash \varphi(g)$ are injective, then $\bar{\varphi} \colon A \slash fg \rightarrow B \slash \varphi(fg)$ is injective.
\end{enumerate}

In particular, if $G \subseteq S$ is a generating set, $\bar{\varphi} \colon A \slash g \rightarrow B \slash \varphi(g)$ is an isomorphism for all $g \in G$, and the diagram
\[
\vcenter{\xymatrix{A \ar[r]^-{\lambda_S} \ar[d]_{\varphi} & A_S \ar[d]^{\varphi_S} \\ B \ar[r]_-{\lambda_{\varphi(S)}} & B_{\varphi(S)} }}
\]
 is a pullback diagram of $R$-algebras, then $\varphi$ is an analytic isomorphism along $S$.
 \end{lemma}

\begin{proof}
 First note that the final claim follows from (1) and (2): the fact that the diagram is a pullback implies that $\varphi$ induces a bijection between the kernels of the vertical morphisms, so the premise of Claim~(2) holds.
 
 Claim~(1) follows directly from the snake lemma. In more detail, multiplication by $f$ (respectively $\varphi(f)$) yields right exact sequences
 \[
 \xymatrix{ A \slash g \ar[r] & A \slash fg \ar[r] & A \slash f \ar[r] & 0}
 \]
 and
 \[
 \xymatrix{B  \slash \varphi(g) \ar[r] & B \slash \varphi(fg) \ar[r] & B \slash \varphi(f) \ar[r] & 0}
 \]
 of $R$-modules. Let $K$ denote the image of the morphism $B\slash \varphi(g) \rightarrow B \slash \varphi(fg)$ given by multiplication with $\varphi(f)$ and write 
  \[
  \pi \colon B \slash \varphi(g) \rightarrow K
  \]
   for the resulting surjection. The snake lemma applied to the diagram
  \[
  \xymatrix{& A \slash g \ar[r] \ar[d]_{\pi \bar{\varphi}} & A \slash fg \ar[d]_{\bar{\varphi}} \ar[r] & A \slash f \ar[d]^{\bar{\varphi}} \ar[r] & 0 \\
  0 \ar[r] & K \ar[r] & B \slash \varphi(fg) \ar[r] & B \slash \varphi(f) \ar[r] & 0}
  \]
 of $R$-modules shows that Claim~(1) holds.
 
 To see Claim~(2), we can adapt the argument given by Bhatwadekar \cite[Proposition~2.1]{BHATWADEKAR_ANALYTIC} in the case of a single generator. Namely, assume that $[a] \in \ker(\bar{\varphi})$, that is, there exists a $b \in B$ such that $\varphi(a)=\varphi(fg) b$. Since $A \slash f \rightarrow B \slash \varphi(f)$ is injective, it follows that there exists an element $a_0 \in A$ such that $a=f \cdot a_0$. From the fact that $\varphi(f) \bigl( \varphi(g)b - \varphi(a_0) \bigr)=0$ it follows that $\varphi(g)b - \varphi(a_0) \in J$, so there exists an element $c \in I$ with $\varphi(c)=\varphi(g)b - \varphi(a_0)$. Moreover, $fc \in I$ and $\varphi(fc)=0$, hence the assumption that $\varphi \vert_I$ is bijective implies that $fc=0$.
 
 Since $\varphi(c+a_0)=\varphi(g)b$ and $A \slash g \rightarrow B \slash \varphi(g)$ is injective, there exists an element $a_1 \in A$ such that $c + a_0=ga_1$. Thus
 \[
 a=fa_0=f(c +a_0)=(fg)a_1 \smash{\rlap{,}}
 \]
 so $[a]=0$ in $A \slash fg$, as claimed.
\end{proof}

 As noted in \cite[\S 2]{ROY}, the pullback condition is automatically satisfied if the multiplicative sets $S$ and $\varphi(S)$ consist of nonzerodivisors.

 \begin{lemma}\label{lemma:analytic_iso_nonzerodivisors}
  Let $\varphi \colon A \rightarrow B$ be a homomorphism of $R$-algebras and let $S \subseteq A$ be a multiplicative set. Suppose that all elements of $S$ are nonzerodivisors in $A$ and all elements of $\varphi(S)$ are nonzerodivisors in $B$. If
 \[
 \bar{\varphi} \colon A \slash g \rightarrow B \slash \varphi(g)
 \] 
 is an isomorphism for all elements $g$ in a generating set of $S$, then $\varphi \colon A \rightarrow B$ is an analytic isomorphism along $S$.
 \end{lemma}
 
 \begin{proof}
  In the notation of Lemma~\ref{lemma:analytic_iso_generators}, we have $I=0$ and $J=0$, so the assumption, combined with Claims~(1) and (2) of Lemma~\ref{lemma:analytic_iso_generators}, implies that  $\bar{\varphi} \colon A \slash s \rightarrow B \slash \varphi(s)$ is an isomorphism for all $s \in S$.
  
  It only remains to check that the diagram
  \[
  \vcenter{\xymatrix{A \ar[r]^-{\lambda_S} \ar[d]_{\varphi} & A_S \ar[d]^{\varphi_S} \\ B \ar[r]_-{\lambda_{\varphi(S)}} & B_{\varphi(S)} }}
  \]
  is a pullback diagram. This follows readily from the injectivity of $\lambda_S$ and $\lambda_{\varphi(S)}$. Indeed, if $b \slash 1=\varphi(a) \slash \varphi(s)$ in $B_{\varphi(S)}$, then $\varphi(s) b=\varphi(a)$ holds in $B$. The fact that $\bar{\varphi}$ is an isomorphism implies that $a=sa_0$ for some $a_0 \in A$. It follows that $b=\varphi(a_0)$ and $a \slash s=a_0 \slash 1=\lambda_S(a_0)$, which establishes the existence part of a pullback diagram. Uniqueness is immediate from the fact that $\lambda_S$ is injective.
 \end{proof}
 
  The following lemma shows that if $\varphi \colon A \rightarrow B$ is an analytic isomorphism along $S$, then $\varphi$ is also an analytic isomorphism along the saturation $\bar{S}$ of $S$. Here $\bar{S}$ denotes the set of all $u \in A$ such that there exists a $v \in A$ with $uv \in S$. This coincides with $\lambda_S^{-1}(A_S^{\times})$, the set of all $u \in A$ such that $u \slash 1$ is a unit in $A_S$.
  
  \begin{lemma}\label{lemma:analytic_iso_saturation}
  Let $\varphi \colon A \rightarrow B$ be a homomorphism of $R$-algebras and let $u,v \in A$. Let $s=uv$ and assume that $\bar{\varphi} \colon A \slash s \rightarrow B \slash \varphi(s)$ is an isomorphism. Then the homomorphisms $\bar{\varphi} \colon A \slash u \rightarrow B \slash \varphi(u)$ and $\bar{\varphi} \colon A \slash v \rightarrow B \slash \varphi(v)$ are isomorphisms as well. 
  \end{lemma}
  
  \begin{proof}
  As in the proof of Lemma~\ref{lemma:analytic_iso_generators}, let $K=\ker\bigl(B \slash \varphi(s) \rightarrow B \slash \varphi(v) \bigr)$ and let $\pi \colon B \slash \varphi(u) \rightarrow K$ be the surjection induced by $\bar{\varphi}$. The commutative diagram
  \[
    \xymatrix{& A \slash u \ar[r] \ar[d]_{\pi \bar{\varphi}} & A \slash s \ar[d]^{\cong}_{\bar{\varphi}} \ar[r] & A \slash v \ar[d]^{\bar{\varphi}} \ar[r] & 0 \\
  0 \ar[r] & K \ar[r] & B \slash \varphi(s) \ar[r] & B \slash \varphi(v) \ar[r] & 0}
  \]
  implies that $\bar{\varphi} \colon A \slash v \rightarrow B \slash \varphi(v)$ is surjective. Reversing the roles of $u$ and $v$ we find that $\bar{\varphi} \colon A \slash u \rightarrow B \slash \varphi(u)$ is surjective as well. Thus the left vertical morphism in the above diagram of $R$-modules is surjective, so the snake lemma implies that $\bar{\varphi} \colon A \slash v \rightarrow B \slash \varphi(v)$ is injective, hence bijective. Reversing the roles of $u$ and $v$ again we get the conclusion. 
  \end{proof}

 We next discuss some examples of analytic patching diagrams which will be used in later sections. In \S \ref{section:excision}, we will consider three classes of analytic patching diagrams. The first two are closely related and relevant for the local-global principle of Quillen (respectively its converse due to Roitman). The third type of patching diagram are affine Nisnevich squares, which we will discuss in \S \ref{section:henselian_pairs}.
  
  \begin{notation}\label{notation:Vorst_patching_diagrams}
 For any commutative $R$-algebra $A$, we write $A[x]$ for the $A$-algebra of polynomials in the variable $x$ and we write $\pi \colon A[x] \rightarrow A$ for the homomorphism of $A$-algebras sending $x$ to $0$, that is, $\pi$ is given by evaluation in $0$. We write $\iota \colon A \rightarrow A[x]$ for the inclusion of $A$ as the constant polynomials in $A[x]$.
 
 For any element $f \in A$, we write $A \pb{A_f} A_f[x]$ for the $R$-algebra defined by the pullback diagram
 \[
 \vcenter{\xymatrix{
  A \pb{A_f} A_f[x] \ar[r]^-{\lambda^{\prime}} \ar[d]_{\pi^{\prime}} & A_f[x] \ar[d]^{\pi} \\
  A \ar[r]_-{\lambda_f} & A_f
 }} \quad \quad \quad
\vcenter{\xymatrix{
 A \ar[d]_{\iota^{\prime}}  \ar[r]^-{\lambda_f} & A_f \ar[d]^{\iota} \\ 
 A \pb{A_f} A_f[x] \ar[r]_-{\lambda^{\prime}} & A_f[x]
}} 
 \] 
 on the left. We write $\iota^{\prime} \colon A \rightarrow A \pb{A_f} A_f[x]$ for the unique $R$-algebra homomorphism such that $\pi^{\prime} \iota^{\prime}=\id_A$ and $\lambda^{\prime} \iota^{\prime}=\iota \lambda_f$. In particular, the diagram on the right above is commutative by construction.
  \end{notation}
  
 \begin{prop}\label{prop:Vorst_patching_diagrams}
 For any commutative $R$-algebra $A$ and any $f \in A$, the homomorphisms $\pi^{\prime} \colon A \pb{A_f} A_f[x] \rightarrow A$ and $\iota^{\prime} \colon A \rightarrow A \pb{A_f} A_f[x]$ are analytic isomorphisms along $(f,f \slash 1)$ respectively along $f$. The morphism $\lambda^{\prime} \colon A \pb{A_f} A_f[x] \rightarrow A_f[x]$ is (up to isomorphism) the localization at $(f,f \slash 1)$, so both diagrams appearing in Notation~\ref{notation:Vorst_patching_diagrams} are analytic patching diagrams. Moreover, the homomorphism $\iota^{\prime} \colon A \rightarrow A \pb{A_f} A_f[x]$ is always flat.
 \end{prop}
 
 \begin{proof}
  Note that $A \pb{A_f} A_f[x]$ is an $A$-algebra via $\iota^{\prime}$ and that $\iota^{\prime}(f)=(f,f \slash 1)$. If we localize the pullback diagram
  \[
 \vcenter{\xymatrix{
  A \pb{A_f} A_f[x] \ar[r]^-{\lambda^{\prime}} \ar[d]_{\pi^{\prime}} & A_f[x] \ar[d]^{\pi} \\
  A \ar[r]_-{\lambda_f} & A_f
 }}  \]
  at $\iota^{\prime}(f)$, we obtain a pullback diagram whose bottom horizontal morphism is an isomorphism. This shows that, up to isomorphism, $\lambda^{\prime}$ is the localization at $\iota^{\prime}(f)$, as claimed.
  
  If we paste the right diagram of Notation~\ref{notation:Vorst_patching_diagrams} on top of the left diagram, we obtain a pullback square whose vertical morphisms are identities. From the cancellation law for pullbacks it follows that both diagrams of Notation~\ref{notation:Vorst_patching_diagrams} are pullback diagrams.
  
  It only remains to show that $\iota^{\prime}$ and $\pi^{\prime}$ induce isomorphisms modulo $f$ (respectively modulo $\iota^{\prime}(f)$). Since $\pi^{\prime} \iota^{\prime}=\id$, we have in particular $\pi^{\prime} \bigl(\iota^{\prime}(f) \bigr)=f$, so it suffices to check the claim for $\iota^{\prime}$.
  
  To see this, note that $\pi$ and $\pi^{\prime}$ have the same kernel by construction. Moreover, the short exact sequence
  \[
  \xymatrix{0 \ar[r] & \ker(\pi) \ar[r] & A \pb{A_f} A_f[x] \ar[r]^-{\pi} & A \ar[r] & 0 }
  \]
  is split exact, with splitting given by $\iota^{\prime}$. Thus $\iota^{\prime}$ is, up to isomorphism, given by the inclusion of $A$ in $A \oplus \bigoplus_{i  \geq 1} A_f$. Since $A_f=f \cdot A_f$, it follows that
  \[
  \bar{\iota^{\prime}} \colon A \slash f \rightarrow A \pb{A_f} A_f[x] \slash \iota^{\prime}(f)
  \]
  is indeed an isomorphism.
  
  From the above description of the $A$-module $A \pb{A_f} A_f[x]$ it is also clear that $\iota^{\prime}$ is flat.
 \end{proof}
 
 We now turn to some examples of functors which satisfy weak analytic excision. Let $M \colon \CAlg_R \rightarrow \Set$ be the functor which sends $A$ to the set of isomorphism classes of finitely presentable $A$-modules. The action of $M$ on morphisms $\varphi \colon A \rightarrow B$ is given by $M(\varphi) \bigl([N] \bigr)=[B \ten{A} N]$.
 
 We write $P \colon \CAlg_R \rightarrow \Set$ for the subfunctor of $M$ which sends $A$ to the set of isomorphism classes of finitely generated projective $A$-modules. Finally, we write $P_r \colon \CAlg_R \rightarrow \Set$ for the subfunctor of $P$ which sends $A$ to the set of isomorphism classes of finitely generated projective modules of constant rank $r \geq 1$. Note that $P_r$ has a natural lift to a functor $P_r \colon \CAlg_R \rightarrow \Set_{\ast}$, with base point $\ast$ of $P_r(A)$ given by the class $[A^r]$ of the free $A$-module of rank $r$.
 
 \begin{rmk}\label{rmk:M_P_Pr_finitary}
 Since a finitely presentable module can be described by a finite matrix, the functors $M$, $P$, and $P_r$ preserve filtered colimits, that is, they are finitary.
 \end{rmk}
 
 \begin{prop}\label{prop:M_weak_analytic_excision}
 The functor $M \colon \CAlg_R \rightarrow \Set$ sends the analytic patching diagram
 \[
 \vcenter{\xymatrix{A \ar[r]^-{\lambda_S} \ar[d]_{\varphi} & A_S \ar[d]^{\varphi_S} \\ B \ar[r]_-{\lambda_{\varphi(S)}} & B_{\varphi(S)} }}
 \]
 to a weak pullback diagram if $\varphi \colon A \rightarrow B$ is flat.
 \end{prop}
 
 \begin{proof}
 Since $M$ commutes with filtered colimits, we can reduce to the case $S=\{1,f,f^2,\ldots\}$ for some $f \in A$. In this case, the claim follows from the fact that the functor
 \[
 \Mod_A \rightarrow \Mod_B \pb{\Mod_{B_{\varphi(f)}}} \Mod_{A_f}
\]
 is an equivalence, where the target is the iso-comma category induced by the span given by base change along $\lambda_{\varphi(f)}$ and $\varphi_f$ (see \cite[\href{https://stacks.math.columbia.edu/tag/05ES}{Theorem 05ES}]{stacks-project}), combined with the observation that this equivalence restricts to an equivalence of the categories of finitely presentable modules (see \cite[\href{https://stacks.math.columbia.edu/tag/05EU}{Remark 05EU}]{stacks-project}).
 \end{proof}
 
 \begin{rmk}
  Bhatwadekar has shown that the assumption of flatness can be dropped in the case where both $A$ and $B$ are noetherian, see \cite[Theorem~2.4]{BHATWADEKAR_ANALYTIC}. This is not true for general $R$-algebras, see \cite[\href{https://stacks.math.columbia.edu/tag/0BNY}{Example 0BNY}]{stacks-project}.
 \end{rmk}
 
 To see that $P$ and $P_r$ satisfy weak analytic excision, we can follow the arguments given in \cite[\S 2]{ROY} and show that the natural functor
 \[
 \Proj(A) \rightarrow \Proj(B)  \pb{\Proj(B_{\varphi(S)}) } \Proj(A_S)  
 \]
 is an equivalence whenever $\varphi \colon A \rightarrow B$ is an analytic isomorphism along $S$.
 
  This, in turn, is based on two ingredients: an observation of Milnor that every object in the target category is a direct summand of an object of the form $(B^n, A_S^n, \alpha)$ for some $n \in \mathbb{N}$ (see \cite[\S 2]{MILNOR}) and on a lemma of Vorst which shows that elementary $n \times n$-matrices over $B_{\varphi(S)}$ can be factored into a product of elementary matrices over $B$ and $A_S$ as long as $n \geq 3$ \cite[Lemma~2.4]{VORST_GLN}. One can check that these arguments also work for the slightly more general notion of analytic isomorphism given in Definition~\ref{dfn:analytic_patching_diagram}. We will nevertheless give detailed proofs of these two facts in Lemma~\ref{lemma:Vorsts_lemma} and Proposition~\ref{prop:patching_projective_modules} below.
  
  Vorst's lemma is also needed to show that the unstable $K_1$-functors satisfy weak analytic excision. Recall from \cite[Lemma~1.4]{SUSLIN_SPECIAL} that the group $\mathrm{E}_r(A)$ of elementary $r \times r$-matrices is normal in $\mathrm{GL}_r(A)$ if $r \geq 3$. The functor
  \[
  K_{1,r} \colon \CAlg_R \rightarrow \Grp
  \]
  is defined by $K_{1,r}(A)=\mathrm{GL}_r(A) \slash \mathrm{E}_r(A)$.
  
  \begin{lemma}[Vorst]\label{lemma:Vorsts_lemma}
  Let $\varphi \colon A \rightarrow B$ be an analytic isomorphism along $S$. Then the following hold:
  \begin{enumerate}
  \item[(1)] If $r \geq 3$ and $\varepsilon \in \mathrm{E}_r(B_{\varphi(S)})$, then there exists an $\varepsilon_1 \in \mathrm{E}_r(B)$ and an $\varepsilon_2 \in \mathrm{E}_r(A_S)$ such that 
  \[
  \varepsilon=\lambda_{\varphi(S)}(\varepsilon_1) \varphi_S(\varepsilon_2)
  \]
   holds;
  \item[(2)] For all $r \geq 3$, the functors $K_{1,r} \colon \CAlg_R \rightarrow \Set$ satisfy weak analytic excision.
  \end{enumerate}
  \end{lemma}
  
  \begin{proof}
  We follow the proof given by Vorst, see \cite[Lemma~2.4]{VORST_GLN}. Vorst makes the additional assumption that $\varphi$ is injective and that $S$ is generated by a single non-nilpotent element. Since any $\varepsilon \in \mathrm{E}_r(B_{\varphi(S)})$ is already defined over $B_{\varphi(f)}$ for some $f \in A$, we can assume without loss of generality that $S=\{1,f,f^2, \ldots\}$. The assumption that $f$ is not nilpotent is superfluous (since the zero-ring is a commutative $R$-algebra).
  
  The assumption that $\varphi$ is injective is also unnecessary: the assumption on $\varphi$ implies that for all $s \in \mathbb{N}$ and all $c \in B_{\varphi(f)}$, there exists an element $b \in B$, an element $a \in A$, and a natural number $m \in \mathbb{N}$ such that the equation
  \[
  c= \frac{\varphi(f)^s b}{1}+\frac{\varphi(a)}{\varphi(f)^m}
  \]
  holds (write $c=c^{\prime} \slash \varphi(f)^m$ with $c^{\prime} \in B$ and use the fact that $c^{\prime} + \varphi(f)^{s+m} B$ lies in the image of $A \slash f^{s+m} \rightarrow B \slash \varphi(f)^{s+m}$). The desired factorization of $\varepsilon$ can now be constructed as in the proof of \cite[Lemma~2.4.(i)]{VORST_GLN}.
  
 Indeed, let $\varepsilon \in \mathrm{E}_r(B_{\varphi(f)})$ be given by $\varepsilon=\prod_{k=1}^m e_{i_k j_k}(c_k)$ for suitable $c_k \in B_{\varphi(f)}$ and let $\sigma_p = \prod_{k=1}^{p-1} e_{i_k j_k}(c_k)$ for $1 \leq p \leq m$. From \cite[Lemma~3.3]{SUSLIN_SPECIAL} (see also \cite[Lemma~2.3]{VORST_GLN} and \cite[Lemma~VI.1.7]{LAM}) it follows that there exist natural numbers $s_p \in \mathbb{N}$ and matrices $\tau_p(t) \in \mathrm{E}_r\bigl(B[t],(t) \bigr)$ such that
 \[
\lambda_{\varphi(f)} \bigl( \tau_p(t) \bigr)=\sigma_p \cdot e_{i_p j_p}\bigl(\varphi(f)^{s_p} t\bigr) \cdot \sigma_p^{-1}
 \]
 holds for all $1 \leq p \leq m$.
 
 Choose elements $a_k \in A$, $b_k \in B$, and $m_k \in \mathbb{N}$ such that
 \[
 c_k = \varphi(f)^{s_k} b_k + \varphi(a_k \slash f^{m_k})
 \]
 holds for all $1 \leq k \leq m$ (this is possible by the above observation). It follows that the equation
 \[
 \varepsilon= \prod_{k=m}^{1} \sigma_k \cdot e_{i_k j_k} \bigl(\varphi(f)^{s_k} b_k\bigr) \sigma_k^{-1} \cdot \prod_{k=1}^m e_{i_k j_k}(\varphi(a_k \slash f^{m_k}))
 \]
 holds.
 
  If we let $\varepsilon_1= \prod_{k=m}^1 \tau_k(b_k) \in \mathrm{E}_r(B)$ and $\varepsilon_2=\prod_{k=1}^m e_{i_k j_k}(a_k \slash f^{m_k}) \in \mathrm{E}_r(A_f)$, then
  \[
  \varepsilon=\lambda_f(\varepsilon_1) \cdot \varphi_f(\varepsilon_2)
  \]
  holds, as claimed.
  
  To deduce Claim~(2) from this, let $\sigma_1 \in \mathrm{GL}_r(B)$ and let $\sigma_2 \in \mathrm{GL}_r(A_S)$. To say that the classes they represent in $K_{1,r}$ are sent to the same element in $K_{1,r}(B_{\varphi(S)})$ means that there exists an $\varepsilon \in \mathrm{E}_r(B_{\varphi(S)})$ such that $\lambda_{\varphi(S)}(\sigma_1) \varepsilon=\varphi_S(\sigma_2)$ holds. Choosing $\varepsilon_1$ and $\varepsilon_2$ as in the conclusion of Claim~(1), we find that the equality
  \[
  \lambda_{\varphi(S)}(\sigma_1  \varepsilon_1)=\varphi_S(\sigma_2 \varepsilon_2^{-1})
  \]
  holds in $\mathrm{GL}_r(B_{\varphi(S)})$. Since $\mathrm{GL_r}(-) \colon \CAlg_R \rightarrow \Set$ is a representable functor, it preserves pullback diagrams. It follows that there exists a $\sigma \in \mathrm{GL}_r(A)$ such that $\varphi(\sigma)=\sigma_1 \varepsilon_1$ and $\lambda_S(\sigma)=\sigma_2 \varepsilon_2^{-1}$ hold. This proves that $K_{1,r} \colon \CAlg_R \rightarrow \Set$ satisfies weak analytic excision.
 \end{proof}
 
 Using this lemma, we get the above mentioned patching result for projective modules.
 
 \begin{prop}[Milnor--Roy]\label{prop:patching_projective_modules}
 Let $\varphi \colon A \rightarrow B$ be an analytic isomorphism along $S$. Then the canonical functor
 \[
 \rho \colon \Proj(A)  \rightarrow \Proj(B)  \pb{\Proj(B_{\varphi(S)}) } \Proj(A_S) 
 \]
 sending $P$ to $(B \ten{A} P,P_S, \kappa_P)$, where $\kappa_P$ denotes the canonical isomorphism, is an equivalence of categories.
 \end{prop}
 
 \begin{proof}
 The argument can be summarized as follows: from the fact that the patching diagram associated to $\varphi$ is cartesian it follows that $\rho$ is full and faithful. Since idempotents split in the domain, it suffices to show that every object in the codomain is a direct summand of $\rho(A^{r})$ for a suitable $r \in R$. The argument of Milnor reduces this to the case of objects of the form $(B^n, A_S^n, \alpha)$. For such objects, the claim follows from Whiteheads lemma, applied to the matrix
 \[
 \begin{pmatrix}
 \alpha & 0 \\
 0 & \alpha^{-1}
 \end{pmatrix}  \in \mathrm{GL_{2n}(B_{\varphi(S)})},
 \]
 combined with Vorst's lemma.
 
 Here are the details. The condition that
 \[
 \vcenter{\xymatrix{A \ar[r]^-{\lambda_S} \ar[d]_{\varphi} & A_S \ar[d]^{\varphi_S} \\ B \ar[r]_-{\lambda_{\varphi(S)}} & B_{\varphi(S)} }}
 \]
 is cartesian implies that $\rho_{A,A}$ is an isomorphism. Since the full subcategory of objects where a natural transformation between additive functors is an isomorphism is closed under finite direct sums and under direct summands, it follows that $\rho$ is indeed full and faithful.
 
 As Roy observed in \cite[\S 2]{ROY}, we can now apply the same reasoning that Milnor gave in \cite[Lemma~2.6]{MILNOR} under slightly different assumptions on the patching diagram. Specifically, given an object $(P,Q,\beta)$ of the target category of $\rho$, we can choose a projective $B$-module $P^{\prime}$ and a projective $A_S$-module $Q^{\prime}$ such that $P \oplus P^{\prime} \cong B^{n}$ and $Q \oplus Q^{\prime} \cong A_S^{m}$. To aid legibility, we let $T \defl \varphi(S)$ and for an $A_S$-module $M$ we write $M_{B_T}$ for $B_T \ten{A_S} M$. Then we get an isomorphism
 \[
 \beta^{\prime} \colon P^{\prime}_{T} \oplus B_{T}^m \cong P^{\prime}_{T} \oplus Q_{B_T} \oplus Q^{\prime}_{B_T} \cong (P^{\prime} \oplus P)_T \oplus Q^{\prime}_{B_T} \cong (Q^{\prime} \oplus A_S^{n})_{B_T},
 \]
 so $(P,Q,\beta)$ is a direct summand of $(B^{n+m},A_{S}^{n+m},\beta \oplus \beta^{\prime})$. Let $r=n+m$ and let $\alpha \in \mathrm{GL}_r(B_T)$ be the unique element such that $\beta \oplus \beta^{\prime}=\kappa_{A^{r}} \circ \alpha$. Since the image of $\rho$ is closed under direct summands (a consequence of the fact that idempotents split in the domain of $\rho$), it only remains to check that the object
 \[
 \Biggl(B^{2r}, A_S^{2r}, \kappa_{A^{2r}} \begin{pmatrix}
 \alpha & 0 \\ 0 & \alpha^{-1} 
\end{pmatrix}  \Biggr)
 \]
 lies in the image of $\rho$. From Whitehead's lemma we know that
 \[
  \begin{pmatrix}
 \alpha & 0 \\
 0 & \alpha^{-1}
 \end{pmatrix}  \in \mathrm{E}_{2r}(B_T) \smash{\rlap{,}}
 \]
 and we can assume that $2r \geq 3$ by increasing the rank of $P^{\prime}$ and $Q^{\prime}$ if necessary. Applying Vorst's Lemma (Claim~(1) of Lemma~\ref{lemma:Vorsts_lemma}) to the inverse of this elementary matrix, we obtain matrices $\varepsilon_1 \in \mathrm{E}_{2r}(B)$ and $\varepsilon_2 \in \mathrm{E}_{2r}(A_S)$ which give an isomorphism
 \[
 (\varepsilon_1, \varepsilon_2^{-1}) \colon \rho(A^{2r}) \rightarrow  \Biggl(B^{2r}, A_S^{2r}, \kappa_{A^{2r}} \begin{pmatrix}
 \alpha & 0 \\ 0 & \alpha^{-1} 
\end{pmatrix}  \Biggr) \smash{\rlap{,}}
 \]
 so $\rho$ is indeed essentially surjective on objects.
   \end{proof}
   
\begin{cor}\label{cor:weak_analytic_excision_for_P_and_P_r}
 The functors $P \colon \CAlg_R \rightarrow \Set$ and $P_r \colon \CAlg_R \rightarrow \Set_{\ast}$ satisfy weak analytic excision for all $r \geq 1$.
\end{cor} 

\begin{proof}
 Weak analytic excision for $P$ is a direct consequence of Proposition~\ref{prop:patching_projective_modules}. Indeed, if $(\lambda_{\varphi(S)})_{\ast}\bigl([P]\bigr)=(\varphi_S)_{\ast} \bigl([Q]\bigr)$, then any choice of isomorphism $\beta \colon P_{\varphi(S)} \cong B_{\varphi(S)} \ten{A_S} Q$ and choice of projective $A$-module $L$ such that $\rho(L) \cong (P,Q,\beta)$ gives an element $[L] \in P(A)$ with $\varphi_{\ast} [L]=[P]$ and $(\lambda_S)_{\ast} [L]=[Q]$.
 
 It only remains to show that $L$ has constant rank $r$ if both $P$ and $Q$ have rank $r$ to get the conclusion for $P_r$. If $\mathfrak{p} \in \mathrm{Spec}(A)$ and $\mathfrak{p} \cap S= \varnothing$, the localization at $\mathfrak{p}$ factors through $A_S$, so $\mathrm{rk} (L_{\mathfrak{p}})=r$. If not, then there exists an element $s \in \mathfrak{p} \cap S $. In this case, the projection $A \rightarrow A \slash \mathfrak{p}$ factors through $\varphi \colon A \rightarrow B$ since $\varphi$ is an analytic isomorphism along $s$. This again implies that $\mathrm{rk}(L_{\mathfrak{p}})=r$.
\end{proof}  

It follows from an observation of Gubeladze that Vorst's Lemma can also be used to show that for each $r \geq 3$, the functor
\[
W_r \colon \CAlg_R \rightarrow \Set_{\ast}
\]
which sends $A$ to $W_r(A) \defl \mathrm{Um}_r(A) \slash \mathrm{E}_{r}(A)$, the set of elementary orbits of unimodular rows pointed by $e_1=(1,0,\ldots,0)$, satisfies weak analytic excision.

\begin{prop}\label{prop:W_r_weak_analytic_excision}
 For all $r \geq 3$, the functor $W_r \colon \CAlg_R \rightarrow \Set_{\ast}$ satisfies weak analytic excision.
\end{prop}

\begin{proof}
 Let $\varphi \colon A \rightarrow B$ be an analytic isomorphism along $S$. Given unimodular rows $u=(u_1, \ldots, u_r)$ over $B$ and $v=(v_1, \ldots, v_r)$ over $A_S$ which become equal in $W_r(B_{\varphi(S)})$, we can apply Vorst's Lemma (Claim~(1) of Lemma~\ref{lemma:Vorsts_lemma}) to any matrix $\varepsilon \in \mathrm{E}_{r}(B_{\varphi(S)})$ witnessing this equality (that is, any $\varepsilon$ such that $(\lambda_{\varphi(S)})_{\ast} u \varepsilon=(\varphi_S)_{\ast} v$) to get $\varepsilon_1 \in \mathrm{E}_r(B)$ and $\varepsilon_2 \in \mathrm{E}_r(A_S)$ such that
 \[
 (\lambda_{\varphi(S)})_{\ast} (u \varepsilon_1) = (\varphi_S)_{\ast} (v \varepsilon_2^{-1})
 \]
 holds.
 
 Since the patching diagram associated to $\varphi$ is cartesian, there exists a row $(w_1, \ldots, w_r)$, $w_i  \in A$, such that $\varphi_{\ast} w=u \varepsilon_1$ and $(\lambda_S)_{\ast}=v \varepsilon_{2}^{-1}$. It only remains to show that $w$ is unimodular, which follows from the proof of \cite[Proposition~9.1]{GUBELADZE_ELEMENTARY}.
 
 Indeed, since the localization of $w$ at $S$ is unimodular, the ideal $I \subseteq A$ generated by the $w_i$ contains some $s \in S$. The image $\bar{w}$ of $w$ in $A \slash s$ coincides with the image of $u\varepsilon_1$ under the composite $B \rightarrow B \slash \varphi(s) \cong A \slash s$, so it is also unimodular. Thus $1 +sa \in I$ for some $a \in A$, which implies that $I=A$, as claimed.
\end{proof}

 The examples $P_{r} \colon \CAlg_R \rightarrow \Set_{\ast}$ and $K_{1,r} \colon \CAlg_R \rightarrow \Grp$ can both be generalized considerably.
 
 Recall that an \emph{Adams stack} over $R$ is a stack for the $\fpqc$-topology on $\CAlg_R$ which is quasi-compact, has affine diagonal, and which has the resolution property. Any Adams stack $X \colon \CAlg_R \rightarrow \Gpd$ yields a functor $\pi_0 X \colon \CAlg_R \rightarrow \Set$ which sends $A$ to the set $\pi_0 X(A)$ of isomorphism classes of objects of the groupoid $X(A)$. The following proposition is a consequence of generalized Tannaka duality.
   
\begin{prop}\label{prop:Adams_stack_weak_analytic_excision}
Let $X \colon \CAlg_R \rightarrow \Gpd$ be an Adams stack and let $\varphi \colon A \rightarrow B$ be an analytic isomorphism along $S$. Then the induced functor
\[
X(A) \rightarrow X(B) \pb{X(B_{\varphi(S)})} X(A_S)
\]
is an equivalence of groupoids. In particular, $\pi_0 X \colon \CAlg_R \rightarrow \Set$ satisfies weak analytic excision.
\end{prop}

\begin{proof}
 Since affine schemes are examples of Adams stacks, the first claim amounts to showing that the diagram
 \[
 \xymatrix{\Spec(B_{\varphi(S)}) \ar[r] \ar[d] & \Spec(B) \ar[d] \\ \Spec(A_S) \ar[r] & \Spec(A)}
 \]
 is a pushout in the 2-category of Adams stacks. By \cite[Theorem~1.1.1]{SCHAEPPI_DESCENT}, we need to check two claims: that
 \[
 \VB^c\bigl(\Spec(A) \bigr)  \rightarrow \VB^c\bigl(\Spec(B) \bigr) \pb{\VB^c\bigl(\Spec(B_{\varphi(S)}) \bigr)} \VB^c\bigl(\Spec(A_S) \bigr)
 \]
 is an equivalence, where $\VB^c$ denotes the category of vector bundles of constant finite rank, and that $M \in \Mod_{A,\mathrm{fp}}$ is isomorphic to the zero module if both $B \ten{A} M \cong 0$ and $M_S \cong 0$.
 
 The first claim follows directly from Proposition~\ref{prop:patching_projective_modules} and the proof of Corollary~\ref{cor:weak_analytic_excision_for_P_and_P_r}.
 
 Let $M$ be a finitely presentable $A$-module such that $B \ten{A} M \cong 0$ and $M_S \cong 0$. Since $M$ is finitely generated and $M_S \cong 0$, there exists an $s \in S$ such that $sM=0$. This implies that $ M \cong A \slash s \ten{A}M$. Since the projection $A \rightarrow A \slash s$ factors through $\varphi \colon A \rightarrow B$, it follows from $B \ten{A} M \cong 0$ that $M \cong 0$, as claimed.
\end{proof}

 A flat affine group scheme $G$ over $R$ gives rise to an Adams stack $BG$ if it has the (equivariant) resolution property (here $BG(A)$ denotes the groupoid of $G$-torsors over $\Spec(A)$). If $R$ is noetherian and $G$ is reductive, this is the case if and only if $G$ is linear, which in turn is the case if and only if the radical torus $\mathrm{rad}(G)$ is isotrivial, see \cite[Theorem~6.2]{GILLE_LINEAR}. This condition is for example satisfied for all reductive group schemes $G$ over $R$ if $R$ is normal (see \cite[Th{\'e}or{\`e}me~X.5.16]{SGA3_NEW}). Note that $\pi_0 BG(A)$ is precisely the set $H^1(G,A)$ of isomorphism classes of $G$-torsors over $\Spec(A)$.
 
 \begin{cor}\label{cor:torsors_weak_analytic_excision}
 If $G$ is a flat affine group scheme over $R$ with the resolution property, then the functor
 \[
 H^1(G,-) \colon \CAlg_R \rightarrow \Set_{\ast}
 \]
 which sends $A$ to the set of isomorphism classes of $G$-torsors over $\Spec(A)$, pointed by the trivial torsor, satisfies weak analytic excision.
 \end{cor}
 
 \begin{proof}
 As noted above, $X=BG$ is an Adams stack in this case and $H^1(G,A)=\pi_{0} X$, so the claim follows directly from Proposition~\ref{prop:Adams_stack_weak_analytic_excision}. 
 \end{proof}

 Similarly, the definition of $K_{1,r}$ can be extended to obtain functors 
 \[
 K_{1}^G \colon \CAlg_R \rightarrow \Grp
 \]
 for more general affine group schemes $G$ over $R$ in such a way that $K_{1,r}$ corresponds to the functor $K_1^{\mathrm{GL}_r}$. This relies on generalization the elementary subgroup and a corresponding form of Vorst's Lemma. Such an version of Vorst's Lemma appears in \cite[Lemma~3.7]{ABE}. The most general form known to the author appears in \cite[Lemma~3.4]{STAVROVA_HOMOTOPY}. We note that, in its most general form, the functor $K_1^{G}$ also depends on a choice of parabolic subgroup $P$ of $G$ and is denoted by $K_{1}^{G,P} \colon \CAlg_R \rightarrow \Set_{\ast}$. The elementary subgroup $\mathrm{E}_P(A)$ is the subgroup of $G(A)$ generated by the unipotent radical $U_P(A)=\mathrm{rad}^{u} P(A)$ and $U_{P^{-}}$, where $P^{-}$ is a parabolic subgroup opposite to $P$ (see \cite[Definition~2.1]{STAVROVA_HOMOTOPY}). The pointed set $G(A) \slash \mathrm{E}_{P}(A)$ gives the value of $K_1^{G,P}$ on $A$.

\begin{lemma}\label{lemma:K_1_GP_finitary}
 If $G$ is a reductive group over $R$ and $P$ is a parabolic subgroup, then the functor
 \[
 K_1^{G,P} \colon \CAlg_R \rightarrow \Set_{\ast}
 \]
 is finitary.
\end{lemma}

\begin{proof}
 Since $G$ is (locally) of finite presentation over $R$, the functor $G(-)$ is finitary, so we only need to check that $E_{P}(-)$ is also finitary. This follows from the fact that $U_P$ and $U_{P^{-}}$ are finitary.
 
 Here are the details. The unipotent radicals $U_P$ and $U_{P^{-}}$ are isomorphic (as schemes) to certain vector bundles
 \[
 U_P \cong \Spec\bigl(\Sym(V)\bigr) \quad \text{and} \quad U_{P^{-}} \cong \Spec \bigl( \Sym(W) \bigr)
 \]
 for certain finitely generated projective $R$-modules $V$ and $W$ (see \cite[Corollaire~XXVI.2.5]{SGA3_NEW}). It follows that the functor
 \[
 U_{P}(-) \ast U_{P^{-}}(-) \colon \CAlg_R \times \CAlg_R \rightarrow \Grp
 \]
(where $\ast$ denotes the free product) preserves filtered colimits in each variable. Thus the functor given by the diagonal
 \[
 A \mapsto  U_{P}(A) \ast U_{P^{-}}(A)
 \]
 is finitary. The group $\mathrm{E}_P(A)$ is the image of the natural transformation 
 \[
 \psi_A \colon U_{P}(A) \ast U_{P^{-}}(A) \rightarrow G(A)
 \]
 induced by the two natural inclusions. Since the image coincides with the quotient of the domain modulo the kernel, it only remains to check that the kernel of $\psi$ commutes with filtered colimits. This follows from the fact that finite limits commute with filtered colimits in the category $\Grp$.
\end{proof}

\begin{prop}[Stavrova]\label{prop:unstable_K1_excision_reductive_group}
 Let $G$ be be an isotropic reductive group scheme over $R$ and let $P$ be a strictly proper parabolic subgroup of $G$ (that is, $P$ intersects every normal semi-simple subgroup of $G$ properly). Assume that the rank of $\Phi_P$ (see \cite[Definition~2.5]{STAVROVA_HOMOTOPY}) is $\geq 2$ everywhere on $\Spec(R)$. Then
 \[
 K_1^{G,P} \colon \CAlg_R \rightarrow \Set_{\ast}
 \]
 satisfies weak analytic excision.
\end{prop}

\begin{proof}
 Let $\varphi \colon A \rightarrow B$ be an analytic isomorphism along $S$. Since $K_1^{G,P}$ preserves filtered colimits, it suffices to treat the case where $S$ is generated by a single element $f \in A$. In this case, the claim follows from \cite[Lemma~3.4(i)]{STAVROVA_HOMOTOPY}. Strictly speaking, there are three additional assumptions in this lemma: that $A=R$, that $\varphi$ is injective, and that $f$ is not nilpotent. As in the case of Vorst's Lemma, the latter two assumptions are unnecessary (see also \cite[Corollary~3.5]{STAVROVA_HOMOTOPY} regarding the injectivity of $\varphi$). Moreover, the bulk of the argument takes place over the $A$-algebra $B_{\varphi}(S)$ anyway, so we can safely assume that $G$ is already defined over $R$ and satisfies the relevant conditions over $\Spec(R)$. This gives the analogue of Vorst's Lemma for the subgroup $\mathrm{E}_P(B_{\varphi(S)})$ of $G(B_{\varphi(S)})$. Since the anlytic patching diagram induced by $\varphi$ is cartesian and $G$ preserves cartesian diagrams, we can conclude as in the case of Vorst's Lemma (Claim~(2) of Lemma~\ref{lemma:Vorsts_lemma}).
\end{proof}

\begin{rmk}\label{rmk:stavrova_conditions_satisfied_isotropic_rank}
 The conditions on $G$ and $P$ are for example satisfied if $R$ is noetherian and $G$ is of isotropic rank $\geq 2$, that is, every non-trivial normal semi-simple $R$-subgroup of $G$ contains $(\mathbb{G}_{m,R})^{2}$ and $P$ is any strictly proper parabolic subgroup of $G$ (see the proof of \cite[Lemma~2.7]{STAVROVA_EQUICHARACTERISTIC}). In this case, $\mathrm{E}(A)=\mathrm{E}_P(A)$ is independent of $P$ and normal by \cite[Theorem~1]{PETROV_STAVROVA}. Note that a strictly proper parabolic subgroup always exists for such $G$ (see the discussion after \cite[Definition~2.2]{STAVROVA_EQUICHARACTERISTIC}).
\end{rmk}

The final class of examples is quite different in flavour from the ones discussed so far. It concerns invariants of $R$-algebras $A$ which can be defined using the derived category of $A$. The higher $K$-groups of Quillen are the most prominent example of such an invariant. The following lemma is useful for dealing with such functors.

\begin{lemma}\label{lemma:analytic_iso_induces_iso_on_kernels}
 Let $\varphi \colon A \rightarrow B$ be an analytic isomorphism along $S \subseteq A$. Then the restriction of $\varphi$ induces an isomorphism
 \[
 \varphi^{s} \colon \ker(s \cdot \colon A \rightarrow A) \rightarrow \ker(\varphi(s) \cdot \colon B \rightarrow B )
 \]
 of $A$-modules for all $s \in S$.
\end{lemma}

\begin{proof}
 Let $I = \ker(\lambda_S) \subseteq A$ and let $J=\ker(\lambda_{\varphi(S)}) \subseteq B$. Since the patching diagram associated to $\varphi$ is cartesian, it follows that $\varphi \vert_{I} \colon I \rightarrow J$ is bijective. This shows in particular that $\varphi^{s}$ is injective since the kernels are contained in $I$ respectively $J$.
 
 To see surjectivity, let $b \in B$ be an element such that $\varphi(s)b=0$. Then there exists an $a \in I$ with $b=\varphi(a)$. Since $\varphi(sa)=\varphi(s)b=0$ and $\varphi \vert_I$ is injective, we conclude that $a \in \ker(s\cdot \colon A \rightarrow A)$.
\end{proof}

\begin{prop}\label{prop:K_theory_weak_analytic_excision}
 For $i \geq 0$, the higher $K$-groups
 \[
 K_i \colon \CAlg_R \rightarrow \Ab
 \]
 of Quillen satisfy weak analytic excision.
\end{prop}

\begin{proof}
 Let $\varphi \colon A \rightarrow B$ be an analytic isomorphism along $S$. If $S$ and $\varphi(S)$ consist of nonzerodivisors, then this is well-known, see for example \cite[Proposition~1.5]{VORST_POLYNOMIAL}. The case of an arbitrary multiplicative set $S$ is covered in \cite[Example~2.9.(ii)]{LAND_TAMME} (which is applicable by Lemma~\ref{lemma:analytic_iso_induces_iso_on_kernels}), where it is deduced from a more general result about localizing invariants of $E_1$-ring spectra.
 
 We present here a different proof which follows \cite[Example~1.4.2]{CALMES_ET_AL_II}, which in turn is based on \cite[Lemma~7.2.3.13]{LURIE_HA}. This proof only relies on the Waldhausen Approximation and Localization Theorems and on the fact that the triangulated category of perfect complexes over a commutative ring is idempotent complete, see \cite[Proposition~2.2.13]{THOMASON_TROBAUGH} for the corresponding statement about perfect complexes over arbitrary schemes. A more elementary proof in the case of affine schemes can be found in \cite[Proposition~3.4]{BOECKSTEDT_NEEMAN}.  We denote the $K$-theory spectrum of $A$ by $K(A)$, so that $K_i A=\pi_i K(A)$ for all $i \geq 0$.
  
 Let $\mathrm{Ch}^b_S\bigl(\Proj(A)\bigr)$ denote the Waldhausen category of bounded complexes which become exact after localization at $S$ and define $\mathrm{Ch}^b_{\varphi(S)}\bigl(\Proj(B)\bigr)$ similarly.  From a combination of Waldhausen Approximation and Waldhausen Localization it follows that there exists a long exact sequence
 \[
\ldots \rightarrow  K_{i+1}(A_S) \rightarrow \pi_i K\Bigl(\mathrm{Ch}^b_S\bigl(\Proj(A)\bigr)\Bigr) \rightarrow  K_i(A) \rightarrow K_i(A_S) \rightarrow \ldots
 \]
 for $i>0$ (see \cite[Theorem~V.2.6.3]{WEIBEL_KBOOK}). We can use Corollary~\ref{cor:weak_analytic_excision_for_P_and_P_r} to check the claim for $i=0$. Since we have an analogous sequence involving the localization $B \rightarrow B_{\varphi(S)}$, it suffices to show that the functor
 \[
 B \ten{A} - \colon  \mathrm{Ch}^b_S\bigl(\Proj(A)\bigr) \rightarrow \mathrm{Ch}^b_{\varphi(S)}\bigl(\Proj(B)\bigr)
 \]
 satisfies the conditions of the Waldhausen Approximation Theorem.
 
 This follows from the fact that the category $\mathrm{Ch}^b_S\bigl(\Proj(A)\bigr)$ is generated in a suitable sense by the objects $C_s \defl (s\cdot \colon A \rightarrow A)$, where the copies of $A$ are in degree $0$ and $1$ (and the category $\mathrm{Ch}^b_{\varphi(S)}\bigl(\Proj(B)\bigr)$ is similarly generated by the objects $C_{\varphi(s)}\defl (\varphi(s) \cdot \colon B \rightarrow B)$). More precisely, if $\ca{C}$ is the smallest full subcategory of $\mathrm{Ch}_S\bigl(\Proj(A)\bigr)$ containing all the objects $C_s$, $s \in S$, which is closed under shifts, finite direct sums, retracts, mapping cones, and quasi-isomorphisms, then $\ca{C}= \mathrm{Ch}^b_S\bigl(\Proj(A)\bigr)$. We defer the proof of this claim to Lemma~\ref{lemma:cofibers_of_s_generate} below.
 
 Using this claim, we find that $B \ten{A} -$ induces quasi-isomorphisms
 \[
 \mathrm{Hom}_A(X,Y) \rightarrow \mathrm{Hom}_B(B \ten{A} X, B \ten{A} Y)
 \]
 on mapping complexes for any two $X, Y \in  \mathrm{Ch}^b_S\bigl(\Proj(A)\bigr)$. Indeed, the closure properties of the category $\ca{C}$ reduce this to the case $Y=C_s$. Lemma~\ref{lemma:analytic_iso_induces_iso_on_kernels} and the definition of analytic isomorphisms imply that the canonical morphism $C_s \rightarrow B \ten{A} C_s$ is a quasi-isomorphism in $\mathrm{Ch}(\Mod_A)$. Since $X$ is cofibrant in the projective model structure on chain complexes, the first morphism in the composite
 \[
 \mathrm{Hom}_A(X,C_s) \rightarrow \mathrm{Hom}_A(X, B \ten{A} C_s) \cong \mathrm{Hom}_{B}(B \ten{A} X, B \ten{A} C_s)
 \]
 is a quasi-isomorphism. The second morphism above is the adjunction isomorphism of the left adjoint differential graded functor $B \ten{A} -$. This establishes the fact that $B \ten{A} -$ induces quasi-isomorphisms on mapping complexes.
 
 From the fact that the homotopy category of $\mathrm{Ch}^b(A)$ is idempotent complete (see for example \cite[Proposition~3.4]{BOECKSTEDT_NEEMAN}) and the fact that $\mathrm{Ch}^b_{\varphi(S)}(B)$ is similarly generated by the objects $C_{\varphi(s)}\cong B \ten{A} C_s$, it follows that
 \[
 B \ten{A} - \colon  \mathrm{Ch}^b_S\bigl(\Proj(A)\bigr) \rightarrow \mathrm{Ch}^b_{\varphi(S)}\bigl(\Proj(B)\bigr)
 \]
 is essentially surjective up to quasi-isomorphism. Since $B \ten{A} -$ preserves mapping cones, it follows that the conditions of the Waldhausen Approximation Theorem \cite[Theorem~1.9.1]{THOMASON_TROBAUGH} are indeed satisfied. The proof is completed by Lemma~\ref{lemma:cofibers_of_s_generate} below.
\end{proof}

\begin{lemma}\label{lemma:cofibers_of_s_generate}
Let $A$ be an $R$-algebra, let $S \subseteq A$ be a multiplicative set, and let $\mathrm{Ch}_S^b\bigl(\Proj(A)\bigr)$ be the category of bounded complexes whose entries are finitely generated projective and which become exact after localization at $S$. Let $\ca{C}$ be the smallest full subcategory of $\mathrm{Ch}^b_S\bigl(\Proj(A)\bigr)$ containing the objects 
\[
C_s\defl (s\cdot \colon A \rightarrow A)
\] 
(concentrated in degrees $0$ and $1$) for all $s \in S$ which is closed under shifts, finite direct sums, retracts, mapping cones, and quasi-isomorphisms. Then $\ca{C}=\mathrm{Ch}^b_S\bigl(\Proj(A)\bigr)$.
\end{lemma}

\begin{proof}
 Let $X \in \mathrm{Ch}^b_S\bigl(\Proj(A)\bigr)$ be an arbitrary object. Note that, if two out of three objects in a distinguished triangle lie in $\ca{C}$, then so does the third. Indeed, since $\ca{C}$ is closed under shifts, we can always rotate the triangle so that the ``missing'' object is the mapping cone of a morphism in $\ca{C}$. 
 
 First, we can shift $X$ so that $X_r=0$ for $r < 0$ and $r > n$, and $X_0 \neq 0$, $X_n \neq 0$, where $n \in \mathbb{N}$ is uniquely determined by $X$. Assume that $n > 1$. Since $H_0(X)=\mathrm{coker} X_1 \rightarrow X_0$ is finitely generated and the localization $\bigl(H_0(X)\bigr)_S\cong 0$, we can find an $s \in S$ and a homomorphism
 \[
 \alpha \colon \bigoplus_{i=1}^k C_s \rightarrow X
 \]
 which induces an epimorphism on $H_0$ (and an isomorphism on $H_r$ for $r<0$). Since $n \geq 0$, the cone $Y$ of $\alpha$ is again concentrated in degrees $0$ up to $n$. Moreover, from the long exact homology sequence it follows that $d_1 \colon Y_1 \rightarrow Y_0$ is surjective. Since $Y_0$ is projective, this epimorphism is split, so the chain complex $Y$ is quasi-isomorphic to a chain complex concentrated in degrees $1$ through $n$. Since it suffices to check that $Y$ lies in $\ca{C}$, we can iterate this argument and reduce to the case $n=1$, that is, $X$ is given by a single differential $d\defl d_1 \colon X_1 \rightarrow X_0$.
 
 Next we show that we can further reduce to the case where $X_1$ and $X_2$ are free $A$-modules. By adding a suitable complex $\id \colon P \rightarrow P$ to $X$ we can assume that $X_1 \cong A^m$ for some $m \in \mathbb{N}$. Pick an epimorphism $A^k \rightarrow X_0$ and consider the pullback
 \[
 \xymatrix{Y_1 \ar[d]_e \ar[r] & A^m \ar[d]^d \\ A^k \ar[r]_-{p} & X_0}
 \]
 of $d$ along $p$. The snake lemma implies that $A^{k} \cong X_0 \oplus Q$ and $Y_1 \cong Q \oplus A^m \oplus Q$ where $Q=\ker(p)$. In particular, $X_0 \oplus Y_1$ is a free $A$-module, so the chain complex
 \[
 Z \defl d \oplus \Sigma^{-1} e
 \]
 is concentrated in degree $-1$ through $1$ and all its entries are free modules. Since $\ca{C}$ is closed under retracts, it suffices to check that $Z$ lies in $\ca{C}$. If we construct a homomorphism
 \[
 \alpha \colon \bigoplus_{i=1}^\ell C_s \rightarrow \Sigma Z
 \]
 as above and let $Z^{\prime} \defl C_{\alpha}$ be its cone, we find that it suffices to check that $Z^{\prime}$ lies in $\ca{C}$, where $Z^{\prime}$ is concentrated in degrees $0$ through $2$, all its entries are free modules, and $d_1 \colon Z_1^{\prime} \rightarrow Z_0^{\prime}$ is a (split) epimorphism. The kernel $F$ of $d_1$ is thus a stably free $A$-module and $Z^{\prime}$ is quasi-isomorphic to a chain complex $Z_2^{\prime} \rightarrow F$. Adding the identity morphism of a finitely generated free module of sufficiently large rank (and shifting by $-1$) we have indeed reduced the problem to the following claim: if $\alpha \colon A^m \rightarrow A^k$ is a homomorphism of $A$-modules such that the localization $\alpha_S$ is an isomorphism, then $X \defl (\alpha \colon A^m \rightarrow A^k)$ lies in $\ca{C}$.
 
 There are now two cases. If $A_S=0$, then $0 \in S$, which implies that $A \in \ca{C}$. In this case, $X$ is the cone of $\alpha \colon A^m \rightarrow A^k$, thought of as a morphism between two chain complexes concentrated in degree $0$.
 
 If $A_S \neq 0$, then the fact that $\alpha_S$ is an isomorphism implies that $m=k$, so $\alpha$ is given by an $m$-by-$m$ matrix. Let $\alpha^{\prime}$ be its adjugate, so that $\alpha \alpha^{\prime}=\alpha^{\prime} \alpha=\mathrm{det}{\alpha} \cdot \id_{A^{m}}$. Since $\mathrm{det}(\alpha)$ is a unit in $A_S$, there exists an element $u \in A$ such that $s \defl u \cdot \mathrm{det}(\alpha)$ lies in $S$. Letting $\beta \defl u \alpha^{\prime}$ we thus have $\alpha \beta=\beta \alpha =s \cdot \id$.
 
 The cone of the homomorphism $\oplus_{i=1}^m C_s \rightarrow \oplus_{i=1}^m C_s$ induced by $\alpha$ and given by the diagram
 \[
 \xymatrix{A^{m} \ar[r]^-{s \cdot \id} \ar[d]_{\alpha} & A^m \ar[d]^{\alpha} \\ A^m \ar[r]_-{s \cdot \id} & A^m}
 \]
 (here morphisms of chain complexes are written vertically) is the chain complex
 \[
Y \defl  \xymatrix@C=30pt{A^m \ar[r]^-{ \bigl( \begin{smallmatrix}
 \alpha \\ -s \cdot \id
 \end{smallmatrix} \bigr) }& A^{2m} \ar[r]^{( \begin{smallmatrix}
 s \cdot \id & \alpha
\end{smallmatrix} )  } & A^m} 
 \]
 concentrated in degrees $0$ through $2$.
 
 The diagram
 \[
 \xymatrix@C=30pt{ & A^m \ar[d]^(0.4){ \bigl( \begin{smallmatrix}
 0 \\ 1
\end{smallmatrix} \bigr )  } \ar[r]^-{\alpha} & A^m \ar@{=}[d] \\
 A^m \ar@{=}[d] \ar[r]^-{\bigl(\begin{smallmatrix}
 \alpha \\ -s \cdot \id
 \end{smallmatrix} \bigr) }& A^{2m} \ar[d]_{ (\begin{smallmatrix}
 1 & 0
\end{smallmatrix} )  } \ar[r]^{(\begin{smallmatrix}
 s \cdot \id & \alpha
\end{smallmatrix} ) } & A^m \\ A^m \ar[r]_-{\alpha} & A^m}
 \]
 gives a short exact sequence $X \rightarrow Y \rightarrow \Sigma X$, where $Y \in \ca{C}$ by construction. Since the chain map $\Sigma X \rightarrow Y$ given by the identity in degree $2$ and by $\bigl(\begin{smallmatrix}
 1 \\ -\beta
 \end{smallmatrix}  \bigr) \colon A^m \rightarrow A^{2m}$ in degree $1$ gives a splitting of this short exact sequence, we have $Y \cong X \oplus \Sigma X$. Since $\ca{C}$ is closed under retracts, we have $X \in \ca{C}$, as claimed.
\end{proof}

Our final example concerns Grothendieck--Witt spectra arising from projective modules equipped with certain kinds of non-singular bilinear forms. Specifically, we are interested in the case of alternating forms: bilinear forms $\beta \colon P \times P \rightarrow A$ such that $\beta(p,p)=0$ for all $p \in P$. In the case where $2$ is a unit in $A$, much is known about the $K$-theory of such objects. In later sections, we will however be interested in the case where $R=\mathbb{F}_2$. Recent result of Calmes, Dotto, Harpaz, Hebestreit, Land, Moi, Nardin, Nikolaus and Steimle \cite{CALMES_ET_AL_I, CALMES_ET_AL_II} show that the relevant results also hold in this context.

The above mentioned authors construct a spectrum, called the \emph{Grothendieck--Witt spectrum} starting from any Poincar\'{e} $\infty$-category. Here a Poincar\'{e} $\infty$-category is a small stable $\infty$-category $\ca{C}$ equipped with a special kind of $\infty$-functor 
\[
q \colon \ca{C}^{\op} \rightarrow \mathbf{Sp}
\] 
 to the $\infty$-category of spectra, see \cite[Definition~1.2.8]{CALMES_ET_AL_I}\footnote{We write $q$ for the $\infty$-functor denoted by a capital qoppa in \cite{CALMES_ET_AL_I}. }. Very roughly, one can think of this functor as a derived version of the functor $\Proj(A)^{\op} \rightarrow \Ab $ which sends a finitely generated projective $P$ module to the abelian group given by a certain type of bilinear or quadratic form on $P$.
 
 We are only interested in certain special examples of Poincar\'{e} $\infty$-categories where $\ca{C}$ is a subcategory of the $\infty$-category $\ca{D}^{p}(A)$ of perfect complexes of $A$-modules, where $A$ is a (discrete) commutative $R$-algebra. Specifically, we write $-A$ for the $A$-module $A$, equipped with the involution $\sigma(a)=-a$. Then we can equip $-A$ with the structure of an $A \otimes A$-module given by $(a \otimes a^{\prime})a^{\prime\prime}=a a^{\prime} a^{\prime \prime}$. This is an \emph{invertible $\mathbb{Z}$-module with involution over $A$} in the sense of \cite[Definition~4.2.4]{CALMES_ET_AL_I} (by \cite[Remark~4.2.5]{CALMES_ET_AL_I}). To any such module, we can associate the family
 \[
 (\ca{D}^{p}(A),q_{-A}^{\geq m}), \quad m \in \mathbb{Z}
 \]
 of Poincar\'{e} $\infty$-categories (see \cite[Example~4.2.13]{CALMES_ET_AL_I}), where $\ca{D}^{p}(A)$ denotes the category of perfect complexes over $A$. Given a subgroup $c \subseteq K_0(A)$, we write $\ca{D}^{c}(A)$ for the full subcategory of $\ca{D}^p(A)$ consisting of complexes $X$ whose class $[X]$ in $K_0(A)$ lies in $c$. The duality functor induced by $q_{-A}^{\geq m}$ on $\ca{D}^{\mathrm{perf}}(A)$ is given by $\mathrm{Hom}_A(-,A)$ since $q_{-A}^{\geq m}$ is compatible with $-A$ (again by \cite[Example~4.2.13]{CALMES_ET_AL_I}). Thus $q_{-A}^{\geq m}$ restricts to a Poincar\'{e} structure on $\ca{D}^{c}(A)$ whenever $[A]$ lies in $c$ and $c$ is stable under the involution on $K_0 A$ induced by $\mathrm{Hom}_A(-,A)$. This is for example the case if $\varphi \colon A \rightarrow B$ is a ring homomorphism and $c=\mathrm{im}K_0 A$ is the image of $K_0 A \rightarrow K_0 B$.
  
 For each $m \in \mathbb{Z}$, the construction of the Poincar\'{e} $\infty$-category $(\ca{D}^{p}(A),q_{-A}^{\geq m})$ is functorial in $A$. Indeed, given a homomorphism $\varphi \colon A \rightarrow B$ of commutative $R$-algebras, we can take $\delta \defl \varphi \colon -A \rightarrow -B$ and use the functoriality of the construction of the $\mathrm{H} \mathbb{Z}$-module with genuine involution in \cite[Example~4.2.13]{CALMES_ET_AL_I} which underlies $q_{-A}^{\geq m}$ to define the remaining data necessary to obtain a Poincar\'{e} $\infty$-functor
 \[
 (\ca{D}^{p}(A),q_{-A}^{\geq m}) \rightarrow (\ca{D}^{p}(B),q_{-B}^{\geq m})
 \]
 via \cite[Corollary~3.4.2]{CALMES_ET_AL_I} and \cite[Lemma~3.4.3]{CALMES_ET_AL_I}. On the underlying $\infty$-categories, the functor is by construction given by the scalar extension $B \ten{A} -$. There is an $\infty$-functor $\mathrm{GW} \colon \mathbf{Cat}_{\infty}^p \rightarrow \mathbf{Sp}$ from the $\infty$-category of Poincar\'{e} $\infty$-categories to spectra (see \cite[Definition~4.2.1]{CALMES_ET_AL_II}).
 
 \begin{prop}\label{prop:Grothendieck_Witt_weak_excision}
 Let $A$, $B$ be commutative $R$-algebras and let $\varphi \colon A \rightarrow B$ be an analytic isomorphism along $S$. If $\varphi$ induces an injection $\bigl( \begin{smallmatrix} \bar{\varphi} \\ \bar{\lambda_S} \end{smallmatrix} \bigr) \colon A \slash 2 \rightarrow B \slash 2 \oplus A_S \slash 2$, then the diagram
 \[
 \xymatrix{ 
 \mathrm{GW}  (\ca{D}^{p}(A),q_{-A}^{\geq m}) \ar[r] \ar[d] & \mathrm{GW}  (\ca{D}^{\mathrm{im}K_0 A}(A_S),q_{-A_S}^{\geq m}) \ar[d] \\
 \mathrm{GW}  (\ca{D}^{p}(B),q_{-B}^{\geq m}) \ar[r] & \mathrm{GW}  (\ca{D}^{\mathrm{im}K_0 B}(B_{\varphi(S)}),q_{-B_{\varphi(S)}}^{\geq m})
 }
 \]
 of spectra is cartesian for all $m \in \mathbb{Z}$.
 \end{prop}
 
 \begin{proof}
 The special case where $\varphi \colon A \rightarrow B$ is given by the localization $\lambda_g$ at some element $g \in A$ and $S$  is generated by a single element $f$ is treated in \cite[Corollary~4.4.22]{CALMES_ET_AL_II}. The case of general $\varphi$ works analogously, in other words, we need to check that Conditions~(i)-(v) of \cite[Proposition~4.4.21]{CALMES_ET_AL_II} hold.
 
 To see Condition~(i), note that $\eta=\varphi$ in our case, so this induces the desired isomorphism $B \ten{A} A \rightarrow B$.
 
 Condition~(ii) asserts that $M=A$ is compatible with $S$, which follows from \cite[Example~1.4.4.iii)]{CALMES_ET_AL_II} since the two $A$-actions on $M=A$ agree by definition.
 
 Condition~(iii) holds since $\varphi$ is an analytic isomorphism along $S$. Indeed, this condition is precisely the conclusion of Lemma~\ref{lemma:analytic_iso_induces_iso_on_kernels}, combined with the definition of analytic isomorphisms.
 
 Both $(A,S)$ and $\bigl(B,\varphi(S)\bigr)$ satisfy the left Ore condition since the $R$-algebras $A$ and $B$ are commutative, so Condition~(iv) holds as well.
 
 It remains to check the final Condition~(v), which states that the boundary map $\hat{\mathrm{H}}^{-m} (C_2,B_{\varphi(S)}) \rightarrow \hat{\mathrm{H}}^{-m+1}(C_2,A)$ in Tate cohomology induced by the short exact sequence
 \[
 \xymatrix@C=30pt{ A \ar[r]^-{\bigl( \begin{smallmatrix} -\varphi \\ \lambda_S \end{smallmatrix} \bigr) } & B \oplus A_S \ar[r]^-{ (\begin{smallmatrix} \lambda_{\varphi(S)} & \varphi \end{smallmatrix}) } & B_{\varphi(S)} }
 \]
 vanishes, where all three abelian groups are equipped with the $C_2$-action given by $\sigma(x)=-x$. Equivalently, this means that $\bigl(\begin{smallmatrix} \varphi \\ \lambda_S \end{smallmatrix} \bigr) \colon A \rightarrow B \oplus A_S$ induces a monomorphism
 \[
 \hat{\mathrm{H}}^{-m+1}(C_2,A) \rightarrow \hat{\mathrm{H}}^{-m+1}(C_2,B \oplus A_S)
 \]
 of abelian groups for all $m \in \mathbb{Z}$.
 
 Since $C_2$ is cyclic, the Tate-cohomology groups of any $\mathbb{Z}[C_2]$-module $M$ alternate (up to natural isomorphism) between the kernel and the cokernel of the norm map $\id + \sigma \colon M_{C_2} \rightarrow M^{C_2}$. From the exact sequence
 \[
 \xymatrix{0 \ar[r] & M^{C_2} \ar[r] & M \ar[r]^-{1-\sigma} & M \ar[r] & M_{C_2} \ar[r] & 0}
 \]
 it follows that the norm map on the $C_2$-module $(M, x \mapsto -x)$ is given by the zero map $0 \colon M \slash 2 \rightarrow M^{C_2}$. Since $M^{C_2} \rightarrow M$ is monic, it suffices to check that the two morphisms
 \[
 \bigl( \begin{smallmatrix} \bar{\varphi} \\ \bar{\lambda_S} \end{smallmatrix} \bigr) \colon A \slash 2 \rightarrow B \slash 2 \oplus A_S \slash 2 
 \quad \text{and} \quad
 \bigl( \begin{smallmatrix} \varphi \\ \lambda_S \end{smallmatrix} \bigr) \colon A \rightarrow B \oplus A_S
\]
 are monic. The left map is monic by assumption and the right map is monic since the patching diagram associated to $\varphi$ is cartesian (by definition of analytic isomorphisms).
 \end{proof}
 
 \begin{rmk}
 We have used the special form of our module with involution $M=-A$ in two places: first, we used the fact that the duality induced by $M$ on $K_0(A)$ is given by $\mathrm{Hom}_A(-,A)$. For this, we need the underlying module of $M$ to be $A$. Secondly, we used the special form of the involution $\sigma$ to check Condition~(v) above. If we had used a different involution, we would have to adapt the condition on $\varphi$ resulting from Condition~(v) of \cite[Proposition~4.4.21]{CALMES_ET_AL_II} accordingly.
 \end{rmk}
 
 The above result does not directly lead to new examples of functors which satisfy weak analytic excision since not all the categories appearing in the square consist of \emph{all} perfect complexes. However, it turns out that in certain cases, the groups $\pi_i\bigl( \mathrm{GW}  (\ca{D}^{p}(A),q_{-A}^{\geq m}) \bigr)$ are insensitive to this distinction when $i \geq 1$.
 
 For $m=0$, $1$, and $2$, the Poincar\'{e} structures $q_{-A}^{\geq m}$ are closely linked to certain kinds of quadratic and bilinear forms on projective $A$-modules. Specifically, applying \cite[Definition~4.2.20]{CALMES_ET_AL_I} to the invertible $\mathbb{Z}$-module $-A$ with involution over $A$, we obtain the three functors
 \[
q_{\mathrm{proj}}^{-s},\; q_{\mathrm{proj}}^{-\mathrm{ev}},\; q_{\mathrm{proj}}^{-q} \colon \Proj(A)^{\op} \rightarrow \Ab
 \]
 which send a projective $A$-module $P$ to the abelian group of \emph{skew-symmetric}, \emph{alternating} (=negative even), respectively \emph{negative quadratic} forms, see \cite[Remark~4.2.21]{CALMES_ET_AL_I}). By \cite[Proposition~4.2.18]{CALMES_ET_AL_I}, these three functors have derived quadratic functors, which by \cite[Proposition~4.2.22]{CALMES_ET_AL_I} are given by $q^{\geq 0}_{-A}$, $q^{\geq 1}_{-A}$, and $q^{\geq 2}_{-A}$ respectively.
 
 The main result of \cite{HEBESTREIT_STEIMLE} therefore allows us to relate the abstract definition of $\mathrm{GW}\bigl(\ca{D}^p(A),q_{-A}^{\geq 1}\bigr)$ to more classical definitions of higher $K$-theory of projective modules equipped with alternating bilinear forms. We will discuss this equivalence in detail in \S \ref{section:symplectic}.
\section{Weak analytic excision and \texorpdfstring{$\mathbb{A}^1$}{A1}-invariance}\label{section:excision}

 Throughout this section, we fix a commutative ring $R$ and we consider various functors from the category $\CAlg_R$ of commutative $R$-algebras to sets, pointed sets, or (abelian) groups. For each $n \in \mathbb{N}$, we write $\iota \colon A \rightarrow A[x_1, \ldots, x_n]$ for the inclusion of constant polynomials and $\pi \colon A[x_1, \ldots, x_n] \rightarrow A$ for the unique $A$-homomorphism which sends each $x_1$ to $0$ (in accordance with Notation~\ref{notation:Vorst_patching_diagrams}). 
 
 Given a functor $F \colon \CAlg_R \rightarrow \Set$ and an element $\sigma \in F(A[x_1, \ldots, x_n])$, we say that \emph{$\sigma$ is extended from $A$} if $\sigma$ lies in the image of $\iota_{\ast} \colon FA \rightarrow F(A[x_1, \ldots, x_n])$. We write $\ca{E}_n^F \subseteq \CAlg_R$ for the full subcategory of commutative $R$-algebras for which $\iota_{\ast} \colon FA \rightarrow F(A[x_1, \ldots, x_n])$ (equivalently $\pi_{\ast}$) is a bijection. In other words, $A \in \ca{E}_n^F$ if all $\sigma \in F(A[x_1, \ldots, x_n])$ are extended from $A$.
 
 Note that $\ca{E}_{n+1}^F$ is contained in $\ca{E}_n^F$: if the extension of $\sigma \in F(A[x_1, \ldots, x_n])$ to $F(A[x_1, \ldots, x_{n+1}])$ is extended from $A$, then so is $\sigma$. The commutative $R$-algebras in the intersection $\ca{E}_{\infty}^F \defl \bigcap_{n \geq 1} \ca{E}_n^F$ are called \emph{$F$-regular}.
 
 Note that the question whether or not all fields are $P$-regular is precisely Serre's problem on projective modules over polynomial rings, while the Bass--Quillen conjecture is the statement that all regular rings lie in $\ca{E}_1^P$.
 
 In this section, we will discuss three key principles which imply that the categories $\ca{E}_n^F$ have some good closure properties. In \S \ref{section:monic}, we will discuss a fourth principle which makes it possible to identify certain basic examples of rings in $\ca{E}_{\infty}^F$.
 
 Given a multiplicative set $S \subseteq A$ and $\sigma \in FA$, we use the notation $\sigma_S \in F(A_S)$ for $(\lambda_S)_{\ast} \sigma$. We also use this notation for localization at prime ideals, possibly of a polynomial extension, if no confusion is possible. For example, given $\sigma \in F(A[x])$, we write $\sigma_{\mathfrak{p}} \in F(A_{\mathfrak{p}}[x])$ for $(\lambda_{A \setminus \mathfrak{p}})_{\ast} \sigma$.
 
 \begin{dfn}\label{dfn:principles_QR}
 Let $F \colon \CAlg_R \rightarrow \Set$ be a functor. We say that \emph{$F$ satisfies the Quillen principle $(\mathrm{Q})$} if
 \begin{itemize}
  \item[$(\mathrm{Q})$] For all $A \in \CAlg_R$ and all $\sigma \in F(A[x])$, if $\sigma_{\mathfrak{m}} \in F(A_{\mathfrak{m}}[x])$ is extended from $A_{\mathfrak{m}}$ for all maximal ideals $\mathfrak{m} \subseteq A$, then $\sigma$ is exteded from $A$
 \end{itemize}
 holds. We say that \emph{$F$ satisfies the Roitman principle $(\mathrm{R})$} if
 \begin{itemize}
 \item[$(\mathrm{R})$] For all $A \in \ca{E}_1^F$, if $S \subseteq A$ is multiplicative, then $A_S \in \ca{E}_1^F$
 \end{itemize}
 holds.
 \end{dfn}
 
 The functors $M, P \colon \CAlg_R \rightarrow \Set$ sending $A$ to the set of isomorphism classes of finitely presentable (respectively finitely generated projective) $A$-modules satisfy $(\mathrm{Q})$ by \cite[Theorem~1']{QUILLEN}. The functor $P$ satisfies $(\mathrm{R})$ by \cite[Proposition~2]{ROITMAN_POLYNOMIAL}. These principles imply corresponding statements for polynomial extensions in several variables.
 
 \begin{prop}\label{prop:R_several_variables}
 If $F \colon \CAlg_R \rightarrow \Set$ satisfies $(\mathrm{R})$, then it satisfies
 \begin{itemize}
 \item[$(\mathrm{R}_n)$] For all $A \in \ca{E}_n^F$, if $S \subseteq A$ is multiplicative, then $A_S \in \ca{E}_n^F$
 \end{itemize}
 \end{prop}
 
 \begin{proof}
 This can be proved by a straightforward induction on $n$. Since $(\mathrm{R}_1)=(\mathrm{R})$, we can assume that $n>1$ and that $(\mathrm{R}_{n-1})$ holds. Note that $A \in \ca{E}_n^F$ implies that $A[x_1, \ldots x_{n-1}] \in \ca{E}_1^F$, so from $(\mathrm{R})$ it follows that $A_S[x_1, \ldots, x_{n-1}] \in \ca{E}_1^F$. This reduces the problem to showing that $A_S \in \ca{E}_{n-1}^F$. Since $A \in \ca{E}_n^F \subseteq \ca{E}_{n-1}^F$, this follows from the inductive assumption.
\end{proof}
 
 The situation for the Quillen principle is a bit more complicated since the inductive assumption does not directly imply that $A[x_1, \ldots, x_{n-1}]$ satisfies the relevant conditions (since we do not have any information concerning the local rings of $A[x_1,\ldots, x_{n-1}]$). The implication does hold under the additional assumption that $F$ is finitary. The proof given in \cite[Theorem~V.1.6]{LAM} for $F=M$ works for any $F$ satisfying $(\mathrm{Q})$.
 
\begin{prop}\label{prop:Q_several_variables}
 Let $F \colon \CAlg_R \rightarrow \Set$ be finitary. If $F$ satisfies $(\mathrm{Q})$, then the following hold for all $n \geq 1$:
 \begin{itemize}
 \item[$(\mathrm{A}_n)$] For each $A \in \CAlg_R$ and each $\sigma \in F(A[x_1, \ldots, x_n])$, the set $Q(\sigma) \subseteq A$ consisting of all $f \in A$ such that $\sigma_f \in F(A_f[x_1, \ldots, x_n])$ is extended from $A_f$ is an ideal of $A$;
 \item[$(\mathrm{Q}_n)$] If $\sigma \in F(A[x_1, \ldots, x_n])$ and $\sigma_{\mathfrak{m}} \in F(A_{\mathfrak{m}}[x_1, \ldots x_n])$ is extended from $A_{\mathfrak{m}}$ for all maximal ideals $\mathfrak{m} \subseteq A$, then $\sigma$ is extended from $A$.
 \end{itemize}
\end{prop}

\begin{proof}
Note that $(\mathrm{Q}_1)$ is precisely $(\mathrm{Q})$. We first show that these two statements are equivalent. To see that $(\mathrm{Q}_n) \Rightarrow (\mathrm{A}_n)$, note that $A \cdot Q(\sigma) \subseteq Q(\sigma)$ since the localization of an extended object is extended. For $f_0, f_1 \in Q(\sigma)$, we can apply $(\mathrm{Q}_n)$ to the $R$-algebra $A_{f_0 + f_1}$ and $\sigma_{f_0 + f_1}$ to conclude that $f_0 + f_1 \in Q(\sigma)$.

 To see that $(\mathrm{A}_n) \Rightarrow (\mathrm{Q}_n)$, we fix a maximal ideal $\mathfrak{m} \subseteq A$ and an object $\sigma$ satisfying the premise of $(\mathrm{Q}_n)$. Since $\sigma_{\mathfrak{m}}$ is extended and $F$ is finitary, it follows that there exists an $f \notin \mathfrak{m}$ such that $\sigma_f \in F(A_f[x_1,\ldots,x_n])$ is extended from $A_f$, that is, $f \in Q(\sigma)$. Since $Q(\sigma)$ is an ideal by $(\mathrm{A}_n)$, it follows that $Q(\sigma)=A$ since the maximal ideal was arbitray. Thus $\sigma=(\lambda_1)_{\ast} \sigma$ is extended, as claimed.
 
 Finally, we show the implication $(\mathrm{A}_1) \Rightarrow (\mathrm{A}_n)$ by induction. Since $(\mathrm{Q})=(\mathrm{Q}_1) \Leftrightarrow (\mathrm{A}_1)$, we can assume that $n >1$ and that $(\mathrm{A}_{n-1}) \Leftrightarrow (\mathrm{Q}_{n-1})$ holds. Since $A \cdot Q(\sigma) \subseteq Q(\sigma)$, it remains to check that $f_0 + f_1 \in Q(\sigma)$ if $f_0, f_1 \in Q(\sigma)$.
 
 Let $f=f_0+f_1$. By applying $(\mathrm{A}_1)$ to the $R$-algebra $B=A_f[x_1, \ldots, x_{n-1}]$, we find that $\sigma_f \in F(B[x_n])$ is extended from some $\tau \in F(B)$ (since $f_0$ and $f_1$ lie in the corresponding Quillen ideal $Q(\sigma_f)$). Writing $\pi \colon B[x_n] \rightarrow B$ for the $B$-homomorphism sending $x_n$ to $0$, we find $\tau=\pi_{\ast} \iota_{\ast} \tau=\pi_{\ast} \sigma_f$, hence $\tau_{f_i}$ is extended from $(A_f)_{f_i}$ for $i=0,1$. Since $f_0$ and $f_1$ are comaximal in $A_f$, this implies that for all maximal ideals $\mathfrak{m} \subseteq A_f$, the object $\tau_{\mathfrak{m}} \in F\bigl( (A_f)_{\mathfrak{m}}[x_1, \ldots, x_{n-1}] \bigr)$ is extended from $(A_f)_{\mathfrak{m}}$. Applying the induction assumption $(\mathrm{Q}_{n-1}) \Leftrightarrow (\mathrm{A}_{n-1})$, we find that $\tau$ is extended from $A_f$. Thus $\sigma_f=\iota_{\ast} \tau$ is also extended from $A_f$, so $f \in Q(\sigma)$, as claimed.
\end{proof}

 The main result of this section establishes a strong connection between weak analytic excision discussed in \S \ref{section:analytic} and the two principles $(\mathrm{Q})$ and $(\mathrm{R})$. It turns out that weak analytic excision always implies $(\mathrm{Q})$, and that it implies $(\mathrm{R})$ under some mild additional assumptions.
 
 \begin{dfn}\label{dfn:F_contractible_and_transitive_group_action}
 Let $F \colon \CAlg_R \rightarrow \Set_{\ast}$ be a functor. A commutative $R$-algebra $A$ is called \emph{$F$-contractible} if $FA \cong \ast$ is a singleton set.
  
 We say that \emph{$F$ admits a natural transitive group action} if there exists a functor $G \colon \CAlg_R \rightarrow \Grp$ and an action $G(A) \times F(A) \rightarrow F(A)$, natural in $A$, turning $F(A)$ into a transitive $G(A)$-set. The action is not required to preserve the base-point of the pointed set $F(A)$. 
 \end{dfn}
 
\begin{rmk}\label{rmk:group_valued_implies_natural_action}
 If $F \colon \CAlg_R \rightarrow \Grp$ is itself a group valued functor, then the underlying functor to pointed sets admits a natural transitive group action. Indeed, in this case we can take $G=F$ and the action given by left multiplication.
\end{rmk}

 \begin{thm}\label{thm:weak_excision_implies_QR}
 Let $F \colon \CAlg_R \rightarrow \Set_{\ast}$ be a finitary functor which satisfies weak analytic excision. Then $F$ satisfies the Quillen principle $(\mathrm{Q}_n)$ for all $n \geq 1$. If in addition one of the following conditions holds:
 \begin{enumerate}
 \item[(A)] all commutative $R$-algebras which are local rings are $F$-contractible, or
 \item[(B)] the functor $F$ admits a natural transitive group action,
 \end{enumerate}
 then $F$ also satisfies $(\mathrm{R}_n)$ for all $n \geq 1$. 
 \end{thm}
 
 For example, for each $r \geq 1$, the functor $P_r \colon \CAlg_R \rightarrow \Set_{\ast}$ satisfies weak analytic excision by Corollary~\ref{cor:weak_analytic_excision_for_P_and_P_r} and all local rings are $P_r$-contractible since there is a unique free module of rank $r$. It follows from the above theorem that $P_r$ satisfies $(\mathrm{Q}_n)$ and $(\mathrm{R}_n)$ for all $n$, so we recover the motivating examples due to Quillen and Roitman mentioned above.
 
 From now until the end of the proof of Theorem~\ref{thm:weak_excision_implies_QR} we fix a functor $F \colon \CAlg_R \rightarrow \Set$. We will in fact prove a sequence of propositions which slightly refine the above theorem. For example, we will indicate precisely what kind of weak excision suffices for each of the statements of Theorem~\ref{thm:weak_excision_implies_QR}.
 
 Let $B$ be a commutative $R$-algebra and let $f \in B$. Recall from Notation~\ref{notation:Vorst_patching_diagrams} that we write $\iota \colon B \rightarrow B[y]$ for the inclusion of the constant polynomials and $\pi \colon B[y] \rightarrow B$ for $B$-algebra homomorphism sending $y$ to $0$, and that
 \[
 \xymatrix{ B \pb{B_f} B_f[y] \ar[r]^-{\lambda^{\prime}}  \ar[d]_{\pi^{\prime}} & B_f[y] \ar[d]^{\pi} \\ B \ar[r]_-{\lambda_f} & B_f }
 \]
 denotes the resulting pullback diagram. Moreover, if $\iota^{\prime} \colon B \rightarrow B \pb{B_f} B_f[y]$ is the unique homomorphism with $\pi^{\prime} \iota^{\prime}=\id$ and $\lambda^{\prime} \iota^{\prime}=\iota \lambda_f$, then
 \[
 \xymatrix{B \ar[d]_{\iota^{\prime}} \ar[r]^{\lambda_f} & B_f \ar[d]^{\iota} \\ 
 B \pb{B_f} B_f[y] \ar[r]^-{\lambda^{\prime}} & B_f[y] }
 \]
 is also a pullback diagram (by the cancellation law for pullback squares). From Proposition~\ref{prop:Vorst_patching_diagrams} we know that $\iota^{\prime}$ is a flat homomorphism and that it is an analytic isomorphism along $f$.
 
 For any $i \geq 0$, we let $\hat{\kappa_i} \colon B[y] \rightarrow B_f[y]$ be the unique homomorphism given by $\lambda_f$ on $B$ and sending $y$ to $f^{-i} y$. We get a unique induced homomorphism $\kappa_i \colon B[y] \rightarrow B\pb{B_f}B_f[y]$ such that $\pi^{\prime} \kappa_i=\pi$ and $\lambda^{\prime} \kappa_i=\hat{\kappa_i}$ hold. For $i=0$ we have $\hat{\kappa_0}=\lambda_f$ and we write $\kappa$ for $\kappa_0$.
 
 \begin{lemma}\label{lemma:Vorst_diagram_filtered_colimit}
 The homomorphisms $\kappa_i \colon B[y] \rightarrow B \pb{B_f} B_f[y]$ exhibit $B \pb{B_f} B_f[y]$ as colimit of the chain
 \[
 \xymatrix{B[y] \ar[r]^{\mu_f} & B[y] \ar[r]^{\mu_f} & B[y] \ar[r]^{\mu_f} & \ldots } \smash{\rlap{,}}
 \]
 where $\mu_f$ denotes the unique $B$-algebra homomorphism with $\mu_f(y)=f \cdot y$.
 \end{lemma}
 
 \begin{proof}
 From the construction of $\kappa_i$ above we find that $\kappa_{i+1} \circ \mu_f=\kappa_i$ holds for all $i \geq 0$. We can check the colimit property on the underlying diagram of sets, where we can compose the diagram with the bijection
 \[
 B \pb{B_f} B_f[x] \rightarrow B \times \{\, p(y) \in B_f[y] \mid p(0)=0 \, \}
 \]
 sending $\bigl(b,p(y) \bigr)$ to $\bigl(b,p(y)-p(0) \bigr)$. Under this bijection, $\kappa_i$ sends $p(y)$ to the pair $\bigl(p(0),p(f^{-i}y)-p(0) \bigr)$. From this description it follows that the $\kappa_i$ are jointly surjective, so it only remains to check that whenever $\kappa_i \bigl(p(y)\bigr)=0$, there exists some $n \in \mathbb{N}$ such that $\mu_{f^n}\bigl(p(y)\bigr)=0$. But $\kappa_i\bigl(p(y) \bigr)=0$ implies $p(0)=0$ and $p(f^{-i} y)=0$ in $B_f[y]$, so the coefficients of $p(y)$ vanish after multiplication with a sufficiently high power of $f$. Thus $p(f^n y)=0$ in $B[y]$ for $n$ sufficiently large, which establishes the claim. 
\end{proof}

 The proof that $F$ satisfies $(\mathrm{Q})$ whenever it satisfies weak analytic excision closely follows a similar argument due to Vorst, see \cite[Theorem~1.2]{VORST_POLYNOMIAL}. The difference is that the functors we consider are valued in sets rather than abelian groups, so we have to adapt the constructions accordingly, working with kernel pairs and equalizers instead of kernels, for example. The basic structure of the proof is however very similar: to an object $\sigma \in F(A[x])$, we associate an ideal $I_{\sigma} \subseteq A$ with the following two properties:
 \begin{enumerate}
 \item[(1)] If $1 \in I_{\sigma}$, then $\sigma$ is extended;
 \item[(2)] If $f \in A$ is an element such that $\sigma_f \in F(A_f[x])$ is extended, then there exists some $n \in \mathbb{N}$ such that $f^n \in I_\sigma$.
 \end{enumerate}
 
 Using the fact that $F$ is finitary, it is then straightforward to show that $F$ satisfies the Quillen principle $(\mathrm{Q})$. Establishing Property~(2) involves a slightly intricate argument and uses the filtered colimit describtion of Lemma~\ref{lemma:Vorst_diagram_filtered_colimit} above; Property~(1) on the other hand is a straightforward consequence of the definition of $I_{\sigma}$.
 
 For any commutative $R$-algebra $A$, and any element $a \in A$, we write
 \[
 \varphi_a^{A} \colon A[x] \rightarrow A[x,y]
 \]
 for the unique $A$-algebra homomorphism sending $x$ to $x+ay$, that is, $\varphi_{a}^{A} \bigl(p(x) \bigr)=p(x+ay)$ for all $p(x) \in A[x]$.
 
 \begin{dfn}\label{dfn:Vorst_ideal}
 Let $A$ be a commutative $R$-algebra and let $\sigma \in F(A[x])$. Then we write $I_{\sigma}$ for the subset
 \[
 I_{\sigma} \defl \{\, a \in A \mid (\varphi_a^{A})_{\ast} \sigma=(\varphi_0^{A})_{\ast} \sigma \,\}
 \]
 of $A$.
 \end{dfn}
 
 \begin{lemma}\label{lemma:I_sigma_ideal}
 For any $\sigma \in F(A[x])$, the set $I_{\sigma}$ is an ideal of $A$. Moreover, if $1 \in I_{\sigma}$, then $\sigma$ is extended from $A$.
 \end{lemma}
 
 \begin{proof}
 Since the $R$-algebra $A$ is fixed, we suppress the superscripts of $\varphi_{a}^{A}$ and write $\varphi_{a}$ instead until the end of this proof. We first show the second claim. If $1 \in I_{\sigma}$, then $(\varphi_1)_{\ast} \sigma = (\varphi_0)_{\ast} \sigma$. If we write $\psi \colon A[x,y] \rightarrow A[x]$ for the $A$-homomorphism sending $x$ to $x$ and $y$ to $-x$, then we have
 \[
 \psi \varphi_1\bigl(p(x)\bigr)=\psi \bigl(p(x+y) \bigr)=p(0)=\iota \pi\bigl(p(x)\bigr)
 \]
 and
 \[
 \psi \varphi_0\bigl(p(x)\bigr)=\psi \bigl(p(x) \bigr)=p(x)=\id\bigl(p(x)\bigr) \smash{\rlap{,}}
 \]
 so $(\iota \pi)_{\ast} \sigma=\sigma$, as claimed.
 
 To show that $I_{\sigma}$ is an ideal, we need two auxiliary $A$-homomorphisms. For any $a^{\prime} \in A$, we let
 \[
 \mu_{a^{\prime}} \colon A[x,y] \rightarrow A[x,y] \quad \text{and} \quad \alpha_{a^{\prime}} \colon A[x,y] \rightarrow A[x,y]
 \]
 be the $A$-homomorphisms given by $\mu_{a^{\prime}}(x)=x$, $\mu_{a^{\prime}}(y)=a^{\prime} y$ and $\alpha_{a^{\prime}}(x)=x+a^{\prime} y$, $\alpha_{a^{\prime}}(y)=y$ respectively.
 
 Since $\mu_{a^{\prime}} \varphi_{a^{\prime}}=\varphi_{a^{\prime} a}$ and $\mu_{a^{\prime}} \varphi_0=\varphi_0$, it follows that $a^{\prime} a \in I_\sigma$ whenever $a \in I_{\sigma}$. It remains to show that $I_{\sigma}$ is closed under addition, so let $a, a^{\prime} \in I_{\sigma}$. Since
 \[
 \alpha_{a^{\prime}} \varphi_a\bigl(p(x)\bigr)=\alpha_{a^{\prime}}\bigl(p(x+ay) \bigr)=p(x+a^{\prime}y+ay)=\varphi_{a^{\prime} + a}\bigl(p(x)\bigr) \smash{\rlap{,}}
 \]
 we have in particular $\alpha_{a^{\prime}} \varphi_0^{A}=\varphi_{a^{\prime}}^{A}$. Thus
 \[
 (\varphi_{a+a^{\prime}})_{\ast} \sigma=(\alpha_{a^{\prime}} \varphi_a)_{\ast} \sigma=(\alpha_{a^{\prime}} \varphi_0)_{\ast} \sigma=(\varphi_{a^{\prime}})_{\ast} \sigma=(\varphi_0)_{\ast} \sigma \smash{\rlap{,}}
 \]
 where the second and last equality follow from the definition of $I_{\sigma}$. Thus $a+a^{\prime} \in I_{\sigma}$, which concludes the proof that $I_{\sigma}$ is an ideal.
 \end{proof}
 
 We come to the heart of the argument, where the link between the excision property and the Quillen principle $(\mathrm{Q})$ finally appears. We need to introduce a bit more terminology. Given a function of sets $\alpha \colon X \rightarrow Y$, the set
 \[
 \mathrm{Kep}(\alpha) \defl \{\, (x,x^{\prime}) \in X\times X \mid \alpha(x)=\alpha(x^{\prime})  \,\}
 \]
 is called the \emph{kernel pair} of $\alpha$. For $i=1,2$, we write $p_i \colon \mathrm{Kep}(\alpha) \rightarrow X$ for the projection on the first respectively second factor.
 
 In the particular case where $\alpha=(\lambda^{\prime})_{\ast} \colon F(B \pb{B_f} B_f[y]) \rightarrow F(B_f[y])$, we write $\mathrm{Kep}^e(\lambda)$ for the subset of all pairs $(\tau,\tau^{\prime}) \in \mathrm{Kep}(\lambda)$ such that $\lambda^{\prime}_{\ast} \tau=\lambda^{\prime}_{\ast} \tau^{\prime} \in F(B_f[y])$ is extended from $B_f$. Here and in the sequel we often omit the symbol $\ast$ in diagrams and in the names of objects in order to simplify the notation.
 
 The diagrams defining $B \pb{B_f} B_f[y]$ and $\iota^{\prime}$ (see Notation~\ref{notation:Vorst_patching_diagrams}) induce functions $\mathrm{Kep}(\lambda_f) \rightarrow \mathrm{Kep}(\lambda^{\prime})$ and $\mathrm{Kep}(\lambda^{\prime}) \rightarrow \mathrm{Kep}(\lambda_f)$ which compose to the identity by functoriality of kernel pairs. Moreover, the first of these functions factors through $\mathrm{Kep}^e(\lambda^{\prime})$. Thus we get induced functions denoted by dashed arrows
 \[
 \xymatrix{ \mathrm{Kep}(\lambda_f) \ar@{-->}[d]_{\iota^{\prime \prime}} \ar[r]^-{p_i} & F(B) \ar[d]^{\iota^{\prime}} \ar[r]^-{\lambda_f} & F(B_f) \ar[d]^{\iota} \\
\mathrm{Kep}^e(\lambda^{\prime}) \ar@{-->}[d]_{\pi^{\prime \prime}} \ar[r]^-{p_i} & F(B \pb{B_f} B_f[y]) \ar[d]^{\pi^{\prime}} \ar[r]^-{\lambda^{\prime}} & F(B_f[y]) \ar[d]^{\pi} \\
  \mathrm{Kep}(\lambda_f) \ar[r]^-{p_i} & F(B) \ar[r]^-{\lambda_f} & F(B_f) }
 \]
 such that the above diagram commutes for $i=1,2$ and $\pi^{\prime \prime} \iota^{\prime \prime}=\id$.
 
 \begin{lemma}\label{lemma:I_sigma_and_localization}
 Let $A$ be a commutative $R$-algebra, let $f \in A$ and let $\sigma \in F(A[x])$ be an object such that $\sigma_f \in F(A_f[x])$ is extended from $A_f$. Assume that $F$ is finitary. If the function $\pi^{\prime \prime}$ above is injective in the case where $B=A[x]$, then there exists an $n \in \mathbb{N}$ such that $f^n \in I_{\sigma}$.
 \end{lemma}
 
 \begin{proof}
 Let $K \subseteq F(A[x])$ denote the subset of all $\sigma$ such that $\sigma_f \in F(A_f[x])$ is extended from $A_f$ and fix an element $a \in A$. Since any extended object $\tau$ is necessarily extended from $\pi_{\ast} \tau$, we find that the top square of the diagram
 \[
 \xymatrix{ F(A[x]) \ar[rd]^{\lambda_f} \ar[d]_{\varphi_{a}^{A}} &  K \ar[l]_-{\mathrm{incl}} \ar[rd]^{\pi \lambda_f} \\
 F(A[x,y]) \ar[d]_{\kappa} \ar[rd]^{\lambda_f} & F(A_f[x]) \ar[d]^{\varphi_{a}^{A_f}} & F(A_f) \ar[l]_-{\iota} \ar[d]^{\iota} \\
 F(A[x] \pb{A_f[x]} A_f[x,y]) \ar[r]_-{\lambda^{\prime}} & F(A_f[x,y]) & F(A_f[x]) \ar[l]_-{\iota}
 }
 \]
 is commutative. The two remaining squares commute for all $a \in A$ since the corresponding diagrams in $\CAlg_R$ are already commutative. Similarly, the lower triangle commutes by definition of $\kappa=\kappa_0$ above Lemma~\ref{lemma:Vorst_diagram_filtered_colimit}. Composing the arrows from $K$ to $F(A_f[x,y])$ clockwise, we see that the counterclockwise composite is independent of $a \in A$. Applying this to the elements $a=1$ and $a=0$, we find that there exists a unique function $\psi \colon K \rightarrow \mathrm{Kep}^e(\lambda^{\prime})$ such that the two equations $p_1 \psi(\sigma)=(\kappa \varphi_1^{A})_{\ast} \sigma$ and $p_2 \psi(\sigma)=(\kappa \varphi_{0}^{A})_{\ast} \sigma$ hold for all $\sigma \in K$. That this function takes values in $\mathrm{Kep}^e \subseteq \mathrm{Kep}$ also follows from the commutativity of the above diagram.
 
 We claim that $\pi^{\prime \prime} \psi(\sigma)=(\sigma, \sigma) \in \mathrm{Kep}(\lambda_f)$, or equivalently that the equation $\pi^{\prime} p_i \psi(\sigma)=p_i \pi^{\prime \prime} \psi(\sigma)=\sigma$ holds for $i=1,2$. From the construction of $\kappa$ we know that $\pi^{\prime} \kappa=\pi$ is the $A$-homomorphism $A[x,y] \rightarrow A[x]$ sending $y$ to $0$ (and $x$ to $x$). Thus $\pi^{\prime} \kappa \varphi_a^{A}=\id_{A[x]}$ for all $a \in A$ and the claim follows from the defining formulas for $\psi$ above. Since we also have $(\sigma,\sigma)=\pi^{\prime \prime} \iota^{\prime \prime} \bigl((\sigma, \sigma) \bigr)$, the assumption that $\pi^{\prime \prime}$ is injective implies that the equation
 \[
 \psi(\sigma) = \iota^{\prime \prime}\bigl((\sigma,\sigma) \bigr)=(\iota^{\prime}_{\ast} \sigma, \iota^{\prime}_{\ast} \sigma )
 \]
 holds for all $\sigma \in K$. Applying $p_i$ for $i=1,2$ we find that the equality
 \[
 (\kappa \varphi_1^{A})_{\ast} \sigma=(\kappa \varphi_0^{A})_{\ast} \sigma
 \]
 holds in $F(A[x] \pb{A_f[x]} A_f[x,y])$.
 
 Since $F$ is finitary, it preserves the filtered colimit
 \[
 \xymatrix{A[x,y] \ar[r]^{\mu_f} & A[x,y] \ar[r]^{\mu_f} & A[x,y] \ar[r]^{\mu_f} & \ldots}
 \]
 of Lemma~\ref{lemma:Vorst_diagram_filtered_colimit}. Since $(\kappa)_{\ast}=(\kappa_0)_{\ast}$ is the first morphism of a colimit cocone, it follows that there exists some $n \in \mathbb{N}$ such that $(\mu_{f^n} \varphi_1^{A})_{\ast} \sigma = (\mu_{f^n} \varphi_0^{A})_{\ast} \sigma$ holds in $F(A[x,y])$. But $\mu_{f^n} \varphi_1^{A}=\varphi_{f^n}^{A}$ and $\mu_{f^n} \varphi_0^{A}=\varphi_0^{A}$. From the definition of $I_{\sigma}$ it follows that $f^n \in I_{\sigma}$, as claimed (see Definition~\ref{dfn:Vorst_ideal}).
 \end{proof}
 
 With this final lemma in hand, we can prove that $F$ satisfies the Quillen property $(\mathrm{Q})$ if it has suitable weak excision properties.
 
 \begin{prop}\label{prop:excision_implies_Q}
 Let $F \colon \CAlg_R \rightarrow \Set$ be a finitary functor and let $A$ be a commutative $R$-algebra. Let $\sigma \in F(A[x])$ be an object such that for each maximal ideal $\mathfrak{m} \subseteq A$, the localization $\sigma_{\mathfrak{m}} \in F(A_{\mathfrak{m}}[x])$ is extended. If the diagram
 \[
 \xymatrix{ F(A[x]) \ar[d]_{\iota^{\prime}} \ar[r]^-{\lambda_f} & F(A_f[x]) \ar[d]^{\iota} \\ 
 F(A[x] \pb{A_f[x]} A_f[x,y]) \ar[r]_-{\lambda^{\prime}} & F(A_f[x,y])}
 \]
 is a weak pullback diagram for all $f \in A$, then $\sigma$ is extended. In particular, if the above diagram is a weak pullback for all commutative $R$-algebras $A$ and all $f \in A$, then $F$ satisfies the Quillen principle $(\mathrm{Q})$.
 \end{prop}
 
 \begin{proof}
 We claim that the assumptions imply that $\pi^{\prime \prime} \colon \mathrm{Kep}^e(\lambda^{\prime}) \rightarrow \mathrm{Kep}(\lambda_f)$ of Lemma~\ref{lemma:I_sigma_and_localization} is injective for all $f \in A$. Since $\pi^{\prime \prime}$ has a section
 \[
 \iota^{\prime \prime} \colon \mathrm{Kep}(\lambda_f) \rightarrow \mathrm{Kep}^e(\lambda^{\prime}) \smash{\rlap{,}}
 \]
 it suffices to check that $\iota^{\prime \prime}$ is surjective (for then $\pi^{\prime \prime}$ and $\iota^{\prime \prime}$ are mutually inverse bijections). 
 
 To see this, let $(\beta_1, \beta_2)$ be a pair of objects of $F(A[x] \pb{A_f[x]} A_f[x,y])$ such that $(\lambda^{\prime})_{\ast} \beta_1=(\lambda^{\prime})_{\ast} \beta_2$ and such that this object of $F(A_f[x,y])$ is extended from $A_f[x]$. In other words, we assume that there exists some $\tau \in F(A_f[x])$ such that $(\lambda^{\prime})_{\ast} \beta_i = \iota_{\ast} \tau$ holds; such a pair $(\beta_1, \beta_2)$ is precisely a generic element of $\mathrm{Kep}^e(\lambda^{\prime})$.
 
 Since
 \[
 \xymatrix{ F(A[x]) \ar[d]_{\iota^{\prime}} \ar[r]^-{\lambda_f} & F(A_f[x]) \ar[d]^{\iota} \\ 
 F(A[x] \pb{A_f[x]} A_f[x,y]) \ar[r]_-{\lambda^{\prime}} & F(A_f[x,y])}
 \]
 is a weak pullback diagram, there exist objects $\alpha_1, \alpha_2 \in F(A[x])$ such that $(\iota^{\prime})_{\ast} \alpha_i  = \beta_i$ and $(\lambda_f)_{\ast} \alpha_i=\tau$ holds for $i=1,2$. It follows that $(\alpha_1,\alpha_2) \in \mathrm{Kep}(\lambda_f)$ and that $\iota^{\prime \prime}(\alpha_1,\alpha_2)=(\beta_1,\beta_2)$ holds. This shows that $\iota^{\prime \prime}$ is indeed surjective, hence that $\pi^{\prime \prime}$ is injective.
 
 Thus Lemma~\ref{lemma:I_sigma_and_localization} is applicable and we have finally established the second property of the ideal $I_{\sigma}$: if $\sigma_f$ is extended, then there exists some $n \in \mathbb{N}$ such that $f^n \in I_{\sigma}$. To show that $\sigma$ is extended, it only remains to check that $I_{\sigma}=A$ (see Lemma~\ref{lemma:I_sigma_ideal}). If $\mathfrak{m} \subseteq A$ is a maximal ideal, then $\sigma_{\mathfrak{m}}$ is extended by assumption. Since $F$ is finitary, there exists some $f \in A \setminus \mathfrak{m}$ such that $\sigma_f$ is extended. Thus some power of $f$ lies in $I_{\sigma}$, so $I_{\sigma}$ is not  contained in $\mathfrak{m}$. Since $\mathfrak{m}$ was arbitrary, we conclude that $I_{\sigma}=A$ and thus that $\sigma$ is extended.
 \end{proof}
 
 \begin{cor}\label{cor:analytic_excision_implies_Q}
 If $F \colon \CAlg_R \rightarrow \Set$ is finitary and satisfies weak analytic excision, then $F$ satisfies the Quillen principle $(\mathrm{Q})$. In fact, it suffices that $F$ sends the patching diagrams associated to each \emph{flat} analytic isomorphism $\varphi \colon A \rightarrow B$ along $S$ to a weak pullback diagram.
 \end{cor}
 
 \begin{proof}
 From Proposition~\ref{prop:Vorst_patching_diagrams} we know that $\iota^{\prime}$ is flat and a weak analytic isomorphism along $f$. Thus Proposition~\ref{prop:excision_implies_Q} is applicable and proves the claim.
 \end{proof}
 
 If we apply this corollary to the functor $M \colon \CAlg_R \rightarrow \Set$, we recover Quillen's original result \cite[Theorem~1']{QUILLEN} in a rather roundabout way: instead of analysing all possible ways of patching, we have merely used the fact that finitely presented modules \emph{can} be patched (albeit in a more general context than the Zariski patching considered by Quillen). The above corollary is of course applicable in many other cases, for example, to all the functors for which weak analytic excision was established in \S \ref{section:analytic}. We also note that the argument above, while inspired by \cite[Theorem~1.2]{VORST_POLYNOMIAL}, also differs in some key aspects from Vorst's proof. For example, Vorst used the fact that $\pi^{\prime} \colon A[x] \pb{A_f[x]} A_f[x,y] \rightarrow A[x]$ is an analytic isomorphism along $(f,f\slash 1)$ (instead of its section $\iota^{\prime}$, which is an analytic isomorphism along $f$), see the proof \cite[Lemma~1.4]{VORST_POLYNOMIAL}.
 
 Showing that weak analytic excision implies that the Roitman priniciple $(\mathrm{R})$ holds is much more straightforward, but the argument does not work at the same level of generality. We need to assume that the functor $F\colon \CAlg_R \rightarrow \Set_{\ast}$ takes values in \emph{pointed} sets. Then we get a new functor $NF \colon \CAlg_R \rightarrow \Set_{\ast}$ given by $NF(A)=\ker\bigl( \pi_{\ast} \colon F(A[x]) \rightarrow F(A) \bigr)$.
 
 \begin{prop}\label{prop:NF_trivial_on_localization}
 Let $F \colon \CAlg_R \rightarrow \Set_{\ast}$ be a finitary functor and let $A$ be a commutative $R$-algebra such that $NF(A)=\ast$. Suppose that the diagram
 \[
 \xymatrix{F(A \pb{A_f} A_f[x] ) \ar[r]^-{\lambda^{\prime}} \ar[d]_{\pi^{\prime}} & F(A_f[x]) \ar[d]^{\pi} \\
 F(A) \ar[r]_-{\lambda_f} & F(A_f) }
 \]
 is a weak pullback diagram for all $f \in A$. Then for any multiplicative subset $S \subseteq A$ we have $NF(A_{S})=\ast$.
 \end{prop}
 
 \begin{proof}
 Fix an element $f \in A$. The fact that $F$ is finitary and Lemma~\ref{lemma:Vorst_diagram_filtered_colimit} imply that the morphisms $(\kappa_i)_{\ast} \colon F(A[x]) \rightarrow F(A \pb{A_f} A_f[x])$ exhibit the target as colimit of the diagram
  \[
  \xymatrix{F(A[x]) \ar[r]^{\mu_f} & F(A[x]) \ar[r]^{\mu_f} & F(A[x]) \ar[r]^{\mu_f} & \ldots}
  \]
  of pointed sets. In particular, the $(\kappa_i)_{\ast}$ are jointly surjective.
  
  Now let $\sigma \in NF(A_f)$ be an arbitrary object. The weak pullback assumption implies that there exists a $\sigma^{\prime} \in F(A \pb{A_f} A_f[x])$ with $(\lambda^{\prime})_{\ast} \sigma^{\prime}=\sigma$ and $(\pi^{\prime})_{\ast} \sigma^{\prime}=\ast$. As noted above, there exists some $i \in \mathbb{N}$ and a $\tau \in F(A[x]) $ such that $\sigma^{\prime}=(\kappa_i)_{\ast} \tau $. Since $\pi^{\prime} \kappa_i=\pi$ by definition, it follows that $(\pi)_{\ast} \tau=\ast$. In other words, we have $\tau \in NF(A)=\ast$, so $\sigma=(\lambda^{\prime} \kappa_i)_{\ast} \tau=\ast$. 
  
  Since $\sigma \in NF(A_f)$ was arbitrary, we have shown that $NF(A_S)=\ast$ if $S$ is generated by a single element $f \in A$. The general case follows from this and the fact that $F$, hence also $NF$, is finitary.
 \end{proof}
 
 To deduce the Roitman principle $(\mathrm{R})$ from this, we need conditions which ensure that $A \in \ca{E}_1^F$ if $NF(A)=\ast$. We will use the following two lemmas for this.
 
 \begin{lemma}\label{lemma:F_contractible_implies_extended_if_trivial_NF}
 Let $F \colon \CAlg_R \rightarrow \Set_{\ast}$ be a functor which satisfies the Quillen principle $(\mathrm{Q})$. Assume that the conclusion of Proposition~\ref{prop:NF_trivial_on_localization} holds: if $NF(A)=\ast$ and $S \subseteq A$ is multiplicative, then $NF(A_S)=\ast$. If all local $R$-algebras are $F$-contractible (see Definition~\ref{dfn:F_contractible_and_transitive_group_action}), then $NF(A)=\ast$ implies that $A \in \ca{E}_1^{F}$, that is, all $\sigma \in F(A[x])$ are extended.
 \end{lemma}
 
 \begin{proof}
 Since $F$ satisfies $(\mathrm{Q})$, it suffices to check that $A_{\mathfrak{m}} \in \ca{E}_1^F$ for all maximal ideals $\mathfrak{m} \subseteq A$. The assumption that $F$ satisfies the conclusion of Proposition~\ref{prop:NF_trivial_on_localization} implies that $NF(A_{\mathfrak{m}})=\ast$. But $F(A_{\mathfrak{m}})=\ast$ by $F$-contractibility of $A_{\mathfrak{m}}$, so we have $NF(A_{\mathfrak{m}})=F(A_{\mathfrak{m}}[x])$. Thus $\pi_{\ast} \colon F(A_{\mathfrak{m}}[x]) \rightarrow F(A_{\mathfrak{m}})$ is the unique bijection between these singleton sets. If follows that the section $\iota_{\ast}$ of $\pi_{\ast}$ is surjective, so $A_{\mathfrak{m}} \in \ca{E}_1^F$, as claimed.
 \end{proof}
 
 \begin{lemma}\label{lemma:transitive_action_implies_extended_if_trivial_NF}
  Let $F \colon \CAlg_R \rightarrow \Set_{\ast}$ be a functor which admits a natural transitive group action (see Definition~\ref{dfn:F_contractible_and_transitive_group_action}). If $A$ is a commutative $R$-algebra such that $NF(A)=\ast$, then $A \in \ca{E}_1^{F}$, that is, all $\sigma \in F(A[x])$ are extended.
 \end{lemma}
 
 \begin{proof}
 Let $G \colon \CAlg_R \rightarrow \Grp$ be a functor and let $G \times F \rightarrow F$ be a natural action which is transitive for each commutative $R$-algebra $A$. Let $\sigma \in F(A[x])$ and choose some $g \in G(A)$ such that $g \cdot \pi_{\ast} \sigma=\ast$. Then $(\iota_{\ast} g) \cdot \sigma$ lies in $NF(A)$ by naturality of the group action. Indeed, we have
 \[
 \pi_{\ast}(\iota_{\ast} g \cdot \sigma)=(\pi_\ast \iota_{\ast} g) \cdot \pi_{\ast} \sigma=g \cdot \pi_{\ast} \sigma= \ast \smash{ \rlap{,}}
 \]
 as claimed.
 
 It follows that $(\iota_{\ast} g) \cdot \sigma=\ast$ by assumption. Thus
 \[
 \sigma=(\iota_{\ast} g^{-1} \cdot \iota_{\ast} g ) \cdot \sigma=(\iota_{\ast} g^{-1}) \cdot \ast=(\iota_{\ast} g^{-1}) \cdot \iota_{\ast} \ast=\iota_{\ast}(g^{-1} \cdot \ast)
 \]
 is extended, as claimed.
 \end{proof}
 
 By Remark~\ref{rmk:group_valued_implies_natural_action}, the above lemma is in particual applicable if $F$ is group-valued.  With this in hand, we can now prove Theorem~\ref{thm:weak_excision_implies_QR}.
  
 \begin{proof}[Proof of Theorem~\ref{thm:weak_excision_implies_QR}]
 Since $F$ satisfies weak analytic excsision, it satisfies the Quillen principle $(\mathrm{Q})$ by Corollary~\ref{cor:analytic_excision_implies_Q}. Since $F$ is finitary, it follows from Proposition~\ref{prop:Q_several_variables} that $F$ also satisfies $(\mathrm{Q}_n)$ for all $n \geq 1$.
 
  It remains to check that $F$ satisfies $(\mathrm{R})$, so let $A \in \ca{E}_1^F$ and let $S \subseteq A$ be a multiplicative set. It follows that $NF(A)=\ast$, so Proposition~\ref{prop:NF_trivial_on_localization} implies that $NF(A_S)=\ast$. From Lemmas~\ref{lemma:F_contractible_implies_extended_if_trivial_NF} and \ref{lemma:transitive_action_implies_extended_if_trivial_NF} it follows that $A_S \in \ca{E}_1^F$. This shows that $(\mathrm{R})$ holds, which implies that $(\mathrm{R}_n)$ holds for all $n \geq 1$ (see Proposition~\ref{prop:R_several_variables}).
 \end{proof}
  
 Results of Lindel \cite[Theorem on p.~319]{LINDEL} and Popescu \cite[Theorem~2.5]{POPESCU_DESINGULARIZATION} imply that all regular rings containing a field are $P$-regular. Poperscu's Theorem makes it possible to reduce this to the case of smooth algebras over a perfect field, in which case Lindel's observation can be applied, which shows that the local rings of such algebras are {\'e}tale neighbourhoods of localizations of polynomial rings. The argument then proceeds by using a patching diagram and induction on the dimension.
 
 It is thus natural to expect that this works for more general functors which send the relevant patching diagrams to weak pullback diagrams. This is not a new observation, such axiomatic treatments appear for example in \cite{COLLIOT-THELENE_OJANGUREN}, later generalized in \cite{ASOK_HOYOIS_WENDT}. A different axiomatic approach appears in \cite{LAVRENOV_SINCHUCK_VORONETSKY}. These results are even applicable in the case where the functor $F$ takes values in pointed sets but not all local rings are $F$-contractible (if one restricts attention to objects which are Zariski-locally trivial). However, the expense for this is that one has to consider patching diagrams which are a bit more complicated than the ones considered in \cite{LINDEL}. In \S \ref{section:pseudoelementary}, we will consider a functor which satisfies weak excision for patching diagrams arising from {\'e}tale neighbourhoods, but weak excision for general Nisnevich squares is not known to hold. For this reason, we present here an axiomatic approach based solely on Lindel's ideas.
 
 The following definition captures the formal aspects of the argument. Recall that a local ring $B$ is an \emph{{\'e}tale neighbourhood} of a local ring subring $(A,\mathfrak{m})$ if the inclusion is essentially {\'e}tale and induces an isomorphism of residue fields (note that this implies that $A \rightarrow B$ is a local homomorphism). From the local structure of {\'e}tale ring maps (see \cite[\href{https://stacks.math.columbia.edu/tag/00UE}{Proposition 00UE}]{stacks-project}) it follows that there exists a monic polynomial $f(t) \in A[t]$ with $f(0) \in \mathfrak{m}$ and $f^{\prime}(0) \notin \mathfrak{m}$ such that $(A[t] \slash f)_{(\mathfrak{m},t)} \cong B$.

\begin{dfn}\label{dfn:local_rings_are_etale_neighbourhoods}
 Let $R^{\prime}$ be a commutative $R$-algebra and let $R^{\prime} \rightarrow C$ be a smooth $R$-homomorphism. We say that \emph{the local rings of $C$ are {\'e}tale neighbourhoods over $R^{\prime}$} if for all $\mathfrak{q} \in \Spec(C)$ with $\mathrm{ht}(\mathfrak{q}) > \mathrm{dim}(R^{\prime})$, there exists an {\'e}tale neighbourhood
 \[
 A \rightarrow C_{\mathfrak{q}}
 \]
 in the category $\CAlg_R$ such that $A$ is the localization of some polynomial algebra over $R^{\prime}$ at some prime ideal.
\end{dfn}

 This definition only applies under rather restrictive assumptions on $R^{\prime}$: to apply Lindel's result, one needs to assume that the extension on residue fields is separable, so ideally one wants all the residue fields of $R^{\prime}$ to be separable. For this reason we will only consider the two cases where $R^{\prime}$ is regular of dimension $0$ and $1$, hence either a field or a discrete valuation ring.
 
 \begin{prop}[Lindel]\label{prop:Lindel_etale_neighbourhood}
 If $R^{\prime}$ is a perfect field or a discrete valuation ring such that both the residue field and the field of fractions are perfect, then for \emph{any} smooth $R$-homomorphism $R^{\prime} \rightarrow C$, the local rings of $C$ are {\'e}tale neighbourhoods over $R^{\prime}$. 
 \end{prop}
 
 \begin{proof}
 The case where $R^{\prime}$ is a perfect field is treated in \cite{LINDEL}. The discussion above \cite[Proposition~2]{LINDEL} shows that $C_{\mathbb{q}}$ contains the field of fractions $L$ of some polynomial algebra over $R^{\prime}$ in such a way that the premise of \cite[Proposition~2]{LINDEL} is satisfied, so the conclusion follows from Part~(ii) of \cite[Proposition~2]{LINDEL}. Since the homomorphism constructed in this manner is $R^{\prime}$-linear, it is in particular $R$-linear.
 
 The case of discrete valuation rings is treated in \cite[Proposition~2.1]{POPESCU_POLYNOMIAL}. Only checking Condition~(ii) of that proposition needs a small argument. Namely, if $p$ is a uniformizer of $R^{\prime}$ and $\mathfrak{q} \in \Spec(C)$, we need to show that $p$ does not lie in the square of the maximal ideal $\mathfrak{q}C_{\mathfrak{q}}$ of $C_\mathfrak{q}$. If $p$ lies in $\mathfrak{q} C_{\mathfrak{q}}$, then this follows from the essential smoothness of $R^{\prime} \rightarrow C_{\mathfrak{q}}$. Indeed, since smooth homomorphisms are stable under base change, it follows that $R^{\prime} \slash p \rightarrow C_{\mathfrak{q}} \slash p$ is essentially smooth, hence that the codomain is regular. Since $p \in \mathfrak{q}C_{\mathfrak{q}}$, the preimage of $\mathfrak{q} C_{\mathfrak{q}}$ in $R^{\prime}$ is $(p)$, so $R^{\prime} \rightarrow C_{\mathfrak{q}}$ is faithfully flat and therefore injective. This implies that $\mathrm{dim} C_{\mathfrak{q}} \slash p < \mathrm{dim} C_{\mathfrak{q}}$. By lifting a minimal set of generators of the maximal ideal in $C_{\mathfrak{q}} \slash p$ we find from this dimension consideration that $p$ cannot lie in $(\mathfrak{q}C_{\mathfrak{q}})^2$. This concludes the case where $p$ lies in the maximal ideal of $C_{\mathfrak{q}}$. If this is not the case, then $C_{\mathfrak{q}}$ is an {\'e}tale neighbourhood over $R^{\prime}_p$ (since this is a perfect field by assumption), so the conclusion follows as well.
 \end{proof}
 
 From \cite[Lemma on p.~321]{LINDEL} we know that for each {\'e}tale neighbourhood $\varphi \colon (A,\mathfrak{m}) \rightarrow (B,\mathfrak{n})$, there exists an element $h \in \mathfrak{m}$ such that $\varphi$ is an analytic isomorphism along $h$. We have $\varphi(h) \in \mathfrak{n}$ since $A \rightarrow B$ is local. We call both the resulting patching diagrams
 \[
 \vcenter{\xymatrix{ 
 A \ar[r]^-{\lambda_h} \ar[d]_{\varphi} & A_h \ar[d]^{\varphi_h} \\
 B \ar[r]_-{\lambda_{\varphi(h)}} & B_{\varphi(h)}
 }} 
 \quad \text{and} \quad
 \vcenter{\xymatrix{ 
 A[t] \ar[r]^-{\lambda_h} \ar[d]_{\varphi} & A_h[t] \ar[d]^{\varphi_h} \\
 B[t] \ar[r]_-{\lambda_{\varphi(h)}} & B_{\varphi(h)}[t]
 }} 
 \]
 a \emph{patching diagram associated to an {\'e}tale neighbourhood}. We will discuss these kinds of patching diagram in more detail in \S \ref{section:henselian_pairs}, see in particular Proposition~\ref{prop:etale_neighbourhood_standard_Nisnevich}.

\begin{prop}\label{prop:Lindel_result_for_general_F} 
 Let $F \colon \CAlg_R \rightarrow \Set_{\ast}$ be a functor which satisfies $(\mathrm{Q})$ and $(\mathrm{R})$ and such that the implication
 \[
 NF(A) = \ast \quad \Rightarrow \quad A \in \ca{E}_1^F
 \]
 holds. Assume that $F$ satisfies weak excision for patching diagrams associated to {\'e}tale neighbourhoods. Let $R^{\prime}$ be a regular commutative $R$-algebra of Krull dimension $d$ and let $R^{\prime} \rightarrow C$ be a smooth $R$-homomorphism.
 
 If $R^{\prime}$ is $F$-regular, $\ca{E}_1^F$ contains all regular local rings of Krull dimension $\leq d$, and all the local rings of $C$ are {\'e}tale neighbourhoods over $R^{\prime}$, then $C$ lies in $\ca{E}_1^F$.
\end{prop} 

\begin{proof}
 Since $F$ satisfies $(\mathrm{Q})$, it suffices to check that the local rings of $C_{\mathfrak{m}}$ lie in $\ca{E}_1^F$, where $\mathfrak{m} \subseteq C$ is maximal. We will in fact show by induction on the height that the local rings $C_{\mathfrak{q}}$ lie in $\ca{E}_1^F$ for all prime ideals $\mathfrak{q} \in \Spec(C)$.
 
 Since $R^{\prime}$ is regular and $R^{\prime} \rightarrow C$ is regular it follows that $C$ is regular as well (smooth ring maps send regular sequences to regular sequences, so the claim follows from the fact that the fibers of a smooth ring map are regular). The assumption therefore implies that $C_{\mathfrak{q}}$ lies in $\ca{E}_1^F$ for all $\mathfrak{q}$ with $\mathrm{ht}(\mathfrak{q}) \leq d$.
 
 Thus we can assume that $\mathrm{ht}(\mathfrak{q}) >d$ and that $C_{\mathfrak{p}} \in \ca{E}_1^F$ for all $\mathfrak{p} \in \Spec(C)$ with $\mathrm{ht}(\mathfrak{p}) < \mathrm{ht}(\mathfrak{q})$. Since the local rings of $C$ are {\'e}tale neighbourhoods over $R^{\prime}$, there exists an {\'e}tale neighbourhood
 \[
 \varphi \colon (A,\mathfrak{m}) \rightarrow (C_{\mathfrak{q}}, \mathfrak{q} C_{\mathfrak{q}})
 \]
 such that $A$ is a localization of some polynomial algebra $R^{\prime}[x_1, \ldots, x_n]$ (see Definition~\ref{dfn:local_rings_are_etale_neighbourhoods}). Since $R^{\prime}$ is $F$-regular and $F$ satisfies $(\mathrm{R})$, it follows that $A$ is also $F$-regular, hence it lies in particular in $\ca{E}_1^F$.
 
 Now we can pick an element $h \in \mathfrak{m}$ as in \cite[Lemma on p.~321]{LINDEL} such that $\varphi$ is an analytic isomorphism along $h$. Then $\varphi(h)$ lies in the maximal ideal of $C_\mathfrak{q}$, so all the local rings of $C_{\mathfrak{q}, \varphi(h)}$ are of the form $C_{\mathfrak{p}}$ for some $\mathfrak{p} \in \Spec(C)$ with $\mathrm{ht}(\mathfrak{p}) < \mathrm{ht}(\mathfrak{q})$. The inductive assumption, combined with the Quillen principle $(\mathrm{Q})$, therefore implies that $C_{\mathfrak{q}, \varphi(h)}$ lies in $\ca{E}_1^F$.
 
 Now pick an object $\sigma \in NF(C_{\mathfrak{q}})$ and consider the weak pullback diagram
 \[
 \xymatrix{ 
 F(A[t]) \ar[r]^-{\lambda_h} \ar[d]_{\varphi} & F( A_h[t] )\ar[d]^{\varphi_h} \\
 F(C_{\mathfrak{q}}[t]) \ar[r]_-{\lambda_{\varphi(h)}} &F( C_{\mathfrak{q},\varphi(h)}[t] )
 }
 \]
 associated to the {\'e}tale neighbourhood $\varphi \colon A \rightarrow C_{\mathfrak{q}}$.
 
 Since $\sigma$ lies in $NF(C_{\mathfrak{q}})$, it follows that $\sigma_{\varphi(h)}$ lies in $NF(C_{\mathfrak{q},\varphi(h)})$. The fact that $C_{\mathfrak{q},\varphi(h)}$ lies in $\ca{E}_1^F$ therefore implies that $\sigma_{\varphi(h)}=\ast$. The weak pullback property implies that we can find some $\tau \in F(A[t])$ such that $\varphi_{\ast} \tau=\sigma$ and $\tau_h=\ast$. Since all objects of $F(A[t])$ are extended from $A$ it follows that $\tau$, hence $\sigma=\varphi_{\ast} \tau$, are both extended. But $\sigma$ lies by assumption in $NF(C_{\mathfrak{q}})$, so it is extended from $\ast$ and therefore equal to $\ast$. Since $\sigma$ was arbitrary, it follows that $NF(C_{\mathfrak{q}})=\ast$.
 
 From our assumption it follows that all $\sigma \in F(C_{\mathfrak{q}}[t])$ are extended, that is, $C_{\mathfrak{q}}$ lies in $\ca{E}_1^F$. A final application of $(\mathrm{Q})$ implies that $C$ lies in $\ca{E}_1^F$, as claimed.
\end{proof}

 The above proposition can be combined with Popescu's Theorem to show that $F$-regularity for fields or discrete valuation rings implies $F$-regularity for a much larger class of regular rings. Recall that a ring homomorphism $R \rightarrow A$ is called \emph{regular} (sometimes \emph{geometrically regular} for emphasis) if it is flat and the fiber rings $\kappa(\mathfrak{p}) \ten{R} A$, $\mathfrak{p} \in \Spec(R)$ are geometrically regular, that is, they are regular and remain so after tensoring with any finite purely inseparable field extension of $\kappa(\mathfrak{p})$. Popescu showed that any regular homomorphism $R \rightarrow A$ between noetherian rings is a filtered colimit of smooth $R$-algebras (see \cite[Theorem~2.5]{POPESCU_DESINGULARIZATION}).
 
 We will apply this result in two cases. If $R$ is a regular local ring containing a field $k$, then let $k_0$ be the prime field of $k$. Since $k_0$ is perfect, the composite $k_0 \rightarrow R$ is geometrically regular. If $(R,\mathfrak{m})$ is a regular local ring of mixed characteristic $(0,p)$, then $R$ is called \emph{unramified} if $p \notin \mathfrak{m}^{2}$. In this case, the homomorphism $\mathbb{Z}_{(p)} \rightarrow R$ is geometrically regular (since $R \slash p R$ is regular and $\mathbb{Z} \slash p$ is perfect). In general, a regular local ring is called unramified if it is either unramified of mixed characteristic or if it contains a field.
 
\begin{thm}\label{thm:F_regularity_for_unramified_regular_rings}
 Let $F \colon \CRing \rightarrow \Set_{\ast}$ be a finitary functor which satisfies $(\mathrm{Q})$ and $(\mathrm{R})$. Assume that the implication
 \[
 NF(A)=\ast \quad \Rightarrow \quad A \in \ca{E}_1^F
 \]
 holds and that $F$ satisfies weak excision for patching diagrams associated to {\'e}tale neighbourhoods. Then the following hold:
 \begin{enumerate}
 \item[(i)] If all fields are $F$-regular, then all regular local rings containing a field are $F$-regular;
 \item[(ii)] If, in addition, all discrete valuation rings are $F$-regular, then all unramified regular local rings are $F$-regular.
 \end{enumerate}
\end{thm}
 
\begin{proof}
 Since $F$ is finitary, Popescu's Theorem \cite[Theorem~2.5]{POPESCU_DESINGULARIZATION} and the above discussion reduces the problem to showing that all smooth algebras $A$ over a prime field $k_0$ (respectively over $\mathbb{Z}_{(p)}$, $p$ a prime) are $F$-regular. Since $A \rightarrow A[t]$ is smooth, it suffices to show that these rings lie in $\ca{E}_1^F$.
 
 Since prime fields are perfect (and both $\mathbb{Q}=\mathrm{Frac}(\mathbb{Z}_{(p)})$ and $\mathbb{Z} \slash p$ are perfect), we know from Proposition~\ref{prop:Lindel_etale_neighbourhood} that the local rings of $A$ are {\'e}tale neighbourhoods over $k_0$ (respectively over $\mathbb{Z}_{(p)}$). In order to apply Proposition~\ref{prop:Lindel_result_for_general_F}, it only remains to check that $\ca{E}_1^F$ contains all the regular local rings of dimension $0$ (respectively $\leq 1$). Since these are precisely the fields and the discrete valuation rings, this follows from the stronger assumption that such rings are $F$-regular in the respective cases.
\end{proof} 
 
 With this in hand, we can prove the following theorem of the introduction. 
 
\begin{cit}[Theorem~\ref{thm:excision_implies_three_principles}]
Let $F \colon \CRing \rightarrow \Set_{\ast}$ be a finitary functor which satisfies weak analytic excision. Then $F$ satisfies Quillen's principle $(\mathrm{Q})$. Assume that one of the following conditions holds:
 \begin{enumerate}
 \item[(A)] all local rings are $F$-contractible;
 \item[(B)] the functor $F$ admits a natural transitive group action.
 \end{enumerate}
 Then $F$ also satisfies Roitman's principle $(\mathrm{R})$. Moreover, the following hold:
 \begin{enumerate}
 \item[(i)] if all fields are $F$ regular, then all regular rings containing a field are $F$-regular;
 \item[(ii)] if, in addition, all discrete valuation rings are $F$-regular, then all unramified regular local rings are $F$-regular.
\end{enumerate} 
\end{cit}
 
\begin{proof}[Proof of Theorem~\ref{thm:excision_implies_three_principles}]
 From Theorem~\ref{thm:weak_excision_implies_QR} it follows that $F$ satisfies $(\mathrm{Q})$ and $(\mathrm{R})$. In order to apply Theorem~\ref{thm:F_regularity_for_unramified_regular_rings}, it only remains to check that the implication
 \[
 NF(A)=\ast \quad \Rightarrow \quad A \in \ca{E}_1^F
 \]
 holds. If Condition~(B) holds, this follows from Lemma~\ref{lemma:transitive_action_implies_extended_if_trivial_NF}. If Condition~(A) holds, this follows from Proposition~\ref{prop:NF_trivial_on_localization} and Lemma~\ref{lemma:F_contractible_implies_extended_if_trivial_NF}.
\end{proof} 
 
 We conclude this section with a discussion of some additional closure properties of the class of $F$-regular algebras. For example, for certain functors $F$, it follows that $A$ is $F$-regular whenever $A_{\mathrm{red}}$ is. In other cases, we can even conclude that local rings are $F$-contractible if all fields are $F$-contractible.

 \begin{dfn}\label{dfn:ideal_injective_and_ideal_invariant}
 Let $F \colon \CAlg_R \rightarrow \Set$ be a functor. We say that
 \begin{enumerate}
 \item[(i)] The functor $F$ is \emph{$j$-injective} if $F(A) \rightarrow F(A \slash I)$ is injective whenever $I$ is an ideal contained in the Jacobson radical of $A$;
 \item[(ii)] The functor $F$ is \emph{$h$-injective} if $F(A) \rightarrow F(A \slash I)$ is injective whenever $(A,I)$ is a henselian pair;
 \item[(iii)] The functor $F$ is \emph{$\ell n$-injective} if $F(A) \rightarrow F(A \slash I)$ is injective whenever $I \subseteq A$ is a locally nilpotent ideal (that is, $I$ consists of nilpotent elements);
 \item[(iv)] The functor $F$ is \emph{$n$-injective} if $F(A) \rightarrow F(A \slash I)$ is injective whenever $I \subseteq A$ is a nilpotent ideal, that is, $I^k=0$ for some $k \in \mathbb{N}$.
 \end{enumerate}
 The properties \emph{$j$-}, \emph{$h$-}, \emph{$\ell n$-}, and \emph{$n$-invariant} are defined analogously by requiring that $F(A) \rightarrow F(A \slash I)$ is bijective.
 \end{dfn}
 
 \begin{example}\label{example:P_r_is_j_injective_and_h_invariant}
 The functors $P_r \colon \CAlg_R \rightarrow \Set_{\ast}$ and $P \colon \CAlg_R \rightarrow \Set$ are $j$-injective (by the Nakayama lemma, see for example \cite[Corollary~I.1.6]{LAM}) and $h$-invariant (by \cite[\href{https://stacks.math.columbia.edu/tag/0D4A}{Lemma 0D4A}]{stacks-project}).
 \end{example}
 
 \begin{lemma}
 For any functor $F \colon \CAlg_R \rightarrow \Set$, the implications
 \[
 \begin{matrix}
 j\text{-injective} & \Rightarrow & h\text{-injective} & \Rightarrow & \ell n\text{-injective} & \Rightarrow & n\text{-injective} \\
\Uparrow && \Uparrow && \Uparrow && \Uparrow \\
j\text{-invariant} & \Rightarrow & h\text{-invariant} & \Rightarrow &  \ell n\text{-invariant} & \Rightarrow &  n\text{-invariant} \\ 
 \end{matrix}
 \]
 hold.
 \end{lemma}
 
 \begin{proof}
 The vertical implications follow directly from the definition. Clearly a nilpotent ideal is locally nilpotent. If $I$ is locally nilpotent, then $(A,I)$ is henselian by \cite[\href{https://stacks.math.columbia.edu/tag/0ALI}{Lemma 0ALI}]{stacks-project}. Finally, if $(A,I)$ is henselian, then $I$ is contained in the Jacobson radical (see for example \cite[\href{https://stacks.math.columbia.edu/tag/09XI}{Lemma 09XI}]{stacks-project}).
 \end{proof}
 
 \begin{lemma}\label{lemma:SL_r_j_invariant}
  Let $I \subseteq A$ be an ideal contained in the Jacobson radical and let $\sigma \in \mathrm{SL}_r(A)$ for some $n \geq 2$. If $\bar{\sigma} \in \mathrm{SL}_r(A \slash I)$ is elementary, then $\sigma$ is elementary. Moreover, for $n \geq 3$ the functors
  \[
  SK_{1,r} \colon \CAlg_R \rightarrow \Grp
 \]
 and the functor $SK_1 \colon \CAlg_R \rightarrow \Ab$ are $j$-invariant.
 \end{lemma}
 
 \begin{proof}
 If $\bar{\sigma}$ is elementary, then there exists an $\varepsilon \in \mathrm{E}_r(A)$ such that $\bar{\varepsilon}=\bar{\sigma}^{-1}$, so $\overline{\sigma \varepsilon}=\id$. It thus suffices to check that a matrix whose diagonal entries lie in $1+I$ and whose off-diagonal entries lie in $I$ is elementary. Since $I$ is contained in the Jacobson radical, the diagonal entries are in particular units. Using only column transformations, we can thus successively ``clear'' the off-diagonal entries while retaining a matrix of the same form. At the end we obtain a diagonal matrix with determinant $1$, which is elementary by Whitehead's Lemma.
 
 This shows the first part, which implies that $SK_{1,r}(A) \rightarrow SK_{1,r}(A \slash I)$ is injective. To see that it is surjective, note that any lift $\sigma$ of $\bar{\sigma} \in SK_{1,r}(A \slash I)$ lies in $\mathrm{GL}_{r}(A)$ since $d \defl \mathrm{det} (\sigma)$ lies in $1+I$. Multiplying $\sigma$ with $d^{-1}$ we obtain the desired preimage in $SK_{1,r}(A)$. The claim about $SK_1$ follows from the fact that it can be written as filtered colimit of the $SK_{1,r}$, with connecting morphisms given by the inclusion sending $\sigma$ to $\bigl( \begin{smallmatrix} \sigma & 0\\ 0 &1 \end{smallmatrix} \bigr)$.
 \end{proof}
 
 \begin{lemma}\label{lemma:W_r_ln_invariant}
 For each $r \geq 3$, the functor
 \[
 W_r \colon \CAlg_R \rightarrow \mathrm{Set}_{\ast} 
 \]
 is $\ell n$-invariant.
 \end{lemma}
 
 \begin{proof}
 This is for example proved in \cite[Lemma~1.4.2]{RAO_DIMENSION_THREE}, using an argument of Mohan Kumar. Since we have established all the prerequisites, we can reproduce the argument here.
 
  Let $I \subseteq A$ be an ideal consisting of nilpotent elements. Note that the function $W_r(A) \rightarrow W_r(A \slash I)$ is surjective since any lift of a unimodular row over $A \slash I$ is unimodular over $A$. To see injectivity, it suffices to check that for any unimodular row $a=(a_1, \ldots, a_r)$ over $A$ and any nilpotent element $n \in A$, the unimodular row
 \[
 (a_1, \ldots, a_i+n, \ldots, a_r)
 \]
 lies in the same elementary orbit as $a$. Since $W_r$ satisfies $(\mathrm{Q})$ (see Corollary~\ref{cor:analytic_excision_implies_Q} and Proposition~\ref{prop:W_r_weak_analytic_excision}), it suffices to check that the class of
 \[
 (a_1, \ldots, a_i+n \cdot x, \ldots, a_r)
 \]
 in $W_r(A[x])$ is extended after passage to any local ring of $A$, which follows from the fact that one of the entries of this row must be a unit if $A$ is local.
 \end{proof}
 
\section{The monic inversion principle}\label{section:monic}

 In order to apply the results of \S \ref{section:excision} (specifically Theorem~\ref{thm:F_regularity_for_unramified_regular_rings}), we need a methood which allows us to check if fields and discrete valuation rings are $F$-regular. The principle we will use is often called the \emph{monic inversion principle}. In the case of the functor $P_{r}$, it is due to Horrocks (see \cite[Theorem~1]{HORROCKS} and the preamble of \cite[Theorem~2]{HORROCKS}), so we also call it the \emph{Horrocks principle} (in accordance with the naming convention of \S \ref{section:excision}).
 
 As usual, we fix a commutative ground ring $R$ and work in the category $\CAlg_R$ of commutative $R$-algebras. Recall that $A \langle t \rangle$ denotes the localization of $A[t]$ at the set of monic polynomials with coefficients in $A$.
 
 \begin{dfn}\label{dfn:Horrocks_principle}
 We say that $F \colon \CAlg_R \rightarrow \Set_{\ast}$ satisfies the \emph{local Horrocks principle $(\mathrm{H}_{\mathrm{loc}})$} if
 \[
 F(A[t]) \rightarrow F(A \langle t \rangle)
 \]
 has trivial kernel.
 
 We say that $F \colon \CAlg_R \rightarrow \Set$ satisfies the \emph{Horrocks principle $(\mathrm{H})$} if the following holds: if $A$ is any $R$-algebra and $\sigma \in F(A[t])$ is any object whose image in $F(A \langle t\rangle)$ is extended from $A$, then $\sigma$ is extended from $A$.
 \end{dfn}
 
 Here an object $\sigma \in F(A \langle t \rangle)$ is said to be extended from $A$ if it there exits a $\tau \in F(A)$ such that the composite
 \[
 \xymatrix{F(A) \ar[r]^-{\iota} & F(A[t]) \ar[r] & F(A \langle t \rangle)}
 \]
 sends $\tau$ to $\sigma$. In order for these definitions to behave as one might expect, we need to make some additional assumptions on the functor $F$.
 
 \begin{prop}\label{prop:uniquely_etxtended_after_monic_inversion}
 Let $F \colon \CAlg_R \rightarrow \Set$ be a finitary functor. Let $A$ be a commutative $R$-algebra and let $S$ denote the set of monic polynomials in $A[t]$. If $F$ satisfies the Horrocks principle $(\mathrm{H})$ and weak Zariski excision, then the function $F(A) \rightarrow F(A\langle t \rangle)$ induced by the composite
 \[
 \xymatrix{ A \ar[r]^-{\iota} & A[t] \ar[r]^-{\lambda_S} & A\langle t \rangle}
 \]
 is injective.
 \end{prop}
 
 \begin{proof}
 The proof given in \cite[Proposition~V.2.4]{LAM} for projective modules works for any functor $F$ as above. Indeed, let $\sigma_0, \sigma_1 \in FA$ be objects such that $(\iota_{\ast} \sigma_0)_S=(\iota_\ast \sigma_1)_S$. We need to check that $\sigma_0=\sigma_1$. Since $F$ is finitary, there exits a monic polynomial $f \in S$ such that $(\lambda_f \iota)_{\ast} \sigma_0=(\lambda_f \iota)_{\ast} \sigma_1$ holds in $F(A[t]_f)$.
 
 Let $n=\mathrm{deg}(f)$  and let $g(s)=s^n f(s^{-1}) \in A[s]$. Since $g(0)=1$ by construction, the elements $g(s)$ and $s$ are comaximal in $A[s]$. If we let $\psi \colon A[s,s^{-1}] \rightarrow A[t,t^{-1}]$ be the unique $A$-algebra isomorphism with $\psi(s)=t^{-1}$. Since $\psi(g)$ and $f$ are associates in $A[t,t^{-1}]$ by construction, there exists a unique $R$-algebra isomorphism $\varphi$ making the diagram
 \[
 \xymatrix{A[s,s^{-1}] \ar[d]_{\psi} \ar[r]^-{\lambda_g} & A[s,s^{-1}]_g \ar[d]^{\varphi} \\ A[t,t^{-1}] \ar[r]^-{\lambda_f} & A[t,t^{-1}]_f }
 \]
 commutative.
 
 Since we now have two polynomial rings over $A$, we write $\iota^s$ (respectively $\iota^{t}$) for the canonical inclusion of $A$ in $A[s]$ (respectively in $A[t]$). From the fact that $\psi$ is an $A$-algebra homomorphism it follows that $\psi \lambda_s \iota^s=\lambda_t \iota^{t}$ holds. This implies that the equalities
 \[
 \varphi_{\ast} (\lambda_g \lambda_s \iota^s)_{\ast}\sigma_1=(\lambda_f \psi \lambda_s \iota^s)_{\ast} \sigma_1=(\lambda_f \lambda_t \iota^t)_{\ast} \sigma_1
 \]
 and
 \[
 \varphi_{\ast}(\lambda_s \lambda_g \iota^s)_{\ast} \sigma_0=\varphi_\ast (\lambda_g \lambda_s \iota^s)_{\ast} \sigma_0=(\lambda_f \lambda_t \iota^{t})_{\ast} \sigma_0
 \]
 hold. The right hand side of the above two equations is equal by choice of $f$ and $\varphi$ is an isomorphism, so it follows that
 \[
( \lambda_g)_{\ast} (\lambda_s \iota^s)_{\ast} \sigma_1=(\lambda_s)_{\ast} (\lambda_g \iota^s)_{\ast} \sigma_0
 \]
 holds.
 Now we use the fact that $g$ and $s$ are comaximal in $A[s]$. Since $F$ satisfies weak Zariski excision, there exists an object $\omega$ in $F(A[s])$ such that $\omega_s=(\lambda_s \iota^s)_{\ast} \sigma_1$ and $\omega_g=(\lambda_g \iota^{s}) \sigma_0$. Since $s$ is a monic polynomial in $s$, it follows that the image of $\omega$ in $F(A \langle s \rangle)$ is extended from $A$. From $(\mathrm{H})$ it follows that there exists some $\omega_{0} \in F(A)$ such that $\omega=(\iota^{s})_{\ast} \omega_0$.
 
 The two commutative triangles
 \[
\vcenter{ \xymatrix{
A[s] \ar[rr]^-{\lambda_g} \ar[rd]_{\pi=\mathrm{ev}_0} && A[s]_g \ar[ld]^{\mathrm{ev}_0} \\
& A 
}} \quad \text{and} \quad
\vcenter{ \xymatrix{
A[s] \ar[rr]^-{\lambda_s} \ar[rd]_{\mathrm{ev}_1} && A[s,s^{-1}] \ar[ld]^{\mathrm{ev}_1} \\
& A 
}}
 \]
 show that $\omega_0=\sigma_0$ and that $\omega_0=\sigma_1$. Thus $\sigma_0=\sigma_1$, as claimed.
 \end{proof}
 
 \begin{cor}
 If $F \colon \CAlg_R \rightarrow \Set_{\ast}$ is finitary and satisfies weak Zariski excision, then the implication $(\mathrm{H}) \Rightarrow (\mathrm{H}_{\mathrm{loc}})$ holds.
 \end{cor}
 
 \begin{proof}
 If $A$ is local and $\sigma \in F(A[t])$ is sent to the basepoint $\ast \in F(A \langle t \rangle)$, then $(\mathrm{H})$ implies that $\sigma$ is extended from some $\sigma_0 \in F(A)$. But the basepoint in $F(A \langle t \rangle)$ is extended from $\ast \in F(A)$, so from Proposition~\ref{prop:uniquely_etxtended_after_monic_inversion} it follows that $\sigma_0=\ast$. Thus $\sigma=\iota_{\ast} \ast$ is trivial, which establishes $(\mathrm{H}_{\mathrm{loc}})$.
 \end{proof}
 
 Quillen showed that Serre's problem about projective modules over polynomial rings has an affirmative answer by showing that the local Horrocks principle and the local-global principle established in \cite{QUILLEN} together imply the global Horrocks principle. This suggests that the principles $(\mathrm{Q})$ and $(\mathrm{H}_{\mathrm{loc}})$ together should imply the global Horrocks principle $(\mathrm{H})$. It turns out that we need to assume a bit more for this to hold. Each one of the two possible assumptions of Theorem~\ref{thm:weak_excision_implies_QR} is sufficient.
 
 \begin{prop}\label{prop:QHloc_implies_H}
 Let $F \colon \CAlg_R \rightarrow \Set_{\ast}$ be a functor which satisfies $(\mathrm{Q})$ and assume that one of the following conditions holds:
  \begin{enumerate}
 \item[(A)] All commutative $R$-algebras which are local rings are $F$-contractible;
 \item[(B)] The functor $F$ admits a natural transitive group action.
 \end{enumerate}
 Then $(\mathrm{H}_{\mathrm{loc}})$ implies $(\mathrm{H})$.
 \end{prop}
 
\begin{proof}
 We first show the implication under Assumption~(A). Let $\sigma \in F(A[t])$ be an object whose image in $F(A \langle t \rangle)$ is extended. It follows that for any maximal ideal $\mathfrak{m} \subseteq A$, the image of $\sigma_{\mathfrak{m}}$ in $F(A_{\mathfrak{m}} \langle t \rangle)$ is extended from $A_{\mathfrak{m}}$, hence trivial. From $(\mathrm{H}_{\mathrm{loc}})$ it follows that $\sigma_{\mathfrak{m}}$ is trivial, so in particular extended. Thus $(\mathrm{Q})$ implies that $\sigma$ is extended, which establishes $(\mathrm{H})$.
 
 Assume now that (B) holds and let $G(A) \times F(A) \rightarrow F(A)$ be a natural transitive group action. If $\sigma \in F(A[t])$ is an object whose image in $F(A \langle t \rangle)$ is extended, there exists some $\tau \in F(A)$ such that $\sigma$ and $\iota_\ast \tau \in F(A[t])$ have the same image in $F(A \langle t \rangle)$. Write $S \subseteq A[t]$ for the set of monic polynomials and let $\tilde{\iota} \defl \lambda_S \iota \colon A \rightarrow A \langle t \rangle$ be the canonical inclusion.
 
 Pick an element $g \in G(A)$ such that $g \cdot \tau=\ast$. Then
 \[
 \bigl((\iota_\ast g) \cdot \sigma \bigr)_S=\tilde{\iota}_{\ast} g \cdot \sigma_S=\tilde{\iota}_{\ast}(g \cdot \tau)=\ast
 \]
 holds, so the image of $(\iota_{\ast} g \cdot \sigma)_{\mathfrak{m}}$ in $F(A_{\mathfrak{m}}\langle t \rangle )$ is trivial for all maximal ideals $\mathfrak{m} \subseteq A$. From $(\mathrm{H}_{\mathrm{loc}})$ it follows that $(\iota_{\ast} g \cdot \sigma)_{\mathfrak{m}}$ is trivial, hence extended. Thus $(\mathrm{Q})$ implies that $(\iota_{\ast} g) \sigma$ is extended from some object $\sigma_0 \in F(A)$. Thus $\sigma=\iota_{\ast}(g^{-1} \cdot \sigma_0)$ is extended, as claimed.
\end{proof}
 
 There is one more way the Horrocks principle is sometimes stated, which is equivalent to $(\mathrm{H})$ under favourable conditions.
 
\begin{lemma}\label{lemma:pullback_version_of_H}
 Let $F \colon \CAlg_R \rightarrow \Set$ be a functor which satisfies weak Zariski excsision. If for every commutative $R$-algebra $A$, the diagram
 \[
 \xymatrix{F(A) \ar[r] \ar[d] & F(A[t^{-1}]) \ar[d] \\ F(A[t]) \ar[r] & F(A[t,t^{-1}])}
 \]
 is a pullback diagram, then $F$ satisfies $(\mathrm{H})$.
 
 The converse holds if $F$ takes values in groups and $F$ is in addition finitary and $j$-injective.
\end{lemma}

\begin{proof}
 Recall that there is a canonical isomorphism $A[t]_{1+tA[t],t} \cong A \langle s \rangle$ obtained by sending $s$ to $t^{-1}$. We simply write this as $A[t]_{1+tA[t],t}=A \langle t^{-1} \rangle$. The assumptions imply that both squares in the diagram
 \[
 \xymatrix{F(A) \ar[d] \ar[r] & F(A[t^{-1}]) \ar[d] \\ F(A[t]) \ar[d] \ar[r] & F(A[t,t^{-1}]) \ar[d] \\ F(A[t]_{1+tA[t]}) \ar[r] & F(A \langle t^{-1} \rangle)}
 \]
 are weak pullback diagrams. Since any $\sigma \in F(A[t^{-1}])$ whose image in the bottom right corner is extended from $A$ lies in particular in the image of
 \[
 F(A[t]_{1+tA[t]}) \rightarrow F(A \langle t^{-1} \rangle)
 \]
 it follows that $(\mathrm{H})$ holds.
 
 To see the converse, assume that $F$ takes values in groups, that it is finitary, and that $F$ is $j$-invariant. We claim that under these assumptions, the principle $(\mathrm{H})$ implies that the homomorphism $F(A[t]) \rightarrow F(A[t,t^{-1}])$ is injective. To see this, it suffices to check that it has trivial kernel. Since any element $\sigma$ in the kernel is clearly extended in $F(A \langle t \rangle)$, we concude that $\sigma$ is extended from some $\sigma_0$ in $A$. But then composing with the evaluation $\mathrm{ev}_1 \colon  A[t,t^{-1}] \rightarrow A$ in $t=1$ shows that $\sigma_0=e$ is the neutral element. Thus the lower square above is a pullback square (the injectivity of the middle vertical homomorphism in the diagram gives the required uniqueness). By the cancellation law for pullbacks, it only remains to check that the total diagram is a pullback.
 
 From the $j$-injectivity of $F$ it follows that $F(A[t]_{1+t[t]}) \rightarrow F(A)$ is an isomorphism, so the left vertical morphism $F(A) \rightarrow F(A[t]_{1+tA[t]})$ is also an isomorphism. Since $F$ is finitary and satisfies weak Zariski excision, we know that the diagonal homomorphism $FA \rightarrow F(A[t^{-1}]) \rightarrow F(A \langle t^{-1} \rangle)$ is injective (see Proposition~\ref{prop:uniquely_etxtended_after_monic_inversion}). Taken together, this implies in particular that the bottom horizontal homomorphism is injective. Now the claim follows from principle $(\mathrm{H})$: if $\sigma \in F(A[t^{-1}])$ has the same image as $\tau \in F(A[t]_{1+tA[t]})$, then this image is in particular extended from $A$, so from $(\mathrm{H})$ we conclude that $\sigma$ is extended from some $\sigma_0 \in FA$. The injectivity properties that we have established above imply that the image of $\sigma_0$ under the left vertical homomorphism must be $\tau$, which shows that the total square is indeed  a pullback diagram.
\end{proof}

 The following lemma shows that we can restrict attention to subsets of monic polynomials to check whether or not all objects are extended.
 
\begin{lemma}\label{lemma:monic_inversion_detects_extended_objects}
 Let $F \colon \CAlg_R \rightarrow \Set$ be a functor which satisfies the Horrocks principle $(\mathrm{H})$. If $S \subseteq A[t]$ is a set of monic polynomials and
 \[
 \xymatrix{ F(A) \ar[r]^-{\iota} & F(A[t]) \ar[r]^-{\lambda_S} & F(A[t]_S)}
 \]
 is surjective, then $\iota_{\ast} \colon F(A) \rightarrow F(A[t])$ is bijective, in other words, $A$ lies in $\ca{E}_1^F$.
\end{lemma}

\begin{proof}
 The assumption implies that the image of each $\sigma \in F(A[t])$ in $F(A \langle t \rangle)$ is extended from $A$, hence $\sigma$ is extended by $(\mathrm{H})$.
\end{proof}

 Given a commutative $R$-algebra $R^{\prime}$, we inductively define
 \[
 R^{\prime} \langle t_1 \ldots, t_n \rangle \defl R^{\prime} \langle t_1 \ldots, t_{n-1} \rangle \langle t_n \rangle
 \]
 for $n>1$ as an iterated localization.
 
 \begin{dfn}\label{dfn:strong_extension_property}
 Let $F \colon \CAlg_R \rightarrow \Set$ be a functor, let $R^{\prime}$ be a commutative $R$-algebra and let $A$ be a commutative $R^{\prime}$-algebra. We say that \emph{$A$ satisfies $(\mathrm{SE}_n^{R^{\prime},F})$} if the function 
 \[
 F(A) \rightarrow F(A \ten{R^{\prime}} R^{\prime} \langle t_1, \ldots, t_k \rangle)
 \]
 induced by the natural inclusion $R^{\prime} \rightarrow R^{\prime} \langle t_1,\ldots,t_k \rangle$ is surjective for all $1 \leq k \leq n$.
 
 We say that \emph{$A$ has the strong $F$-extension property over $R^{\prime}$} if $(\mathrm{SE}_n^{R^{\prime},F})$ holds for all $n \geq 1$, that is, the above function is surjective for all $k \geq 1$. In the case where $A=R^{\prime}$ is the initial $R^{\prime}$-algebra, we simply say that \emph{$R^{\prime}$ has the strong $F$-extension property} (or the strong $F$-extension property \emph{over itself}).
 \end{dfn}
 
 We will often apply this in the case where $R^{\prime}$ is an arbitrary field or a principal ideal domain, in which case the algebras $R^{\prime} \langle t_1,\ldots t_k \rangle$ are also fields respectively principal ideal domains. The following proposition shows that the strong extension property implies that all objects over $A[t_1, \ldots, t_n]$ are extended if the Horrocks principle holds, which justifies the name of the strong extension property.
 
 \begin{prop}\label{prop:SE_implies_extended}
 Let $F \colon \CAlg_R \rightarrow \Set$ be a functor which satisfies the Horrocks principle $(\mathrm{H})$. Let $R^{\prime}$ be a commutative $R$-algebra and let $A$ be a commutative $R^{\prime}$-algebra which satisfies $(\mathrm{SE}_n^{R^{\prime},F})$ for some $n \geq 1$. Then $A$ lies in $\ca{E}_n^F$, that is, the function
 \[
 F(A) \rightarrow F(A[t_1,\ldots, t_n])
 \]
 is bijective. Moreover, if $n \geq 2$, then $A \ten{R^{\prime}} R^{\prime} \langle t \rangle$ satisfies $(\mathrm{SE}_{n-1}^{R^{\prime} \langle t \rangle,F})$.
 \end{prop}
 
 \begin{proof}
 We prove this by induction on $n$. If we let $S \subseteq A[t_1]$ be the image of the set of monic polynomials in $R^{\prime}[t_1]$, the assumption $(\mathrm{SE}_1^{R^{\prime},F})$ yields the premise of Lemma~\ref{lemma:monic_inversion_detects_extended_objects} and thus the base case holds.
 
 We can therefore assume that $n > 1$ and that the conclusion holds for $n-1$. Let $R^{\prime \prime} \defl R^{\prime} \langle t_1 \rangle$ and note that $R^{\prime \prime} \langle t_2, \ldots, t_k \rangle=R^{\prime} \langle t_1, \ldots, t_k \rangle$ for all $k \geq 2$ by construction. It follows that there is a canonical isomorphism
 \[
 (A \ten{R^{\prime}} R^{\prime \prime}) \ten{R^{\prime \prime}} R^{\prime \prime} \langle t_2, \ldots, t_k \rangle \cong A \ten{R^{\prime}} R^{\prime} \langle t_1 ,\ldots, t_k \rangle 
 \] 
 such that the diagram
 \[
 \xymatrix{F(A) \ar[r] \ar[d] & F(A \ten{R^{\prime}} R^{\prime} \langle t_1 ,\ldots, t_k \rangle) \ar[d]^{\cong} \\ F(A \ten{R^{\prime}} R^{\prime \prime} ) \ar[r] & F\bigl((A \ten{R^{\prime}} R^{\prime \prime}) \ten{R^{\prime \prime}} R^{\prime \prime} \langle t_2, \ldots, t_k \rangle \bigr)}
 \]
 is commutative.
 
 Since the top horizontal function is surjective for all $2 \leq k \leq n$ by assumption, it follows that the bottom horizontal function is also surjective for these $k$. This means that $A \ten{R^{\prime}} R^{\prime \prime}$ satisfies $(\mathrm{SE}_{n-1}^{R^{\prime \prime},F})$, as claimed. The inductive assumption (applied to the commutative $R$-algebra $R^{\prime\prime}=R^{\prime} \langle t_1 \rangle$) implies that the bottom horizontal function in the diagram
 \[
 \xymatrix{F(A) \ar[r] \ar[d] & F(A[t_2, \ldots, t_n]) \ar[r] & F(A[t_1, \ldots, t_n]) \ar[d]^{\lambda_S} \\
 F(A \ten{R^{\prime}} R^{\prime}\langle t_1 \rangle) \ar[rr] && F(A \ten{R^{\prime}} R^{\prime} \langle t_1 \rangle[t_2,\ldots, t_n] )}
 \]
 is surjective. Since the left vertical function is surjective by $(\mathrm{SE}_n^{R^{\prime},F})$, it follows that
 \[
 F(A[t_2, \ldots, t_n]) \rightarrow F(A[t_1,\ldots, t_n]_S)
 \]
 is surjective. Thus Lemma~\ref{lemma:monic_inversion_detects_extended_objects} implies that $A[t_2,\ldots,t_n]$ lies in $\ca{E}_1^{F}$. We also know that $A$ satisfies $(\mathrm{SE}_{n-1}^{R^{\prime},F})$, so the inductive assumption implies that $A$ lies in $\ca{E}_{n-1}^{F}$, so the composite
 \[
 F(A) \rightarrow F(A[t_2,\ldots, t_k]) \rightarrow F(A[t_1,\ldots,t_n])
 \]
 is bijective, as claimed.
 \end{proof}
 
 \begin{cor}\label{cor:strong_F_extension_implies_F-regular}
 Let $F \colon \CAlg_R \rightarrow \Set$ be a functor which satisfies the Horrocks principle $(\mathrm{H})$ and let $R^{\prime}$ be a commutative $R$-algebra. If $A$ has the strong $F$-extension property over $R^{\prime}$, then $A$ is $F$-regular.
 \end{cor}
 
 \begin{proof}
 Recall that $A$ is $F$-regular if it lies in $\ca{E}_n^{F}$ for all $n \geq 1$, so the conclusion follows directly from Proposition~\ref{prop:SE_implies_extended}.
 \end{proof}
 
 \begin{thm}\label{thm:strong_F_extension_for_Laurent}
 Let $F \colon \CAlg_R \rightarrow \Set$ be a functor which satisfies the Horrocks principle $(\mathrm{H})$ and let $R^{\prime}$ be a commutative $R$-algebra. If $A$ has the strong $F$-extension property over $R^{\prime}$, then so does $A[t_1, \ldots, t_n]$ for all $n \geq 1$. Moreover, the $R^{\prime} \langle t_1, \ldots, t_n \rangle$-algebra $A \ten{R^{\prime}} R^{\prime} \langle t_1, \ldots, t_n \rangle$ has the strong $F$-extension property over $R^{\prime} \langle t_1, \ldots, t_n \rangle$.
 
 If $F$ additionally satisfies weak Zariski excision, then the $R^{\prime}$-algebra $A[t_1^{\pm}, \ldots, t_n^{\pm}]$ has the strong $F$-extension over $R^{\prime}$ and the function
 \[
 F(A) \rightarrow F(A[t_1^{\pm}, \ldots, t_n^{\pm}])
 \]
 is bijective. In fact, weak excision for patching diagrams of the form
 \[
 \xymatrix{B \ar[r]^-{\lambda_s} \ar[d]_{\lambda_{1+sB}} & B_s \ar[d]^{\lambda_{1+sB}} \\ B_{1+sB} \ar[r]_-{\lambda_s} & B_{1+sB,s} }
 \]
 (where $B$ is any $R$-algebra and $s \in B$ is an arbitrary element) suffices for this. 
 \end{thm}
 
 \begin{proof}
 We first show that $A \ten{R^{\prime}} R^{\prime} \langle t_1, \ldots, t_n \rangle$ has the strong $F$-extension property over $R^{\prime} \langle t_1, \ldots, t_n \rangle$ for all $n \geq 1$. Since
 \[
A \ten{R^{\prime}} R^{\prime} \langle t_1, \ldots, t_n \rangle  \ten{R^{\prime} \langle t_1, \ldots, t_n \rangle} R^{\prime} \langle t_1, \ldots, t_n \rangle \langle t_{n+1} , \ldots, t_{n+k} \rangle 
 \]
 is canonically isomorphic to $ A \ten{R^{\prime}} R^{\prime} \langle t_1, \ldots, t_{n+k} \rangle$, the strong $F$-extension property of $A$ over $R^{\prime}$ implies the strong $F$-extension property of $A \ten{R^{\prime}} R^{\prime} \langle t_1, \ldots, t_n \rangle$ over $R^{\prime} \langle t_1, \ldots, t_n \rangle$ for all $n \geq 1$. 
 
 From Corollary~\ref{cor:strong_F_extension_implies_F-regular} it follows that both $A$ and $A \ten{R^{\prime}} R^{\prime} \langle t_1, \ldots, t_n \rangle$ are $F$-regular. The commutative diagram
 \[
 \xymatrix{F(A) \ar[r]^-{\mathrm{surj}} \ar[d]_{\cong} & F(A \ten{R^{\prime}} R^{\prime} \langle t_1, \ldots, t_n \rangle) \ar[d]^{\cong} \\ 
 F(A[x_1, \ldots, x_k]) \ar[r] & F(A \ten{R^{\prime}} R^{\prime} \langle t_1, \ldots, t_n \rangle[x_1, \ldots, x_k])}
 \]
 and the isomorphism
 \[
 A \ten{R^{\prime}} R^{\prime} \langle t_1, \ldots, t_n \rangle [x_1, \ldots, x_k] \cong A[x_1, \ldots, x_k] \ten{R^{\prime}} R^{\prime} \langle t_1, \ldots, t_n \rangle
 \]
 imply that $A[x_1, \ldots, x_k]$ has the strong $F$-extension property over $R^{\prime}$. It follows in particular that $A[x_1]$ has the strong $F$-extension property over $R^{\prime}$.
 
 From now on, we assume that $F$ has the relevant weak excision property and we first show by induction on $k$ that the function
 \[
 F(A) \rightarrow F(A[x_1^{\pm}, \ldots, x_k^{\pm}])
 \]
 is surjective, with base case $k=0$ tautologically true. Since $A$ is $F$-regular and $A[x_1]$ has the strong $F$-extension property over $R^{\prime}$, the inductive assumption implies that
 \[
 F(A) \rightarrow F(A[x_1]) \rightarrow F(A[x_1,x_2^{\pm}, \ldots, x_k^{\pm}])
 \]
 is surjective. It only remains to check that the top horizontal function in the diagram
 \[
 \xymatrix{ F(A[x_1, x_2^{\pm}, \ldots, x_k^{\pm}]) \ar[r] \ar[d] & F(A[x_1^{\pm}, \ldots, x_k^{\pm}]) \ar[d] \\
 F\bigl(A[x_1, x_2^{\pm}, \ldots, x_k^{\pm}]_{1+x_1 A[x_1, x_2^{\pm}, \ldots, x_k^{\pm}]}\bigr) \ar[r] & F(A[x_2^{\pm}, \ldots, x_k^{\pm}] \langle x_1^{-1} \rangle) }
 \]
 is surjective, so let $\sigma \in F(A[x_1^{\pm}, \ldots, x_k^{\pm}])$. By assumption, this diagram is a weak pullback diagram, so we only need to show that the image of $\sigma$ in the lower right corner lies in the image of the bottom horizontal function. For this, it suffices to show that the image of $\sigma$ is extended from $A$.
 
 Since the right vertical map factors through the ``partial'' localization 
 \[
 F(A \ten{R^{\prime}} R^{\prime} \langle x_1^{-1} \rangle [x_2^{\pm}, \ldots, x_k^{\pm}] ) \smash{\rlap{,}}
 \]
 we can show this by checking that the composite
 \[
 F(A) \rightarrow F(A \ten{R^{\prime}} R^{\prime} \langle x^{-1} \rangle ) \rightarrow F(A \ten{R^{\prime}} R^{\prime} \langle x_1^{-1} \rangle [x_2^{\pm}, \ldots, x_k^{\pm}] )
 \]
 is surjective. The left function is surjective since $A$ has the strong $F$-extension property over $R$. To see that the right function is surjective, note that $A \ten{R^{\prime}} R^{\prime} \langle x_1^{-1} \rangle$ has the strong $F$-extension property over $R^{\prime} \langle x_1^{-1} \rangle$ by the already established first part of the theorem. We can thus apply the inductive assumption to the $R$-algebra $R^{\prime} \langle x_1^{-1} \rangle$ to conclude that the right function is indeed surjective. This shows that the function
 \[
 F(A) \rightarrow F(A[x_1^{\pm}, \ldots, x_k^{\pm}])
 \]
 is bijective for all $k \in \mathbb{N}$ if $A$ is $F$-regular over some commutative $R$-algebra $R^{\prime}$.
 
 Finally, to see that $A[x_1^{\pm}, \ldots, x_k^{\pm}]$ has the strong $F$-extension property over $R^{\prime}$, we can argue as in the case of polynomial algebras over $A$. Indeed, the commutative diagram
 \[
 \xymatrix{F(A) \ar[r]^-{\mathrm{surj}} \ar[d]_{\cong} & F(A \ten{R^{\prime}} R^{\prime} \langle t_1, \ldots, t_n \rangle) \ar[d]^{\cong} \\ 
 F(A[x_1^{\pm}, \ldots, x_k^{\pm}]) \ar[r] & F(A \ten{R^{\prime}} R^{\prime} \langle t_1, \ldots, t_n \rangle[x_1^{\pm}, \ldots, x_k^{\pm}])}
 \]
 and the isomorphism
 \[
 A \ten{R^{\prime}} R^{\prime} \langle t_1, \ldots, t_n \rangle [x_1^{\pm}, \ldots, x_k^{\pm}] \cong A[x_1^{\pm}, \ldots, x_k^{\pm}] \ten{R^{\prime}} R^{\prime} \langle t_1, \ldots, t_n \rangle
 \]
 show that $A[x_1^{\pm}, \ldots, x_k^{\pm}]$ has the strong $F$-extension property over $R^{\prime}$.
 \end{proof}
 
 The strong $F$-extension property is easiest to check if both the domain and the codomain of the function in question consist of a single object. If $F \colon \CAlg_R \rightarrow \Set_{\ast}$ is a functor and $\ca{C}$ is a class of $F$-contractible objects with the property that $R^{\prime} \in \ca{C}$ implies that $R^{\prime} \langle t \rangle \in \ca{C}$, then all algebras in $\ca{C}$ have the strong $F$-extension property. The following lemma gives three examples of such classes.
 
 \begin{lemma}\label{lemma:strong_F_extension_from_F_contractible}
 If $\ca{C}$ is the class of all fields (respectively all principal ideal domains, respectively all noetherian rings of Krull dimension $\leq d$) and $R^{\prime} \in \ca{C}$, then $R^{\prime} \langle t \rangle \in \ca{C}$. If $F \colon \CAlg_R \rightarrow \Set_{\ast}$ is a functor such that all algebras of $\ca{C}$ are $F$-contractible, then they have the strong $F$-extension property.
 \end{lemma}
 
 \begin{proof}
 As noted above, the second claim follows from the closure property of $\ca{C}$. If $\ca{C}$ is the class of fields, then $R^{\prime} \langle t \rangle$ is the field of fractions of $R^{\prime}[t]$, hence lies in $\ca{C}$. If $R^{\prime}$ is a principal ideal domain, then $R^{\prime} \langle t \rangle$ is a regular unique factorization domain as localization of $R^{\prime}[t]$. To establish claim about the second and third class $\ca{C}$, it thus suffices to show that the Krull dimension of $R^{\prime} \langle t \rangle$ coincides with the Krull dimension of $R^{\prime}$ if $R^{\prime}$ is noetherian. This follows from the fact that all maximal ideals of $R[t]$ of height $d+1$ contain a monic polynomial, see \cite[Proposition~IV.1.2]{LAM}.
 \end{proof}
 
 \begin{example}\label{example:strong_F_extension_property_from_F_contractible}
 For each $r \geq 1$, all fields and all principal ideal domains are $P_r$-contractible (since all projective modules over such rings are free). Thus fields and principal ideal domains have the strong $P_r$-extension property. Since the Bass stable range of a noetherian ring of Krull dimension $d$ is at most $d+1$, all noetherian rings of Krull dimension $d$ are $W_{r}$-contractible for $r > d+1$. The above Lemma thus shows that such rings have the strong $W_{r}$-extension property.
 \end{example}
 
 The following result provides a useful summary of our discussion of the Horrocks principle.
 
 \begin{cor}\label{cor:Horrocks_and_analytic_excision}
 Let $F \colon \CAlg_R \rightarrow \Set_{\ast}$ be a finitary functor which satisfies weak analytic excision. Assume that one of the following conditions holds:
  \begin{enumerate}
 \item[(A)] All commutative $R$-algebras which are local rings are $F$-contractible;
 \item[(B)] The functor $F$ admits a natural transitive group action.
 \end{enumerate}
 Finally, we assume that $F$ satisfies the local Horrocks principle $(\mathrm{H}_{\mathrm{loc}})$.
 
 Let $R^{\prime}$ be a commutative $R$-algebra and assume that $A$ satisfies $(\mathrm{SE}_n^{R^{\prime},F})$ for some $n \geq 1$. Then $A$ lies in $\ca{E}_n^F$. If $A$ has the strong $F$-extension property over $R^{\prime}$, then $A[t_1, \ldots, t_k, x_1^{\pm}, \ldots, x_n^{\pm}]$ is $F$-regular for all $k,n \in \mathbb{N}$ and the function
 \[
 F(A) \rightarrow F(A[t_1, \ldots, t_k, x_1^{\pm}, \ldots, x_n^{\pm}])
 \]
 is bijective.
 \end{cor}
 
 \begin{proof}
 From Corollary~\ref{cor:analytic_excision_implies_Q} we know that $F$ satisfies $(\mathrm{Q})$. Thus Proposition~\ref{prop:QHloc_implies_H} is applicable and implies that the (global) Horrocks principle $(\mathrm{H})$ holds (under either of the assumptions (A) or (B)). The first claim was established in Proposition~\ref{prop:SE_implies_extended} and the second claim is direct consequence of Theorem~\ref{thm:strong_F_extension_for_Laurent}.
 \end{proof}
 
 This corollary makes it possible to check $F$-regularity for certain basic $R$-algebras such as fields. Combined with Theorem~\ref{thm:F_regularity_for_unramified_regular_rings}, we can often establish $F$-regularity for unramified regular local rings from this basic case.
 
\begin{example}
 As originally observed by Quillen in \cite{QUILLEN}, the principle $(\mathrm{Q})$, combined with Horrocks' Theorem \cite[Theorem~1]{HORROCKS} implies that all fields and all principal ideal domains are $F$-regular (see Proposition~\ref{prop:QHloc_implies_H} for the implication and Example~\ref{example:strong_F_extension_property_from_F_contractible} for the strong $P_r$-extension property for fields and principal ideal domains). Applying Theorem~\ref{thm:excision_implies_three_principles} to $P_r$, we recover the Lindel--Popescu Theorem that all unramified regular local rings are $P$-regular (for local rings, more generally connected rings, we can deduce $P$-regularity from $P_r$-regularity for all $r \geq 1$).
\end{example}

\begin{example}
 From the Fundamental Theorem of algebraic K-theory \cite{GRAYSON_FUNDAMENTAL} we know that there are short exact sequences
 \[
 \xymatrix@C=15pt{0 \ar[r] & K_i(A) \ar[r] & K_i(A[t]) \oplus K_i(A[t^{-1}]) \ar[r] & K_i(A[t,t^{-1}]) \ar[r] & K_{i-1}(A) \ar[r] & 0 }
 \]
 for all commutative rings $A$ and all $i \in \mathbb{N}$ (in fact, by definition of negative $K$-groups, this holds for all $i \in \mathbb{Z}$). Thus
 \[
 \xymatrix{K_i(A) \ar[r] \ar[d] & K_i(A[t^{-1}]) \ar[d] \\ K_i(A[t]) \ar[r] & K_i(A[t,t^{-1}])}
 \]
 is a pullback square, so Lemma~\ref{lemma:pullback_version_of_H} implies that each $K_i \colon \CRing \rightarrow \Ab$ satisfies the Horrocks principle $(\mathrm{H})$.
\end{example}

\begin{example}\label{example:W_r_strong_extension_property}
 For each $r \geq 3$, the functor $W_r \colon \CAlg_R \rightarrow \Set_{\ast}$ which sends $A$ to $\mathrm{Um}_r(A) \slash \mathrm{E}_r(A)$ satisfies $(\mathrm{H}_{\mathrm{loc}})$ by \cite[Theorem~1.1]{RAO_TWO_EXAMPLES} (which is based on Suslin's factorization theorem for $\mathrm{E}_r(R[x,x^{-1}], \mathfrak{m}[x,x^{-1}])$ for a local ring $(R,\mathfrak{m})$). Since these functors satisfy weak analytic excision (see Proposition~\ref{prop:W_r_weak_analytic_excision}) and local rings are $W_r$-contractible for all $R$, Corollary~\ref{cor:Horrocks_and_analytic_excision} implies that all rings with the strong $W_r$-extension property are $W_r$-regular. From Example~\ref{example:strong_F_extension_property_from_F_contractible} it follows that all noetherian rings of Krull dimension $d$ are $W_r$-regular for all $r \geq d+2$. For a proof of this based on Suslin's monic polynomial theorem, see \cite[Theorem~III.3.1]{LAM}.
\end{example}

\begin{example}
If $G$ is a reductive group scheme over $R$, then the functor
\[
H^1(-,G) \colon \CAlg_R \rightarrow \Set_{\ast}
\]
satisfies the local Horrocks priniciple $(\mathrm{H}_{\mathrm{loc}})$ under suitable conditions (see for example \cite[Proposition~8.4]{CESNAVICIUS} for a very general version). However, it is rarely the case that all $G$-torsors over all local $R$-algebras are trivial, so neither Theorem~\ref{thm:weak_excision_implies_QR} nor Corollary~\ref{cor:Horrocks_and_analytic_excision} are applicable in this case. For this reason, one often restricts attention to Zariski-locally trivial torsors. An axiomatic treatment of one such situation (where the base $R$ is a field) can be found in \cite[Th{\'e}or{\`e}me~1.1]{COLLIOT-THELENE_OJANGUREN}.
\end{example}

\begin{example}
For $r \geq 3$, the functors $K_{1,r}$ and $SK_{1,r} \colon \CAlg_R \rightarrow \Grp$ satisfy the (local) Horrocks principle $(\mathrm{H}_{\mathrm{loc}})$ by \cite[Corollary~5.7]{SUSLIN_SPECIAL}. Since $K_{1,r}(A) \cong SK_{1,r} \rtimes A^{\times} $ and every reduced ring is $(-)^{\times}$-regular, we can restrict attention to $SK_{1,r}$.

 All fields are $SK_{1,r}$-contractible, so they have the strong $SK_{1,r}$-extension property by Lemma~\ref{lemma:strong_F_extension_from_F_contractible} for all $r \geq 3$. Thus Part~(i) of Theorem~\ref{thm:F_regularity_for_unramified_regular_rings} is applicable and recovers Vorst's Theorem that all regular rings containing a field are $SK_{1,r}$-regular (see \cite[Theorem~3.3]{VORST_GLN}). The following result extends this to all unramified regular local rings.
\end{example}

\begin{prop}\label{prop:SK_1r_regularity_for_dvr_and_unramified}
 Let $r \geq 3$. Then all discrete valuation rings have the strong $SK_{1,r}$-extension property. Moreover, all unramified regular local rings are $SK_{1,r}$-regular.
\end{prop}

\begin{proof}
 We will prove the second statement by applying Theorem~\ref{thm:excision_implies_three_principles}, so we only need to check that all discrete valuation rings are $SK_{1,r}$-regular. To do this, it suffices to check that all discrete valuation rings have the strong $SK_{1,r}$-extension property. We prove this by establishing the following claim: the category $\ca{C}$ of principal ideal domains which are $SK_{1,r}$-contractible is closed under $R^{\prime} \mapsto R^{\prime} \langle x \rangle$. Granting this claim, we can apply Lemma~\ref{lemma:strong_F_extension_from_F_contractible} to conclude that all rings in $\ca{C}$ have the strong $SK_{1,r}$-extension property. Since discrete valuation rings are $SK_{1,r}$-contractible, they lie in $\ca{C}$, hence they have the strong $SK_{1,r}$-extension property, as claimed.
 
 To finish the proof, it remains to show that $R^{\prime} \langle x \rangle$ lies in $\ca{C}$ for all $R^{\prime} \in \ca{C}$. Since $R^{\prime}$ lies in $\ca{C}$ and all the functors $SK_{1,r}$ are $j$-invariant, it follows that $R^{\prime}[t]_{1+tR^{\prime}[t]}$ is $SK_{1,r}$-contractible for all $r \geq 3$. From \cite[Proposition~IV.6.1]{LAM}, applied to $B=R^{\prime}[t]_{1+tR^{\prime}[t]}$ and $s=t$ it follows that $B_s \cong R^{\prime} \langle t^{-1} \rangle$ is $SK_{1,r}$-contractible. Thus $R^{\prime} \langle x \rangle$ lies in $\ca{C}$, as claimed (recall that $R^{\prime} \langle x \rangle$ is a principal ideal domain by \cite[Corollary~IV.1.3]{LAM}).
\end{proof}

 This example can be extended from $\mathrm{GL}_r$ and $\mathrm{SL}_r$ to more general reductive groups. Recall that a reductive group $G$ over $R$ has isotropic rank $\geq n$ if every non-trivial normal semisimple $R$-subgroup contains $(\mathbb{G}_{m,R})^{n}$. This condition is stable under scalar extension (see the discussion before \cite[Definition~2.1]{STAVROVA_EQUICHARACTERISTIC}). Moreover, for such $G$, the group $E_{P}(A) \subseteq G(A)$ is independent of the choice of parabolic subgroup $P$ and $E(A) \defl E_P(A)$ is normal in $G(A)$ (see \cite[Theorem~2.3]{STAVROVA_EQUICHARACTERISTIC}), so we can write $K_1^G$ for $K_1^{G,P} \colon \CAlg_R \rightarrow \Grp$ in this case. From Remark~\ref{rmk:stavrova_conditions_satisfied_isotropic_rank} we know that $K_1^G$ satisfies weak analytic excision if $G$ has isotropic rank $\geq 2$. The functor $K_1^G$ is finitary by Lemma~\ref{lemma:K_1_GP_finitary}.
 
\begin{example}
 If $G$ is a reductive group over $R$ of isotropic rank $\geq 2$, then $K_1^G$ satisfies the local Horrocks principle $(\mathrm{H}_{\mathrm{loc}})$ by \cite[Theorem~1.1]{STAVROVA_HOMOTOPY}. In fact, as in the case for $G=\mathrm{SL}_r$, this follows from the analogue of Suslin's factorization theorem over local rings (see \cite[Corollary~5.2]{STAVROVA_HOMOTOPY}).
 
 If $k$ is a field (and an $R$-algebra), and if $G$ is moreover simply connected and absolutely almost simple, then
 \[
 K_1^G(k) \rightarrow K_1^G(k \langle t \rangle)
 \]
 is surjective by \cite[Th{\'e}or{\`e}me~5.8]{GILLE_KNESER-TITS}. Since this condition is stable under scalar extension (see \cite[Expos{\'e}~XXII, D{\'e}finition~2.7]{SGA3_NEW}), this theorem shows in fact that
 \[
 K_1^G(k) \rightarrow K_1^G(k \langle t_1, \ldots, t_n \rangle)
 \]
 is surjective for all $n \geq 1$. Thus $k$ has the strong $K_1^G$-extension property, so it is $K_1^G$-regular by Corollary~\ref{cor:strong_F_extension_implies_F-regular}. If the field is perfect, we can apply Popescu's Theorem and Propositions~\ref{prop:Lindel_etale_neighbourhood} and \ref{prop:Lindel_result_for_general_F} to show that all regular rings containing $k$ are $K_1^G$-regular. In \cite[\S 6]{STAVROVA_HOMOTOPY}, one can find results which generalize this to larger classes of reductive groups $G$.
\end{example}
\section{Henselian pairs}\label{section:henselian_pairs}

 Let $R$ be a commutative ring and let $I \subseteq R$ be an ideal. Recall that the pair $(R,I)$ is called a \emph{henselian pair} if the following condition holds: given any {\'e}tale ring homomorphism $\varphi \colon R \rightarrow R^{\prime}$, if there exists a homomorphism $\psi \colon R^{\prime} \rightarrow R \slash I$ such that the triangle
 \[
 \xymatrix{R \ar[rr]^{\varphi} \ar[rd]_{\pi} & & R^{\prime} \ar@{-->}[ld]^{\psi} \\ & R \slash I }
 \]
 commutes (where $\pi$ denotes the canonical projection), then there exists a ring homomorphism $\lambda \colon R^{\prime} \rightarrow R$ such that $\lambda \varphi=\id_R$ and $\pi \lambda=\psi$. Such an $R^{\prime}$ is always finitely presentable. If we choose a system of equations giving such a presentation, then the condition states that solutions in the quotient $R \slash I$ can be lifted to $R$ (as long as the system of equations defines an {\'e}tale ring homomorphism).
 
 There are many equivalent characterizations of henselian pairs, see for example \cite[\href{https://stacks.math.columbia.edu/tag/09XI}{Lemma 09XI}]{stacks-project}. We will only need the following characterizations of henselian pairs. Recall from \cite[Definition~1.2]{GRECO} that a monic polynomial $f(t)=\sum_{i=0}^n a_i t^i \in R[t]$ is called an \emph{$N$-polynomial over $(R,I)$} if $a_0 \in I$ and $a_1$ is a unit modulo $I$.
 
 \begin{prop}\label{prop:characterizations_of_henselian_pairs}
 Let $R$ be a commutative ring and let $I \subseteq R$ be an ideal. Then the following are equivalent:
 \begin{enumerate}
 \item[(i)] The pair $(R,I)$ is henselian;
 \item[(ii)] The ideal $I$ is contained in the Jacobson radical of $R$ and for every monic polynomial $f(t) \in R[t]$ of the form
 \[
 f(t)=t^n(t-1)+a_n t^n + \ldots + a_1 t + a_0
 \]
 with $a_0, \ldots, a_n \in I$, there exists an $\alpha \in 1+I \subseteq R$ with $f(\alpha)=0$;
 \item[(iii)] The ideal $I$ is contained in the Jacobson radical of $I$ and every $N$-polynomial has a root $\alpha \in I$.
 \end{enumerate}
 \end{prop}
 
  \begin{proof}
  The implication $(ii)\Rightarrow (i)$ is the implication $(5) \Rightarrow (2)$ of \cite[\href{https://stacks.math.columbia.edu/tag/09XI}{Lemma 09XI}]{stacks-project}. 
  
  To see that $(i) \Rightarrow (iii)$, note that the image $\bar{f}$ of an $N$-polynomial in $R \slash I$ factors as product $\bar{f} = x \cdot \bar{g}$ for some monic polynomial $\bar{g}$ with $\bar{g}(0)$ a unit. Thus $x$ and $\bar{g}$ are comaximal in $R \slash I [t]$, so the implication $(2) \Rightarrow (1)$ of \cite[\href{https://stacks.math.columbia.edu/tag/09XI}{Lemma 09XI}]{stacks-project} (combined with the alternative definition of henselian pairs used there, see \cite[\href{https://stacks.math.columbia.edu/tag/09XE}{Definition 09XE}]{stacks-project}) shows that $f$ has a root in $I$.
  
  Finally, to show $(iii) \Rightarrow (ii)$, it suffices to observe that the change of variables $x=t-1$ sends a polynomial as in $(ii)$ to an $N$-polynomial.
  \end{proof}
  
  Henselian pairs can also be defined by a factorization property for monic polynomial (see \cite[\href{https://stacks.math.columbia.edu/tag/09XE}{Definition 09XE}]{stacks-project}). We will frequently use the fact that the polynomial does not need to be monic, as long as \emph{one} of the factors in the quotient is monic. This is proved in Part~(6) \cite[\href{https://stacks.math.columbia.edu/tag/04GG}{Lemma 04GG}]{stacks-project} but only in the case where $R$ is a local ring with maximal ideal $I=\mathfrak{m}$ (taking into account that polynomials in the quotient are monic up to a unit since the quotient is a field in this case).
  
 In order to prove this for general henselian pairs, we will use the following construction. Given two natural numbers $n, m$ with $n+m >0$ and two polynomials $f(t)=\sum_{i=0}^n a_i t^i$ and $g(t)=\sum_{j=0}^m b_j t^j$ in $R[t]$, the element
  \[
\mathrm{Res}_{n,m}(f,g) \defl  \mathrm{det} \begin{pmatrix}
  a_n & a_{n-1} &   \ldots & a_1      & a_0 \\
        & a_n       & a_{n-1} &  \ldots & a_1       & a_0 \\
        &             & \ddots  & \ddots &             & \ddots  &   \ddots \\
        &             &             & a_n       & a_{n-1} & \ldots & a_1 & a_0 \\
   b_m & b_{m-1} &   \ldots & b_1      & b_0 \\
        & b_m       & b_{m-1} &  \ldots & b_1       & b_0 \\
        &             & \ddots  & \ddots &             & \ddots  &   \ddots \\
        &             &             & b_m       & b_{m-1} & \ldots & b_1 & b_0 \\       
  \end{pmatrix}
  \]  
  of $R$ is called the \emph{resultant} of $f$ and $g$. The above matrix is called the \emph{Sylvester matrix} of $f$ and $g$. Here the first $m$ rows contain copies of the coefficients of $f$ and the final $n$ rows contain copies of the coefficients of $g$, so that the matrix is a square matrix of size $m+n$. We only need the following basic fact about resultants: if $R$ is a field, then the resultant is $0$ if and only if either both $a_n$ and $b_m$ are $0$, or the two polynomials have a common non-constant factor (see for example the summary in \cite[p.~89]{VAN_DER_WAERDEN}). The following lemma shows that one of these implications holds over arbitrary commutative rings if in addition one of the polynomials is monic.
  
 \begin{lemma}\label{lemma:resultant_unit_if_comaximal}
 Let $R$ be a commutative ring. Let $n > 0$, let
 \[
 f(t)=t^n+a_{n-1} t^{n-1} + \ldots + a_0
 \]
 be a monic polynomial in $R[t]$ and let $g(t)=b_m t^m + \ldots + b_0$ be an arbitrary polynomial. If $f(t)$ and $g(t)$ are comaximal in $R[t]$, then the resultant $\mathrm{Res}_{n,m}(f,g) \in R$ is a unit.
 \end{lemma}
  
 \begin{proof}
 Let $\mathfrak{m} \subseteq R$ be a maximal ideal and consider the resultant $\mathrm{Res}_{n,m}(\bar{f},\bar{g})$ of the reductions $\bar{f}$ and $\bar{g}$ in $R \slash \mathfrak{m}$ of $f$ and $g$. Since $\bar{f}$ is monic of degree $\geq 1$ and the two polynomials $\bar{f}$ and $\bar{g}$ are comaximal, it follows from the above recollection that $\mathrm{Res}_{n,m}(\bar{f},\bar{g})$ is not equal to $0$, that is, $\mathrm{Res}_{n,m}(\bar{f},\bar{g})$ is a unit in $R \slash \mathfrak{m}$. Since $\mathfrak{m}$ was arbitrary and $\overline{\mathrm{Res}_{n,m}(f,g)}=\mathrm{Res}_{n,m}(\bar{f},\bar{g})$, it follows that $\mathrm{Res}_{n,m}(f,g) \in R$ is a unit. 
 \end{proof}

 \begin{thm}\label{thm:generalized_hensel_lemma}
 Let $(R,I)$ be a henselian pair and let $f(t)=\sum_{i=0}^{n+m} c_i t^i$ be a polynomial in $R[t]$, not necessarily monic, with $n,m \geq 1$. Suppose that there exists a factorization $\bar{f}=\bar{g} \bar{h}$ in $R \slash I[t]$ such that $\bar{g}$ is monic of degree $n$, and the two polynomials $\bar{g}$ and $\bar{h}$ are comaximal in $R \slash I [t]$. Then there exists a factorization $f=gh$ in $R[t]$ with $g$ monic of degree $n$ lifting the above factorization.
 \end{thm}
 
 \begin{proof}
 Let $R^{\prime \prime}$ be the $R$-algebra obtained by freely factoring $f$ as a product of a monic polynomial $G(t)=t^n + \sum_{i=0}^{n-1} A_i t^i$ and a polynomial $H(t)=\sum_{j=0}^m B_j t^j$, defined as follows. The algebra $R^{\prime \prime}$ is obtained in two steps. First, we freely add variables $A_0, \ldots, A_{n-1}$ and $B_0, \ldots, B_m$ to $R$. We set $A_n=1$. For $0 \leq k \leq n+m$, let $p_k=\sum_{i+j=m+n-k} A_i B_j-c_{m+n-k}$, where $0 \leq i \leq n$ and $0 \leq j \leq m$. The algebra $R^{\prime \prime}$ is given by the quotient
 \[
 R^{\prime \prime} \defl R[A_0, \ldots, A_{n-1},B_0, \ldots, B_m] \slash (p_0, \ldots, p_{n+m})
 \]
 modulo the ideal generated by the $p_i$.
 
 This description can be simplified a bit. The relation $p_0$ is given by $A_n B_m-c_{n+m}$, which is equal to $B_m-c_{n+m}$ by our convention that $A_n=1$. Thus we can remove the variable $B_m$ from the above presentation and replace each occurrence of $B_m$ by $c_{n+m}$ in the remaining relations $p_1, \ldots, p_{n+m}$. In this way, the algebra $R^{\prime \prime}$ is described by adding $n+m$ variables and imposing $n+m$ relations.
 
 Now we order the variables $A_i$ and $B_j$ as follows. We set
 \[
  x_1=B_{m-1},\; x_2=B_{m-2},\; \ldots,\; x_m=B_0,\; x_{m+1}=A_{n-1},\; \ldots,\; x_{n+m}=A_0 \smash{\rlap{.}}
 \]
 The Jacobian matrix of this presentations is thus given by $(\partial p_j \slash \partial x_i)_{i,j=1}^{n+m}$. By construction we have for $0 \leq j <m$:
 \[
 \partial p_k \slash \partial B_j = A_i \quad \text{where} \quad i=n+m-k-j
 \]
 if $0 \leq i \leq n$ and $ \partial p_k \slash \partial B_j=0$ if $i<0$ or $i>n$ (with the convention that $A_n=1$). Similarly, for $0 \leq i < n$ we have
\[
\partial p_k \slash \partial A_i = B_j \quad \text{where} \quad j=n+m-k-i
\]
if $0 \leq j \leq m$ and $\partial p_k \slash \partial A_i=0$ if $j<0$ or $j>m$ (here we again use our convention $B_m=c_{n+m}$). From these equalities it follows that the above Jacobian matrix is precisely the Sylvester matrix of the polynomials $G(t)$ and $H(t)$.
 
 Now let $R^{\prime}$ be obtained from $R^{\prime \prime}$ by inverting the resultant of $G(t)=t^n + \sum_{i=0}^{n-1} A_i t^i$ and $H(t)=c_{n+m}t^m +\sum_{j=0}^{m-1} B_j t^j$. Since we have adjoined $n+m$ variables and imposed $n+m$ equations, and the determinant of the Jacobian matrix of the above presentation is invertible in $R^{\prime}$ (it coincides with resultant by the above discussion), the algebra $R^{\prime}$ is {\'e}tale over $R$.
  
 By assumption, there exists a homomorphism $R^{\prime \prime} \rightarrow R \slash I$ of $R$-algebras sending the generic polynomials $G$ and $H$ to $\bar{g}$ and $\bar{h}$ respectively. From Lemma~\ref{lemma:resultant_unit_if_comaximal} it follows that this homomorphism factors through the localization $R^{\prime \prime} \rightarrow R^{\prime}$. Since $R \rightarrow R^{\prime}$ is {\'e}tale, the conclusion follows from our definition of henselian pairs.
 \end{proof}
 
 We will frequently use the following special case of the above result. Recall that for any ideal $I \subseteq R$, we write $I[t]$ for the ideal in $R[t]$ generated by $I$.
 
 \begin{cor}\label{cor:splitting_into_monic_times_1+I_polynomial}
 Let $(R,I)$ be a henselian pair and let $f(t) \in R[t]$ be a polynomial. If the leading coefficient of $\bar{f} \in  R \slash I[t]$ is a unit, then there exists a factorization $f=u \cdot g \cdot h$ where $u \in R^{\times}$, $g$ is monic of the same degree as $\bar{f}$, and $h \in 1+tI[t]$. If the leading coefficient of $\bar{f}$ is $1$, then one can arrange additionally that $u \equiv 1 \mod I$.
 \end{cor}
 
 \begin{proof}
 If $f$ and $\bar{f}$ have the same degree, we can take $h=1$ and $u$ equal to the leading coefficent of $f$ . In this case, we have $u \in R^{\times}$ since $I$ is contained in the Jacobson radical. If not and if $\mathrm{deg}(\bar{f})=0$, then we can take $u=f(0)$ and $g=1$. In the remaining case, we thus have $0<\mathrm{deg}(\bar{f})<\mathrm{deg}(f)$ and therefore $\mathrm{deg}(f) \geq 2$.
 
 In this case, we can take the factorization $\bar{f}=\bar{g} \cdot v$ where $\bar{g}$ is monic of the same degree as $\bar{f}$ and $v$ is the leading coefficient of $\bar{f}$. Since this leading coefficient is a unit in $R \slash I$ by assumption, the pair $(\bar{g},v)$ is comaximal. The above inequality involving the degrees of $f$ and $\bar{f}$ implies that Theorem~\ref{thm:generalized_hensel_lemma} is applicable. It follows that there exists a monic polynomial $g(t) \in R[t]$ and a polynomial $p(t) \in R[t]$ such that $f=gp$, $g$ is a lift of $\bar{g}$ (so $\mathrm{deg}(g)=\mathrm{deg}(\bar{g})=\mathrm{deg(\bar{f})}$), and $p$ is a lift of the constant polynomial $v$. Since $v$ is a unit, the image of $u=p(0)$ modulo $I$ is invertible. Since $I$ is contained in the Jacobson radical, it follows that $u$ is a unit. Setting $h(t)=u^{-1} p(t)$ we obtain the desired factorization (note that $h(t) \in 1+tI[t]$ by construction).
 
 The assertion about $u$ in the case where the leading coefficient of $\bar{f}$ is $1$ follows from the fact that this leading coefficient is given by the equivalence class $\bar{u}$ of $u$ modulo $I$ (since $h \equiv 1 \mod I$).
 \end{proof}
 
 Recall that $R \langle t \rangle$ denotes the localization of $R[t]$ at the multiplicative set of monic polynomials.
 
 \begin{cor}\label{cor:localization_of_monic_polynomials_jacobson_radical}
 Let $(R,I)$ be a henselian pair. Then the kernel of the canoncial projection
 \[
 \pi \colon R \langle t \rangle_{1+I[t]} \rightarrow R \slash I \langle t \rangle
 \]
 is contained in the Jacobson radical of $R \langle t \rangle_{1+I[t]}$.
 \end{cor}
 
 \begin{proof}
 It suffices to check that each element $p$ of $R \langle t \rangle_{1+I[t]}$ which is  sent to a unit by $\pi$ is itself a unit. Without loss of generality we can assume that $p \in R[t]$. Note that a polynomial in $R \slash I [t]$ is sent to a unit in $R \slash I \langle t \rangle$ if and only if its leading coefficient is a unit. By the above corollary we can thus find a factorization $p=ugh$ where $u \in R^{\times}$, $g$ is monic, and $h \in 1+I[t]$, so $p$ is indeed a unit in $R \langle t \rangle_{1+I[t]}$.
 \end{proof}
 
 If we define the category of \emph{pairs} to have objects $(R,I)$ where $R$ is a commutative ring and $I \subseteq R$ is an ideal and morphisms $(R,I) \rightarrow (R^{\prime},I^{\prime})$ the homomorphisms $\varphi \colon R \rightarrow R^{\prime}$ with $\varphi(I) \subseteq I^{\prime}$, then the inclusion of all henselian pairs has a left adjoint, called the \emph{henselization} and denoted by $(R, I) \mapsto (R^h_{I},I^h)$. In \S \ref{section:pseudoelementary}, we want to establish weak excision results for certain patching diagrams arising from the henselization functor. In order to do this, it is convenient to use an explicit construction of $(R^{h}_I, I^h)$ which only relies on certain simple building blocks. In fact, we will provide two such constructions, relying on either \emph{standard Nisnevich} respectively \emph{basic Nisnevich} homomorphisms. The former have the advantage that they can also be used to describe {\'e}tale neighbourhoods of local rings (which is a crucial ingredient in the proof of the main result \S \ref{section:pseudoelementary}), but for the application of the henselization in \S \ref{section:pseudoelementary} we will need to use basic Nisnevich homomorphisms. 
 
 Recall that a ring homomorphism $R \rightarrow R^{\prime}$ is called \emph{standard {\'e}tale} if it is (up to isomorphism) of the form $R \rightarrow (R[t] \slash f)_g$ where $f(t)$ is monic and the derivative $f^{\prime}(t)$ is invertible in the localization of $R[t] \slash f$ at $g(t)$ (see \cite[\href{https://stacks.math.columbia.edu/tag/00UB}{Definition 00UB}]{stacks-project}).
 
 \begin{dfn}\label{dfn:standard_Nisnevich_L_homomorphism}
 Let $I \subseteq R$ be an ideal. A ring homomorphism $R \rightarrow R^{\prime}$ is called a \emph{$L$-homomorphism along $I$} if there exists a polynomial $f(t)=a_0 + a_1 t + \ldots + a_n t^n \in R[t]$ with $a_0 \in I$, $a_1 \in R^{\times}$ and a polynomial $g(t) \in R[t]$ with $g(0)=1$ such that the $R$-algebra $R^{\prime}$ is isomorphic to $(R[t] \slash f)_g$ and the following two conditions hold:
 \begin{enumerate}
 \item[(1)] The derivative $f^{\prime}(t)$ of $f(t)$ is a unit in $(R[t] \slash f)_g$;
 \item[(2)] The unique polynomial $\tilde{f}(t)$ such that $f(t)-a_0=t \cdot \tilde{f}(t)$ holds is a unit in $(R[t] \slash f)_g$.
  \end{enumerate}
 If $f$ can moreover be chosen to be monic, then $R \rightarrow R^{\prime}$ is called \emph{standard Nisnevich along $I$}.
 \end{dfn}
 
 Both $L$-homomorphisms and standard Nisnevich homomorphisms are by definition {\'e}tale, but only the latter are in general standard {\'e}tale. For example, for any $f$ as above, the homomorphism $R \rightarrow (R[t] \slash f)_{a_1^{-2} f^{\prime} \tilde{f}}$ is an $L$-homomorphism along any ideal $I$ containing $a_0$. It is standard Nisnevich along $I$ if $f$ is monic. This definition is closely related to the notion of $N$-polynomials (see \cite[Definition~1.2]{GRECO}), but note that not every $N$-polynomial can be used in place of $f$ since we assume that $a_1 \in R^{\times}$ (instead of $\bar{a}_1 \in (R\slash I)^{\times}$). This distinction disappears if $I$ is contained in the Jacobson radical but is relevant otherwise.
 
 \begin{lemma}\label{lemma:standard_Nisnevich_and_Jacobson_radical}
 Let $I \subseteq R$ be an ideal and let $a \in I$. Then the localization $R \rightarrow R_{1+a}$ is standard Nisnevich along $I$. If $x \in R$ is an element such that $\bar{x} \in R \slash I$ is a unit, then there exists a homomorphism
 \[
 \varphi \colon R \rightarrow R^{\prime}
 \]
 which is standard Nisnevich along $I$ such that $\varphi(x)$ is a unit in $R^{\prime}$.
 \end{lemma}
 
\begin{proof}
 The first claim can be proved by taking $f(t)=-a+t$, $g(t)=1+t$ and noting that $R_{1+a} \cong (R[t]\slash f)_g$ as $R$-algebras. To see the second claim, let $y \in R$ be an element such that $\bar{x} \bar{y}=1$ in $R \slash I$. Thus $xy=1+a$ for some element $a \in I$. The localization $R \rightarrow R_{1+a}$ yields the desired ring homomorphism (which is standard Nisnevich along $I$ by the first claim).
\end{proof}

\begin{lemma}\label{lemma:base_change_for_L_homomorphism}
If $R \rightarrow R^{\prime}$ is an $L$-homomorphism along $I$ (respectively standard Nisnevich along $I$), then for any ring homomorphism $\varphi \colon R \rightarrow A$, the base change $A \rightarrow A \ten{R} R^{\prime}$ of $\varphi$ is an $L$-homomorphism along $IA$ (respectively standard Nisnevich along $IA$).
\end{lemma}

\begin{proof}
 Pick polynomials $f$ and $g$ in $R[t]$ as in the definition of $L$-homomorphisms along $I$ (and choose $f$ monic in case $R \rightarrow R^{\prime}$ is standard Nisnevich along $I$). Then $\varphi(f)$ and $\varphi(g)$ are polynomials in $A[t]$ satisfying the same conditions, so the claim follows from the isomorphism
 \[
 A \ten{R} R^{\prime} \cong \bigl( A[t] \slash \varphi(f) \bigr)_{\varphi(g)}
 \]
 of $A$-algebras.
\end{proof}

 The following proposition shows that $L$-homomorphisms along $I$ give rise to affine Nisnevich squares. The argument follows closely the proof of \cite[Lemma on p.~321]{LINDEL} in the case of {\'e}tale neighbourhoods.
 
 \begin{prop}\label{prop:L_homomorphism_implies_affine_Nisnevich_square}
 Let $I \subseteq R$ be an ideal and let $R \rightarrow R^{\prime}$ be an $L$-homomorphism along $I$. Then the induced homomorphism $R \slash I^k \rightarrow R^{\prime} \slash I^k R^{\prime}$ is an isomorphism for all $k \geq 1$.
 \end{prop}
  
 \begin{proof}
 We first show that $R \slash I^k \rightarrow R^{\prime} \slash I^k R^{\prime}$ is injective. Pick polynomials $f$ and $g$ as in Definition~\ref{dfn:standard_Nisnevich_L_homomorphism} so that $R \rightarrow R^{\prime}$ is isomorphic to $R \rightarrow (R[t] \slash f)_g$. Since all the elements of $1+I$ are sent to units in $R \slash I^k$, the homomorphism $R \slash I^k \rightarrow R^{\prime} \slash I^k R^{\prime}$ can be obtained as base change of
 \[
 \varphi \colon R_{1+I} \rightarrow (R_{1+I}[t] \slash f)_g
 \]
 along the projection $R_{1+I} \rightarrow R_{1+I} \slash I^k \cong R \slash I^k$. We claim that $\varphi$ is faithfully flat. This implies the desired injectivity since faithfully flat morphisms are universally injective.
 
 To see that $\varphi$ is faithfully flat, we need to check that every maximal ideal $\mathfrak{m}$ of $R_{1+I}$ is of the form $\varphi^{-1}(\mathfrak{p})$ for some prime ideal $\mathfrak{p} \subseteq  (R_{1+I}[t] \slash f)_g$. By construction we have $I \subseteq \mathfrak{m}$, so $f(t)$ lies in $(\mathfrak{m},t)$ since $f(0) \in I$. On the other hand, we have $g(0)=1$, so $g \notin (\mathfrak{m},t)$ (since this is a proper ideal of $R[t]$ with quotient isomorphic to $R \slash \mathfrak{m}$). Thus $(\mathfrak{m},t)$ is a maximal ideal of $(R[t] \slash f)_g$ and clearly $\varphi^{-1}\bigl((\mathfrak{m},t)\bigr)=\mathfrak{m}$. It follows that $\varphi$ is indeed faithfully flat, so $R \slash I^k \rightarrow R^{\prime} \slash I^k R^{\prime}$ is injective.
 
 It remains to check surjectivity. Since the ideal generated by $I$ in $R \slash I^k$ is nilpotent, a suitable form of Nakayama's Lemma implies that it suffices to check surjectivity of $R \slash I \rightarrow R^{\prime} \slash I R^{\prime}$. First note that $t \in I R^{\prime}$. Indeed, recall from Definition~\ref{dfn:standard_Nisnevich_L_homomorphism} that $f(t)=a_0+t \cdot \tilde{f}(t)$, so the invertibility of $\tilde{f}$ in $R^{\prime}$ implies that $t$ and $a_0$ are associates in $R^{\prime}$. A general element of $R^{\prime}$ can be written as $(a + t a^{\prime})(1+tb^{\prime})^{-1}$ where $a \in R$, $a^{\prime}, b^{\prime} \in R[t]$ and $(1+tb^{\prime})=g^n$ for some $n \geq 0$.
 
 Since $(1+tb^{\prime})^{-1}=1-tb^{\prime}(1+tb^{\prime})^{-1}$ in $R^{\prime}$ we find that
 \[
 \Bigl(\frac{a+ta^{\prime}}{1+tb^{\prime}} -a \Bigr)=(a+ta^{\prime})\bigl(1-tb^{\prime}(1+tb^{\prime})^{-1}\bigr) -a \in (t) \subseteq IR^{\prime}
 \]
 holds, so $R \slash I \rightarrow R^{\prime} \slash IR^{\prime}$ is indeed surjective.
 \end{proof}
 
 If $I$ is a principal ideal $I=(a)$, we will see that any $L$-homomorphism $\varphi \colon R \rightarrow R^{\prime}$ along $(a)$ is also an analytic isomorphism along $(a)$. We have thus obtained new examples of analytic patching diagrams. From the above proposition, we already know that the homomorphism $R \slash a \rightarrow R^{\prime} \slash \varphi(a)$ is an isomorphism, so to establish the claim, it only remains to show that the diagram
 \[
 \xymatrix{R \ar[d]_{\varphi} \ar[r]^{\lambda_a} & R_a  \ar[d]^{\varphi_a} \\ R^{\prime} \ar[r]_-{\lambda_{\varphi(a)}} & R^{\prime}_{\varphi(a)} }
 \]
 is a pullback, which follows from the proposition below.
 
 \begin{prop}\label{prop:etale_homomorphism_analytic_isomorphism}
 Let $\varphi \colon R \rightarrow R^{\prime}$ be an {\'e}tale homomorphism and let $a \in R$ be an element such that the induced homomorphism $R \slash a \rightarrow R^{\prime} \slash \varphi(a)$ is an isomorphism. Then for any $R$-module $M$, the diagram
 \[
 \xymatrix{ M \ar[d] \ar[r] & M_a \ar[d] \\ R^{\prime} \ten{R} M \ar[r] & (R^{\prime} \ten{R} M)_{\varphi(a)}}
 \]
 is a pullback. In particular, $\varphi$ is an analytic isomorphism along $a \in R$.
 \end{prop}
 
 \begin{proof}
 Since $\varphi$ is flat, this follows from \cite[\href{https://stacks.math.columbia.edu/tag/05ES}{Theorem 05ES}]{stacks-project}. Since $\varphi$ is {\'e}tale, it is possible to give a more elementary proof of this fact in our particular case.
 
 Since $R \rightarrow R \slash a$ factors through $\varphi \colon R \rightarrow R^{\prime}$, the ring homomorphism
 \[
 R \rightarrow R_a \times R^{\prime}
 \]
 is faithfully flat. Thus it suffices to check that the above diagram is a pullback after tensoring with $R_a$ respectively with $R^{\prime}$. This is immediate for $R_a$ since both vertical homomorphism become isomorphisms after tensoring.
 
 By the cancellation law for pullbacks, it suffices to check that the diagram
 \[
 \xymatrix{ R^{\prime} \otimes R^{\prime} \otimes M \ar[r] \ar[d]_{\mu \otimes M} \ar[r] & (R^{\prime} \otimes R^{\prime} \otimes M)_{1 \otimes a^{\prime} } \ar[d]^{(\mu \otimes M)_{1 \otimes a^{\prime}}} \\ R^{\prime} \otimes M \ar[r] & (R^{\prime} \otimes M)_{a^{\prime}} }
 \]
 is a pullback to conclude the proof (here $a^{\prime}=\varphi(a)$, $\mu$ denotes the multiplication, and all tensor products are over $R$).
 
 Since $R^{\prime} \otimes R^{\prime}$ and $R^{\prime}$ are {\'e}tale $R$-algebras, it follows that $\mu$ is {\'e}tale as well. But it is also surjective, so, up to isomorphism, it is given by the localization $(R^{\prime} \otimes R^{\prime})_e$ at some idempotent $e$ (see \cite[\href{https://stacks.math.columbia.edu/tag/00U8}{Lemma 00U8}]{stacks-project}).
 
 The assumption implies that $\mu$ induces an isomorphism modulo $1 \otimes a^{\prime}$ (since the section $R^{\prime} \otimes \varphi$ of $\mu$ has this property), so $1 \otimes a^{\prime}$ and $e$ are comaximal. Thus the above diagram is isomorphic to a Zariski patching diagram, hence a pullback.
 \end{proof}
 
 Having established the basic properties of $L$-homomorphism and standard Nisnevich morphisms, we now turn to examples. We want to show that both {\'e}tale neighbourhoods and henselizations can be written as filtered colimits of standard Nisnevich homomorphisms. We first consider the case of {\'e}tale neighbourhoods $(B, \mathfrak{n})$ of a local ring $(A, \mathfrak{m})$. Recall from \cite{LINDEL} that there exists a monic polynomial $f(t) \in A[t]$ such that $a_0=f(0) \in \mathfrak{m}$, $f^{\prime}(0) \notin \mathfrak{m} $ and $B$ is isomorphic to the quotient $(A[t]_{(\mathfrak{m},t)}) \slash f$. In the patching diagrams associated to {\'e}tale neighbourhoods, $A \rightarrow B$ is an analytic isomorphism along some element $h \in A$, specifically $h=a_0$.
 
\begin{prop}\label{prop:etale_neighbourhood_standard_Nisnevich}
 Let $\varphi \colon (A, \mathfrak{m}) \rightarrow (B,\mathfrak{n})$ be an {\'e}tale neighbourhood and let $h \in \mathfrak{m}$ be an element such that $\varphi$ is an analytic isomorphism along $h$, chosen as above. Then $B$ is a directed colimit of standard Nisnevich homomorphisms along $a_0$. In particular, all patching diagrams associated to {\'e}tale neighbourhoods are directed colimits of patching diagrams associated to standard Nisnevich homomorphisms (along $a_0$).
\end{prop}
 
\begin{proof}
 The second claim follows from the first since standard Nisnevich homomorphisms are stable under base change (see Lemma~\ref{lemma:base_change_for_L_homomorphism}).
 
To see the first claim, recall that by definition $f^{\prime}(0) \notin \mathfrak{m}$ Writing $f(t)=\sum_{i=0}^n a_i t^i$, we have $f^{\prime}(0)=a_1$, so $a_1$ is a unit in $A$. If we write $\tilde{f}$ for the unique polynomial such that $f(t)-a_0=t \cdot \tilde{f}(t)$ holds, then the constant term of $\tilde{f}$ is $a_1 \in A^{\times}$. Thus both $f^{\prime}$ and $\tilde{f}$ are units in the local ring $A[t]_{(\mathfrak{m},t)}$. It follows that $B$ can be written as the localization of
 \[
 (A[t] \slash f)_{a_1^{-2} f^{\prime} \tilde{f}}
 \]
 at the multiplicative set $A[t] \setminus (\mathfrak{m},t)$. These are precisely the polynomials whose constant term is a unit. To invert these, it suffices to invert all the polynomial $g(t)$ with $g(0)=1$. Since the homomorphisms
 \[
 A \rightarrow (A[t] \slash f)_{a_1^{-2} f^{\prime} \tilde{f} g}
 \]
 are standard Nisnevich along $a_0$ (see Definition~\ref{dfn:standard_Nisnevich_L_homomorphism}), the claim follows by taking the filtered colimit over all such $g$, ordered by divisibility.
\end{proof} 

 Next we show how the henselization of a pair $(R,I)$ can be built using ring homomorphisms which are standard Nisnevich along $I$. The argument is inspired by the construction of the henselization given in \cite{GRECO}, but it differs slightly: the diagram we construct consists of {\'e}tale homomorphisms, while the diagram in \cite{GRECO} uses morphisms which are only essentially {\'e}tale. On the other hand, our diagram is only a filtered diagram, not directed.
 
 \begin{dfn}\label{dfn:diagram_category_for_henselization}
 Let $I \subseteq R$ be an ideal. Let $\ca{H}_{(R,I)}$ denote the category whose objects are the ring homomorphisms $R \rightarrow A$ which can be written as composites
 \[
 R=R_0 \rightarrow R_1 \rightarrow \ldots \rightarrow R_k=A
 \]
 such that $R_i \rightarrow R_{i+1}$ is standard Nisnevich along $IR_i$ for all $i=0, \ldots, k-1$. The morphisms of $\ca{H}_{(R,I)}$ are the morphisms of $R$-algebras, that is, the morphisms of rings making the evident triangle commutative.
 \end{dfn}
 
 We will use the following lemma to show that $\ca{H}_{(R,I)}$ is filtered. Its proof closely follows the proof of \cite[Lemma~1.5]{GRECO}, though the statement and the assumptions are a bit different.
 
\begin{lemma}\label{lemma:weak_coequalizer_in_H}
Let $R \rightarrow R^{\prime}$ be standard Nisnevich along the ideal $I \subseteq R$ and let $\varphi, \psi \colon R^{\prime} \rightarrow A$ be two homomorphisms of $R$-algebras. Then there exists an element $a \in IA$ such that the two composites in the diagram
\[
\xymatrix{R^{\prime} \ar@<-0.5 ex>[r]_{\psi} \ar@<0.5ex>[r]^{\varphi} & A \ar[r] & A_{1+a} }
\]
are equal.
\end{lemma}

\begin{proof}
 Let $f(t)$ and $g(t)$ be as in Definition~\ref{dfn:standard_Nisnevich_L_homomorphism}, so that $R \rightarrow R^{\prime}$ is isomorphic to $R \rightarrow (R[t] \slash f)_g$. In particular, $\varphi$ and $\psi$ are uniquely determined by the choice of a root $x$ respectively $y$ of $f(t)$ in $A$. Thus $f(t)=(t-x)p(t)=(t-y)q(t)$ for some monic polynomials $p, q \in A[t]$. We claim that $q(x) \in A$ is a unit modulo $IA$. Note that the conclusion follows from this claim, for we can then find elements $z \in A$ and $a \in IA$ such that $q(x) \cdot z=1+a$, which implies that $(x-y)=0$ in the localization $A_{1+a}$ (since $(x-y)q(x)=f(x)=0$ in $A$).
 
 To see the claim, we first take the derivative of the equation $f(t)=(t-y)q(t)$ to find that $f^{\prime}(t)=(t-y)q^{\prime}(t)+q(t)$ holds. Evaluating in $x$ yields $f^{\prime}(x)=(x-y)q^{\prime}(x)+q(x)$. By assumption, $\varphi$ sends $f^{\prime}(t)$ to a unit in $A$, so $f^{\prime}(x)=\varphi(f^{\prime})$ is a unit. Similarly $\tilde{f}(x)=\varphi(\tilde{f})$ is a unit in $A$, so $x=-\tilde{f}(x)^{-1} a_0$ lies in $IA$. The same argument applied to $\psi$ shows that $y \in IA$. Taken together, this shows that $q(x)=f^{\prime}(x)+(y-x)q^{\prime}(x)$ is a unit modulo $IA$, as claimed.
\end{proof}

\begin{lemma}\label{lemma:category_H_filtered}
Let $I \subseteq R$ be an ideal. Then the category $\ca{H}_{(R,I)}$ is filtered.
\end{lemma}

\begin{proof}
Since $\id_R \in \ca{H}_{(R,I)}$, we have $\ca{H}_{(R,I)} \neq \varnothing$. Stability of standard Nisnevich homomorphisms under base change (see Lemma~\ref{lemma:base_change_for_L_homomorphism}) implies that any solid span
\[
\xymatrix{R \ar[r] \ar[d] & A^{\prime} \ar@{-->}[d] \\ A \ar@{-->}[r] & A^{\prime \prime} }
\]
can be completed to a commutative square in $\ca{H}_{(R,I)}$ indicated by the dashed arrows.

 Finally, given a pair of morphisms $\varphi, \psi \colon A \rightarrow B$ in $\ca{H}_{(R,I)}$, we can iteratively apply Lemma~\ref{lemma:weak_coequalizer_in_H} above and the fact that $B \rightarrow B_{1+b}$ is standard Nisnevich along $IB$ if $b \in IB$ (see Lemma~\ref{lemma:standard_Nisnevich_and_Jacobson_radical}) to find a morphism $\lambda \colon B \rightarrow C$ in $\ca{H}_{(R,I)}$ such that $\lambda \varphi=\lambda \psi$.
\end{proof}

\begin{thm}\label{thm:henselization}
Let $I \subseteq R$ be an ideal and let $H \in \CRing$ denote the colimit of the diagram
\[
\ca{H}_{(R,I)} \rightarrow \CRing
\]
which sends $R \rightarrow A$ to $A$. Then the henselization of $(R,I)$ is given by $(H,IH)$.
\end{thm}

\begin{proof}
 We need to show two things: that $(H,IH)$ is a henselian pair, and that for any henselian pair $(B,J)$ and any ring homomorphism $\varphi \colon R \rightarrow B$ with $\varphi(I) \subseteq J$, there exists a unique ring homomorphism $\bar{\varphi} \colon H \rightarrow B$ whose composite with the canonical morphism $R \rightarrow H$ is equal to $\varphi$ (note that the condition $\bar{\varphi}(IH) \subseteq J$ is then automatic).
 
 We first establish existence and uniqueness of $\bar{\varphi}$. The proof idea is simple: the polynomials whose roots are gradually adjoined in the construction of $H$ have unique roots in henselian pairs, yielding the desired $\bar{\varphi}$.  To see this, note that for any morphism of pairs
 \[
 \psi \colon (A,K) \rightarrow (B,J)
 \]
 and any $\lambda \colon A \rightarrow A^{\prime}$ which is standard Nisnevich along $K$, there exists a unique $\bar{\psi} \colon A^{\prime} \rightarrow B$ with $\bar{\psi} \lambda=\psi$. Uniqueness follows from Lemma~\ref{lemma:weak_coequalizer_in_H} and the observation that $\psi(K) \subseteq J$ is contained in the Jacobson radical (since $(B,J)$ is henselian), so the localization in question is an isomorphism.
 
 The existence of the desired homomorphism $\bar{\psi}$ follows from the fact that any monic polynomial $f(t)=t^n+a_{n-1} t^{n-1}+ \ldots +a_1 t+a_0$ with $a_0 \in K$ and $a_1 \in A[t]$ is sent to an $N$-polynomial in $B[t]$, so it has a root $\beta$ in $J \subseteq B$ (see Proposition~\ref{prop:characterizations_of_henselian_pairs}). If $g(t) \in A[t]$ is any polynomial with $g(0)=1$, then $\psi(g)(\beta)$ is a unit (again since $J$ is contained in the Jacobson radical), so from the definition of standard Nisnevich homomorphisms (see Definition~\ref{dfn:standard_Nisnevich_L_homomorphism}) it follows that the desired $\bar{\psi} \colon A^{\prime} \rightarrow B$ exists.
 
 Applying this construction iteratively, we find that $\varphi \colon R \rightarrow B$ uniquely extends to a cocone on the diagram $\ca{H}_{(R,I)} \rightarrow \CRing$, and such cocones are in natural bijection with homomorphisms $\bar{\varphi}\colon H \rightarrow B$ as above.
 
  It remains to show that $(H,IH)$ is a henselian pair. If $h \in IH$ and $y \in H$ are arbitrary elements, then there exists some $R \rightarrow A$ in $\ca{H}_{(R,I)}$ such that $h$ (respectively $y$) is the image of some $a \in IA$ (respectively $x \in A$) under the canonical map $A \rightarrow H$ because the diagram $\ca{H}_{(R,I)}$ is filtered (see Lemma~\ref{lemma:category_H_filtered}). Since $A \rightarrow A_{1+ax}$ is standard Nisnevich along $IA$, hence a morphism in $\ca{H}_{(R,I)}$ (see Lemma~\ref{lemma:standard_Nisnevich_and_Jacobson_radical}), it follows that $1+hy$  is a unit in $H$. Since $h \in IH$ and $y \in H$ were arbitrary, it follows that $IH$ is contained in the Jacobson radical of $H$.
  
  From Proposition~\ref{prop:characterizations_of_henselian_pairs} it follows that it only remains to show that every $N$-polynomial $f(t) \in H[t]$ has a root in $IH$. Since the diagram is filtered, we can assume that this polynomial is the image of some $N$-polynomial 
  \[
  f(t)=a_0+a_1 t+ \ldots +t^n \in A[t]
  \]
  for the pair $(A,IA)$ for a suitable $R \rightarrow A$ in $\ca{H}_{(R,I)}$ (to check the condition on $a_1$, note that $R \slash I \rightarrow A\slash IA$ and thus $R \slash I \rightarrow H \slash IH$ are isomorphisms by Proposition~\ref{prop:L_homomorphism_implies_affine_Nisnevich_square}).
  
  Since $a_1$ is a unit modulo $IA$, there exists morphism $A \rightarrow A^{\prime}$ which is standard Nisnevich along $IA$ such that the image of $a_1$ in $A^{\prime}$ is a unit (see Lemma~\ref{lemma:standard_Nisnevich_and_Jacobson_radical}). By renaming $A^{\prime}$ we can assume without loss of generality that $a_1 \in A^{\times}$ is a unit in $A$. In this case, the homomorphism
  \[
  A \rightarrow (A[t] \slash f)_{a_1^{-2} f^{\prime} \tilde{f}}
  \]
  is standard Nisnevich along $IA$ (since $a_0 \in IA$ by definition of $N$-polynomials). By construction, $t$ is a root in the target, so it only remains to check that $t$ lies in the image of $I$. This follows from the fact that $t=-\tilde{f}(t)^{-1} a_0$. Thus the original $N$-polynomial does indeed have a root in $IH$, so the pair $(H,IH)$ is henselian by Proposition~\ref{prop:characterizations_of_henselian_pairs}.
\end{proof}

\begin{cor}\label{cor:weak_excision_for_henselization_diagrams}
 Let $a \in R$ be an element and let $B$ be a commutative $R$-algebra. Let $F \colon \CRing \rightarrow \Set$ be a finitary functor. If $F$ sends patching diagrams associated to homomorphisms $A \rightarrow A^{\prime}$ which are standard Nisnevich along a principal ideal to weak pullback squares, then the diagram
 \[
 \xymatrix{F(B) \ar[d] \ar[r] & F(B_a) \ar[d] \\ F(B \ten{R} R^h_{(a)}) \ar[r] & F \bigl( (B\ten{R} R^h_{(a)} )_a\bigr)}
 \]
 is also a weak pullback square.
\end{cor}

\begin{proof}
Since $F$ preserves filtered colimits and standard Nisnevich homomorphisms are stable under base change (see Lemma~\ref{lemma:base_change_for_L_homomorphism}), the construction of the henselization given in Theorem~\ref{thm:henselization} reduces the question to checking that the outer rectangle of the diagram
\[
\xymatrix{ F(B_0) \ar[d] \ar[r] & F(B_1) \ar[d] \ar[r] & F(B_2) \ar[d] \ar[r] & \ldots \ar[r] & F(B_k)\ar[d]  \\
F\bigl((B_0)_a\bigr) \ar[r] & F\bigl((B_1)_a\bigr) \ar[r] & F\bigl((B_2)_a\bigr) \ar[r] & \ldots \ar[r] & F\bigl((B_k)_a\bigr) }
\]
 is a weak pullback diagram whenever $B_i \rightarrow B_{i+1}$ is standard Nisnevich along (the image of) $a$ for all $0 \leq i <k$. This follows from the assumption and the fact that weak pullback squares are stable under pasting.
\end{proof}

 In \S \ref{section:pseudoelementary}, we will encounter a functor from commutative rings to pointed sets for which it seems difficult to establish weak Nisnevich excision. It is possible to do this for a large class of Nisnevich squares, though, namely for the squares associated to \emph{basic} Nisnevich homomorphisms defined below.
 
 Recall that any {\'e}tale ring homomorphism $R \rightarrow S$ can be written as
 \[
 R \rightarrow R[x_1, \ldots, x_n] \slash (f_1, \ldots, f_n) \cong S
 \]
 where the $f_i$ are polynomials such that the image of $\mathrm{det}(\partial f_j \slash \partial x_i)_{i,j=1}^n$ in $S$ is invertible (see \cite[\href{https://stacks.math.columbia.edu/tag/00U9}{Lemma 00U9}]{stacks-project}).
 
\begin{dfn}
Let $I \subseteq R$ be an ideal. A ring homomorphism $R \rightarrow S$ is called \emph{basic Nisnevich along $I$} if there exists a presentation
\[
 R \rightarrow R[x_1, \ldots, x_n] \slash (f_1, \ldots, f_n) \cong S
\]
such that the image of $\mathrm{det}(\partial f_j \slash \partial x_i)_{i,j=1}^n$ is invertible in $S$ and $f_j \equiv x_j \mod I$ for all $j=1, \ldots, n$.
\end{dfn}

 From the definition it is clear that such a ring homomorphism is {\'e}tale and that $R \slash I \rightarrow S \slash IS$ is an isomorphism. In particular, for $I=(a)$ principal, we find that $R \rightarrow S$ is an analytic isomorphism along $a$ (see Lemma~\ref{prop:etale_homomorphism_analytic_isomorphism}). It is also clear that basic Nisnevich homomorphisms are stable under base change: if $R \rightarrow R^{\prime}$ is any ring homomorphism, then $R^{\prime} \rightarrow R^{\prime} \ten{R} S$ is basic Nisnevich along $IR^{\prime}$.
  
 \begin{example}\label{example:localization_basic_Nisnevich}
 Let $I \subseteq R$ be an ideal and let $a \in I$. Then the localization $R \rightarrow R_{1+a}$ is basic Nisnevich along $I$. Indeed, we can present this homomorphism as $R \rightarrow R[y] \slash \bigl( 1-y(1+a) \bigr)$ and from the change of variables $x=1-y$ we find that $R \rightarrow R_{1+a}$ is isomorphic to $R \rightarrow R[x] \slash \bigl( x+(x-1)a \bigr)$.
 \end{example}
 
 \begin{prop}\label{prop:basic_Nisnevich_refinement}
 Let $I \subseteq R$ be an ideal and let $R \rightarrow S$ be an {\'e}tale ring homomorphism such that $R \slash I \rightarrow S \slash IS$ is an isomorphism. Then there exists an element $a \in IS$ such that the composite
 \[
 R \rightarrow S \rightarrow S_{1+a}
 \]
 is basic Nisnevich along $I$.
 \end{prop}
 
 \begin{proof}
 Choose a presentation $ R \rightarrow R[x_1, \ldots, x_n] \slash (f_1, \ldots, f_n)$ of $R \rightarrow S$ such that $\mathrm{det}(\partial f_j \slash \partial x_i)$ is invertible in $S$. Since $R \slash I \rightarrow S \slash IS$ is an isomorphism, there exist polynomials $h_i \in R[x_j]$ with coefficients in $I$ and elements $r_i \in R$ such that
 \[
 x_i \equiv r_i+h_i \mod (f_1, \ldots, f_n)
 \]
 holds for all $i=1,\ldots, n$. Since we can apply the change of variables $x_i^{\prime}=x_i-r_i$, we can assume without loss of generality that $r_i=0$. Writing $g_i=x_i-h_i$, we thus get a canonical surjection
 \[
 R[x_j] \slash (g_i) \rightarrow R[x_j] \slash (f_i) \cong S
 \]
 sending $x_j$ to $x_j$.

Since $(\partial g_j \slash \partial x_i)$ is congruent to the identity modulo $I$, there exists an element $a \in IR[x_j]$ such that $\mathrm{det} (\partial g_j \slash \partial x_i)=1+a$. It follows that the composite
\[
R \rightarrow R[x_j] \slash (g_i) \rightarrow \bigl( R[x_j] \slash (g_i)\bigr)_{1+a}
\]
 is {\'e}tale and thus that the resulting surjection $\varphi \colon \bigl( R[x_j] \slash (g_i)\bigr)_{1+a} \rightarrow S_{1+a}$ is also {\'e}tale (since any morphism between {\'e}tale $R$-algebras is {\'e}tale). From \cite[\href{https://stacks.math.columbia.edu/tag/00U8}{Lemma 00U8}]{stacks-project} we conclude that $\varphi$ is isomorphic to the localization at some idempotent element $e$.
 
 Since $\varphi$ induces an isomorphism modulo $I$, it follows that $e \equiv 1 \mod I$. This shows that the composite
 \[
 R \rightarrow S \rightarrow S_{1+a}
 \]
 can up to isomorphism be written as $R \rightarrow \bigl( R[x_j] \slash (g_i)\bigr)_{1+b}$ for some $b \in IR[x_j]$. As in Example~\ref{example:localization_basic_Nisnevich}, we can rewrite this localization as
 \[
 R \rightarrow R[x_1, \ldots, x_n, x_{n+1}] \slash \bigl(g_1, \ldots, g_n, x_{n+1}+(x_{n+1}-1)b \bigr) \smash{\rlap{,}}
 \]
 which is basic Nisnevich along $I$ since $\mathrm{det} (\partial g_j \slash \partial x_i)_{i,j=1}^n=1+a$ is by construction invertible in the target.
 \end{proof}
 
 The important application of basic Nisnevich homomorphisms that we will rely on in \S \ref{section:pseudoelementary} is that basic Nisnevich homomorphisms can be used to construct the henselization of a pair.
 
 \begin{prop}\label{prop:henselization_using_basic_Nisnevich}
 The henselization $R \rightarrow R^h_I$ of $(R,I)$ is a filtered colimit of basic Nisnevich homomorphisms along $I$.
 \end{prop} 
 
 \begin{proof}
 From \cite[\href{https://stacks.math.columbia.edu/tag/0A02}{Lemma 0A02}]{stacks-project} we know that the henselization can be written as the filtered colimit of all $R$-algebras $A$ such that $R \rightarrow A$ is {\'e}tale and the induced homomorphism $R \slash I \rightarrow A \slash IA$ is an isomorphism. By Proposition~\ref{prop:basic_Nisnevich_refinement}, the basic Nisnevich homomorphisms along $I$ form a final subdiagram of this diagram, so the colimit can be computed using only basic Nisnevich homomorphisms along $I$.
 \end{proof}
 
 Recall that an affine Nisnevich square is an analytic patching diagram
 \[
 \xymatrix{ A\ar[r]^-{\lambda_f} \ar[d]_{\varphi} & A_f \ar[d]^{\varphi_f} \\ B \ar[r]_-{\lambda_{\varphi(f)}} & B_{\varphi(f)} }
 \]
 where $\varphi \colon A \rightarrow B$ is {\'e}tale (and $A \slash f \rightarrow B \slash \varphi(f)$ is an isomorphism).
 
 \begin{cor}
 To check that a functor from commutative rings sends affine Nisnevich squares to (homotopy) pullback squares, it suffices to check this for the homomorphisms $\varphi \colon A \rightarrow B$ which are basic Nisnevich along $(f) \subseteq A$.
 \end{cor}
 
 \begin{proof}
 This follows from the cancellation property of (homotopy) pullback squares and Proposition~\ref{prop:basic_Nisnevich_refinement} since the localization $B \rightarrow B_{1+a}$ is basic Nisnevich along $\bigl(\varphi(f)\bigr)$ by Example~\ref{example:localization_basic_Nisnevich}.
 \end{proof}
 
 Note that the above corollary does not apply to functors which send these squares to \emph{weak} pullback squares, since the necessary cancellation property does not hold in this case.
\section{Overrings of polynomial rings}\label{section:overrings}

 Throughout this section, we fix a commutative ring $R$. Given an $R$-algebra $A$ and an element $\alpha \in A$, we can form a new $R$-algebra
 \[
 C \defl A[x,y] \slash (xy-\alpha)
 \]
 obtained by formally ``factoring'' $\alpha$. Note that this is an overring of the polynomial ring $A[x]$ in the sense of \cite{BHATWADEKAR_ROY_OVERRINGS}: we have $A[x] \subseteq C \subseteq A[x,x^{-1}]$, in fact, the localization of the homomorphism $A[x] \rightarrow C$ yields an isomorphism $A[x,x^{-1}] \cong C_x$. Broadly speaking, we are interested in the possible values that a functor $F \colon \CAlg_R \rightarrow \Set$ can take on such algebras if $F$ satisfies weak analytic excision. More precisely, we want to relate these values to the corresponding values on the localization $A_{\alpha}$ and the quotient $A \slash \alpha$.
 
 The basic idea is to consider an analytic patching diagram such as
 \[
 \xymatrix{C \ar[r] \ar[d] & C_{\alpha} \ar[d] \\ C^{\wedge}_{(\alpha)} \ar[r] & C^{\wedge}_{(\alpha)}[\frac{1}{\alpha}] }
 \]
 obtained by completion at the ideal $(\alpha)$. This is only useful if we can understand:
 \begin{itemize}
 \item[-] The value of $F$ on $C_{\alpha}$ and $C^{\wedge}_{(\alpha)}$, that is, the ``pieces'' away from $\alpha$ and the pieces ``near'' $\alpha$ respectively;
 
\item[-] The ``ways of gluing pieces'' on the ``intersection'' $C^{\wedge}_{(\alpha)}[\frac{1}{\alpha}]$.
 \end{itemize}
 
 The latter turns out to be difficult in general. In this section, we will see that the slight modification of $C^{\wedge}_{(\alpha)}$ given by the following lemma is often more useful.
 
 \begin{lemma}\label{lemma:henselian_analytic_iso_for_overrings}
 Let $A$ be a commutative $R$-algebra and let $\alpha \in A$ be an arbitrary element. Let $B \defl A^{h}_{(\alpha)}[x,y] \slash (xy-\alpha)$. Then the homomorphism
 \[
 A[x,y] \slash (xy-\alpha) \rightarrow B
 \]
 obtained from the canonical homomorphism $A \rightarrow A^{h}_{(\alpha)}$ is an analytic isomorphism along $\alpha$
 \end{lemma}
 
 \begin{proof}
 The homomorphism in question is obtained by taking the base change of $A \rightarrow A^{h}_{(\alpha)}$ along the inclusion $A \rightarrow A[x,y]\slash (xy-\alpha)$. Since $A \rightarrow A^{h}_{(\alpha)}$ is a filtered colimit of standard Nisnevich homomorphisms along $(\alpha)$ (see Theorem~\ref{thm:henselization}), we see that the base change can also be written as a filtered colimit of such homomorphism. Since each of these is an analytic isomorphism along $\alpha$ (see Propositions~\ref{prop:L_homomorphism_implies_affine_Nisnevich_square} and \ref{prop:etale_homomorphism_analytic_isomorphism}), the claim follows.
 \end{proof}
 
 Note that the localization $B_{\alpha}$ in the above lemma is isomorphic to the Laurent polynomial ring $A^{h}_{(\alpha)}[x,x^{-1}]$. We will see some examples where the ``ways of gluing'' on this algebra can be understood. Since the localization of the domain at $\alpha$ is isomorphic to the Laurent polynomial ring $A_{\alpha} [x,x^{-1}]$, we can hope to understand the ``pieces'' on this ring if we understand the ground algebra $A$ well enough. The first half of this section provides some methods that can be used to understand the ``pieces'' on the ring $B$ of the above lemma. Very roughly, the idea is to use some auxiliary patching diagrams to relate the ``pieces'' on $B$ to the ``pieces'' on the localization $B_{1+\alpha B}$. If the functor $F$ we care about is $j$-injective, it is thus possible to reduce question on $B$ to questions on $B \slash \alpha$ (or even $A \slash \alpha$ in particularly good cases). To carry this out, we need a few additional examples of patching diagrams.
 
\begin{lemma}\label{lemma:zariski_analytic_patching_diagram}
Let $A$ be a commutative $R$-algebra and let $S, T \subseteq A$ be two multiplicatively closed sets. Then the localization $\lambda_T \colon A \rightarrow A_T$ is an analytic isomorphism along $S$ if and only if for all $s \in S$ and all $t \in T$, the ideal $(s,t)$ is equal to $A$.
\end{lemma}

\begin{proof}
 If $A \slash s \rightarrow A_T \slash \lambda_T(s)$ is an isomorphism, then each $t \in T$ is a unit in the quotient $A \slash s$, hence $(s,t)=A$. The converse implication follows since base change along $A \rightarrow A \slash s$ preserves localizations. It only remains to check that
 \[
 \xymatrix{ A \ar[r] \ar[d] & A_S \ar[d] \\ A_T \ar[r] & A_{S,T} }
 \]
 is a pullback diagram in this case. Since $A \rightarrow A_S \times A_T$ is faithfully flat, it suffices to check this after base change to $A_S$ or $A_T$. But then either both vertical or both horizontal morphisms are isomorphisms, so the diagram is indeed a pullback.
\end{proof}

 To get examples of such pairs $S$ and $T$, we will use the following form of the Top-Bottom Lemma (see \cite[Lemma~IV.5.3]{LAM}).
 
\begin{lemma}\label{lemma:top_bottom_lemma}
 Let $A$ be a commutative $R$-algebra with Jacobson radical $J$ and let $f(x)=x^n+a_{n-1} x^{n-1}+ \ldots + a_0 \in A[x]$ be a monic polynomial such that all coefficients $a_i$ lie in $J$. Then for all $g \in 1+xA[x]$ we have $(f,g)=A[x]$.
\end{lemma}

\begin{proof}
 The proof in \cite[Lemma~IV.5.3]{LAM} is stated for local rings $A$ but works verbatim in our case. Indeed, since $A[x] \slash f$ is integral over $A$, any maximal ideal $M \subseteq A[x] \slash f$ intersects $A$ in a maximal ideal $\mathfrak{m}$. Since $J \subseteq \mathfrak{m}$, it follows from $f \in M$ that $M=(\mathfrak{m},x)$, hence there is no maximal ideal $M$ containing both $f$ and $g$.
\end{proof}

 Recall from \cite[Lemma~4.1]{BHATWADEKAR_ROY_OVERRINGS} that each overring of a polynomial ring gives rise to an analytic isomorphism, which yields the following in our case.
 
 \begin{lemma}\label{lemma:analytic_iso_arising_from_overring}
 Let $A$ be a commutative $R$-algebra and let $\alpha \in A$ be a nonzerodivisor. If $B=A[x,y] \slash (xy-\alpha)$, then $A[x] \rightarrow B$ is an analytic isomorphism along $S \defl 1+xA[x] \subseteq A[x]$. Moreover, both $x$ and $y$ are nonzerodivisors in $B$ and the set $S$ and its image in $B$ consist of nonzerodivisors.
 \end{lemma}
 
 \begin{proof}
 The elements of $S$ are clearly nonzerodivisors in $A[x]$ and we first show that the same is true for their images in $B$. To see this, it suffices to check that the localization $B \rightarrow B_x \cong A[x,x^{-1}]$ is injective, hence that the element $x \in B$ is a nonzerodivisor. Suppose that $x \cdot p(x,y)=(xy-\alpha) q(x,y)$ inn $A[x,y]$ and write $q=q_0(y)+q_1(y)x+ \ldots$ as a polynomial in $x$ with coefficients in $A[y]$. It follows that $\alpha q_0(y)=0$, so all coefficients of $q_0$ must be zero since $\alpha$ is a nonzerodivisor in $A$. Thus $x$ (and by symmetry $y$) is a nonzerodivisor and the final claim of the lemma is established.
 
 It only remains to check that $A[x] \slash \bigl( 1+xf(x) \bigr) \rightarrow B \slash \bigl( 1+xf(x) \bigr)$ is an isomorphism for all $f(x) \in A[x]$. Since $x$ is a unit in these quotients, the inverse is obtained by sending $y$ to $\alpha x^{-1}=-\alpha f(x)$.
 \end{proof}
 
 Combining all this, we get the following patching diagram.
 
 \begin{prop}\label{prop:composite_patching_for_overrings}
  Let $A$ be a commutative $R$-algebra, let $\alpha \in A$ be a nonzerodivisor, and let $B=A[x,y]\slash (xy-\alpha)$. If $\alpha$ lies in the Jacobson radical of $A$, then the composite
  \[
  A[x] \rightarrow B \rightarrow B_{1+y A[y]}
  \]
  is an analytic isomorphism along $S \defl 1+xA[x] \subseteq A[x]$. In particular, the diagram
  \[
  \xymatrix{ A[x] \ar[d] \ar[r] & A[x]_{1+xA[x]} \ar[dd] \\ B \ar[d]_{\lambda_{1+yA[y]}}  \\ B_{1+yA[y]} \ar[r] & B_{1+yA[y],1+xA[x]} }
  \]
  is an analytic patching diagram in this case.
 \end{prop}
 
 \begin{proof}
  Since pullback diagrams can be pasted together and isomorphisms can be composed, it follows that analytic isomorphisms can be composed as well. From Lemma~\ref{lemma:analytic_iso_arising_from_overring} we know that $A[x] \rightarrow B$ is an analytic isomorphism along $1+xA[x]$, so it only remains to show that $\lambda_{1+yA[y]} \colon B \rightarrow B_{1+yA[y]}$ is an analytic isomorphism along $1+xA[x]$. Since this is a Zariski patching diagram, this amounts to showing that each element $1+yg(y) \in 1+yA[y]$ is sent to a unit in $B \slash \bigl(1+xf(x)\bigr)$ for all $f(x) \in A[x]$ (see Lemma~\ref{lemma:zariski_analytic_patching_diagram}).
  
  Since $x$ is a unit in the quotient $B \slash \bigl(1+xf(x)\bigr)$, the right vertical morphism in the diagram
  \[
  \xymatrix{B \ar[r]^-{\pi} \ar[d]_{\lambda_x} & B \slash \bigl(1+xf(x)\bigr) \ar[d]_{\lambda_x}^{\cong} \\ 
  B_x \ar[d]_{\varphi}^{\cong} \ar[r]^-{\pi} & B_x \slash \bigl(1+xf(x)\bigr) \ar[d]_{\bar{\varphi}}^{\cong}  \\ A[x,x^{-1}] \ar[r]^-{\pi} & A[x,x^{-1}] \slash \bigl( 1+xf(x) \bigr)  }
  \]
 is an isomorphism, where $\varphi$ is given by $\varphi(x)=x$ and $\varphi(y)=\alpha x^{-1}$. It thus suffices to check that the image of $\varphi(1+yg(y))=1+\alpha x^{-1} \cdot g(\alpha x^{-1})$ in the quotient $A[x,x^{-1}] \slash \bigl( 1+xf(x) \bigr)$ is a unit. Equivalently, we want to show that
 \[
 1+\alpha x^{-1} \cdot g(\alpha x^{-1})
 \quad \text{and} \quad
1+xf(x)
 \]
 are comaximal in $A[x,x^{-1}]$.
 
 For sufficiently large $n$, we have $h(x) \defl x^n \cdot \bigl( 1+\alpha x^{-1} \cdot g(\alpha x^{-1}) \bigr) \in A[x]$. Since $h(x)$ is monic and $\alpha$ lies in the Jacobson radical, the Top-Bottom lemma implies that $h(x)$ and $1+xf(x)$ are comaximal in $A[x]$ (see Lemma~\ref{lemma:top_bottom_lemma}).
 \end{proof}
 
\begin{cor}\label{cor:trivial_kernel_from_jacobson_patching}
Let $A$ be a commutative $R$-algebra, let $\alpha \in A$ be a nonzerodivisor which is contained in the Jacobson radical of $A$, and let $B \defl A[x,y] \slash (xy-\alpha)$. Let $F \colon \CAlg_R \rightarrow \Set_{\ast}$ be a functor and suppose that the following conditions hold:
\begin{enumerate}
\item[(i)] The functor $F$ satisfies weak analytic excision;
\item[(ii)] The algebra $A[x]$ is $F$-contractible.
\end{enumerate}
 Then the function
 \[
 F(B) \rightarrow F(B_{1+xA[x],1+yA[y]})
 \]
 of pointed sets induced by the localization homomorphism has trivial kernel.
\end{cor} 
 
\begin{proof}
 The composite of two functions with trivial kernel also has trivial kernel, so we can treat the two localization morphisms
 \[
 \lambda_1 \colon B \rightarrow B_{1+yA[y]} \quad \text{and} \quad \lambda_2 \colon B_{1+yA[y]} \rightarrow B_{1+yA[y],1+xA[x]}
 \]
 separately.
 
 In both cases, we can find an analytic patching diagram of the form
 \[
 \xymatrix{A[x] \ar[r] \ar[d] & \cdot \ar[d] \\ \cdot \ar[r]_{\lambda_i} & \cdot }
 \]
 (by Lemma~\ref{lemma:analytic_iso_arising_from_overring} for $\lambda_1$ and by Proposition~\ref{prop:composite_patching_for_overrings} for $\lambda_2$). This implies that any object in the kernel of $F(\lambda_i)$ extends to an object of $F(A[x])$, which are all trivial by assumption.
\end{proof}

 To use this corollary, we need to make some additional assumptions. It is here that we apply the results of \S \ref{section:henselian_pairs} about henselian pairs. Our main applications all rely on the following proposition.
 
  \begin{prop}\label{prop:overrings_in_henselian_case}
 Let $A$ be a commutative $R$-algebra and let $\alpha \in A$. Let $B=A[x,y] \slash (xy-\alpha)$ and consider an element $1+\alpha g(x,y) \in 1+\alpha B$. If the pair $\bigr(A,(\alpha)\bigl)$ is henselian, then there exist elements $p(x) \in A[x]$ and $q(y) \in A[y]$ such that the ideal generated by $1+\alpha g(x,y)$ in the localization $B_{1+xp(x)}$ contains the element $1+yq(y)$ (we can even find such a $q(y)$ which is furthermore divisible by $\alpha$). In particular, the localization
 \[
 B \rightarrow B_{1+xA[x],1+yA[y]}
 \]
 factors through the localization $B \rightarrow B_{1+\alpha B}$.
\end{prop}
 
 \begin{proof}
 Let $I$ be the principal ideal of $B$ generated by $1+\alpha g(x,y)$ and let $N \in \mathbb{N}$ be the highest power of $y$ that occurs in a monomial of $g(x,y) \in A[x,y]$ (here we choose a preimage of our element in $1+\alpha B$ which lies in $A[x,y]$). Since $xy=\alpha$ in $B$, there exists some $h(x) \in A[x]$ such that
 \[
 x^N \bigr(1+\alpha g(x,y) \bigl)=x^N + \alpha h(x)
 \]
 holds in $B$.
 
 Since $x^N+\alpha h(x) \equiv x^N \cdot 1 \mod \alpha$ and $\bigr(A,(\alpha) \bigl)$ is a henselian pair, Corollary~\ref{cor:splitting_into_monic_times_1+I_polynomial} is applicable. This corollary shows that there exist polynomials $f_0(x), f_1(x) \in A[x]$ and a unit $u \in A$ such that $f_0$ is monic of degree $N$, $f_1 \in 1 + x(\alpha)[x]$, and $u$ is congruent to $1$ modulo $(\alpha)$, such that $x^N+\alpha h(x)=f_0(x) \cdot u \cdot f_1(x)$ holds. 
 
 It follows that we can write $f_1(x)=1+x p(x)$ for a suitable $p(x) \in A[x]$. Since $u f_1(x)$ is a unit in the localization $B_{1+xp(x)}$, we can conclude that the ideal in $B_{1+xp(x)}$ generated by $1+\alpha g(x,y)$ contains the (image of the) monic polynomial $f_0(x)$. To show the first claim, it therefore suffices to show that the ideal $J \subseteq B$ generated by $1+\alpha g(x,y)$ and $f_0(x)$ already contains the desired element $1+yq(y)$ for some $q(y) \in A[y]$.
 
 First note that $J$ contains elements $1+\alpha^k g^{\prime}(x,y)$ for arbitrarily large $k$ (for example by iteratively multiplying with $1\pm \alpha^k g^{\prime} (x,y)$). Since $\alpha^k=x^k y^k$ in $B$, we can in particular find a $g^{\prime}(x,y) \in A[x,y]$ such that $1+\alpha y^N g^{\prime}(x,y)$ lies in $J$.
 
 Writing $g^{\prime}(x,y)=\sum_{i=0}^m g_i(y) x^i$ with $g_i(y) \in A[y]$, we can gradually subtract multiples of the monic polynomial $f_0(x) \in J$ from $g^{\prime}(x,y)$ to obtain an element
 \[
 1+ \alpha y^N \bigl(\textstyle\sum_{i=0}^m g_i(y)x^i \bigr) \in J
 \]
 such that $m <N$. Since $y^N x^i=y^{N-i} \alpha^i$ in $B$, this yields the desired element in $J$ of the form $1+\alpha y q(y)$ for some $q \in A[y]$.
 
 This concludes the proof of the first claim. To see the second claim, note that the ideal generated in $B_{1+xA[x],1+yA[y]}$ by any element $1+\alpha b$, $b \in B$, contains a unit by the first claim, so we do indeed get the desired factorization through $B_{1+\alpha B}$.
 \end{proof}
 
 By construction, $\alpha$ lies in the Jacobson radical of $B_{1+\alpha B}$. We will need some compatibility of $F$ with ideals contained in the Jacobson radical in order to utilize this fact. It turns out that something slightly weaker than $j$-injectivity suffices.

 \begin{dfn}\label{dfn:partially_j_injective}
 We say that a functor $F \colon \CAlg_R \rightarrow \Set_{\ast}$ is \emph{partially $j$-injective} if $F(A) \rightarrow F(A \slash I)$ has trivial kernel whenever $I \subseteq A$ is contained in the Jacobson radical of $A \in \CAlg_R$.
 \end{dfn}
 
 \begin{example}\label{example:partially_j_injective}
 Clearly any $j$-injective functor such as $P_r$, $SK_{1,r}$, $K_0$, and $SK_1$ is also partially $j$-injective. The reason for introducing this terminology is that the functors
\[
W_r \colon \CAlg_R \rightarrow \Set_{\ast}
\]
are partially $j$-injective for all $r \geq 3$. Indeed, if a unimodular row is elementary completable modulo the Jacobson radical, then its elementary orbit contains a row with one element equal to a unit. The claim follows since such rows are elementary completable.
 \end{example}
 
  The following theorem summarizes the main results.
  
  \begin{thm}\label{thm:henselian_implies_trivial_kernel_for_overring}
  Let $A$ be a commutative $R$-algebra and let $\alpha \in A$ be a nonzerodivisor. Let $B=A[x,y]\slash (xy-\alpha)$ and assume that $\bigl(A,(\alpha)\bigr)$ is a henselian pair. Let $F \colon \CAlg_R \rightarrow \Set_{\ast}$ be a functor and suppose that the following conditions hold:
 \begin{enumerate}
 \item[(i)] The functor $F$ satisfies weak analytic excision;
 \item[(ii)] The algebra $A[x]$ is $F$-contractible;
 \item[(iii)] The functor $F$ is partially $j$-injective.
 \end{enumerate}
 Then the function
 \[
 F(B) \rightarrow F(B \slash \alpha) \cong F\bigl(A \slash \alpha [x,y]\slash (xy)\bigr)
 \]
 of pointed sets induced by the projection $B \rightarrow B \slash \alpha$ has trivial kernel.
  \end{thm}
 
 \begin{proof}
  Note that if a composite $gf$ of maps of pointed sets has trivial kernel, so does $f$. The factorization of Proposition~\ref{prop:overrings_in_henselian_case}, combined with Corollary~\ref{cor:trivial_kernel_from_jacobson_patching}, thus shows that $F(B) \rightarrow F(B_{1+\alpha B})$ has trivial kernel. By assumption, $F(B_{1+\alpha B}) \rightarrow F(B \slash \alpha)$ also has trivial kernel, so the claim follows from the fact that maps with trivial kernel are stable under composition.
 \end{proof}
 
 In order to understand algebras of the form $C[x,y] \slash (xy)$, it is convenient to have additional weak excision properties.
 
 \begin{dfn}\label{dfn:weak_milnor_excision}
 A pullback square
 \[
 \xymatrix{A \ar[r]^-{\varphi^{\prime}} \ar[d]_{\psi^{\prime}} & C \ar[d]^{\psi} \\ B \ar[r]_-{\varphi}  & D }
 \]
 of commutative $R$-algebras is called a \emph{Milnor square} if one of $\varphi$ or $\psi$ is surjective (in which case $\varphi^{\prime}$ respectively $\psi^{\prime}$ is surjective, too). A functor $F \colon \CAlg_R \rightarrow \Set$ satisfies \emph{weak Milnor excision} if $F$ sends Milnor squares to weak pullback diagrams.
 \end{dfn}
 
 \begin{example}\label{example:weak_milnor_squares}
 For any commutative $R$-algebra $C$, the diagrams
 \[
 \vcenter{\xymatrix{
C[x] \ar[r] \ar[d] & C \ar[d] \\ C[x,y] \slash (xy) \ar[r] & C[y] 
 }}
 \quad \text{and} \quad
 \vcenter{\xymatrix{
C[x,y] \slash (xy) \ar[r] \ar[d] & C[y] \ar[d] \\ C[x] \ar[r] & C \\ 
 }}
 \]
 are Milnor squares, where all the unlabelled morphisms send the variable $x$ to $x$ (respectively $y$ to $y$) if these variables occur in the target, and to $0$ otherwise.
 
 More generally, given a commutative $R$-algebra $A$ with ideals $I$, $J \subseteq A$, the diagram
 \[
 \xymatrix{ A \slash I \cap J \ar[r] \ar[d] & A \slash I \ar[d] \\ A \slash J \ar[r] & A \slash I+J}
 \]
 is a Milnor square. This follows from the snake lemma, applied to the diagram
 \[
  \xymatrix@C=15pt{0 \ar[r] & I \cap J \ar[r] \ar[d] & I \oplus J \ar[r] \ar[d] & I+J \ar[d] \ar[r] & 0\\
 0 \ar[r] & A \ar[r]^-{\left(\begin{smallmatrix} 1\\ 1 \end{smallmatrix} \right)} \ar[d] & A \oplus A \ar[r]^-{\left(\begin{smallmatrix} 1 &  -1 \end{smallmatrix} \right)} \ar[d] & A \ar[r] \ar[d] & 0 \\ & A \slash I \cap J \ar[r] & A \slash I \oplus A \slash J \ar[r] & A \slash I+J  }
 \]
 with the top two rows exact. Specializing this to $A=C[x,y]$, $I=(x)$, and $J=(y)$, we get the right Milnor square above. From the pullback cancellation law it follows that the left square is a pullback as well (by stacking the left square on top of the right square).
 \end{example}
 
 Many of the functors we have considered in \S \ref{section:analytic} satisfy weak Milnor excision. Several examples follow from the observation that Vorst's Lemma (and its generalizations to other reductive group schemes over $R$) hold for Milnor squares.
 
 \begin{prop}\label{prop:weak_milnor_excision_group_case}
 For each $r \geq 3$, the functors $W_r$, $K_{1,r}$, and $SK_{1,r}$ satisfy weak Milnor excision. More generally, if $G$ is a reductive group scheme over $R$ with parabolic subgroup $P$, then $K_1^{G,P} \colon \CAlg_R \rightarrow \Set_{\ast}$ satisfies weak Milnor excision.
 \end{prop}
 
 \begin{proof}
 The proof proceeds as in the analytic case (see Lemma~\ref{lemma:Vorsts_lemma}) once we know that $E_r(A) \rightarrow E_r(A \slash I)$ respectively $E_P(A) \rightarrow E_P(A \slash I)$ is surjective for any ideal $I \subseteq A$. This is well-known for $G=\mathrm{GL}_r$ or $\mathrm{SL}_r$, and for general $G$ it is a consequence of \cite[Expos{\'e}~XXVI, Corollaire~2.5]{SGA3_NEW},  which tells us that the unipotent radicals $U_{P}=\mathrm{rad}^u(P)$ and $U_{P^{-}}=\mathrm{rad}^u(P^{-})$ are isomorphic to $\Spec\bigl(\Sym(V_i)\bigr)$ for certain projective modules $V_i$, $i=1,2$. Since $\CAlg_R\bigl(\Sym(V_i),B\bigr) \cong \mathrm{Hom}_R(V_i,B)$, the projectivity of the $V_i$ implies that the homomorphisms $U_{P^{\pm}}(A) \rightarrow U_{P^{\pm}}(A \slash I)$ are indeed surjective.
 
 We can use this to turn ``weak excision problems'' for $K_{1,r}$, $SK_{1,r}$, and $K_1^{G,P}$ into corresponding problems for $\mathrm{GL}_r$, $\mathrm{SL}_r$, and $G$. Since Milnor squares are cartesian, the latter problems have in fact unique solutions, so the proof proceeds exactly as in the analytic case (see the second part of Lemma~\ref{lemma:Vorsts_lemma} for details).
 
 It remains to check weak excision for $W_r \colon \CAlg_r \rightarrow \Set_{\ast}$. A special case of this was first observed by Gubeladze, see \cite[Proposition~9.1]{GUBELADZE_ELEMENTARY}. Given a Milnor patching diagram
 \[
 \xymatrix{A \ar[r]^-{\varphi^{\prime}} \ar[d]_{\psi^{\prime}} & B \ar[d]^{\psi} \\ A \slash I \ar[r]_-{\varphi}  & B \slash J}
 \]
 and $u \in \mathrm{Um}_r(A \slash I)$, $v \in \mathrm{Um}_r(B)$, and $\bar{\varepsilon} \in E_r(B \slash J)$ such that $\varphi(u)=\psi(v) \bar{\varepsilon}$ holds, we can pick $\varepsilon \in E_r(B)$ such that $\psi(\varepsilon)=\bar{\varepsilon}$ to reduce to the case $\varphi(u)=\psi(v)$. Since the diagram is cartesian, there exists a row $w=(w_1, \ldots, w_r) \in A^{r}$ with $\psi^{\prime}(w)=u$ and $\varphi^{\prime}(w)=v$. We claim that $w$ is unimodular, which implies the desired weak excision property of $W_r$.
 
 Thus let $\mathfrak{m} \subseteq A$ be a maximal ideal. If the ideal $I$ is not contained in $\mathfrak{m}$, then we have $A \slash I \ten{A} A_{\mathfrak{m}}=0$ and $B \slash J \ten{A} A_{\mathfrak{m}}=0$ since a unit is sent to $0$ in both these rings. By applying the functor $-\ten{A} A_{\mathfrak{m}}$ to the above Milnor square we thus find that $\varphi^{\prime} \ten{A} A_{\mathfrak{m}}$ is an isomorphism. From the fact that $\varphi^{\prime}(w)$ is unimodular it follows that $(w_1, \ldots, w_r)$ is not contained in $\mathfrak{m}$.
 
 It remains to check the same for maximal ideals $\mathfrak{m}$ with $I \subseteq \mathfrak{m}$. Since $\psi^{\prime}(w)$ is unimodular, there exists an element in $1+I$ contained in the ideal $(w_1, \ldots, w_r) \subseteq A$, so no maximal ideal which contains $I$ can contain all the $w_i$, as claimed.
 \end{proof}
 
 A second class of examples is given by the following result.
 
 \begin{prop}\label{prop:weak_milnor_excision_adams_case}
 The functors $P$ and $P_r$, $r \geq 1$ satisfy weak Milnor excision. More generally, if $X$ is an Adams stack over $R$ and $F \defl \pi_0 X$ is the functor which sends $A$ to the set of isomorphism classes of the groupoid $X(A)$, then $F$ satisfies weak Milnor excision.
 \end{prop}
 
 \begin{proof}
 For $P$ and $P_r$, $r \geq 1$, this follows from Milnor's result \cite[\S 2]{MILNOR} that $\Proj(-)$ sends Milnor squares to cartesian squares of groupoids (note that this can be proved exactly as in the case of analytic patching diagrams, see the proof of Proposition~\ref{prop:patching_projective_modules}). The case of $P_r$ requires one additional observation: given a Milnor square
 \[
 \xymatrix{A \ar[r]^-{\varphi^{\prime}} \ar[d]_{\psi^{\prime}} & B \ar[d]^{\psi} \\ A \slash I \ar[r]_-{\varphi}  & B \slash J}
 \]
 and a finitely generated projective $A$-module $V$, if both $A \slash I \ten{A} V$ and $B \ten{A} V$ have constant rank $r$, then so does $V$. This follows as in the above proof from the fact that $\varphi^{\prime} \ten{A} A_{\mathfrak{m}}$ is an isomorphism for any maximal ideal not containing all of $I$.
 
 Finally, for the case $F=\pi_0 X$, $X$ an Adams stack over $R$, we note that $X$ sends the above diagram to a cartesian diagram of groupoids in the special case where $B \slash J=0$ since (Adams) stacks preserve finite products of $R$-algebras. We know from generalized Tannaka duality that
 \[
 X(A) \rightarrow X(A \slash I) \pb{X(B \slash J)} X(B)
 \]
 is also an equivalence if $B \slash J \neq 0$ (see \cite[Theorem~4.2.5]{SCHAEPPI_DESCENT}). The weak excision property for $F$ follows directly from the description of cartesian squares of groupoids.
 \end{proof}
 
 These results imply that the first two $K$-groups $K_0$ and $K_1$ satisfy weak Milnor excision. In contrast, the higher $K$-groups $K_i$ for $i \geq 2$ in general do \emph{not} satisfy weak Milnor excision (see for example the introduction of \cite{LAND_TAMME} for a discussion).
 
 Applying this to Example~\ref{example:weak_milnor_squares}, we obtain the following consequence.
 
 \begin{lemma}\label{lemma:excision_for_coordinate_cross}
 Let $F \colon \CAlg_R \rightarrow \Set_{\ast}$ be a functor which satisfies weak Milnor excision and let $C$ be an $R$-algebra such that $C[x]$ is $F$-contractible. Then the commutative algebra $C[x,y] \slash (x,y)$ is $F$-contractible as well.
 \end{lemma}
 
\begin{proof}
 Since
 \[
 \xymatrix{F(C[x]) \ar[r] \ar[d] & F(C) \ar[d] \\ F\bigl( C[x,y]\slash (xy)\bigr) \ar[r] & F(C[y])}
 \]
 is a weak pullback diagram by Example~\ref{example:weak_milnor_squares} and $F(C[y])=\ast$, it follows that
 \[
 \ast=F(C[x]) \rightarrow F\bigl(C[x,y] \slash (xy) \bigr)
 \]
 is surjective.
\end{proof} 

This yields the following corollary of Theorem~\ref{thm:henselian_implies_trivial_kernel_for_overring}. 
 
\begin{cor}\label{cor:henselian_implies_trivial_kernel_for_overring}
 Let $F \colon \CAlg_R \rightarrow \Set_{\ast}$ be a functor, let $A$ be a commutative $R$-algebra, and let $\alpha \in A$ be a nonzerodivisor such that $\bigl(A,(\alpha) \bigr)$ is a henselian pair. If, in addition to the conditions of Theorem~\ref{thm:henselian_implies_trivial_kernel_for_overring}, the conditions
\begin{enumerate}
\item[(iv)] The functor $F$ satisfies weak Milnor excision;
\item[(v)] The $R$-algebra $A\slash \alpha [x]$ is $F$-contractible
\end{enumerate}
 hold, then $F\bigl(A[x,y] \slash (xy-\alpha)\bigr)=\ast$.
\end{cor}

\begin{proof}
 From Theorem~\ref{thm:henselian_implies_trivial_kernel_for_overring} we know that
 \[
 F\bigl(A[x,y] \slash (xy-\alpha) \bigr) \rightarrow F\bigl(A \slash \alpha[x,y]\slash (xy)\bigr)
 \]
 has trivial kernel and from Lemma~\ref{lemma:excision_for_coordinate_cross} it follows that the target is trivial.
\end{proof}

Returning to the patching diagram
\[
\xymatrix{A[x,y] \slash (xy-\alpha) \ar[r] \ar[d] & A_{\alpha}[x,x^{-1}] \ar[d] \\ A^h_{(\alpha)}[x,y] \slash(xy-\alpha) \ar[r] & A^h_{(\alpha)}[\frac{1}{\alpha}][x,x^{-1}] }
\]
 discussed at the beginning of this section, we see that the above corollary can often be used very effectively to show that there are no non-trivial ``pieces near $\alpha$'' (that is, the value of $F$ at the bottom left corner is trivial).
 
 To understand the pieces ``away from $\alpha$,'' that is, the value of $F\bigl(A_{\alpha}[x,x^{-1}]\bigr)$, we can sometimes use the following geometric argument. We say that an $R$-algebra $A$ with structural morphism $\eta \colon R \rightarrow A$ is \emph{naively $\mathbb{A}^{1}$-contractible over $R$} if there exists an augmentation $\varepsilon \colon A \rightarrow R$ and an $R$-homomorphism $A \rightarrow A[t]$ such that $\mathrm{ev}_1 \circ h=\id_A$ and $\mathrm{ev}_0 \circ h=\eta \circ \varepsilon$. Note that $h$ gives a ``naive $\mathbb{A}^{1}$-homotopy'' $\eta \varepsilon \sim \id $, so $\eta$ and $\varepsilon$ are indeed mutually inverse $\mathbb{A}^{1}$-equivalences (since $\varepsilon \eta=\id$). This is of limited use as long as we do not know if $A$ is $F$-regular. Nevertheless, we get the following consequence.
 
 \begin{prop}\label{prop:F_and_naive_A1_contractible_A}
 Let $F \colon \CAlg_R \rightarrow \Set_{\ast}$ be a functor such that the ground ring $R$ is $F$-contractible. If $A$ is naively $\mathbb{A}^{1}$-contractible over $R$, then the composite
 \[
 \xymatrix{ NF(A) \ar[r]^-{\mathrm{incl}} & F(A[t]) \ar[r]^-{\mathrm{ev}_1} & F(A) }
 \]
 is surjective. If $B$ is a $F$-regular $R$-algebra and $\varphi \colon A \rightarrow B$ is any $R$-homomorphism, then $F\varphi \colon FA \rightarrow FB$ factors through the trivial pointed set, that is, $F\varphi(\sigma)=\ast$ for all $\sigma \in FA$.
 \end{prop}
 
 \begin{proof}
 Choose $\varepsilon$ and $h$ as above. Since $NF(A)=\ker\bigl(F(\mathrm{ev}_0)\bigr)$ and
 \[
 F(\mathrm{ev}_0) \circ F(h)=F(\mathrm{ev}_0 h)=F(\eta \varepsilon)
 \]
 factors through $F(R)=\ast$, we find that $Fh(\sigma) \in NF(A)$ for all $\sigma \in FA$. But $F(\mathrm{ev}_1) F(h)=\id$, so the map in question is indeed surjective.
 
 Because of this surjectivity, it suffices to show that the diagonal composite in the diagram
 \[
 \xymatrix{NF(A) \ar[d]_{NF (\varphi)} \ar[r]^-{\mathrm{incl}} & F(A[t]) \ar[d]^{\varphi[t]} \ar[r]^{\mathrm{ev_1}}  & FA \ar[d]^{\varphi} \\ NF(B) \ar[r]_-{\mathrm{incl}} & F(B[t]) \ar[r]_-{\mathrm{ev}_1} & FB}
 \]
 factors through $\ast$ in order to establish the second claim. Since the $F$-regularity of $B$ implies that $NF(B)=\ast$, the claim follows from the commutativity of the diagram.
 \end{proof}
 
 In order to apply these results, we need to work with particular kinds of functors $F$ for which it is also possible to understand the ``ways of glueing'' on the $R$-algebra $B_{\alpha}=A^h_{(\alpha)}[\frac{1}{\alpha}][x,x^{-1}]$. In the remainder of this section, we will see that this is for example possible for the functor $SK_1 \colon \CAlg_R \rightarrow \Ab$.
 
 Since the higher $K$-groups are homotopy groups and the underlying functor to spectra sends analytic patching diagrams to cartesian diagrams (see Proposition~\ref{prop:K_theory_weak_analytic_excision}), the ``ways of glueing'' are parametrized by the next higher homotopy group. In other words, for $B \defl A^h_{(\alpha)}[x,y] \slash (xy-\alpha)$ and $C \defl A_{\alpha}[x,x^{-1}]$, there exists an exact sequence
 \[
 \xymatrix@C=15pt@R=15pt{K_2(B) \oplus K_2(C) \ar[r] & K_2(B_{\alpha}) \ar[r]^-{\partial} & SK_1\bigl(A[x,y] \slash (xy-\alpha)\bigr) \ar[r] & SK_1(B) \oplus SK_1(C) }
 \]
 of abelian groups. A priori, we should consider $K_1$ instead of $SK_1$ for the two groups on the right above, but $K_1(A) \cong SK_1(A) \oplus A^{\times}$ and the morphism on units is injective since any analytic patching diagram is a pullback diagram. We have already discussed the rightmost homomorphism of the above sequence at length. To prove the main result of this section, we need conditions which ensure that it is injective, or, equivalently, that $\partial=0$. In other words, we need a convenient formula for the connecting homomorphism $\partial$. 

 In the case where $\alpha$ is a nonzerodivisor, we have a better model of the homotopy fiber of $K(\lambda_{\alpha})$, which makes this analysis possible. Recall that the \emph{Steinberg symbol} is a function
 \[
 \{-,-\} \colon K_1(A) \times K_1(A) \rightarrow K_2(A)
 \]
 which is bilinear and skew-symmetric (see for example \cite[Theorem~III.5.12.2]{WEIBEL_KBOOK}). The formula for $\partial$ can be deduced from the special case appearing in the Fundamental Theorem of $K$-theory.
 
 \begin{thm}[Fundamental Theorem of $K$-theory]\label{thm:Fundamental_Theorem_K_Theory}
  For any commutative $R$-algebra $A$, and any $i \geq 1$, the sequence
  \[
  \xymatrix@C=15pt{0 \ar[r] & K_i (A) \ar[r] & K_i(A[t]) \oplus K_i(A[t^{-1}]) \ar[r] & K_i(A[t,t^{-1}]) \ar[r]^-{\partial} & K_{i-1}(A) \ar[r] & 0 }
  \]
  is split exact. For $i=2$, the splitting of $\partial$ is given by
  \[
  \xymatrix{K_1(A) \ar[r] & K_1(A[t,t^{-1}]) \ar[r]^{\{t,-\}} & K_2(A[t,t^{-1}]) }.
  \] 
 \end{thm}
 
 \begin{proof}
  From the proof of the Fundamental Theorem in \cite{GRAYSON_FUNDAMENTAL}, we know that a splitting is given by the Loday product as in \cite[Corollary~2.3.7]{LODAY}, that is, by $-\star t$. This product factors through the pairing
  \[
  - \star - \colon K_1(A[t,t^{-1}]) \times K_1(A[t,t^{-1}]) \rightarrow K_2(A[t,t^{-1}])
  \]
  which coincides with the inverse of $\{-,-\}$ by \cite[Proposition~2.2.3]{LODAY}. Thus $\partial \{t,u\}=u$ for all $u \in K_1(A)$.
 \end{proof}
 
 In order to describe the connecting homomorphism $\partial$ in the localization sequence explicitly, we need to introduce some notation. Let $\ca{M}$ denote the category of pairs $(A,S)$ where $A$ is a commutative ring and $S \subseteq A$ is a multiplicative set consisting of nonzerodivisors. The morphisms $(A,S) \rightarrow (B,T)$ in $\ca{M}$ are the ring homomorphisms $\varphi \colon A \rightarrow B$ with $\varphi(S) \subseteq T$.
 
 Given $(A,S) \in \ca{M}$, we write $\mathbf{H}_{1,S}(A)$ for the category of $S$-torsion $A$-modules of projective dimension $1$, that is, $A$-modules $M$ which fit in a short exact sequence
 \[
 \xymatrix{0 \ar[r] & P \ar[r] & Q \ar[r] & M \ar[r] & 0}
 \]
 where $P$ and $Q$ are finitely generated projective and $A_S \ten{A} M \cong 0$. Since $\mathbf{H}_{1,S}(A)$ is a full subcategory of the abelian category $\Mod_A$ which is closed under extensions, it inherits the structure of an exact category in the sense of Quillen. We can thus define its $K$-groups using the Quillen $Q$-construction, that is, $K_i \mathbf{H}_{1,S}(A) \defl \pi_{i+1} BQ\bigl(\mathbf{H}_{1,S}(A) \bigr)$.
 
 To describe the connecting homomorphism, we use the functor $U \colon \ca{M} \rightarrow \Set$ given by the set
 \[
 U(A,S) \defl \{ (t,u) \in A^2 \; \vert \; \lambda_S(t) \in A_S^{\times}, \; u \in A^{\times} \}
 \]
  of pairs of elements which are units (respectively sent to a unit in $A_S$). Note that the element $t$ in the above set is a nonzerodivisor in $A$. This follows from the fact that the localization $\lambda_S \colon A \rightarrow A_S$ is injective. We define a natural transformation
  \[
  \psi_{(A,S)} \colon U(A,S) \rightarrow K_1\bigl(\mathbf{H}_{1,S}(A) \bigr)
  \]
 as follows. Given $(t,u) \in A^2$ such that $\lambda_S(t)$ is a unit, it follows from the above observation that $A\slash t$ is an object of $\mathbf{H}_{1,S}(A)$. Thus we can extend the two commutative triangles on the left
 \[
 \xymatrix@R=10pt{ & A \slash t \ar[dd]^{\bar{u}_{!}} \\ 0 \ar[ru]^{0_{!}} \ar[rd]_{0_{!}} && 0 \ar[ld]^{0^{!}} \ar[lu]_{0^{!}} \\ & A \slash t} \quad 
 \quad
 \xymatrix@R=10pt{ & 2 \ar[dd] \\ 0 \ar[rr]|!{[ru];[rd]}\hole \ar[ru] \ar[rd] && 1 \ar[lu] \ar[ld] \\ & 3 }
 \]
 in $Q\mathbf{H}_{1,S}(A)$ to a map $\partial \Delta^3 \rightarrow BQ \mathbf{H}_{1,S}(A)$ by sending the two remaining faces on the right diagram to the same $2$-simplex $\sigma$ which defines the loop in $BQ \mathbf{H}_{1,S}(A)$ represented by $A \slash t$ (for the definition of the symbols $(-)^{!}$ and $(-)_{!}$, see \cite[p.~101]{QUILLEN_HIGHER_K}). We let $\psi_{(A,S)} \bigl( (t,u) \bigr)$ be its homotopy class in $K_1 \bigl( \mathbf{H}_{1,S}(A) \bigr)$.
 
 The important thing for us to note is that $\psi_{(A,S)} \bigl( (t,u) \bigr)=0$ if $\bar{u} \colon A \slash t \rightarrow A \slash t$ is the identity morphism, for then the map constructed above coincides with the boundary of the degenerate $3$-simplex $s_2(\sigma)$.
 
 With these preliminaries in hand, we can now describe the connecting homomorphism $\partial$ in the exact sequence of $K$-groups associated to a localization at a set of nonzerodivisors.
 
 \begin{lemma}\label{lemma:formula_for_connecting_homomorphism}
 Let $A$ be a commutative ring and let $S \subseteq A$ be a multiplicative set of nonzerodivisors. Then there exists an exact sequence
 \[
 \xymatrix{K_2(A) \ar[r]^-{K_2(\lambda_S)} & K_2(A_S) \ar[r]^-{\partial} & K_1\bigl(\mathbf{H}_{1,S}(A)\bigr) \ar[r]^-{\kappa} & K_1(A) \ar[r]^-{K_1(\lambda_S)} & K_1(A_S)}
 \]
 which is natural for morphisms in $\ca{M}$. Moreover, the formula
 \[
 \partial \{t,u\} = \psi_{(A,S)} \bigl((t,u)\bigr)
 \]
 holds for all $u \in A^{\times}$ and all $t \in A$ with $\lambda_S(t) \in A_S^{\times}$. In particular, $\partial \{t,u\}=0$ if $u \equiv 1 \mod t$.
 \end{lemma}
 
\begin{proof}
 The last claim follows from the formula for $\partial$ and the above discussion of the natural transformation $\psi$. This formula is a special case of a result of Grayson, see \cite[Remark~7.8]{GRAYSON_FLAT}. Since Grayson's argument is rather intricate, we indicate a proof of this using only the Fundamental Theorem and basic facts about Loday's product.
 
 If $\bar{S}$ denotes the saturation of $S$, then the identity on $A$ gives a morphism $(A,S) \rightarrow (A,\bar{S})$ in $\ca{M}$ which gives a levelwise isomorphism between the respective long exact sequences. Thus we can assume without loss of generality that $S$ is saturated and thus that $t \in S$. By naturality, we can therefore reduce to the universal case $A=\mathbb{Z}[u,u^{-1},t]$ and $S=\{1,t,t^2,\ldots\}$, in which case the localization sequence is of the form
 \[
 \xymatrix{K_2(B[t]) \ar[r] & K_2(B[t,t^{-1}]) \ar[r]^-{\partial} & K_1 \mathbf{H}_{1,t}(B[t])}
 \]
 for $B=\mathbb{Z}[u,u^{-1}]$.
 
 In this case, $\psi_{(A,S)}$ factors through the inclusion $K_1(B) \rightarrow K_1\bigl( \mathbf{H}_{1,t} (B[t]) \bigr)$ induced by the functor which sends a finitely generated projective $B$ module $P$ to $P$, equipped with the (nilpotent) endomorphism $0$. Since $B$ is regular, this inclusion is in fact an isomorphism, with inverse $h \colon K_1\bigl( \mathbf{H}_{1,t} (B[t]) \bigr) \rightarrow K_1 B$ given by forgetting the $t$-action. The composite $h\partial$ is precisely the morphism $\partial$ appearing in the Fundamental Theorem, so $h \partial \{t,u\}=u \in B^{\times}$ by Theorem~\ref{thm:Fundamental_Theorem_K_Theory}. That this coincides with $\psi_{(B,t)}\bigl((t,u)\bigr)$ follows from the description of the inclusion of $B^{\times}$ in $K_1(B)=\pi_2 BQ\bigl(\Proj(B)\bigr)$, see for example \cite[Exercise~IV.7.9]{WEIBEL_KBOOK}.
 
 There is a subtle point about the naturality of this long exact sequence that should be taken into account. As Swan noted in \cite[Appendix]{SWAN_QUADRIC}, one has to be a bit careful to ensure that the connecting homomorphism $\partial$ is natural.
 
 In order to do this, we first have to give a precise definition of the homomorphisms $\partial$ and of $\kappa$. Following \cite[Proposition~A13]{SWAN_QUADRIC}, we write $\Proj^{\prime}(A)$ for the exact subcategory of $A$-modules $M$ which have projective dimension at most $1$ and which satisfy the condition that the localization $M_S$ is a projective $A_S$-module. Note that we have an exact inclusion $\mathbf{H}_{1,S}(A) \rightarrow \Proj^{\prime}(A)$ and that the resulting square
 \[
 \xymatrix{ \mathbf{H}_{1,S}(A) \ar[r] \ar[d] & \Proj^{\prime}(A) \ar[d] \\ \ast \ar[r]^-{0} & \Proj(A_S) }
 \]
 commutes up to unique natural isomorphism. Applying the $Q$-construction to this square, we obtain a homotopy fiber square by \cite[Proposition~A13]{SWAN_QUADRIC}.
 
 We claim that this construction gives a functor from $\ca{M}$ to such squares of exact categories. To see this, we need to check that for any $\varphi \colon (A,S) \rightarrow (B,T)$ in $\ca{M}$ and any $M \in \Proj^{\prime}(A)$, the module $B \ten{A} M$ lies in $\Proj^{\prime}(B)$. In fact, any exact sequence
 \[
 \xymatrix{0 \ar[r] & P \ar[r]^-{\alpha} & Q \ar[r] & M \ar[r] & 0}
 \]
 with $P$, $Q \in \Proj(A)$ and such that $M_S$ is projective is sent to a short exact sequence by $B \ten{A} -$. To see this, we only need to check that $B \ten{A} \alpha$ is injective, and since $T$ consists of nonzerodivisors, it suffices to check this after applying the localization $B_T \ten{B} -$. But we know that $B_T \ten{A} \alpha$ is a split monomorphism since $A_S \ten{A} \alpha$ is a split monomorphism (which in turn is a consequence of the projectivity of $M_S$).
 
 We let $\partial$ be the connecting homomorphism in the homotopy fiber square resulting from the above square of exact categories, and we let $\kappa$ be the composite
 \[
 \pi_2 BQ \bigl( \mathbf{H}_{1,S}(A) \bigr) \rightarrow \pi_2 BQ \bigl( \Proj^{\prime}(A) \bigr) \cong \pi_2 BQ \bigl( \Proj(A)  \bigr)
 \]
 where the isomorphism is induced from the inclusion $\Proj(A) \rightarrow \Proj^{\prime} (A)$ of exact categories.
 
 To get the desired naturality, it is perhaps easiest to note that the $Q$-construction is a 2-functor from the 2-category of exact categories, exact functors, and natural isomorphisms to the 2-category of categories. Thus any time we apply the $Q$-construction to a coherent diagram of exact functors and natural isomorphisms, we get a fully coherent homotopy commutative diagram of spaces after applying $BQ(-)$.
 
 Note that the alternative method suggested by Swan needs to be modified slightly: the object $0 \in Q\ca{E}$ is rarely a zero-object, but it is \emph{rigid}, that is, it has no non-trivial automorphisms. By considering rigid objects and natural \emph{isomorphisms} (instead of zero-objects and arbitrary natural transformations), one can alternatively use the arguments of \cite[Appendix]{SWAN_QUADRIC} to establish the desired naturality.
\end{proof} 
 
 After this detour, we are now ready to show that the connecting homomorphism in our original exact sequence vanishes in many cases.
 
 \begin{prop}\label{prop:connecting_homomorphism_K_theory_vanishes}
  Let $A$ be a commutative ring, let $\alpha \in A$ be an element, let $C=A[x,y] \slash (xy-\alpha)$ and let $B \defl A^h_{(\alpha)}[x,y] \slash (xy-\alpha)$. Assume the following conditions hold:
  \begin{enumerate}
  \item[(1)] $SK_1(A[x,x^{-1}]) \cong 0$;
  \item[(2)] $A^h_{(\alpha)}$ is $K_2$-regular;
  \item[(3)] $SK_1(A \slash \alpha) \cong 0$;
  \item[(4)] $A^{\times} \rightarrow (A \slash \alpha)^{\times}$ is surjective;
  \item[(5)] The element $\alpha \in A[x,y] \slash (xy -\alpha)$ is a nonzerodivisor. 
  \end{enumerate}
 Then the connecting homomorphism $\partial$ in the exact sequence
 \[
 \xymatrix{K_2(B) \oplus K_2(C_{\alpha}) \ar[r] & K_2(B_{\alpha}) \ar[r]^-{\partial} & SK_1(C ) \ar[r] & SK_1(B) \oplus SK_1(C_{\alpha}) }
 \]  
  (associated to the analytic isomorphism $C \rightarrow B$ along $\alpha$) is equal to $0$.
 \end{prop}

 \begin{proof}
 We first use the Fundamental Theorem and the fact that $SK_1(A[x,x^{-1}])=0$ to reduce the problem to checking that $\partial \{x,u\}=0$ holds for all $u \in (A^h_{(\alpha)})^{\times}$. The assumption on units further simplifies this to the case $u \equiv 1 \mod \alpha$. In this case, the claim follows from the explicit formula for the connecting homomorphism in the localization sequence given in Lemma~\ref{lemma:formula_for_connecting_homomorphism}. We give a detailed proof in five steps, where each step uses the corresponding assumption.
 
 \textbf{Step 1.} If we tensor the patching diagram associated to the analytic isomorphism $A \rightarrow A^h_{(\alpha)}$ along $\alpha$ with $\mathbb{Z}[x,x^{-1}]$, we obtain the patching diagram on the left below
 \[
 \vcenter{\xymatrix{ A[x,x^{-1}] \ar[d] \ar[r] & A_{\alpha}[x,x^{-1}] \ar[d]^{\varphi^{\prime}} \\
 A^h_{(\alpha)}[x,x^{-1}] \ar[r]_-{\varphi} & A^h_{(\alpha)}[\frac{1}{\alpha}][x,x^{-1}] \cong B_{\alpha}
 }} \quad \quad
\vcenter{\xymatrix{ C \ar[d] \ar[r] & A_{\alpha}[x,x^{-1}] \ar[d]^{\varphi^{\prime}} \\ B \ar[r]_-{\lambda_{\alpha}} & B_{\alpha}
}}
 \]
 while the analytic isomorphism $C \rightarrow B$ along $\alpha$ yields the diagram on the right. An important fact that we will use is that the two vertical homomorphisms on the right of each square are equal.
 
 The long exact sequence of $K$-groups for the left diagram and the fact that $SK_1(A[x,x^{-1}])=0$ (Assumption~(1)) imply that $K_2(\varphi^{\prime})$ and $K_2(\varphi)$ are jointly surjective. Since the connecting homomorphism $\partial$ for the right diagram above vanishes on the image of $K_2(\varphi^{\prime})$, we have reduced the problem to checking that the composite $\partial K_2(\varphi)$ vanishes.
 
 \textbf{ Step 2.} We can use the Fundamental Theorem to give a simpler description of the domain $K_2(A^h_{(\alpha)}[x,x^{-1}])$ of $K_2(\varphi)$. Namely, from the Fundamental Theorem we know that the top row of the diagram
 \[
 \xymatrix@C=15pt{ K_2(A^h_{(\alpha)}) \ar[r] & K_2(A^h_{(\alpha)}[x]) \oplus K_2(A^h_{(\alpha)}[x^{-1}]) \ar[r]^-{\pm} & K_2(A^h_{(\alpha)}[x,x^{-1}])  \ar@<-0.5ex>[r] & K_1(A^h_{(\alpha)}) \ar@<-0.5ex>[l]_-{ \{x,-\} } \\ & K_2(A^h_{(\alpha)}) \oplus K_2(A^h_{(\alpha)}) \ar[u]^{\iota \oplus \iota} \ar[r]^-{\pm} & K_2(A^h_{(\alpha)}) \ar[u]^{\iota} }
 \]
 is split exact. Since $A^h_{(\alpha)}$ is $K_2$-regular (Assumption~(2)), the left vertical homomorphism in the above diagram is an isomorphism.
 
 Thus 
 \[
 \{x,-\} \colon K_1(A^h_{(\alpha)}) \rightarrow K_2(A^h_{(\alpha)}[x,x^{-1}]) \quad \text{and} \quad K_2(\iota) \colon K_2(A^h_{(\alpha)}) \rightarrow K_2(A^h_{(\alpha)}[x,x^{-1}])
 \]
 are jointly surjective. Since $\varphi \iota$ factors through the localization $\lambda_{\alpha} \colon B \rightarrow B_{\alpha}$, it follows that $\partial K_2(\varphi) K_2(\iota)=0$. This reduces the problem to checking that $\partial$ vanishes on $K_2(\varphi) \{x,u\}=\{x,K_1(\varphi)(u)\}$ for all $u \in K_1(A^h_{(\alpha)})$.
 
 \textbf{Step 3.} Recall that $K_1(A^h_{(\alpha)}) \cong (A^h_{(\alpha)})^{\times} \oplus SK_1(A^h_{(\alpha)})$. Since the functor $SK_1$ is $j$-invariant (see Lemma~\ref{lemma:SL_r_j_invariant}), we know that
 \[
 SK_1(A^h_{(\alpha)}) \cong SK_1(A \slash \alpha) \cong 0
 \]
 from Assumption~(3). Thus  we only need to consider elements of the form $\{x,\varphi(u)\}$ with $u \in (A^h_{(\alpha)})^{\times}$.
 
 \textbf{Step 4.} Recall that $A \rightarrow A^h_{(\alpha)}$ is an analytic isomorphism along $\alpha$, so the induced homomorphism $A \slash \alpha \rightarrow A^h_{(\alpha)} \slash \alpha$ is an isomorphism. From Assumption~(4) it follows that for each unit $u \in (A^h_{(\alpha)})^{\times}$ there exists a unit $v \in A^{\times}$ such that the image of $v$ under the diagonal composite in the square
 \[
 \xymatrix{A \ar[r] \ar[d] & A^h_{(\alpha)} \ar[d] \\ A \slash \alpha  \ar[r]_-{\cong} & A^h_{(\alpha)} \slash \alpha  }
 \]
 coincides with $\bar{u}$.
 
 From the bilinearity of $\{-,-\}$ it follows that we can treat $\{x,v\}$ and $\{x,v^{-1}u\}$ separately. Note that $\{x,v\}$ lies in the image of $K_2(A[x,x^{-1}]) \rightarrow K_2(B_{\alpha})$, which factors through $K_2(\varphi^{\prime}) \colon K_2(A_{\alpha}[x,x^{-1}]) \rightarrow K_2(B_{\alpha})$. Since $\partial K_2(\varphi^{\prime})=0$, it only remains to check that $\partial(\{x,u\})=0$ whenever $u \in A^h_{(\alpha)}$ is a unit such that $u \equiv 1 \mod \alpha$.
 
 \textbf{Step 5.} Finally we use the fact that $\alpha \in C$ is a nonzerodivisor (Assumption~(5)) to give a simple description of the connecting homomorphism $\partial$ in the long exact sequence coming from the analytic isomorphism $C \rightarrow B$ along $\alpha$. Since $C \rightarrow B$ is flat (as base change of the flat homomorphism $A \rightarrow A^h_{(\alpha)}$ ), it follows that $\alpha$ is also a nonzerodivisor in $B$. Thus the connecting homomorphism $\partial \colon K_2(B_{\alpha}) \rightarrow K_1 (C)$ can be written as the composite of the three solid arrows in the diagram
 \[
 \xymatrix{K_2 (B) \ar@{-->}[r] & K_2(B_{\alpha}) \ar[r]^-{\partial^{\prime}} & K_1 \bigl( \mathbf{H}_{1,\alpha}(B)\bigr ) \ar@{-->}[r] & K_1(B) \ar@{-->}[r] & K_1(B_{\alpha}) \\ 
 K_2 (C) \ar@{-->}[r] \ar@{-->}[u] & K_2(C_{\alpha}) \ar@{-->}[r] \ar@{-->}[u] & K_1 \bigl( \mathbf{H}_{1,\alpha}(C)\bigr ) \ar[r]^-{\kappa} \ar[u]^{\cong} & K_1(C) \ar@{-->}[r] \ar@{-->}[u] & K_1(C_{\alpha}) \ar@{-->}[u] }
 \]
 of the natural long exact sequences obtained from localization at $\alpha$ by Lemma~\ref{lemma:formula_for_connecting_homomorphism} (here the middle vertical homomorphism is an isomorphism since $C \rightarrow B$ is an analytic isomorphism along $\alpha$). Thus it suffices to check that $\partial^{\prime}(\{x,u\})=0$ whenever $u \in A^h_{(\alpha)}$ is a unit such that $u \equiv 1 \mod \alpha$. From Lemma~\ref{lemma:formula_for_connecting_homomorphism} we know that $\partial^{\prime}(\{x,u\})$ is given by $\psi_{(B,\alpha)} \bigl( (x,u)\bigr)$, and that this element is $0$ if $u \equiv 1 \mod x$. Since $\alpha =xy$ in $B$, the latter follows from the assumption $u \equiv 1 \mod \alpha$.
 \end{proof}

 Before coming to the main theorem of this section, we need two more lemmas.
 
 \begin{lemma}\label{lemma:ufd_implies_nonzerodivisor}
 Let $A$ be a unique factorization domain and let $f \neq 0$ be an element of $A$. Then the element $xy+f \in A[x,y]$ is irreducible, so $A[x,y] \slash (xy+f)$ is a domain. Moreover, we have $f \not\equiv 0 \mod xy+f$.
 \end{lemma}
 
 \begin{proof}
 Considering $xy+f$ as a polynomial in $y$ with coefficients in $A[x]$, we see that $f \not\in (xy+f)$ for degree reasons, which establishes the final claim. Since $A$ is a unique factorization domain, so is $A[x,y]$, so the claim that the quotient is a domain follows from the fact that irreducible elements are prime.
 
 To see the irreducibility of $xy+f$, note that any factorization must be of the form $xy+f=(ay+b)c$ for certain elements $a,b,c \in A[x]$, again for degree reasons. It follows that $x=ac$. Moreover, since $A$ is a domain, $x \in A[x]$ is a prime element. From the uniqueness of prime factorizations in $A[x]$ it follows that either $a$ or $c$ must be a unit. To establish the claim, we need to show that $c$ is a unit. If not, then $a \in A[x]$ is a unit, so it follows in particular that $a \in A$. Thus $c=a^{-1}x$ is a multiple of $x$, so from
 \[
 xy+f=(ay+b)c
 \]
 it follows that $f \in A \setminus \{0\}$ satisfies $f \equiv 0 \mod x$ in $A[x]$, a contradiction. Thus $a$ is a multiple of $x$ and $c$ is a unit, so $xy+f$ is indeed irreducible.
 \end{proof}
 
 Recall that the rings $R_{\ell} \defl R \langle z_1, \ldots, z_{\ell} \rangle$ are defined inductively by setting $R_0 \defl R$, $R_{\ell +1} \defl R_{\ell} \langle z_{\ell+1} \rangle$. The second lemma establishes a close relationship between the low-dimensional $K$-groups of $R$ and of $R_{\ell}$.
 
 \begin{lemma}\label{lemma:low_dimensional_K_for_R_l}
 Let $R$ be a regular commutative ring such that $K_0(R) \cong \mathbb{Z}$ and $SK_1(R)\cong 0$. Then $K_0(R_{\ell}) \cong \mathbb{Z}$ and $SK_1(R_{\ell}) \cong 0$ holds for all $\ell \geq 0$.
 \end{lemma}
 
 \begin{proof}
 By induction, it suffices to show this in the case $\ell=1$. To see this case, let $A \defl R[t]_{1+tR[t]}$ and note that the localization $A_t$ is isomorphic to $R \langle z \rangle$, with $R$-isomorphism given by $t \mapsto z^{-1}$.
 
 The regularity assumption implies that $K_0(R[t]) \cong \mathbb{Z}$. The same is therefore true for any localization of $R[t]$, so $K_0(A) \cong \mathbb{Z}$ and $K_0(A_t) \cong \mathbb{Z}$ (write a finitely generated projective $P$ in the localization at $S$ as $M_S$ for some finitely presentable $R[t]$-module $M$, which has a finite resolution by finitely generated free modules since $K_0(R[t]) \cong \mathbb{Z}$).
 
 Since $SK_1$ is $j$-invariant (see Lemma~\ref{lemma:SL_r_j_invariant}), we have $SK_1(A) \cong SK_1(R) \cong 0$. Since $A$ and $A \slash t$ are regular, the localization sequence for $A$ at $t$ is given by
 \[
 \xymatrix@R=5pt{A^{\times} \ar[r]^-{
 \bigr( \begin{smallmatrix}
 0 \\ \lambda_t
 \end{smallmatrix} \bigr) } & SK_1(A_t) \oplus A_t^{\times} \ar[r] & K_0(R) \ar[r]^-{0} & K_0(A) \ar[r]^{\cong} & K_0(A_t) \\
 & & \mathbb{Z} \ar@{=}[u] & \mathbb{Z} \ar@{=}[u] & \mathbb{Z} \ar@{=}[u]
 }
 \]
 where the second morphism has a splitting given sending the generator $[R]$ to $t \in A_t^{\times}$. Since the first morphism and the splitting are jointly epimorphic, it follows that $SK_1(A_t) \cong 0$, as claimed.
 \end{proof}
 
 With these ingredients in place, we can prove the following theorem.
 
 \begin{thm}\label{thm:SK1-contractible}
 Let $R$ be a regular unique factorization domain with $SK_1(R) \cong 0$ and $K_0(R) \cong \mathbb{Z}$. Let $\beta \neq 0$ be a homogeneous polynomial in $R[t_1, \ldots, t_k]$ of degree $\geq 2$. If
 \[
 R_{\ell}[t_1, \ldots, t_k] \slash \beta
 \]
 is reduced and $SK_1$-contractible for all $\ell \geq 0$, then
 \[
 SK_1 \bigl( R_{\ell}[t_1, \ldots, t_k, x_i, y_i \vert i=1, \ldots, n] \slash (\textstyle \sum_{i=1}^n x_i y_i+\beta) \bigr) \cong 0
 \]
 for all $n \geq 1$ and all $\ell \geq 0$.
 \end{thm}
 
 \begin{proof}
 Let $A_{n,\ell} \defl R_{\ell}[t_1, \ldots, t_k, x_i, y_i \vert i=1, \ldots, n]$ be the polynomial $R_{\ell}$-algebra with $2n+k$ variables. Let $\beta_0 \defl \beta \in A_{0,0}$ and let $\beta_{n+1} \defl x_{n+1} y_{n+1} + \beta_n \in A_{n+1,0}$ for all $n \geq 0$. We use the same names for the corresponding polynomials in $A_{n,\ell}$, considered as polynomials with coefficients in $R_{\ell}$.
 
 We first claim that $A_{n,\ell} \slash \beta_n$ is a domain for all $n \geq 1$ and all $\ell \geq 0$. To see this, we can apply Lemma~\ref{lemma:ufd_implies_nonzerodivisor} to the unique factorization domain $A_{n-1,\ell}$ and the element $\beta_{n-1} \neq 0$. Next we prove by induction on $n \geq 0$ that $SK_1(A_{n, \ell} \slash \beta_n) \cong 0$ for all $n \geq 0$ and $\ell \geq 0$, which will establish the claim of the theorem. The proof uses the techniques we have developed in this section to study analytic patching diagrams arising from algebras of the form $A[x,y] \slash (xy-\alpha)$.
 
 The base case of the induction is covered by the assumption of the theorem: $A_{0, \ell} \slash \beta_0=R_{\ell}[t_1, \ldots, t_k] \slash \beta$, so $SK_1(A_{0,\ell} \slash \beta_0) \cong 0$ for all $\ell \geq 0$.
 
 We prove the inductive step $n \rightsquigarrow n+1$ simultaneously for all $\ell \geq 0$, so we fix some $\ell \geq 0$ and set $R^{\prime} \defl R_{\ell}=R \langle z_1, \ldots z_{\ell} \rangle$. Let
 \[
 A \defl A_{n, \ell}=R^{\prime}[t_1, \ldots, t_k, x_i, y_i \vert i=1, \ldots, n]
 \]
 and note that $A[x,y] \slash (xy-\beta_n)$ is isomorphic to $A_{n+1,\ell} \slash \beta_{n+1}$. Indeed, we can send $x$ to $x_{n+1}$ and $y$ to $-y_{n+1}$ to get the desired isomorphism.
 
 To establish the claim, it thus suffices to show that $SK_1\bigl(A[x,y] \slash (xy-\beta_n)\bigr) \cong 0$ holds. If we write $\alpha \defl \beta_n$, then we can consider the analytic patching diagram
 \[
 \xymatrix{A[x,y] \slash (xy-\alpha) \ar[r]^-{\varphi} \ar[d]_{\psi} & A_{\alpha}[x,x^{-1}] \ar[d] \\ B \ar[r] & B_{\alpha} } 
 \]
 where $B=A^h_{(\alpha)}[x,y] \slash (xy-\alpha)$. We will prove the inductive step by establishing the following three claims:
 
 \begin{enumerate}
 \item[(A)] The homomorphism $\varphi_{\ast} \colon SK_1 \bigl(A[x,y] \slash (xy-\alpha)\bigr) \rightarrow SK_1 (A_{\alpha}[x,x^{-1}])$ sends each $\sigma$ in the domain to $0$;
 \item[(B)] The group $SK_1(B)$ is trivial, so $\psi_{\ast} \colon SK_1 \bigl(A[x,y] \slash (xy-\alpha)\bigr) \rightarrow SK_1(B)$ also sends each $\sigma$ in the domain to $0$;
 \item[(C)] The homomorphisms $\varphi_{\ast}$ and $\psi_{\ast}$ are jointly monomorphic.
 \end{enumerate}

 Taken together, these three claims show that $SK_1 \bigl(A[x,y] \slash (xy-\alpha)\bigr) \cong 0$. The argument can now be summarized as follows. We establish Claim~(A) by using the naive $\mathbb{A}^{1}$-contractibility of $A[x,y] \slash (xy-\alpha)$ and the fact that
 \[
 NSK_1(A_{\alpha}[x,x^{-1}])\cong 0
 \]
 holds  (see Proposition~\ref{prop:F_and_naive_A1_contractible_A}). Claim~(B) is proved by applying Theorem~\ref{thm:henselian_implies_trivial_kernel_for_overring} and its Corollary~\ref{cor:henselian_implies_trivial_kernel_for_overring}, while Claim~(C) is proved by applying Proposition~\ref{prop:connecting_homomorphism_K_theory_vanishes}. In both cases, five conditions need to be checked. Most of these are fairly straightforward, only Condition~(v) of Corollary~\ref{cor:henselian_implies_trivial_kernel_for_overring} is a bit more involved. This condition requires that
 \[
 SK_1(A \slash \alpha [x]) \cong 0
 \]
 holds, but a priori we only know that $SK_1(A \slash \alpha)=SK1(A_{n, \ell} \slash \beta_n) \cong 0$ holds (by the inductive assumption). To deduce the above fact, we will use the monic inversion principle for $SK_1$ and the fact that we also know that $SK_1(A_{n, \ell +1} \slash \beta_n) \cong 0$ holds. Next we give the details for these three steps.

\textbf{Step~(A).} Since $R^{\prime}$ is regular, so is the localization $A_{\alpha}[x,x^{-1}]$ of the polynomial algebra $A[x]=R^{\prime}[t_1,\ldots, t_k,x_i,y_i,x]$. Thus it is in particular $SK_1$-regular.

Since $\beta$ is homogeneous of degree $\geq 2$, the polynomial $xy-\alpha=\sum_{i=0}^{n+1} x_i y_i+\beta$ can be made homogeneous of the same degree by choosing appropriate degrees for $x_i$ and $y_i$. It follows that $A[x,y] \slash (xy-\alpha)$ is naively $\mathbb{A}^{1}$-contractible over $R^{\prime}$ by the Swan--Weibel homotopy trick (see \cite[Appendix to \S V.3]{LAM}). Note that we have $SK_1(R^{\prime}) \cong 0$ by Lemma~\ref{lemma:low_dimensional_K_for_R_l}. Thus Proposition~\ref{prop:F_and_naive_A1_contractible_A} is applicable and implies that $\varphi_{\ast} \sigma=0$ for all $\sigma \in SK_1\bigl( A[x,y] \slash (xy-\alpha) \bigr)$.

\textbf{Step~(B).} This step contains the heart of the argument, where we use the factorization property for not necessarily monic polynomials for henselian pairs (see Theorem~\ref{thm:generalized_hensel_lemma}). Namely, we use Corollary~\ref{cor:henselian_implies_trivial_kernel_for_overring} of Theorem~\ref{thm:henselian_implies_trivial_kernel_for_overring} to show that $SK_1(B)=0$. Recall that this corollary relates 
\[
F\bigl(A^h_{(\alpha)}[x,y]\slash (xy-\alpha) \bigr) \quad \text{and} \quad F(A^h_{(\alpha)} \slash \alpha) \cong F(A \slash \alpha)
\]
 for suitable functors $F$.

 Specifically, $\alpha \in A^h_{(\alpha)}$ needs to be a nonzerodivisor, $F$ needs to satisfy weak analytic excision, $A^h_{(\alpha)}[x]$ needs to be $F$-contractible, and $F$ needs to be (partially) $j$-injective (so that Conditions~(i)-(iii) of Theorem~\ref{thm:henselian_implies_trivial_kernel_for_overring} are satisfied), and $F$ needs to satisfy weak Milnor excision and $A^h_{(\alpha)} \slash \alpha [x] \cong A \slash \alpha [x] $ needs to be $F$-contractible (so that Conditions~(iv) and (v) of Corollary~\ref{cor:henselian_implies_trivial_kernel_for_overring}) are satisfied.
 
 Since $A^h_{(\alpha)}$ is flat over $A$, it suffices to check that $\alpha \in A$ is a nonzerodivisor, which is clear from the fact that $A$ is a domain. We have seen in \S \ref{section:analytic} that $SK_1$ satisfies weak analytic excision (either by noting that $SK_1=\mathrm{colim}_r SK_{1,r}$ and appealing to Vorst's Lemma, or by using the corresponding fact for $K_1$, see Proposition~\ref{prop:K_theory_weak_analytic_excision}). Moreover, $SK_1$ satisfies weak Milnor excision by Proposition~\ref{prop:weak_milnor_excision_group_case}. We know from Lemma~\ref{lemma:SL_r_j_invariant} that $SK_1$ is $j$-invariant, hence in particular $j$-injective.
 
 It remains to check that $SK_1(A^h_{(\alpha)}[x]) \cong 0$ and $SK_1(A \slash \alpha [x]) \cong 0$. Since the henselization of a regular ring is regular (see for example \cite[Corollary~5.7 and Theorem~5.10]{GRECO}), it suffices to show that $SK_1(A^h_{(\alpha)}) \cong 0$  to establish the first of these two claims. By $j$-injectivity of $SK_1$, it further suffices to check that $SK_1(A \slash \alpha) \cong 0$ (since $\alpha$ lies in the Jacobson radical of $A^h_{(\alpha)}$). By definition we have $A=A_{n, \ell}$ and $\alpha=\beta_n$, so $A \slash \alpha=A_{n, \ell} \slash \beta_n$, which is $SK_1$-contractible by the inductive assumption.
 
 To establish Claim~(B) it only remains to check that $SK_1(A \slash \alpha [x]) \cong 0$ holds. Since $A_{n, \ell+1} \slash \beta_n$ can be written as localization of $A_{n,\ell} \slash \beta_n [x]$ at the set of monic polynomials in $x$ with coefficients in $R_{\ell}=R^{\prime}$, it follows that 
 \[
 A_{n,\ell} \slash \beta_n [x] \rightarrow A_{n,\ell} \slash \beta_n \langle x \rangle
 \]
 factors through the algebra $A_{n, \ell+1} \slash \beta_n$. From the inductive assumption we know that we have $SK_1(A_{n,\ell+1} \slash \beta_n) \cong 0$, so
 \[
 SK_1( A_{n,\ell} \slash \beta_n [x] ) \rightarrow SK_1( A_{n,\ell} \slash \beta_n \langle x \rangle)
 \]
 is the zero morphism. From the Horrocks principle $(\mathrm{H})$ (which is a consequence of the Fundamental Theorem, see Theorem~\ref{thm:Fundamental_Theorem_K_Theory}, and Lemma~\ref{lemma:pullback_version_of_H} ) it follows that $SK_1(A_{n, \ell} \slash \beta_n[x]) \cong 0$, which concludes the proof of Claim~(B).
 
 \textbf{Step~(C).} Finally, we use Proposition~\ref{prop:connecting_homomorphism_K_theory_vanishes} to show that the connecting homomorphism $\partial$ in the exact sequence
 \[
 \xymatrix{K_2(B_{\alpha}) \ar[r]^-{\partial} & SK_1\bigl(A[x,y] \slash (xy-\alpha) \bigr) \ar[r]^-{\bigl(\begin{smallmatrix} \psi_{\ast} \\ \varphi_{\ast} \end{smallmatrix} \bigr)} & SK_1(B) \oplus SK_1(A_{\alpha}[x,x^{-1}]) }
 \]
 is equal to $0$, yielding the desired injectivity of $\bigl(\begin{smallmatrix} \psi_{\ast} \\ \varphi_{\ast} \end{smallmatrix} \bigr)$.
 
 We have already observed in the proof of Claim~(B) above that $SK_1(A \slash \alpha) \cong 0$, so Condition~(3) of Proposition~\ref{prop:connecting_homomorphism_K_theory_vanishes} holds. Since $A$ is regular, so is the henselization $A^h_{(\alpha)}$ (by \cite[Corollary~5.7 and Theorem~5.10]{GRECO}). Thus $A^h_{(\alpha)}$ is in particular $K_2$-regular (since regular rings are $K_i$-regular for all $i$, see for example \cite[Theorem~V.6.3]{WEIBEL_KBOOK}). This shows that Condition~(2) of Proposition~\ref{prop:connecting_homomorphism_K_theory_vanishes} also holds.
 
 To see that $SK_1(A[x,x^{-1}]) \cong 0$, note that the regularity of $R^{\prime}$ reduces this to the claim that $SK_1(R^{\prime}[x,x^{-1}]) \cong 0$. From the Fundamental Theorem (and the regularity of $R^{\prime}$) we know that $K_1(R^{\prime}[x,x^{-1}]) \cong K_1(R^{\prime}) \oplus K_0(R^{\prime})$, with $K_0(R^{\prime} )\cong \mathbb{Z}$ corresponding to the subgroup of $(R^{\prime}[x,x^{-1}])^{\times}$ generated by $x$. From the fact that $SK_1(R^{\prime}) \cong 0$ (see Lemma~\ref{lemma:low_dimensional_K_for_R_l}), it follows that $SK_1(R^{\prime}[x,x^{-1}]) \cong 0$ as well, so Condition~(1) of Proposition~\ref{prop:connecting_homomorphism_K_theory_vanishes} holds.
 
 From Lemma~\ref{lemma:ufd_implies_nonzerodivisor} we know that $A[x,y] \slash xy-\alpha$ is a domain and that $\alpha \neq 0$ in $A[x,y] \slash (xy-\alpha)$, so it is a nonzerodivisor (Condition~(5) of Proposition~\ref{prop:connecting_homomorphism_K_theory_vanishes}).
 
 It only remains to show that Condition~(4) of Proposition~\ref{prop:connecting_homomorphism_K_theory_vanishes} holds, that is, we need to show that the homomorphism
 \[
 A^{\times} \rightarrow (A \slash \alpha)^{\times}
 \]
 is surjective. To see this, it suffices to check that
 \[
 ( R^{\prime})^{\times} \rightarrow (A \slash \alpha)^{\times}=(A_{n, \ell} \slash \beta_n)^{\times}
 \]
 is surjective. 
 
 We claim that the target $A_{n,\ell} \slash \beta_n$ is reduced. For $n >0$, this follows from the fact that the target is a domain (a consequence of Lemma~\ref{lemma:ufd_implies_nonzerodivisor} established at the beginning of the proof). For $n=0$, we have $A_{0, \ell} \slash \beta_0=R_{\ell}[t_1, \ldots, t_k] \slash \beta$, which is reduced by assumption.
 
 If we let $\mathbb{G}_m(C) \defl C^{\times}$, then the $\mathbb{G}_{m}$-regular rings are precisely the reduced rings. In the proof of Claim~(A) we have seen that $A_{n, \ell} \slash \beta_n$ is naively $\mathbb{A}^{1}$-contractible over $R^{\prime}$. It follows that the inclusion $\eta \colon R^{\prime} \rightarrow A_{n, \ell} \slash \beta_n$ induces an isomorphism after applying $\mathbb{G}_{m}$, so
 \[
  (R^{\prime})^{\times} \rightarrow (A \slash \alpha)^{\times}=(A_{n, \ell} \slash \beta_n)^{\times}
 \]
 is in particular surjective, which establishes Claim~(4) of Proposition~\ref{prop:connecting_homomorphism_K_theory_vanishes}, and thus Claim~(C) above.
 
 Taken together, Claims~(A), (B), and (C) show that $SK_1\bigl(A[x,y] \slash (xy-\alpha) \bigr)=0$. Thus we find that $SK_1\bigl(A[x,y] \slash (xy+\beta_n) \bigr) \cong 0$, that is, $SK_1(A_{n+1, \ell} \slash \beta_{n+1}) \cong 0$. This concludes the proof of the inductive step $n \rightsquigarrow n+1$.
 \end{proof}
 
 With this result in hand, we can now prove two of the theorems stated in the introduction, which we restate for the convenience of the reader.
 
 \begin{cit}[Theorem~\ref{thm:converse_to_Swan's_thm}]
 Let $k$ be a field and let $\alpha \in k[t_1, \ldots, t_n]$ be a homogeneous polynomial of degree $\geq 2$. Assume that for all field extensions $k \subseteq k^{\prime}$, the algebra $k^{\prime}[t_1, \ldots, t_n] \slash \alpha$ is a $K_1$-regular domain. Then for all field extensions $k \subseteq k^{\prime}$, the algebra
 \[
 A \defl k^{\prime}[t_1, \ldots, t_n, x,y] \slash (xy-\alpha)
 \]
 is a $K_1$-regular unique factorization domain.
 \end{cit}
 
 \begin{proof}[Proof of Theorem~\ref{thm:converse_to_Swan's_thm}]
 Since $K_1(A) \cong SK_1(A) \oplus A^{\times}$, it suffices to check that $A$ is $SK_1$-regular and $\mathbb{G}_m$-regular, where $\mathbb{G}_m(C)=C^{\times}$. The latter follows from the fact that $A$ is reduced, which is clear once we establish the claim that $A$ is a unique factorization domain. Since $A_{x} \cong k^{\prime}[t_1, \ldots, t_n, x, x^{-1}]$ is a unique factorization domain and $x \in A$ is prime ($A \slash x \cong k^{\prime}[t_1, \ldots, t_n] \slash \alpha [y]$ is by assumption a domain), this follows from Nagata's criterion (see \cite[\href{https://stacks.math.columbia.edu/tag/0AFU}{Lemma 0AFU}]{stacks-project}).
 
 It remains to show that $A$ is $SK_1$-regular. Note that $k^{\prime}_{\ell} =k^{\prime} \langle z_1, \ldots, z_{\ell} \rangle$ is a field extension of $k^{\prime}$, hence it is also a field extension of $k$. Since $SK_1(k^{\prime})\cong 0$ and $K_0(k^{\prime}) \cong \mathbb{Z}$ hold, Theorem~\ref{thm:SK1-contractible} is applicable with $\beta=-\alpha$. It implies that
 \[
 SK_1\bigl(k^{\prime} \langle z_1, \ldots, z_{\ell} \rangle [t_1, \ldots, t_n,x,y] \slash (xy-\alpha) \bigr) \cong 0
 \]
 holds for all $\ell \geq 0$. Thus the $k^{\prime}$-algebra $A$ satisfies the strong $SK_1$-extension property over $k^{\prime}$, see Definition~\ref{dfn:strong_extension_property}. Since $SK_1$ satisfies the Horrocks principle $(\mathrm{H})$ (by the Fundamental Theorem and Lemma~\ref{lemma:pullback_version_of_H}), the strong $SK_1$-extension property implies $SK_1$-regularity, see Corollary~\ref{cor:strong_F_extension_implies_F-regular}.
 \end{proof}
 
 Next we can prove Theorem~\ref{thm:K_theory_of_simple_Hermite_conjecture_rings}, which we again restate for convenience.
 
 \begin{cit}[Theorem~\ref{thm:K_theory_of_simple_Hermite_conjecture_rings}]
 If $R$ is a regular ring and $n \geq 2$, then the ring $R[x_i,y_i] \slash ( \sum_{i=1}^n x_i y_i )$ is $K_1$-regular. In particular, the induced morphism
 \[
 K_1(R) \rightarrow K_1\bigl( R[x_i,y_i] \slash ( \textstyle\sum_{i=1}^n x_i y_i ) \bigr)
 \]
 is an isomorphism.
 \end{cit}
 
 \begin{proof}[Proof of Theorem~\ref{thm:K_theory_of_simple_Hermite_conjecture_rings}]
 Since the polynomial $\sum_{i=1}^n x_i y_i$ is homogeneous, the ring $R[x_i,y_i] \slash (\sum_{i=1}^n x_i y_i)$ is graded, with degree zero part $R$. The second claim therefore follows from the first and the Swan--Weibel homotopy trick \cite[Appendix to \S V.3]{LAM}.
 
 Since $K_1(A)\cong SK_1(A) \oplus A^{\times}$, we can treat the two functors $\mathrm{SK}_1$ and $(-)^{\times}=\mathbb{G}_m$ separately. From the fact that $\sum_{i=1}^n x_i y_i$ is homogeneous, it follows that the ring $\mathbb{Z}[x_i,y_i]\slash (\sum_{i=1}^n x_i y_i)$ is a graded domain with degree zero part $\mathbb{Z}$. It is therefore $\mathbb{Z}$-torsion free, hence flat over $\mathbb{Z}$. Thus the two functors
 \[
 F,F^{\prime} \colon \CAlg_R \rightarrow \Ab
 \]
 given by  
 \[
 F(A) \defl SK_1 \bigl( A \ten{\mathbb{Z}} \mathbb{Z}[x,y] \slash (\textstyle \sum_{i=1}^n x_i y_i )\bigr)
 \]
 and
 \[
 F^{\prime}(A) \defl \mathbb{G}_m\bigr(A \ten{\mathbb{Z}} \mathbb{Z}[x,y] \slash (\textstyle \sum_{i=1}^n x_i y_i ) \bigl)
 \]
 both satisfy weak analytic excision. The claim amounts to showing that $R$ is both $F$-regular and $F^{\prime}$-regular. Since both these functors satisfy the Quillen principle $(\mathrm{Q})$ (see Theorem~\ref{thm:weak_excision_implies_QR}), it suffices to check this for the \emph{local} rings of $R$. Translating back, we need to show that $R_{\mathfrak{m}}[x_i,y_i] \slash ( \sum_{i=1}^n x_i y_i )$ is $SK_1$-regular and $\mathbb{G}_m$-regular for all regular local rings $R_{\mathfrak{m}}$ of $R$ (localized at a maximal ideal $\mathfrak{m} \subseteq R$). The second of these claims is straightforward: From Lemma~\ref{lemma:ufd_implies_nonzerodivisor} it follows that 
 \[
 R_{\mathfrak{m}}[x_i,y_i] \slash ( \textstyle\sum_{i=1}^n x_i y_i )
 \]
 is a domain, hence in particular reduced and therefore $\mathbb{G}_m$-regular.
 
 Since the local rings $R_{\mathfrak{m}}$ of $R$ are unique factorization domains (and satisfy $SK_1(R_{\mathfrak{m}})=0$, $K_0(R_{\mathfrak{m}})=\mathbb{Z}$), we can apply Theorem~\ref{thm:SK1-contractible} to $\beta=xy \in R[x,y]$. If we can establish the premise that $R_{\mathfrak{m}} \langle z_1, \ldots, z_{\ell} \rangle [x,y] \slash (xy)$ is reduced and $SK_1$-regular for all $\ell \geq 1$, then the theorem implies that $R_{\mathfrak{m}} [x_i,y_i] \slash (\sum_{i=1}^n x_i y_i)$ has the strong $SK_1$-extension property over $R_{\mathfrak{m}}$ (see Definition~\ref{dfn:strong_extension_property}). As in the proof above, the fact that $SK_1$ satisfies the Horrocks prinicple $(\mathrm{H})$ then implies that the ring $R_{\mathfrak{m}} [x_i,y_i] \slash (\sum_{i=1}^n x_i y_i)$ is $SK_1$-regular (see Corollary~\ref{cor:strong_F_extension_implies_F-regular}).
 
 We have thus reduced the problem to showing that $R_{\ell}[x,y] \slash (xy)$ is reduced and $SK_1$-regular, where $R_{\ell}=R \langle z_1, \ldots, z_\ell \rangle$, for any regular \emph{local} ring $R$. The second fact follows from Lemma~\ref{lemma:excision_for_coordinate_cross}, applied to the functor 
 \[
 G(A) \defl \mathrm{SK_1}(A \ten{\mathbb{Z}} \mathbb{Z}[t_1, \ldots, t_k]).
 \]
  Indeed, since $SK_1$ satisfies weak Milnor excision and $\mathbb{Z} \rightarrow \mathbb{Z}[t_1, \ldots, t_k]$ is flat, so does $G$. The isomorphisms
 \[
 SK_1(R_{\ell}[x][t_1, \ldots, t_k]) \cong SK_1(R_{\ell}) \cong 0
 \]
 (coming from the regularity of $R_{\ell}$ and from Lemma~\ref{lemma:low_dimensional_K_for_R_l}) imply that the conditions of Lemma~\ref{lemma:excision_for_coordinate_cross} hold, so $F\bigl(R_{\ell}[x,y]\slash (xy) \bigr) \cong 0$. This shows that $R_{\ell}[x,y] \slash (xy)$ is indeed $SK_1$-regular.
 
 To see that $R_{\ell}[x,y] \slash (xy)$ is reduced, note that $(xy)$ is contained in the prime ideals $(x)$ and $(y)$. So if $f^k \in (xy)$, then $f \in (x)$ and $f \in (y)$. Since $R_{\ell}[x,y]$ is a unique factorization domain and $x$, $y$, are not associates, it follows that $f\in (xy)$.
 \end{proof}
 
 Recall from \cite{SCHAEPPI_HERMITE} that the rings $\mathbb{Z}[x_i,y_i] \slash (\sum_{i=1}^n x_i y_i)$ are part of a ``test set'' $\ca{H}$ for the Hermite ring conjecture, meaning that the Hermite ring conjecture holds if and only if it holds for all $R \in \ca{H}$. The rings in $\ca{H}$ have some nice additional properties, for example, they are all unique factorization domains, see \cite{SCHAEPPI_HERMITE} for more details.
 
 From the $K_1$-regularity of $\mathbb{Z}[x_i,y_i] \slash (\sum_{i=1}^n x_i y_i)$ (see Theorem~\ref{thm:K_theory_of_simple_Hermite_conjecture_rings}) it follows that $\mathbb{Z}[x_i,y_i] \slash (\sum_{i=1}^n x_i y_i)$ is also $K_0$-regular (see \cite[Corollary~2.1]{VORST_POLYNOMIAL}), hence
 \[
 \mathbb{Z} \cong K_0(\mathbb{Z}) \rightarrow K_0 \bigl( \mathbb{Z}[x_i,y_i] \slash (\textstyle\sum_{i=1}^n x_i y_i) \bigr)
 \]
 is an isomorphism by the Swan--Weibel homotopy trick. Thus the Hermite ring conjecture would imply that all finitely generated projective modules over $\mathbb{Z}[x_i,y_i] \slash (\sum_{i=1}^n x_i y_i)$ are free. More can be said: since the Hermite ring conjecture implies its own $K_1$-analogue (see Proposition~\ref{prop:Hermite_ring_conjecture_K1_analogue} below), all unimodular rows over $\mathbb{Z}[x_i,y_i] \slash (\sum_{i=1}^n x_i y_i)$ are \emph{elementary} completable if the Hermite ring conjecture holds (see Corollary~\ref{cor:Hermite_and_A1_contractible_implies_elementary} below). Thus a natural question is whether an isomorphism
 \[
 W_r \bigl( \mathbb{Z}[x_i,y_i] \slash (\textstyle\sum_{i=1}^n x_i y_i) \bigr) \cong \ast 
 \]
 can be established by different means. We first provide some details showing that these two claims are indeed consequences of the Hermite ring conjecture.
 
 \begin{prop}\label{prop:Hermite_ring_conjecture_K1_analogue}
 Assume that the Hermite ring conjecture holds. Then if $\sigma(t) \in \mathrm{SL}_r(A[t])$ is stably elementary and $r \geq 3$, there exists an elementary matrix $\varepsilon(t) \in \mathrm{E}_r(A[t])$ such that $\sigma(t)=\sigma(0) \varepsilon(t)$ holds.
 \end{prop}
 
 \begin{proof}
 This is a consequence of Milnor patching for projective modules. Since $\mathrm{SL}_r(-)$ commutes with filtered colimits, we can reduce to the case where $A$ is finitely generated over $\mathbb{Z}$, so we can assume that there exists a surjective ring homomorphism $\pi \colon \mathbb{Z}[t_1, \ldots, t_n] \rightarrow A$. Let the left square below
 \[
 \vcenter{\xymatrix{
 B \ar[r] \ar[d] & \mathbb{Z}[t_1, \ldots, t_n] \ar[d]^{\pi} \\ \mathbb{Z}[t_1, \ldots, t_n] \ar[r]_-{\pi} & A
 }}\quad \quad
  \vcenter{\xymatrix{
 B[t] \ar[r] \ar[d] & \mathbb{Z}[t_1, \ldots, t_n][t] \ar[d]^{\pi[t]} \\ \mathbb{Z}[t_1, \ldots, t_n][t] \ar[r]_-{\pi[t]} & A[t]
 }}
 \]
 be a pullback diagram in the category of commutative rings. Since the left square is by construction a Milnor square, so is the right square obtained by applying the exact functor $\mathbb{Z}[t] \ten{\mathbb{Z}} -$. Thus we get an equivalence
 \[
 \Proj(B[t]) \rightarrow \Proj(\mathbb{Z}[t_i][t]) \pb{\Proj(A[t])} \Proj(\mathbb{Z}[t_i][t])
 \]
 of categories. Under this equivalence, $\sigma$ corresponds to a stably free module $P$ over $B[t]$ and we claim that the Hermite ring conjecture, applied to $P$, yields the desired elementary matrix (after translating back through the above equivalence of categories).
 
 First we note that the matrix $\sigma^{\prime}(t)=\bigl(\sigma(0)\bigr)^{-1} \sigma(t) $ satisfies $\sigma^{\prime}(0)=\id$, so it suffices to check that any $\sigma(t) \in \mathrm{SL}_r(A[t])$ with $\sigma(0)=\id$ is elementary. Let $P$ be the stably free $B[t]$-module corresponding to $\sigma$ under the above equivalence of categories. The Hermite ring conjecture implies that $P$ is extended from $B$. Since $P \slash tP$ corresponds to $\sigma(0)=\id \in \mathrm{SL}_r(B)$, it follows that $P$ is free. Translating this fact back, we find that there exist matrices $\tau_0, \tau_1 \in \mathrm{GL}_r(\mathbb{Z}[t_1, \ldots, t_n][t])$ such that the left square
 \[
\vcenter{ \xymatrix{A[t]^r \ar[d]_{\pi(\tau_0)} \ar[r]^-{\id} & A[t]^r \ar[d]^{\pi(\tau_1)} \\ A[t]^r \ar[r]_-{\sigma(t)} & A[t]^r}}
\quad \quad
\vcenter{\xymatrix{A[t]^r \ar[d]_{\pi \Bigl( \begin{smallmatrix} \mathrm{det} \tau_0^{-1} &0 \\ 0 & \id_{r-1} \end{smallmatrix} \Bigr) } \ar[r]^-{\id}  & A[t]^r \ar[d]^{\pi \Bigl( \begin{smallmatrix} \mathrm{det} \tau_1^{-1} &0 \\ 0 & \id_{r-1} \end{smallmatrix} \Bigr) } \\ A[t]^r \ar[r]_-{\id} & A[t]^r }}
 \]
 commutes. Since $\mathrm{det}\bigl(\sigma(t)\bigr)=1$, it follows that $\pi \bigl( \mathrm{det}(\tau_0) \bigr)=\pi \bigl( \mathrm{det}(\tau_1) \bigr)$. Thus the right diagram above also commutes. By stacking the right square on top of the left square, we find that we can assume that the matrices $\tau_0$ and $\tau_1$ both have determinant $1$. Thus we have
 \[
 \sigma(t)=\pi(\tau_1 \tau_0^{-1})
 \]
 with $\tau_1 \tau_0^{-1} \in \mathrm{SL}_r(\mathbb{Z}[t_1, \ldots, t_n][t])$. For all $r \geq 3$, $\mathbb{Z}$ is $SK_{1,r}$-regular by Proposition~\ref{prop:SK_1r_regularity_for_dvr_and_unramified}, so $\tau_1 \tau_0^{-1}$ is elementary since $\mathbb{Z}$ is a euclidean domain.
 \end{proof}
 
 \begin{cor}\label{cor:Hermite_and_A1_contractible_implies_elementary}
 Assume that the Hermite ring conjecture holds. If $A$ is naively $\mathbb{A}^{1}$-contractible over $R$ and $SK_{1,r}(R)=0$ for all $r \geq 3$, then all stably elementary matrices $\sigma \in \mathrm{SL}_r(A)$ are elementary.
 \end{cor}
 
 \begin{proof}
 Let $\eta \colon R \rightarrow A$ be the unit and let $\varepsilon \colon A \rightarrow R$ be an augmentation such that there exists a homomorphism $h \colon A \rightarrow A[t]$ with $\mathrm{ev}_1 h=\id$ and $\mathrm{ev}_0 h=\eta \varepsilon$.
 
 Since $\mathrm{ev}_0 h$ factors through $R$, the assumptions imply that $(\mathrm{ev}_0)_{\ast} h_{\ast}\sigma$ is elementary. From the fact that $\sigma$ is stably elementary, it follows that $h_{\ast} \sigma \in \mathrm{SL}_r(A[t])$ is stably elementary. Proposition~\ref{prop:Hermite_ring_conjecture_K1_analogue} therefore implies that $h_{\ast} \sigma$ is equal to $(\mathrm{ev}_0)_{\ast} h_{\ast}\sigma \cdot \varepsilon$ for some $\varepsilon \in \mathrm{E}_r(A[t])$, so $h_{\ast} \sigma$ is elementary. The claim now follows from the fact that $\sigma=(\mathrm{ev}_1)_{\ast} h_{\ast} \sigma$.
 \end{proof}
 
 One way to approach the problem of computing $W_r\bigl(\mathbb{Z}[x_i,y_i]\slash (\sum_{i=1}^n x_i y_i)\bigr)$ would be to adapt the inductive proof of Theorem~\ref{thm:SK1-contractible} with $W_r$ in the place of $SK_1$. Both Theorem~\ref{thm:henselian_implies_trivial_kernel_for_overring} and its Corollary~\ref{cor:henselian_implies_trivial_kernel_for_overring} are applicable in this case. Unfortunately, there is no clear candidate for the analogue of the functor $K_2$. The following proposition sheds some light on what the ``glueing data'' looks like for the functors $W_r$.
 
 \begin{prop}\label{prop:elementary_completable_implies_stably_elementary_patching}
 Let $\varphi \colon A \rightarrow B$ by an analytic isomorphism along $S$ and let $u=(u_1, \ldots, u_{r+1}) \in \mathrm{Um}_{r+1} (A)$ be a unimodular row. Assume that $u$ is elementary completable over $B$ and over $A_S$. Let $P \defl \mathrm{ker}(u \colon A^{r+1} \rightarrow A)$ be the stably free projective module associated to $u$. Then under the equivalence
 \[
 \Proj(A)  \simeq \Proj(B) \pb{\Proj(B_S)} \Proj(A_S)
 \]
 the module $P$ corresponds to a triple $(B^r,A_S^r, \sigma)$ where $\sigma \in \mathrm{SL}_r(B_S)$ is $1$-stably elementary.
 \end{prop}
 
 \begin{proof}
 Pick $\varepsilon_0 \in \mathrm{E}_{r+1}(B)$ and $\varepsilon_1 \in \mathrm{E}_{r+1}(A_S)$ with $e_1 \cdot \varepsilon_0=\varphi(u)$ and $e_1 \cdot \varepsilon_2=\lambda_S(u)$, where $e_1=(1, 0,\ldots, 0)$. These matrices induce isomorphisms $B \ten{A} P \cong B^r$ and $P_S \cong A_S^r$ respectively. The resulting commutative diagram
\[
\xymatrix@!R=15pt@!C=5pt{ 
& B_S^r \ar[rrr]^-{\bigl( \begin{smallmatrix} 0 \\ \id_{r} \end{smallmatrix} \bigr)} &&& B_S^{r+1} \ar[rrr]^-{e_1} &&& B_S \ar@{=}[ddd] \\
B_S^r \ar[ru]^{\sigma} \ar[rrr]^(0.6){\bigl( \begin{smallmatrix} 0 \\ \id_{r} \end{smallmatrix} \bigr)}  &&& B_S^{r+1} \ar[ru]^{\varepsilon} \ar[rrr]^-{e_1} &&& B_S \ar@{=}[ru] \\ 
\\ 
&B_S \ten{A_S} P\ar[rrr] |!{[rrd];[rruu]} \hole \ar[uuu]_(0.4){\cong} |!{[uul];[uurr]}\hole &&& B_S^{r+1} \ar[uuu]_(0.4){\varphi_S(\varepsilon_1)} |!{[uul];[uurr]} \hole \ar[rrr]^(0.4){\varphi_S \lambda_S(u)} |!{[rrd];[rruu]} \hole &&& B_S \\
(B \ten{A} P)_S  \ar[rrr] \ar[uuu]^(0.6){\cong} \ar[ru]^{\cong}_{\kappa} &&& B_S^{r+1} \ar@{=}[ru] \ar[uuu]^(0.6){\lambda_{\varphi(S)}(\varepsilon_0)} \ar[rrr]_-{\lambda_{\varphi(S)} \varphi(u)} &&& B_S \ar@{=}[uuu] \ar@{=}[ru]
 }
\] 
 where $\varepsilon=\varphi_S(\varepsilon_1) \circ \lambda_{\varphi(S)}(\varepsilon_0)^{-1}$ and $\sigma$ is the unique homomorphism induced on kernels shows that the image $(B \ten{A} P, P_S, \kappa)$ of $P$ is isomorphic to $(B^r,A_S^r,\sigma)$. Moreover, since the top diagram commutes we have $\varepsilon= \Bigl( \begin{smallmatrix} 1 & 0 \\ \ast & \sigma \end{smallmatrix} \Bigr)$, so $\Bigl( \begin{smallmatrix} 1 & 0 \\ \ast & \sigma \end{smallmatrix} \Bigr)$ is elementary. By subtracting suitable multiples of the first row from the rows below, it follows that $\Bigl( \begin{smallmatrix} 1 & 0 \\ 0 & \sigma \end{smallmatrix} \Bigr)$ is elementary, so $\sigma$ is indeed $1$-stably elementary.
 \end{proof}
 
 The above proposition shows that the study of unimodular rows $u$ over rings of the form $A[x,y] \slash (xy - \alpha)$ naturally leads to the study of 1-stably elementary matrices over $A^h_{(\alpha)} [\frac{1}{\alpha}][x,x^{-1}]$. Such a matrix arises whenever the uniomdular row $u$ is elementary completable over both the localization $A_{\alpha}[x,x^{-1}]$ and over $B \defl A^h_{(\alpha)}[x,y] \slash (xy - \alpha)$. However, even if the resulting $1$-stably elementary matrix can be factored into a matrix over $A_{\alpha}[x,x^{-1}]$ and one over $B$, it does not follow that the original unimodular row $u$ is \emph{elementary} completable, merely that it can be completed to a matrix in $\mathrm{GL}_r$. This is an obstacle to studying unimodular rows over rings such as $\mathbb{Z}[x_i,y_i] \slash (\sum_{i=1}^n x_i y_i)$ by induction on $n$.
 
 There is a further problem in this context. The simplest case where a $1$-stably elementary matrix over $B_{\alpha}$ can be factored into a matrix over $A_{\alpha}[x,x^{-1}]$ and one over $B$ is if the matrix itself is already elementary and the size $r$ of the matrix is $\geq 3$. This is a consequence of Vorst's lemma, see Part~(1) of Lemma~\ref{lemma:Vorsts_lemma}. This makes it possible to show in some specific cases that unimodular rows of length $\geq 4$ are completable. The aim of the next section is to provide an analogue of Vorst's lemma in the remaining case of $1$-stably elementary $2 \times 2$-matrices. This can be used to study the completability of unimodular rows of length $3$ in some examples.
\section{Pseudoelementary matrices}\label{section:pseudoelementary}

 It is a well-known fact that there are ``not enough'' elementary $2 \times 2$-matrices, meaning that many results involving elementary matrices require that the size $r$ of the matrix in question satisfies $r \geq 3$. In this section, we introduce the class of \emph{pseudoelementary} $2 \times 2$-matrices and we show that some of these problematic assertions become true if we replace the elementary $2\times 2$-matrices by the pseudoelementary ones.
 
 \begin{dfn}\label{dfn:pseudoelementary}
 Let $A$ be a commutative ring and let $\sigma \in \mathrm{SL}_2(A)$. The matrix $\sigma$ is called \emph{pseudoelementary} if there exists a finite set $T$, a ring homomorphism
 \[ 
 \varphi \colon \mathbb{Z}[T] \rightarrow A \smash{\rlap{,}}
 \]
 and a matrix $\sigma^{\prime} \in \mathrm{SL}_2(\mathbb{Z}[T])$ such that $\sigma=\varphi(\sigma^{\prime})$ holds. By Lemma~\ref{lemma:basic_properties_pseudoelementary} below, these form a subgroup of $\mathrm{SL}_2(A)$, which we denote by $\widetilde{\mathrm{E}}_2(A)$. We write $\widetilde{SK}_{1,2}(A)$ for the pointed set $\mathrm{SL}_2(A) \slash \widetilde{\mathrm{E}}_2(A)$ of left cosets of $\widetilde{\mathrm{E}}_2(A)$.
  \end{dfn}
  
  We record some immediate consequences of this definition.
  
  \begin{lemma}\label{lemma:basic_properties_pseudoelementary}
  The pseudoelementary matrices have the following basic properties:
  \begin{enumerate}
  \item[(1)] The subsets $\widetilde{\mathrm{E}}_2(A) \subseteq \mathrm{SL}_2(A)$ form a subgroup functor of $\mathrm{SL}_2$;
  \item[(2)] For any (possibly infinite) set $T$, we have $\mathrm{SL}_2(\mathbb{Z}[T])=\widetilde{\mathrm{E}}_2(\mathbb{Z}[T])$;
  \item[(3)] If $\pi \colon A \rightarrow A^{\prime}$ is surjective, then $\pi_{\ast} \colon \widetilde{\mathrm{E}}_2(A) \rightarrow \widetilde{\mathrm{E}}_2(A^{\prime})$ is surjective;
  \item[(4)] All elementary matrices are pseudoelementary, that is, $\mathrm{E}_2(A) \subseteq \widetilde{\mathrm{E}}_2(A)$;
  \item[(5)] The functor $\widetilde{SK}_{1,2} \colon \CRing \rightarrow \Set_{\ast}$ is finitary.
  \end{enumerate}
  \end{lemma}
  
  \begin{proof}
  All these claims follow readily from the definitions. Detailed arguments can be found below.
  
  \textbf{Claim~(1).} If $\sigma \in \mathrm{SL}_2(A)$ lies in the image of the map induced by $\varphi \colon \mathbb{Z}[T] \rightarrow A$ and $\sigma^{\prime} \in \mathrm{SL}_2(A)$ lies in the image of the map induced by $\varphi^{\prime} \colon \mathbb{Z}[T^{\prime}] \rightarrow A$, then they both lie in the image of the map induced by
  \[
  \mu \circ  \varphi \ten{\mathbb{Z}} \varphi^{\prime} \colon \mathbb{Z}[T] \ten{\mathbb{Z}} \mathbb{Z}[T^{\prime}] \rightarrow A
  \]
  on $\mathrm{SL}_2$, where $\mu$ denotes the multiplication of $A$. Since the domain of this homomorphism is again a polynomial ring over $\mathbb{Z}$, we find that $\sigma (\sigma^{\prime})^{-1}$ is also pseudoelementary.
  
  If $\alpha \colon A \rightarrow B$ is any homomorphism, then $\alpha(\sigma)=\alpha \circ \varphi(\sigma_0)$ for some $\sigma_0 \in \mathrm{SL}_2(\mathbb{Z}[T])$, so $\alpha(\sigma)$ is pseudoelementary. Thus
  \[
  \alpha_{\ast} \colon \mathrm{SL}_2(A) \rightarrow \mathrm{SL}_2(B)
  \]
  sends $\widetilde{\mathrm{E}}_2(A)$ to $\widetilde{\mathrm{E}}_2(B)$.
  
  \textbf{Claim~(2).} It is immediate from the definition that $\mathrm{SL}_2(\mathbb{Z}[T])=\widetilde{\mathrm{E}}_2(\mathbb{Z}[T])$ for $T$ finite (we can take $\varphi=\id$ in the definition). The claim for infinite sets $T$ follows from the fact that $\mathrm{SL}_2(-)$ is a finitary functor.
  
  \textbf{Claim~(3).} Let $\sigma \in \widetilde{\mathrm{E}}_2(A^{\prime})$ be the image of $\sigma_0 \in \mathrm{SL}_2(\mathbb{Z}[T])$ under the map induced by $\varphi^{\prime} \colon \mathbb{Z}[T] \rightarrow A^{\prime}$. By picking elements $a_t \in A$ for all $t \in T$ with the property that $\pi(a_t)=\varphi^{\prime}(t)$, we obtain a homomorphism $\varphi \colon \mathbb{Z}[T] \rightarrow A$ sending $t$ to $a_t$, so $\pi \circ \varphi=\varphi^{\prime}$ by construction. Thus $\sigma=\pi\bigl(\varphi(\sigma_0)\bigr)$ and $\varphi(\sigma_0)$ lies by definition in $\widetilde{\mathrm{E}}_2(A)$.
  
  \textbf{Claim~(4).} For any commutative ring $A$, we can find a surjection $\pi \colon \mathbb{Z}[T] \rightarrow A$ letting $T$ be a generating set of $A$ (for example, we can take $T=A$). Thus the induced map $\pi_{\ast} \colon \mathrm{E}_2(\mathbb{Z}[T]) \rightarrow \mathrm{E}_2(A)$ on elementary matrices is surjective. By Claim~(2), all matrices in the domain are pseudoelementary. Since these are preserved by the induced map on $\mathrm{SL}_2$ (see Claim~(1)), all matrices in the domain are sent to pseudoelementary matrices over $A$. This implies that all elementary matrices are pseudoelementary.
  
 \textbf{Claim~(5).} Since $\mathrm{SL}_2(-)$ is finitary, it suffices to check that $\widetilde{\mathrm{E}}_2(-)$ is finitary to show that $\widetilde{SK}_{1,2}(A)=\mathrm{SL}_2(A) \slash \widetilde{\mathrm{E}}_2(A)$ is finitary.
 
 If $A=\mathrm{colim} A_i$ is a filtered colimit, we need to check that each $\sigma \in \widetilde{\mathrm{E}}_2(A)$ lies in the image of some $\widetilde{\mathrm{E}}_2(A_i)$ and that whenever $\sigma_1 \in \widetilde{\mathrm{E}}_2(A_{i_1})$ and $\sigma_2 \in \widetilde{\mathrm{E}}_2(A_{i_2})$ have the same image in $\widetilde{\mathrm{E}}_2(A)$, then they already have the same image in some $\widetilde{\mathrm{E}}_2(A_{i_3})$ for suitable morphisms $i_1 \rightarrow i_3$  and $i_2 \rightarrow i_3$. 
 
 The second claim follows from the corresponding fact for the functor $\mathrm{SL}_2(-)$ since $\widetilde{\mathrm{E}}_2(-) \subseteq \mathrm{SL}_2(-)$. To see the first claim, we note that any homomorphism $\varphi \colon \mathbb{Z}[T] \rightarrow A$ with $T$ finite factors through some $A_i \rightarrow A$ (since the hom-functor $\CRing(\mathbb{Z}[T],-)$ preserves filtered colimits), from which the claim follows immediately.
  \end{proof}
  
  In order to deduce some non-trivial facts about pseudoelementary matrices, we frequently think of $\sigma \in \mathrm{SL}_2(A)$ as a patching datum. If the corresponding projective module turns out to be free, on can in favourable cases conclude that the matrix $\sigma$ is pseudoelementary. We refer to this argument as the \emph{free patching lemma}.
  
 \begin{lemma}[free patching]\label{lemma:free_patching}
 Suppose the diagram
 \[
 \xymatrix{A \ar[r]^{\psi^{\prime}} \ar[d]_{\varphi^{\prime}} & C \ar[d]^{\varphi} \\ B \ar[r]_-{\psi} & D }
 \]
 is either a Milnor square or an analytic patching diagram, so that there exists a canonical equivalence
 \[
 \Proj(A) \simeq \Proj(B) \pb{\Proj(D)} \Proj(C)
 \]
 of categories.
 
 Let $\sigma \in \mathrm{SL}_2(D)$ and let $P$ be the projective $A$-module corresponding to the patching datum $(B^2,C^2, \sigma \colon D^2 \rightarrow D^2)$ under the above equivalence. If $P$ is free, then there exists a factorization
 \[
\sigma = \varphi(\sigma_1) \psi(\sigma_0)
 \]
 for some $\sigma_0 \in \mathrm{SL}_2(B)$ and $\sigma_1 \in \mathrm{SL}_2(C)$. In particular, if the equalities
 \[
\mathrm{SL}_2(B)=\widetilde{\mathrm{E}}_2(B) \quad \text{and} \quad \mathrm{SL}_2(C)=\widetilde{\mathrm{E}}_2(C)
 \]
 hold, then $\sigma$ is pseudoelementary.
 
 If the diagram is a Milnor square, then the converse holds: whenever $\sigma \in \mathrm{SL}_2(D)$ is pseudoelementary, then the associated projective module is free.
 \end{lemma}
 
 \begin{proof}
 The freeness of $P$ amounts to the existence of $\tau_0 \in \mathrm{GL}_2(B)$ and $\tau_1 \in \mathrm{GL}_2(C)$ such that the left square below
 \[
 \vcenter{ \xymatrix{D^2 \ar[r]^-{\id} \ar[d]_{\psi(\tau_0)} & D^2 \ar[d]^{\varphi(\tau_1)} \\ D^2 \ar[r]_-{\sigma} & D^2} }
  \quad \quad
  \vcenter{ \xymatrix{D^2 \ar[r]^-{\id} \ar[d]_{\psi \left(\begin{smallmatrix} \mathrm{det}(\tau_0)^{-1} & 0 \\ 0 & 1 \end{smallmatrix} \right) } & D^2 \ar[d]^{\varphi \left(\begin{smallmatrix} \mathrm{det}(\tau_1)^{-1} & 0 \\ 0 & 1 \end{smallmatrix} \right)} \\ D^2 \ar[r]_-{\id} & D^2} }
 \]
 is commutative. Since $\sigma$ lies in $\mathrm{SL}_2(D)$, it follows that $\mathrm{det}\bigl(\psi(\tau_0)\bigr)=\mathrm{det}\bigl(\varphi(\tau_1)\bigr)$, so the right diagram above commutes, too. If we stack the right diagram on top of the left, we find that there exist matrices $\tau_0$ and $\tau_1$ as above which furthermore have determinant equal to $1$. This proves the first claim of the lemma. Under the additional assumption that $\tau_0$ and $\tau_1$ are pseudoelementary, we can conclude that $\sigma$ is pseudoelementary, too.
 
 Finally, if the diagram is a Milnor square and $\sigma$ is pseudoelementary, we need to check that the projective module $P$ obtained by patching free modules along $\sigma$ is free. In this case, one of the maps $\varphi$ or $\psi$ is surjective. This implies that $\sigma=\varphi(\sigma_0)$ (respectively $\sigma=\psi(\sigma_0)$) holds by Part~(3) of Lemma~\ref{lemma:basic_properties_pseudoelementary}, which implies that $(B^2,C^2,\sigma)$ is isomorphic to $(B^2,C^2,\id)$.
 \end{proof}
 
 To make use of this lemma, we also need the following naturality property of patching data.
 
 \begin{lemma}\label{lemma:naturality_of_patching_data}
Let
\[
\xymatrix{ A \ar[rr] \ar[dd] \ar[rd] && C \ar[rd] \ar[dd]|!{[ld];[rd]} \hole  \\ & A^{\prime} \ar[rr] \ar[dd] && C^{\prime}  \ar[dd] \\ B \ar[rd] \ar[rr]|!{[ru];[rd]} \hole && D \ar[rd]^(0.4){\varphi} \\ & B^{\prime} \ar[rr] && D^{\prime}}
\]
be a commutative diagram in the category of commutative rings and assume that the front face and the back face are Milnor squares or analytic patching diagrams. Let $\sigma \in \mathrm{SL}_2(D)$ and let $P \in \Proj(A)$ be the projective module corresponding to $(B^2,C^2,\sigma)$. Then $A^{\prime} \ten{A} P$ is isomorphic to the projective module corresponding to the patching datum $\bigl((B^{\prime})^2,(C^{\prime})^2, \varphi(\sigma) \bigr)$. 
 \end{lemma}
 
\begin{proof}
The pseudofunctor $\Proj \colon \CRing \rightarrow \Cat$ sends the above cube to a cube of categories and functors which is 2-commutative, meaning that all the squares in the image of this cube can be canonically filled with natural isomorphisms in such a way that the two resulting natural isomorphisms filling the hexagon
\[
\xymatrix{ & \Proj(A^{\prime}) \ar[r] & \Proj(C^{\prime}) \ar[rd] \\ \Proj(A)  \ar[ru] \ar[rd] & \rtwocell\omit & & \Proj(D^{\prime}) \\ & \Proj(B) \ar[r] & \Proj(D) \ar[ru]_{\varphi_{\ast}}}
\]
coincide.

 It is a formal consequence of this that the square
 \[
 \xymatrix{ \Proj(A) \ar[r]^-{\simeq} \ar[d] & \Proj(B) \pb{\Proj(D)} \Proj(C) \ar[d] \\
 \Proj(A^{\prime}) \ar[r]^-{\simeq}  & \Proj(B^{\prime}) \pb{\Proj(D^{\prime})} \Proj(C^{\prime}) }
 \]
 commutes up to natural isomorphism. Thus the square with the inverse equivalences as the horizontal functors also commutes up to natural isomorphism, from which the claim follows.
\end{proof} 
 
 \begin{rmk}\label{rmk:naturality_of_patching_data_generalized}
 The argument of the above lemma applies more generally to the case of $G$-torsors for flat affine group schemes $G$ with the resolution property. These stacks $BG$ (which send $A$ to the groupoid of $G$-torsors over $A$) also send the front and back faces to 2-categorical pullback diagrams since they are Adams stacks. Moreover, these stacks are pointed by the trivial torsor, so we can associate torsors to patching data as in the case of projective modules. Namely, in the situation of the above lemma we have functions $T \colon G(D) \rightarrow BG(A)$ sending $\sigma$ to the torsor associated to $(\ast, \ast, \sigma)$, where $\ast$ denotes the trivial torsor. Applying the argument of the lemma in this situation shows that the square
 \[
 \xymatrix{G(D) \ar[d]_-{G(\varphi)} \ar[r]^-{T} & BG(A) \ar[d] \\ G(D^{\prime}) \ar[r]_-{T} & BG(A^{\prime}) }
 \]
 commutes up to natural isomorphism. 
 \end{rmk}
 
 Given a commutative ring $A$, we write $\mathbb{Z}[A]$ for the polynomial ring on the underlying set of $A$. We write $\pi \colon \mathbb{Z}[A] \rightarrow A$ for the canonical surjection which sends the variable $a \in \mathbb{Z}[A]$ to the element $a \in A$. Finally, we define the ring $\Omega A$ by the pullback diagram
 \[
 \xymatrix{\Omega A \ar[d]_{\varphi} \ar[r]^-{\psi} & \mathbb{Z}[A] \ar[d]^{\pi} \\ \mathbb{Z}[A] \ar[r]_-{\pi} & A}
 \]
 of commutative rings. Since $\pi$ is surjective, this is a Milnor square. It follows that the square
 \[
 \xymatrix{\Omega A[t] \ar[d]_{\varphi[t]} \ar[r]^-{\psi[t]} & \mathbb{Z}[A][t] \ar[d]^{\pi[t]} \\ \mathbb{Z}[A][t] \ar[r]_-{\pi[t]} & A[t]}
 \]
 is a Milnor square, too.
 
 \begin{prop}\label{prop:local_global_for_pseudoelementary}
 Let $\sigma(t) \in \mathrm{SL}_2(A[t])$ be a matrix with $\sigma(0)=\id$. If the localization $\sigma_{\mathfrak{m}}(t)$ is pseudoelementary for all maximal ideals $\mathfrak{m} \subseteq A$, then $\sigma(t) \in \widetilde{\mathrm{E}}_2(A[t])$. 
 \end{prop}
 
 \begin{proof}
 Let $P \in \Proj(\Omega A[t])$ denote the rank two projective module corresponding to the patching datum $(\mathbb{Z}[A][t]^2, \mathbb{Z}[A][t]^2, \sigma(t))$. Since $\mathrm{SL}_2(\mathbb{Z}[A][t])=\widetilde{\mathrm{E}}_2(\mathbb{Z}[A][t])$, the free patching lemma is applicable, so it only remains to show that $P$ is free.
 
 The fact that $\sigma(0)=\id$ implies that the $\Omega A$-module $P \slash tP \cong \Omega A^2$ is free of rank two (this follows from Lemma~\ref{lemma:naturality_of_patching_data}, applied to the diagram whose vertical morphisms send a polynomial $f(t)$ to $f(0)$). Thus it suffices to show that $P$ is extended from $\Omega A$.
 
 Since $P_2 \colon \CRing \rightarrow \Set$ satisfies the Quillen principle $(\mathrm{Q})$, it suffices to check that $P_{\mathfrak{m}}$ is extended for all maximal ideals $\mathfrak{m} \subseteq \Omega A$. Let $I=\mathrm{ker}(\pi \varphi \colon \Omega A \rightarrow A)$ and let $S \defl \Omega A \setminus \mathfrak{m}$. If $I \not \subseteq \mathfrak{m}$, then the pullback diagram
 \[
 \xymatrix{ \Omega A_{\mathfrak{m}} \ar[r] \ar[d] & \mathbb{Z}[A]_{\psi(S)} \ar[d] \\ \mathbb{Z}[A]_{\varphi(S)} \ar[r] & A_{\pi \varphi(S)} \cong 0 }
 \]
 shows that $\Omega A_{\mathfrak{m}} \cong \mathbb{Z}[A]_{\varphi(S)} \times \mathbb{Z}[A]_{\psi(S)}$. Since this ring is local, it is in particular connected, so one of $\mathbb{Z}[A]_{\varphi(S)}$ or $\mathbb{Z}[A]_{\psi(S)}$ must be the zero ring. Thus $\Omega A_{\mathfrak{m}}$ is a localization of some polynomial ring (usually involving infinitely many variables). Since the functor $P_2 \colon \CRing \rightarrow \Set_{\ast}$ satisfies the Roitman principle $(\mathrm{R})$, it follows that all rank two projective $\Omega A_{\mathfrak{m}}[t]$-modules are extended from the local ring $\Omega A_{\mathfrak{m}}$, hence free. This shows that $P_{\mathfrak{m}}$ is free whenever $I \not \subseteq \mathfrak{m}$.
 
 If $I \subseteq \mathfrak{m}$, then $\mathfrak{m}=(\pi \varphi)^{-1}(\mathfrak{m}_0)$ for some maximal ideal $\mathfrak{m}_0 \subseteq A$. By tensoring the square
 \[
  \xymatrix{\Omega A[t] \ar[d]_{\varphi[t]} \ar[r]^-{\psi[t]} & \mathbb{Z}[A][t] \ar[d]^{\pi[t]} \\ \mathbb{Z}[A][t] \ar[r]_-{\pi[t]} & A[t]}
 \]
 with the localization $\Omega A \rightarrow \Omega A_{\mathfrak{m}}=\Omega A_S$, we obtain a cube as in Lemma~\ref{lemma:naturality_of_patching_data} whose diagonal morphisms are given by the respective localization morphisms. It follows that $P_{\mathfrak{m}}$ is isomorphic to the module associated to the patching datum $\sigma_{\mathfrak{m}_0}(t)$. Since $\sigma_{\mathfrak{m}_0}(t)$ is by assumption pseudoelementary and the square is a Milnor square, the surjectivity of the localization of $\pi[t] \colon \mathbb{Z}[A][t] \rightarrow A[t]$ implies that $P_{\mathfrak{m}}$ is free (this also follows from the final part of the free patching lemma).
 
 We have established that $P_{\mathfrak{m}}$ is free for all $\mathfrak{m} \subseteq \Omega A$ maximal, so $P$ is extended from $\Omega A$, hence free. Thus the free patching lemma implies that $\sigma(t)$ is pseudoelementary.
  \end{proof}
  
 \begin{prop}\label{prop:pseudoelementary_monic_inversion}
 Let $g \in A[t]$ be a monic polynomial and let $\sigma(t) \in \mathrm{SL}_2(A[t])$. If the localization $\sigma(t)_g \in \widetilde{\mathrm{E}}_2(A[t]_g)$ is pseudoelementary, then $\sigma(t)$ is pseudoelementary.
 \end{prop}
 
 \begin{proof}
 We use the notation introduced above Proposition~\ref{prop:local_global_for_pseudoelementary}. Let $h \in \Omega A[t]$ be a monic polynomial with $\pi \varphi(h)=g$. The patching datum $\sigma(t) \in \mathrm{SL}_2(A[t])$ corresponds to a projective module $P \in \Proj(\Omega A[t])$.
 
 If we tensor the Milnor patching diagram
 \[
  \xymatrix{\Omega A[t] \ar[d]_{\varphi[t]} \ar[r]^-{\psi[t]} & \mathbb{Z}[A][t] \ar[d]^{\pi[t]} \\ \mathbb{Z}[A][t] \ar[r]_-{\pi[t]} & A[t]}
 \]
 with $\Omega A[t] \rightarrow \Omega A[t]_h$, we can apply Lemma~\ref{lemma:naturality_of_patching_data} and conclude that $P_h$ corresponds to the patching datum $\sigma(t)_g$. Since $\sigma(t)_g$ is by assumption pseudoelementary and the localized square is still a Milnor square, it follows from the free patching lemma that $P_h$ is free.
 
 Thus the fact that $P_2 \colon \CRing \rightarrow \Set_{\ast}$ satisfies the Horrocks principle $(\mathrm{H})$ implies that $[P] \in P_2(\Omega A[t]) $ is extended. Moreover, since $P_h$ is free, it follows from Proposition~\ref{prop:uniquely_etxtended_after_monic_inversion} that $P$ is extended from the free module, hence free. Thus $\sigma(t)$ is pseudoelementary by the free patching lemma.
 \end{proof}
 
 \begin{cor}\label{cor:SK_12_satisfies_Q_and_H}
 The functor $\widetilde{SK}_{1,2} \colon \CRing \rightarrow \Set_{\ast}$ satisfies the Quillen principle $(\mathrm{Q})$ and the Horrocks principle $(\mathrm{H})$.
 \end{cor}
 
 \begin{proof}
 We first show that the functor satisfies the Quillen principle $(\mathrm{Q})$. Let $\sigma(t) \in \mathrm{SL}_2(A[t])$ be a matrix such that the class of $\sigma_{\mathfrak{m}}(t)$ in $\widetilde{SK}_{1,2}(A_{\mathfrak{m}}[t])$ is extended for all maximal ideals $\mathfrak{m} \subseteq A$. For a fixed maximal ideal we thus have $\sigma_{\mathfrak{m}}(t)=\sigma_{\mathfrak{m}}(0) \varepsilon(t)$ for some $\varepsilon(t) \in \widetilde{\mathrm{E}}_2(A_{\mathfrak{m}}[t])$. It follows that the matrix $\tau(t) \defl \sigma(0)^{-1} \sigma(t)$ satisfies $\tau(0)=\id$ and $\tau_{\mathfrak{m}}(t) \in \widetilde{\mathrm{E}}_2(A_{\mathfrak{m}}[t])$ for all maximal ideals $\mathfrak{m} \subseteq A$. From Proposition~\ref{prop:local_global_for_pseudoelementary} we can conclude that $\tau(t)$ is pseudoelementary. Thus $[\sigma(t)]=[\sigma(0)\tau(t)]=[\sigma(0)]$ is extended from $\widetilde{SK}_{1,2}(A)$, which shows that $\widetilde{SK}_{1,2}$ does indeed satisfy the Quillen principle $(\mathrm{Q})$.
 
  To see that the Horrocks principle $(\mathrm{H})$ holds, let $\sigma(t) \in \mathrm{SL}_2(A[t])$ be a matrix such that the class of $\sigma(t) \in \widetilde{SK}_{1,2}(A \langle t \rangle)$ is extended from $A$. This means that the localization of $\sigma(t)$ is equal to $\tau \varepsilon$ for some $\tau \in \mathrm{SL}_2(A)$ and $\varepsilon \in \widetilde{\mathrm{E}}_2(A \langle t \rangle)$. Thus $\tau^{-1} \sigma(t)$ is sent to the trivial class in $\widetilde{SK}_{1,2}(A \langle t \rangle)$. Since the functor $\widetilde{SK}_{1,2}$ is finitary, it follows that this class is already trivial after localization at a single monic polynomial. Thus Proposition~\ref{prop:pseudoelementary_monic_inversion} is applicable and implies that $\tau^{-1} \sigma(t)$ is pseudoelementary. Therefore the class of $\sigma(t)=\tau\bigl(\tau^{-1} \sigma(t) \bigr)$ is extended from $\widetilde{SK}_{1,2}(A)$, which concludes the proof that $\widetilde{SK}_{1,2}$ satisfies the Horrocks principle $(\mathrm{H})$.
 \end{proof}
 
 In order to show that $\widetilde{SK}_{1,2}$ also satisfies the Roitman principle $(\mathrm{R})$, we need the following lemma.
 
 \begin{lemma}\label{lemma:pseudoelementary_Roitman}
 Let $f \in A$ and let $[\sigma(t)] \in N \widetilde{SK}_{1,2}(A)$. For $n \in \mathbb{N}$, let
 \[
 \mu_{f^n} \colon A_f[t] \rightarrow A_f[t]
 \]
 be the $A_f$-linear isomorphism sending $t$ to $f^n t$. Then there exists a $k \in \mathbb{N}$ and a matrix $\tau(t) \in \mathrm{SL}_2(A[t])$ with $\tau(0)=\id$ such that $[\lambda_f\bigl(\tau(t)\bigr)]=[\mu_{f^k} \sigma(t)]$ holds.
 \end{lemma}
 
 \begin{proof}
 By definition of $N \widetilde{SK}_{1,2}$, we have $\sigma(0) \in \widetilde{\mathrm{E}}_2(A_f)$. Since $\mu_{f^k}\bigl( \sigma(0) \bigr)=\sigma(0)$, the equalities
 \[
 [\mu_{f^k} \sigma(t)]=[\mu_{f^k} \bigl( \sigma(t) \sigma(0)^{-1} \bigr) \sigma(0)]=[\mu_{f^k} \bigl( \sigma(t) \sigma(0)^{-1} \bigr)]
 \]
 show that we can replace $\sigma(t)$ with $\sigma(t) \sigma(0)^{-1}$, which reduces the problem to the case where $\sigma(0)=\id$.
 
 For $k$ large enough, all the coefficients of the polynomials occurring in $\mu_{f^k} \bigl( \sigma(t) \bigr)$ lie in the image of $\lambda_f \colon A \rightarrow A_f$. Since $\sigma(0)=\id$, there exists a $2 \times 2$-matrix $\tau^{\prime}(t)$ with entries in $A[t]$ such that $\lambda_f \bigl( \tau^{\prime}(t) \bigr)=\mu_{f^k} \bigl( \sigma(t) \bigr)$ holds. Moreover, we can choose the matrix in such a way that $\tau^{\prime}(0)=\id$. The matrix $\tau^{\prime}(t)$ need however not be invertible, so we let $g(t) \in A[t]$ be its determinant.
 
 It follows that $g(0)=1$ and $\lambda_f\bigl(g(t)\bigr)=1$. This implies that for all $i >0$, the coefficient of $t^i$ in $g(t)$ is annihilated by sufficiently large powers of $f$. Thus for $\ell \in \mathbb{N}$ large enough, we have $g(f^{\ell} t)=1$. It follows that $\tau(t) \defl \tau^{\prime}(f^{\ell} t)$ lies in $\mathrm{SL}_2(A[t])$. We have $\tau(0)=\id$ and $\lambda_f\bigl(\tau(t)\bigr)=\mu_{f^{k+\ell}} \bigl( \sigma(t) \bigr)$ by construction.
 \end{proof}
 
 \begin{prop}\label{prop:pseudoelementary_Roitman}
 The functor $\widetilde{SK}_{1,2} \colon \CRing \rightarrow \Set_{\ast}$ satisfies the Roitman principle $(\mathrm{R})$. In particular, if $A$ is $\widetilde{SK}_{1,2}$-regular and $S \subseteq A$ is a multiplicative set, then $A_S$ is $\widetilde{SK}_{1,2}$-regular.
 \end{prop}
 
 \begin{proof}
 Recall that $\ca{E}_1^{\widetilde{SK}_{1,2}}$ denotes the category of commutative rings $A$ such that all $[\sigma(t)] \in \widetilde{SK}_{1,2}(A[t])$ are extended from $\widetilde{SK}_{1,2}(A)$. We need to check that this category is closed under the operation $A \mapsto A_S$ for all multiplicative sets $S \subseteq A$.
 
 We first note that $A \in \ca{E}_1^{\widetilde{SK}_{1,2}}$ holds if and only if $N \widetilde{SK}_{1,2}(A) \cong \ast$ holds. If all $[\sigma(t)]$ are extended from $A$, then we clearly have $N \widetilde{SK}_{1,2}(A) \cong \ast$. To see the converse, note that we have a natural transitive group action 
 \[
 \mathrm{SL}_2(A) \times \widetilde{SK}_{1,2}(A) \rightarrow \widetilde{SK}_{1,2}(A)
 \]
  given by left multiplication (see Definition~\ref{dfn:F_contractible_and_transitive_group_action}). The converse implication therefore follows from Lemma~\ref{lemma:transitive_action_implies_extended_if_trivial_NF}.
 
 We have reduced the problem to checking that $N \widetilde{SK}_{1,2}(A_S) \cong \ast$ whenever $N \widetilde{SK}_{1,2}(A) \cong \ast$ holds. This is a direct consequence of Lemma~\ref{lemma:pseudoelementary_Roitman} above. Indeed, the automorphism $\mu_{f^n}$ defined in Lemma~\ref{lemma:pseudoelementary_Roitman} satisfies $\mathrm{ev}_0 \mu_{f^n}=\mathrm{ev}_0$, so it induces an automorphism of $N\widetilde{SK}_{1,2}(A_f)$. The conclusion of the lemma therefore shows that $N \widetilde{SK}_{1,2}(A_f) \cong \ast$ whenever $N \widetilde{SK}_{1,2}(A) \cong \ast$. The case for general multiplicative sets follows from this since $\widetilde{SK}_{1,2}$ commutes with filtered colimits.
 
 This concludes the proof that $\widetilde{SK}_{1,2}$ satisfies the Roitman priniciple $(\mathrm{R})$. The claim about $\widetilde{SK}_{1,2}$-regular rings follows from this by Proposition~\ref{prop:R_several_variables}.
 \end{proof}
 
 We have now established that $\widetilde{SK}_{1,2}$ satisfies the three principles $(\mathrm{Q})$, $(\mathrm{R})$, and $(\mathrm{H})$ discussed in \S \ref{section:excision} and \S \ref{section:monic}. The first two would be a consequence of weak analytic excision, but it is unclear if this holds in full generality. Before discussing various cases where weak excision \emph{does} hold, we record one more general fact about $\widetilde{SK}_{1,2}$.
 
 \begin{prop}\label{prop:pseudoelementary_j_invariant}
 The functor $\widetilde{SK}_{1,2} \colon \CRing \rightarrow \Set_{\ast}$ is $j$-invariant.
 \end{prop}
 
 \begin{proof}
 Let $A$ be a commutative ring and let the ideal $I\subseteq A$ be contained in the Jacobson radical. If $\sigma, \tau \in \mathrm{SL}_2(A)$ are matrices such that $[\bar{\sigma}]=[\bar{\tau}]$ holds in $\widetilde{SK}_{1,2}(A \slash I)$, then $\tau^{-1} \sigma$ is sent to a pseudoelementary matrix in $\widetilde{\mathrm{E}}_2(A \slash I)$. Thus it suffices to show that $\sigma$ is pseudoelementary whenever its image $\bar{\sigma}$ in $\mathrm{SL}_2(A \slash I)$ is pseudoelementary. To see this, pick $\varepsilon \in \widetilde{\mathrm{E}}_2(A)$ with $\bar{\varepsilon}=\bar{\sigma}$. Then $\varepsilon^{-1} \sigma$ is sent to the identity matrix in $\mathrm{SL}_2(A\slash I)$, hence it is elementary by Whitehead's lemma (see Lemma~\ref{lemma:SL_r_j_invariant} for details).
 
 To see surjectivity, pick any lift $\sigma$ of $\bar{\sigma} \in \mathrm{SL}_2(A \slash I)$. Then $d \defl \mathrm{det}(\sigma) \in 1+I$ is a unit, so $\sigma$ lies in $\mathrm{GL}_2(A)$. If we multiply the first row of $\sigma$ with $d^{-1}$, we obtain the desired lift of $\bar{\sigma}$ in $\mathrm{SL}_2(A)$.
 \end{proof}
 
 In order to establish weak excision for pullback square
 \[
 \xymatrix{A \ar[d] \ar[r] & C \ar[d] \\ B \ar[r] & D}
 \]
 in the category of commutative rings, we need to show that any pseudoelementary matrix over $D$ factors into a product of pseudoelementary matrices over $B$ and $C$.
 
 We first show that Vorst's lemma (and Suslin's lemma \cite[Lemma~3.3]{SUSLIN_SPECIAL} it is based upon) holds if $\varphi \colon A \rightarrow B$ is an analytic isomorphism along $f$ and the matrix over $D=B_{\varphi(f)}$ is moreover elementary (instead of merely pseudoelementary).
 
 \begin{lemma}\label{lemma:pseudoelementary_Suslin_lemma}
  Let $A$ be a commutative ring, let $f \in A$, and let $\gamma \in \mathrm{SL}_2(A_f)$. Let $g(t) \in A_f[t]$ be any polynomial and let $\varepsilon(t)=e_{12}\bigl(t g(t) \bigr) \in \mathrm{E}_2(A_f[t])$ or $\varepsilon(t)=e_{21}\bigl(t g(t) \bigr)$. Then there exist a $k \in \mathbb{N}$ and a pseudoelementary matrix $\tau(t) \in \widetilde{\mathrm{E}}_2(A[t])$ such that
  \[
  \lambda_f\bigl( \tau(t) \bigr)=\gamma \varepsilon(f^k t) \gamma^{-1}
  \]
  holds.
 \end{lemma}
 
 \begin{proof}
 The two cases are symmetric, so let $\varepsilon(t)=e_{21}\bigl(t g(t) \bigr)$. Let $\left( \begin{smallmatrix} b \\ d \end{smallmatrix} \right)$ be the second column of $\gamma$. It follows that the second row of $\gamma^{-1}$ is $\left(\begin{smallmatrix} d & -b \end{smallmatrix}\right)$. Multiplying out we find that
 \[
 \gamma \varepsilon(t) \gamma^{-1}= \begin{pmatrix} 1 + tg(t)bd & -tg(t) b^2 \\ tg(t) d^2 & 1- tg(t)bd \end{pmatrix}
 \]
 holds. Now we can choose $k$ large enough so that $b^{\prime}=f^k b$, $d^{\prime}=f^k d$ both lie in $A$, and $g^{\prime}(t)=f^k g(t)$ lies in $A[t]$. Thus the equality
 \[
 \gamma \varepsilon(f^{3k} t) \gamma^{-1}= \begin{pmatrix} 1 + tg^{\prime}(f^{3k} t)b^{\prime}d^{\prime} & -tg^{\prime}(f^{3k}t) (b^{\prime})^2 \\ tg^{\prime}(f^{3k}t) (d^{\prime})^2 & 1- tg^{\prime}(f^{3k}t)b^{\prime}d^{\prime} \end{pmatrix}
 \]
 holds. The matrix on the right above is the (localization of the) image of the matrix
 \[
 \begin{pmatrix}
 1+xyz & -x^2 z \\ y^2 z & 1-xyz 
 \end{pmatrix} \in \mathrm{SL}_2(\mathbb{Z}[x,y,z])
 \]
 under the ring homomorpism $\varphi \colon \mathbb{Z}[x,y,z] \rightarrow A[t]$ given by $x \mapsto b^{\prime}$, $y \mapsto d^{\prime}$, and $z \mapsto tg^{\prime}(f^{3k} t)$. Thus $\gamma \varepsilon(f^{3k} t) \gamma^{-1}$ is indeed the localization of a matrix in $\widetilde{\mathrm{E}}_2(A[t])$, as claimed.
 \end{proof}
 
 With the above adaptation of \cite[Lemma~3.3]{SUSLIN_SPECIAL} to the case of $2 \times 2$-matrices, we obtain the following version of Vorst's lemma.
 
 \begin{lemma}\label{lemma:partial_Vorst_lemma}
 Let $\varphi \colon A \rightarrow B$ be an analytic isomorphism along $f$. Then for every elementary matrix $\varepsilon \in \mathrm{E}_2(B_{\varphi(f)})$, there exist a pseudoelementary matrix $\varepsilon_1 \in \widetilde{\mathrm{E}}_2(B)$ and an elementary matrix $\varepsilon_2 \in \mathrm{E}_2(A_f)$ such that
 \[
 \varepsilon=(\varepsilon_1)_{\varphi(f)} \cdot \varphi_f(\varepsilon_2)
 \]
 holds in $\mathrm{SL}_2(B_{\varphi(f)})$.
 \end{lemma}
 
 \begin{proof}
 This proof follows verbatim the proof of Vorst's lemma (see Lemma~\ref{lemma:Vorsts_lemma}), the only difference is that we appeal to Lemma~\ref{lemma:pseudoelementary_Suslin_lemma} instead of \cite[Lemma~3.3]{SUSLIN_SPECIAL}.
 
 Here are the details. Let $\varepsilon=\prod_{k=1}^m e_{i_k j_k}(c_k)$ with $i_k, j_k \in \{1,2\}$, $i_k \neq j_k$, and $c_k \in B_{\varphi(f)}$. For $1 \leq p \leq m$, let $\sigma_p \defl \prod_{k=1}^{p-1} e_{i_k j_k}(c_k)$. From Lemma~\ref{lemma:pseudoelementary_Suslin_lemma} it follows that there are matrices $\tau_p(t) \in \widetilde{\mathrm{E}}_2(B[t])$ and $s_p \in \mathbb{N}$ such that 
 \[
 \sigma_p e_{i_p j_p}\bigl(\varphi(f)^{s_p} t\bigr) \sigma_p^{-1}=\bigl(\tau_p(t)\bigr)_{\varphi(f)}
 \]
 holds. As in the proof of Lemma~\ref{lemma:Vorsts_lemma}, we can find elements $a_k \in A$, $b_k \in B$, and $n_k \in \mathbb{N}$ such that
 \[
 c_k=\varphi(f)^{s_k} b_k + \varphi(a_k \slash f^{n_k})
 \]
 for all $1 \leq k \leq m$. It follows that
 \[
\varepsilon = \textstyle \prod_{k=m}^1  \sigma_k e_{i_k j_k} \bigl( \varphi(f)^{s_k} b_k\bigr) \sigma_k^{-1} \cdot \textstyle \prod_{k=1}^m e_{i_k j_k}\bigl(\varphi(a_k \slash f^{n_k}) \bigr)
 \]
 holds (this is a Formula valid for any group, see \cite[Formula~(VI.1.15)]{LAM}). If we write $\varepsilon_1 = \prod_{k=m}^1 \tau_k(b_k)$ and $\varepsilon_2=\prod_{k=1}^m e_{i_k j_k}(a_k \slash f^{n_k})$, then $\varepsilon_1$ lies in $\widetilde{\mathrm{E}}_2(B)$, $\varepsilon_2$ lies in $\mathrm{E}_2(A_f)$, and the above equation shows that the desired equality
 \[
 \varepsilon=(\varepsilon_1)_{\varphi(f)} \cdot \varphi_f(\varepsilon_2)
 \]
 holds.
 \end{proof}
 
 The following lemma gives a useful class of rings $A$ with the property that $\mathrm{SL}_2(A)=\mathrm{E}_2(A)$.
 
 \begin{lemma}\label{lemma:SL2=E2_for_A<t>}
 Let $A$ be a principal ideal domain such that $\mathrm{SL}_2(A)=\mathrm{E}_2(A)$. Then $A \langle t \rangle$ is also a principal ideal domain such that
 \[
 \mathrm{SL}_2(A \langle t \rangle) = \mathrm{E}_2(A \langle t \rangle)
 \]
 holds.
 \end{lemma}
 
 \begin{proof}
 This follows from \cite[Lemma~IV.6.1]{LAM}. To see this, let $B \defl A[s]_{1+sA[s]}$. Then $\mathrm{SL}_2(B)=\mathrm{E}_2(B)$ holds by Lemma~\ref{lemma:SL_r_j_invariant}. Moreover, $\mathrm{SL}_2( B \slash s)=\mathrm{E}_2(B \slash s)$ holds by assumption on $A \cong B \slash s$. From \cite[Lemma~IV.6.1]{LAM} it follows $\mathrm{SL}_2(B_s)=\mathrm{E}_2(B_s)$ holds, so the second claim follows from the isomorphism $A \langle t \rangle \cong B_s$ given by $t \mapsto s^{-1}$. That $A \langle t \rangle$ is a principal ideal domain is for example shown in \cite[Corollary~IV.1.3.(3)]{LAM}.
 \end{proof}
 
 With these ingredients in place, we can show that $\widetilde{SK}_{1,2}$ is trivial for some important ``building blocks.'' The vanishing of $\widetilde{SK}_{1,2}$ on this set of rings will be used to establish all of our weak excision results for $\widetilde{SK}_{1,2}$.
 
\begin{prop}\label{prop:pseudoelementary_trivial_for_single_Laurent_variable}
 For each $n \in \mathbb{N}$ we have $\widetilde{SK}_{1,2}(\mathbb{Z}[x,x^{-1},t_1, \ldots, t_n]) \cong \ast$.
\end{prop}

\begin{proof}
 The ring $\mathbb{Z}[x]$ is by definition of pseudoelementary matrices $\widetilde{SK}_{1,2}$-regular. Since $\widetilde{SK}_{1,2}$ satisfies the Roitman principle $(\mathrm{R})$ (see Proposition~\ref{prop:pseudoelementary_Roitman}), it follows that the localization $\mathbb{Z}[x,x^{-1}]$ is also $\widetilde{SK}_{1,2}$-regular. Thus it suffices to check that $\widetilde{SK}_{1,2}(\mathbb{Z}[x,x^{-1}])\cong \ast$.
 
 To see this, consider the analytic patching diagram
 \[
 \xymatrix{\mathbb{Z}[x] \ar[r] \ar[d]_{\varphi} & \mathbb{Z}[x,x^{-1}] \ar[d]^{\varphi_x} \\ \mathbb{Z}[x]_{1+x\mathbb{Z}[x]} \ar[r] & \mathbb{Z} \langle x^{-1} \rangle }
 \]
 and fix $\sigma \in \mathrm{SL}_2(\mathbb{Z}[x,x^{-1}])$. By Lemma~\ref{lemma:SL2=E2_for_A<t>}, the image of $\sigma$ in $\mathbb{Z} \langle x^{-1} \rangle$ is an elementary matrix. Thus the partial version of Vorst's lemma implies that there exists a factorization 
 \[
 \varphi_x(\sigma)=(\varepsilon_1)_x \varphi_x(\varepsilon_2)
 \]
 for some $\varepsilon_1 \in \widetilde{\mathrm{E}}_2(\mathbb{Z}[x]_{1+x \mathbb{Z}[x]})$ and some $\varepsilon_2 \in \mathrm{E}_2(\mathbb{Z}[x,x^{-1}])$, see Lemma~\ref{lemma:partial_Vorst_lemma}.
 
 Since $\mathrm{SL}_2(-)$ preserves pullbacks, it follows that there exists a unique matrix $\tau \in \mathrm{SL}_2(\mathbb{Z}[x])$ such that $\varphi(\tau)=\varepsilon_1$ and $\tau_x=\sigma (\varepsilon_2)^{-1}$. It follows that $\sigma$ is pseudoelementary, as claimed.
\end{proof}

In order to establish weak Zariski excision for $\widetilde{SK}_{1,2}$, we use a technique due to Plumstead. We call two matrices $\sigma_0, \sigma_1 \in \mathrm{SL}_r(A)$ \emph{isotopic} if there exists a $\tau(t) \in \mathrm{SL}_2(A[t])$ with $\sigma_0=\tau(0)$ and $\sigma_1 =\tau(1)$. This is a symmetric relation since $\tau(1-t)$ gives an isotopy ``in the other direction.''

\begin{lemma}[Plumstead]\label{lemma:Plumstead}
 Let $A$ be a commutative ring and let $f, g \in A$ be comaximal elements. Let $\sigma_0, \sigma_1 \in \mathrm{SL}_2(A_{fg})$ and let $P_i \in \Proj(A)$, $i=1,2$ be the projective module corresponding to the patching data $(A_f^2, A_g^2, \sigma_i)$. If $\sigma_0$ and $\sigma_1$ are isotopic, then there exists an isomorphism $P_0 \cong P_1$ of $A$-modules.
\end{lemma}

\begin{proof}
 This is proved in \cite[Lemma~1]{PLUMSTEAD}. To see this, let $P \in \Proj(A[t])$ be the projective module corresponding to the patching datum given by a fixed isotopy $\bigl(A_f[t]^2,A_g[t]^2,\tau(t) \bigr)$. Note that, by construction, the localized modules $P_{\mathfrak{m}} \in \Proj(A_{\mathfrak{m}}[t])$ are free for all maximal ideals $\mathfrak{m} \subseteq A$ (since each localization $A \rightarrow A_{\mathfrak{m}}$ factors either through $A \rightarrow A_f$ or through $A \rightarrow A_g$ and the modules $P_f$ and $P_g$ are both free by definition of $P$). 
 
 Since $P_2 \colon \CRing \rightarrow \Set_{\ast}$ satisfies the Quillen principle $(\mathrm{Q})$, it follows that $P$ is extended. Thus the patching data $\sigma_0=\tau(0)$ and $\sigma_1=\tau(1)$ define isomorphic modules since they correspond to $P \slash t P$ and $P \slash (t-1)P$ respectively (see Lemma~\ref{lemma:naturality_of_patching_data}).
\end{proof}

\begin{lemma}\label{lemma:pseudoelementary_isotopies}
 Each pseudoelementary matrix $\sigma \in \widetilde{\mathrm{E}}_2(A)$ is isotopic to an elementary matrix and each elementary matrix is isotopic to the identity matrix.
\end{lemma}

\begin{proof}
 The claim about elementary matrices is straightforward: if the elementary matrix $\varepsilon$ is given by $\prod_{k=1}^m e_{i_k j_k}(c_k)$, then $\tau(t) \defl \prod_{k=1}^m e_{i_k j_k}(c_k \cdot t)$ gives the desired isotopy between $\varepsilon$ and $\id$.
 
 To see the first claim, let $\sigma$ be the image of $\sigma^{\prime} \in \mathrm{SL}_2(\mathbb{Z}[x_1, \ldots, x_n])$ under $\varphi \colon \mathbb{Z}[x_1, \ldots, x_n] \rightarrow A$. Let 
 \[
 h \colon \mathbb{Z}[x_1,\ldots, x_n] \rightarrow \mathbb{Z}[x_1, \ldots, x_n][t]
\]
be the homomorphism given by $x_i \mapsto x_i t$. Then we have $\mathrm{ev}_1 \circ h(\sigma^{\prime})=\sigma^{\prime}$ and the matrix $\mathrm{ev}_0 \circ h (\sigma^{\prime})$ is elementary since $\mathrm{ev}_0 h$ factors through $\mathrm{SL}_2(\mathbb{Z})=\mathrm{E}_2(\mathbb{Z})$. The matrix
\[
\tau(t) \defl \varphi[t] \circ h (\sigma)
\]
gives an isotopy between $\sigma$ and an elementary matrix since $\mathrm{ev}_i \circ \varphi[t]=\varphi \circ \mathrm{ev}_i$ for $i=0,1$.
\end{proof}

\begin{thm}\label{thm:pseudoelementary_weak_Zariski_excision}
 Let $A$ be a commutative ring and let $S, T \subseteq A$ be two multiplicative sets such that for each $f \in S$ and each $g \in T$ we have $Af+Ag=A$. Then the following hold:
 \begin{enumerate}
 \item[(a)] For each $\varepsilon \in \widetilde{\mathrm{E}}_2(A_{S,T})$, there exist $\varepsilon_1 \in \widetilde{\mathrm{E}}_2(A_S)$ and $\varepsilon_2 \in \widetilde{\mathrm{E}}_2(A_T)$ such that $\varepsilon=(\varepsilon_1)_T (\varepsilon_2)_S$ holds;
 \item[(b)] The square
 \[
 \xymatrix{ \widetilde{SK}_{1,2}(A) \ar[r] \ar[d] & \widetilde{SK}_{1,2}(A_S) \ar[d] \\ \widetilde{SK}_{1,2}(A_T) \ar[r] & \widetilde{SK}_{1,2}(A_{S,T})}
 \]
 is a weak pullback diagram.
 \end{enumerate}
\end{thm}

\begin{proof}
 Part~(b) follows from Part~(a) and the fact that $\mathrm{SL}_2(-)$ preserves pullbacks: if $(\sigma)_S=(\tau)_T \varepsilon$ in $\mathrm{SL}_2(A_{S,T})$, then $\bigl(\sigma (\varepsilon_2)^{-1} \bigr)_S=(\tau \varepsilon_1)_T$, so both these matrices come from a unique matrix in $\mathrm{SL}_2(A)$.
 
 To see Part~(a), let $\varepsilon \in \widetilde{\mathrm{E}}_2(A_{S,T})$ and choose a homomorphism
 \[
 \varphi \colon \mathbb{Z}[t_1, \ldots, t_n] \rightarrow A_{S,T}
 \]
 such that $\varepsilon$ is the image of some $\varepsilon^{\prime} \in \mathrm{SL}_2(\mathbb{Z}[t_1, \ldots, t_n])$ under $\varphi$. Since the domain of $\varphi$ is finitely presentable, we can reduce to the case where both $S$ and $T$ are generated by a single element $f$ respectively $g$.
 
 Now we let $B$ be the ``free'' ring with two comaximal elements, that is, $B=\mathbb{Z}[f,\bar{f},g,\bar{g}] \slash ( f\bar{f}+g \bar{g} -1)$. Any choice of ``witnesses'' that $f,g \in A$ are comaximal induces a ring homomorphism $B \rightarrow A$ sending $f$ to $f$ and $g$ to $g$. We can extend this to a ring homomorphism $\psi \colon B[t_1, \ldots, t_n] \rightarrow A$ in such a way that the image of $\varphi$ above is contained in the image of $\psi_{fg} \colon B_{fg}[t_1, \ldots, t_n] \rightarrow A_{fg}$. It follows that there exists a lift
 \[
 \xymatrix{ & B_{fg}[t_1, \ldots, t_n] \ar[d]^{\psi_{fg}} \\ \mathbb{Z}[t_1, \ldots, t_n] \ar[r]_-{\varphi} \ar@{-->}[ru]^{\varphi^{\prime}} & A_{fg}}
 \]
 making the above triangle commutative. We have reduced the problem to checking that the pseudoelementary matrix $\varphi^{\prime}(\varepsilon^{\prime})$ has the desired factorization, for then we get the factorization of $\varepsilon$ by applying $\psi_f$ and $\psi_g$ respectively.
 
 Since $\varphi^{\prime}(\varepsilon^{\prime})$ is isotopic to an elementary matrix and all elementary matrices are isotopic to the identity (see Lemma~\ref{lemma:pseudoelementary_isotopies}), Plumstead's lemma implies that there exist an isomorphism
 \[
 (B_f[t_i]^2,B_g[t_i], \varphi^{\prime}(\varepsilon^{\prime})) \cong (B_f[t_i]^2,B_g[t_i], \id )
 \]
 of patching data (see Lemma~\ref{lemma:Plumstead}). From the free patching lemma it thus follows that we have a factorization $\varphi^{\prime}(\varepsilon^{\prime})=(\sigma_1)_g (\sigma_2)_f$ for some $\sigma_1 \in \mathrm{SL}_2(B_f[t_i])$ and $\sigma_2 \in \mathrm{SL}_2(B_g[t_i])$.
 
 To conclude the proof, it thus suffices to check that the equalities
 \[
 \mathrm{SL}_2(B_f[t_i]) =\widetilde{\mathrm{E}}_2(B_f[t_i]) \quad \text{and} \quad  \mathrm{SL}_2(B_g[t_i]) =\widetilde{\mathrm{E}}_2(B_g[t_i])
 \]
 hold. This follows from the two isomorphisms $B_{f} \cong \mathbb{Z}[f,f^{-1},g,\bar{g}]$ and $B_g \cong \mathbb{Z}[f,\bar{f},g,g^{-1}]$ and Proposition~\ref{prop:pseudoelementary_trivial_for_single_Laurent_variable} (which tells us that $\widetilde{SK}_{1,2}$ is trivial on any polynomial ring over $\mathbb{Z}[x,x^{-1}]$).
\end{proof}

 As a consequence, we can now extend Proposition~\ref{prop:pseudoelementary_trivial_for_single_Laurent_variable} to an arbitrary number of Laurent polynomial variables.

 \begin{cor}\label{cor:pseudoelementary_over_several_Laurent_variables}
 If $R$ is either a field or a principal ideal domain such that the equality $\mathrm{SL}_2(R)=\mathrm{E}_2(R)$ holds, then
 \[
 \widetilde{SK}_{1,2}(R[x_1^{\pm}, \ldots, x_n^{\pm}, t_1, \ldots, t_m]) \cong \ast
 \]
 holds for all $n, m \in \mathbb{N}$.
 \end{cor}
 
 \begin{proof}
 Both these types of rings are closed under $R \mapsto R\langle t \rangle$, see Lemma~\ref{lemma:SL2=E2_for_A<t>}. It follows that $R$ has the strong $\widetilde{SK}_{1,2}$-extension property over itself, that is, the function
 \[
 \widetilde{SK}_{1,2}(R) \rightarrow \widetilde{SK}_{1,2}(R \langle t_1, \ldots, t_k \rangle) \cong \ast
 \]
 is tautologically surjective for all $k$ (see Definition~\ref{dfn:strong_extension_property}). Since $\widetilde{SK}_{1,2}$ satisfies the Horrocks principle $(\mathrm{H})$ (see Corollary~\ref{cor:SK_12_satisfies_Q_and_H}) and weak Zariski excision (see Theorem~\ref{thm:pseudoelementary_weak_Zariski_excision}), we can apply Theorem~\ref{thm:strong_F_extension_for_Laurent}, which establishes the claim.
\end{proof}

 With this in hand, we can prove our first result whose statement does not involve the newly introduced notion of pseudoelementary matrices. Namely, we can now show that all finitely generated projective modules over the coordinate ring of $\mathrm{SL}_2(R)$ are free for suitable ground rings $R$.
 
 \begin{cor}\label{cor:projectives_free_over_coordinate_ring_of_SL2}
 Let  $R$ be either a field or a principal ideal domain such that the equality $\mathrm{SL}_2(R)=\mathrm{E}_2(R)$ holds. Then all finitely generated projective modules over the two rings
 \[
 A=R[x,y,z,w] \slash (xy-zw-1) \quad \text{and} \quad B = R[x,y,z,w] \slash (xy+zw-1)
 \]
 are free.
 \end{cor}
 
 \begin{proof}
 The map $B \rightarrow A$ sending $z$ to $-z$ and all other variables to themselves is an isomorphism, so it suffices to consider the ring $B$. Applying Nagata's criterion (see \cite[\href{https://stacks.math.columbia.edu/tag/0AFU}{Lemma 0AFU}]{stacks-project}) to the prime element $x \in B$, we find that $B$ is a unique factorization domain, so all projective modules of rank $1$ are free (see for example \cite[Theorem~II.1.3]{LAM}).
 
 Note that all finitely generated projective modules over the rings
 \[
  B_x \cong R[x,x^{-1},z,w] \quad \text{and} \quad B_z \cong R[x,y,z,z^{-1}]
 \]
 are free. (This follows for example from the fact that these rings $R$ have the strong $P_r$-extension property over themselves for all $r \geq 1$ and Theorem~\ref{thm:strong_F_extension_for_Laurent}.) Thus each finitely generated projective $B$-module $P$ corresponds to a patching datum of the form $(B_x^2, B_z^2,\sigma)$ for some $\sigma \in \mathrm{GL}_r(B_{xz})$.
 
 Since all rank $1$ projectives over $B$ are free, it follows that each unit in $B_{xy}$ decomposes as a product of a unit in $B_x$ and a unit in $B_z$. Applying this to the determinant of $\sigma$, we can further reduce this to the case of patching datum where $\sigma \in \mathrm{SL}_r(B)$.
 
 Note that we have an isomorphism $B_{xz} \cong R[x^{\pm},z^{\pm},w]$, so if $r > 2$, we have $SK_{1,r}(B_{xz}) \cong 0$ since $R$ satisfies the strong $SK_{1,r}$-extension property (in the case where $R$ is a principal ideal domain, we can see this by using \cite[Lemma~IV.6.1]{LAM} as in the proof of Lemma~\ref{lemma:SL2=E2_for_A<t>}, and the fact that $\mathrm{SL}_2(R) \rightarrow SK_{1,r}(R)$ is surjective since $W_{r}(R) \cong \ast$ for all $r \geq 3$). For $r=2$, we have $\widetilde{SK}_{1,2}(B_{xz}) \cong \ast$ by Corollary~\ref{cor:pseudoelementary_over_several_Laurent_variables}.
 
 Thus Vorst's lemma (see Lemma~\ref{lemma:Vorsts_lemma}) in the case $r \geq 3$ and Theorem~\ref{thm:pseudoelementary_weak_Zariski_excision} in the case $r=2$ imply that there exist matrices $\sigma_1 \in \mathrm{SL}_r(B_x)$ and $\sigma_2 \in \mathrm{SL}_r(B_z)$ such that $\sigma=(\sigma_1)_z (\sigma_2)_x$ holds. These matrices give the desired isomorphism $P \cong B^r$. 
 \end{proof}
 
 In order to prove further weak excision results, we introduce some terminology.
 
 \begin{dfn}\label{dfn:good_lifts}
 Let $\varphi \colon A \rightarrow B$ be an analytic isomorphism along $f$. A commutative diagram
 \[
 \xymatrix{ \mathbb{Z}[t_1, \ldots, t_n][x] \ar[d]_{\psi} \ar[r]^-{\varphi^{\prime}} & B^{\prime} \ar[d]^{\psi^{\prime}} \\ A \ar[r]_-{\varphi} & B}
 \]
 is called a \emph{good lift of $\varphi$} if the following conditions hold:
 \begin{enumerate}
 \item[(1)] We have $\psi(x)=f$;
 \item[(2)] The homomorphism $\varphi^{\prime} \colon \mathbb{Z}[t_i,x] \rightarrow B^{\prime}$ is an analytic isomorphism along $x$;
 \item[(3)] The $\widetilde{SK}_{1,2}(B^{\prime}) \rightarrow \widetilde{SK}_{1,2}(B^{\prime}_{\varphi^{\prime}(x)})$ has trivial kernel.
 \end{enumerate}
 
  We say that $\varphi$ \emph{admits enough good lifts} if for every finitely generated subring $B_0 \subseteq B_{\varphi(f)}$, there exists a good lift as above such that $B_0$ is contained in the image of the homomorphisms $\psi^{\prime}_{\varphi^{\prime}(x)} \colon B^{\prime}_{\varphi^{\prime}(x)} \rightarrow B_{\varphi(f)}$. 
 \end{dfn}
 
 \begin{prop}\label{prop:good_lifts_implies_weak_excision}
 Let $\varphi \colon A \rightarrow B$ be an analytic isomorphism along $f$ which admits enough good lifts. Then the following hold:
 \begin{enumerate}
 \item[(a)] For each $\varepsilon \in \widetilde{\mathrm{E}}_2(B_{\varphi(f)})$, there exist $\varepsilon_1 \in \widetilde{\mathrm{E}}_2(B)$ and $\varepsilon_2 \in \widetilde{\mathrm{E}}_2(A_f)$ such that $\varepsilon=(\varepsilon_1)_{\varphi(f)} \varphi_f (\varepsilon_2)$ holds;
 \item[(b)] The square
 \[
 \xymatrix{ \widetilde{SK}_{1,2}(A) \ar[r]^-{\lambda_f} \ar[d]_{\varphi} & \widetilde{SK}_{1,2}(A_f) \ar[d]^{\varphi_f} \\ \widetilde{SK}_{1,2}(B) \ar[r]_-{\lambda_{\varphi(f)}} & \widetilde{SK}_{1,2}(B_{\varphi(f)}) }
 \]
 is a weak pullback diagram.
 \end{enumerate}
 \end{prop}
 
 \begin{proof}
 As in Theorem~\ref{thm:pseudoelementary_weak_Zariski_excision}, Part~(b) follows from Part~(a) and the fact that $\mathrm{SL}_2(-)$ preserves pullback diagrams. Thus let $\varepsilon \in \widetilde{\mathrm{E}}_2(B_{\varphi(f)})$ be pseudoelementary and choose a homomorphism $\gamma \colon \mathbb{Z}[x_1, \ldots, x_k] \rightarrow B_{\varphi(f)}$ and $\varepsilon^{\prime} \in \mathrm{SL}_2(\mathbb{Z}[x_i])$ such that $\varepsilon=\gamma(\varepsilon^{\prime})$. By assumption, we can find a good lift 
 \[
 \xymatrix{ \mathbb{Z}[t_1, \ldots, t_n][x] \ar[d]_{\psi} \ar[r]^-{\varphi^{\prime}} & B^{\prime} \ar[d]^{\psi^{\prime}} \\ A \ar[r]_-{\varphi} & B}
 \]
 such that the image of $\gamma$ is contained in the image of $\psi^{\prime}_{\varphi^{\prime}(x)}$. Thus there exists a homomorphism $\gamma^{\prime}$ making the triangle
 \[
 \xymatrix{ & B^{\prime}_{\varphi^{\prime}(x)}  \ar[d]^{\psi^{\prime}_{\varphi^{\prime}(x)}} \\ \mathbb{Z}[x_i] \ar@{-->}[ru]^{\gamma^{\prime}} \ar[r]_-{\gamma} & B_{\varphi(f)}}
 \]
 commutative. To establish the claim, it therefore suffices to show that the pseudoelementary matrix $\gamma^{\prime}(\varepsilon^{\prime})$ admits a factorization as in Part~(a).
 
 Since
 \[
 \xymatrix{ \mathbb{Z}[t_1, \ldots, t_n][x] \ar[d]_{\varphi^{\prime}} \ar[r] & \mathbb{Z}[t_1, \ldots, t_n][x,x^{-1}] \ar[d]^{\varphi^{\prime}_x} \\ B^{\prime} \ar[r] & B^{\prime}_{\varphi^{\prime}(x)} }
 \]
 is an analytic patching diagram, the free patching lemma implies that $\gamma^{\prime}(\varepsilon^{\prime})$ factors as  $(\sigma_1)_{\varphi^{\prime}(x)} \varphi^{\prime}_x (\sigma_2)$ for some $\sigma_1 \in \mathrm{SL}_2(B^{\prime})$ and $\sigma_2 \in \mathrm{SL}_2(\mathbb{Z}[t_i][x,x^{-1}])$. From Proposition~\ref{prop:pseudoelementary_trivial_for_single_Laurent_variable} it follows that $\sigma_2$ is pseudoelementary. This implies that the image $(\sigma_1)_{\varphi^{\prime}(x)}=\gamma^{\prime} (\varepsilon^{\prime}) \varphi^{\prime}_x(\sigma_2^{-1})$ of $\sigma_1$ under localization at $\varphi^{\prime}(x)$ is pseudoelementary. From Part~(3) of the definition of good lifts it therefore follows that $\sigma_1$ is pseudoelementary, too, which concludes the proof of Part~(a).
 \end{proof}
 
 In order to construct good lifts, we need criteria which ensure that the homomorphism $\widetilde{SK}_{1,2}(B^{\prime}) \rightarrow \widetilde{SK}_{1,2}(B^{\prime}_{\varphi^{\prime}}(x))$ induced by localization has trivial kernel. One such criterion is given by the next lemma.
 
 \begin{lemma}\label{lemma:localization_induces_injection_on_SK12}
 Let $R$ be a commutative ring such that $\widetilde{SK}_{1,2}(R[x]) \cong \ast$ holds. Then for any polynomial $h(x) \in R[x]$, the function
 \[
 \widetilde{SK}_{1,2}(R[x]_{1+xh(x)}) \rightarrow \widetilde{SK}_{1,2}(R[x]_{1+xh(x),x})
 \]
 induced by localization at $x$ has trivial kernel.
 \end{lemma}
 
 \begin{proof}
 Since $x$ and $1+xh(x)$ are comaximal, the square
 \[
 \xymatrix{\widetilde{SK}_{1,2}(R[x]) \ar[d] \ar[r] & \widetilde{SK}_{1,2}(R[x,x^{-1}]) \ar[d] \\ \widetilde{SK}_{1,2}(R[x]_{1+xh(x)}) \ar[r] & \widetilde{SK}_{1,2}(R[x]_{1+xh(x),x})}
 \]
 is a weak pullback diagram (see Theorem~\ref{thm:pseudoelementary_weak_Zariski_excision}). Thus any $\sigma$ in the kernel of the lower map is the image of some $\tau \in \widetilde{SK}_{1,2}(R[x])$, hence pseudoelementary by assumption.
 \end{proof}
 
 \begin{prop}\label{prop:L_homomorphism_good_lifts}
 Let $A \rightarrow B \defl \bigl(A[x] \slash f(x) \bigr)_{g(x)}$ be an $L$-homomorphism (for example, a standard Nisnevich homomorphism) along $a_0 \defl f(0)$. Then the patching diagram
 \[
 \xymatrix{ A \ar[d] \ar[r] & A_{a_0} \ar[d] \\ B \ar[r] & B_{a_0}}
 \]
 admits enough good lifts.
 \end{prop}
 
 \begin{proof}
 Let $a_i$ denote the coefficients of $f$ and let $b_i$ denote the coefficients of $g$. By definition, we thus have $b_0=1$ and $a_1$ is a unit (see Definition~\ref{dfn:standard_Nisnevich_L_homomorphism}). By dividing $f$ through $a_1$, we can furthermore assume that $a_1=1$. Consider the polynomial algebra $\mathbb{Z}[a_i,b_j,t_1, \ldots, t_k]$ (excluding $b_0=1$ and $a_1=1$) and the homomorphism
 \[
 \psi \colon \mathbb{Z}[a_i,b_j,t_1, \ldots, t_k] \rightarrow A
 \]
 sending $a_i$ to $a_i$, $b_i$ to $b_i$ and the $t_i$ to arbitrary elements of $A$. We write $f(x)$ and $g(x)$ for the polynomials over $\mathbb{Z}[a_i,b_j,t_1, \ldots, t_k]$ lifting the corresponding polynomials over $A$. Recall that $\tilde{f}$ denotes the unique polynomial such that $f(x)-a_0=\tilde{f}(x) \cdot x$ holds. Let 
 \[
 B^{\prime} \defl \bigl( \mathbb{Z}[a_i,b_j,t_1, \ldots, t_k][x] \slash f(x) \bigr)_{f^{\prime} \tilde{f} g}
 \]
 and let $\varphi^{\prime} \colon \mathbb{Z}[a_i,b_j,t_1, \ldots, t_k] \rightarrow B^{\prime}$ be the corresponding $L$-homomorphism along $a_0$. Note that sending $x \in B^{\prime}$ to $x \in B$ induces a homomorphism $\psi^{\prime} \colon B^{\prime} \rightarrow B$ such that the square
 \[
 \xymatrix{\mathbb{Z}[a_i,b_j,t_1, \ldots, t_k]\ar[r]^-{\varphi^{\prime}} \ar[d]_{\psi} & B^{\prime} \ar[d]^{\psi^{\prime}} \\ A \ar[r]_-{\varphi} & B}
 \]
 is commutative (since both the image of $f^{\prime}$ and of $\tilde{f}$ are invertible in $B$, see Definition~\ref{dfn:standard_Nisnevich_L_homomorphism}).
 
 Since the images of the $t_i$ in $A$ can be chosen arbitrarily, it follows that any fixed finitely generated subring of $B_{a_0}$ is contained in the image of $\psi^{\prime}_{a_0}$ for a suitably chosen $\psi$, so it only remains to check that the above square defines a good lift of $\varphi$.
 
 By construction, $\varphi^{\prime}$ is an $L$-homomorphism along $a_0$, hence an analytic isomorphism along $a_0$ (see Proposition~\ref{prop:L_homomorphism_implies_affine_Nisnevich_square}). We have $\psi(a_0)=a_0$ by definition of $\psi$, so it only remains to check that
 \[
 \widetilde{SK}_{1,2}(B^{\prime}) \rightarrow \widetilde{SK}_{1,2}(B^{\prime}_{a_0})
 \]
 has trivial kernel.
 
 By ``solving for $a_0$'' we get an isomorphism
 \[
 \mathbb{Z}[\{a_i,b_j,t_{\ell}, x\} \setminus \{a_0\}] \rightarrow \mathbb{Z}[a_i,b_j,t_1, \ldots, t_k][x] \slash f(x) \smash{\rlap{,}}
 \]
 so $B^{\prime}$ can be written as localization of the domain of this isomorphism at $f^{\prime} \tilde{f} g$. From the facts that $g(0)=1$ and $a_1=1$, it follows that this is a polynomial of the form $1+xh(x)$. Since $a_0$ and $x$ are associates in $B^{\prime}$, we are reduced to checking that
 \[
 \widetilde{SK}_{1,2} \bigl( \mathbb{Z}[\{a_i,b_j,t_{\ell}, x\} \setminus \{a_0\}]_{1+xh(x)} \bigr) \rightarrow  \widetilde{SK}_{1,2}  \bigl( \mathbb{Z}[\{a_i,b_j,t_{\ell}, x\} \setminus \{a_0\}]_{1+xh(x),x} \bigr)
 \]
 is injective, which follows from Lemma~\ref{lemma:localization_induces_injection_on_SK12}.
 \end{proof}
 
 \begin{lemma}\label{lemma:good_lifts_for_filtered_colimit}
 Let $A \rightarrow B_i$ be a filtered diagram of $A$-algebras and let $f \in A$. Assume that $A \rightarrow B_i$ is an analytic isomorphism along $f$ and that $A \rightarrow B_i$ admits enough good lifts for all $i$. Then $A \rightarrow \mathrm{colim} B_i$ is an analyitic isomorphism along $f$ which admits enough good lifts.
 \end{lemma}
 
 \begin{proof}
 The homomorphism $A \rightarrow \mathrm{colim} B_i$ is an analytic isomorphism along $f$ since filtered colimits commute with the formation of quotients and with pullbacks. The claim about good lifts follows from the fact that any finitely generated subring of $(\mathrm{colim} B_i)_f$ is contained in the image of some $(B_i)_f$.
 \end{proof}
 
 \begin{cor}\label{cor:good_lifts_for_etale_nbhd}
 Let $(R, \mathfrak{m})$ be a local ring and let $R \rightarrow R^{\prime}$ be an {\'e}tale neighbourhood of $R$, exhibited as such by $R^{\prime} \cong \bigl(R[x] \slash f(x) \bigr)_{(\mathfrak{m},x)}$ with $f(x)$ monic. Let $a_0=f(0)$ and let $\gamma \colon R \rightarrow A$ be any ring homomorphism. Then the analytic isomorphism $A \rightarrow A \ten{R} R^{\prime}$ along $\gamma(a_0)$ admits enough good lifts.
 \end{cor}
 
 \begin{proof}
 The  ring $R^{\prime}$ can be written as filtered colimit of the rings $\bigl( R[x] \slash f(x) \bigr)_{g(x)}$ where $f(x)$ is fixed and $g(x)$ varies among all polynomials with $g(0)=1$ and such that $f^{\prime}$ and $\tilde{f}$ are invertible in the localization at $g$ (see Proposition~\ref{prop:etale_neighbourhood_standard_Nisnevich}). The base change to $A$ is of the same form and thus admits enough good lifts by Proposition~\ref{prop:L_homomorphism_good_lifts} and Lemma~\ref{lemma:good_lifts_for_filtered_colimit}. 
 \end{proof}
 
 With this in hand, we can prove Theorem~\ref{thm:Vorsts_Theorem_for_2x2_matrices} of the introduction.
 
 \begin{cit}[Theorem~\ref{thm:Vorsts_Theorem_for_2x2_matrices}]
 Let $R$ be a commutative ring whose local rings are unramified regular local rings. For each $\sigma(t_1, \ldots, t_n) \in \mathrm{SL}_2(R[t_1, \ldots, t_n])$, there exists a pseudoelementary matrix $\varepsilon(t_1, \ldots, t_n) \in \widetilde{\mathrm{E}}_2(R[t_1, \ldots, t_n])$ such that the equation
 \[
 \sigma(t_1, \ldots, t_n)=\sigma(0, \ldots, 0) \cdot \varepsilon(t_1, \ldots, t_n)
 \]
 holds.
 \end{cit}
 
 \begin{proof}[Theorem~\ref{thm:Vorsts_Theorem_for_2x2_matrices}]
 We first note that the claim is equivalent to the assertion that $R$ is $\widetilde{SK}_{1,2}$-regular. The claim of the theorem implies that the function
 \[
 \widetilde{SK}_{1,2}(R) \rightarrow \widetilde{SK}_{1,2}(R[t_1, \ldots, t_n])
 \]
 is bijective for all $n$, so conversely assume that $R$ is $\widetilde{SK}_{1,2}$-regular. Then for each $\sigma$ as in the theorem, there exist $\tau \in \mathrm{SL}_2(R)$ and $\varepsilon^{\prime} \in \widetilde{\mathrm{E}}_2(R[t_1, \ldots, t_n])$ such that
 \[
 \sigma(t_1, \ldots, t_n)=\tau \cdot \varepsilon^{\prime}(t_1, \ldots, t_n)
 \]
 holds. Evaluating this in $t_i=0$, we find that $[\sigma(0,\ldots,0)]=[\tau]$ in $\widetilde{SK}_{1,2}(R)$, from which the claim of the theorem follows.
 
 Since $\widetilde{SK}_{1,2}$ satisfies the Quillen principle $(\mathrm{Q})$ (see Corollary~\ref{cor:SK_12_satisfies_Q_and_H}), it only remains to check that all unramfied regular local rings are $\widetilde{SK}_{1,2}$-regular. To see this, it suffices to check that the premise of Theorem~\ref{thm:F_regularity_for_unramified_regular_rings} holds.
 
 Since $\widetilde{SK}_{1,2}$ admits a natural transitive group action, the implication
 \[
 N \widetilde{SK}_{1,2}(R) \cong \ast \quad \Rightarrow \quad R \in \ca{E}_1^{\widetilde{SK}_{1,2}}
 \]
 holds (see Lemma~\ref{lemma:transitive_action_implies_extended_if_trivial_NF}). Note that $\widetilde{SK}_{1,2}$ satisfies the Roitman principle $(\mathrm{R})$ by Proposition~\ref{prop:pseudoelementary_Roitman}.
 
 From Corollary~\ref{cor:good_lifts_for_etale_nbhd} we know that patching diagrams associated to {\'e}tale neighbourhoods admit enough good lifts. Thus Part~(b) of Proposition~\ref{prop:good_lifts_implies_weak_excision} implies that $\widetilde{SK}_{1,2}$ satisfies weak excision for patching diagrams associated to {\'e}tale neighbourhoods.
 
 Finally, we know from Corollary~\ref{cor:pseudoelementary_over_several_Laurent_variables} that all fields and all discrete valuation rings are $\widetilde{SK}_{1,2}$-regular, so both Part~$(i)$ and Part~$(ii)$ of Theorem~\ref{thm:F_regularity_for_unramified_regular_rings} are applicable, which concludes the proof that all unramified regular local rings are $\widetilde{SK}_{1,2}$-regular.
 \end{proof}
 
 In order to show that analytic isomorphisms arising from completions admit enough good lifts, we need to use for the first time the fact $\widetilde{SK}_{1,2}$ is $j$-invariant.
 
 \begin{prop}\label{prop:good_lifts_for_completion}
 Let $A$ be a finitely generated commutative ring and let $f \in A$. Then the analytic isomorphism $A \rightarrow A^{\wedge}_{(f)}$ along $f$ admits enough good lifts.
 \end{prop}
 
 \begin{proof}
 Choose a surjective homomorphism $\psi \colon \mathbb{Z}[t_1, \ldots, t_n][x] \rightarrow A$ sending $x$ to $f$ and let $B^{\prime} \defl \mathbb{Z}[t_1, \ldots, t_n]\llbracket x \rrbracket$. The completion morphism
 \[
 \varphi^{\prime} \colon \mathbb{Z}[t_1, \ldots, t_n][x] \rightarrow \mathbb{Z}[t_1, \ldots, t_n]\llbracket x \rrbracket=B^{\prime}
 \]
 is an analytic isomorphism along $x$. Moreover, since $A \cong \mathbb{Z}[t_i][x] \slash I$ for some (necessarily finitely generated) ideal $I \subseteq \mathbb{Z}[t_i][x]$, we have 
 \[
 A^{\wedge}_{(f)} \cong B^{\prime} \otimes_{\mathbb{Z}[t_i][x]} \mathbb{Z}[t_i][x] \slash I \cong B^{\prime} \slash I B^{\prime} \smash{\rlap{,}}
 \]
 so the homomorphism $\psi^{\prime} \colon B^{\prime} \rightarrow A^{\wedge}_{(f)}$ induced by $\psi$ is surjective. It follows that its localization $\psi^{\prime}_x \colon B^{\prime}_x \rightarrow A^{\wedge}_{(f)}[\frac{1}{f}]$ is surjective, too. It thus suffices to check that $\varphi^{\prime}$ is a good lift.
 
 To see this, it only remains to show that the function
 \[
 \widetilde{SK}_{1,2}(B^{\prime}) \rightarrow \widetilde{SK}_{1,2}(B^{\prime}_x)
 \]
 has trivial kernel. Since $x$ lies in the Jacobson radical of $B^{\prime}=\mathbb{Z}[t_i] \llbracket x \rrbracket$, we have
 \[
 \widetilde{SK}_{1,2}(\mathbb{Z}[t_i] \llbracket x \rrbracket) \cong \widetilde{SK}_{1,2}(\mathbb{Z}[t_i]) \cong \ast
 \]
 by Proposition~\ref{prop:pseudoelementary_j_invariant} and by definition of pseudoelementary matrices. Thus the domain of the above function is trivial, so the function clearly has trivial kernel.
 \end{proof}
 
 Our final (and most technical) example of analytic isomorphisms with enough good lifts concerns basic Nisnevich homomorphisms along $f$. Recall that 
 \[
 R \rightarrow R[x_1, \ldots, x_n] \slash (g_1, \ldots, g_c)
 \]
 is called a \emph{relative global complete intersection} if all the non-zero fiber rings over the residue fields of $R$ have dimension $n-c$ (here we assume $n \geq c$), see \cite[\href{https://stacks.math.columbia.edu/tag/00SP}{Definition 00SP}]{stacks-project}.
 
 \begin{lemma}\label{lemma:ideal_intersection_equals_product_relative_gci}
 Let $R \rightarrow R[x_1, \ldots, x_n] \slash (g_1, \ldots, g_c)$ be a relative global complete intersection. Then the ideals $I=(t_1, \ldots, t_c)$ and $J=(t_1-g_1, \ldots, t_c-g_c)$ in $R[x_1, \ldots, x_n,t_1, \ldots, t_c]$ satisfy $I \cap J=I \cdot J$.
 \end{lemma}
 
 \begin{proof}
 Let $A \defl R[x_1, \ldots, x_n, t_1, \ldots, t_c]$. Since $\mathrm{Tor}_1^A(A \slash I, A \slash J) \cong I \cap J \slash IJ$, it suffices to show that $\mathrm{Tor}^A_1(A \slash I, A \slash J) \cong 0$ (see \cite{SPEYER_IDEALS_MATHOVERFLOW}). We recall this argument here. If we tensor the exact sequence
 \[
 \xymatrix{0 \ar[r] & J \ar[r] & A \ar[r] & A \slash J \ar[r] & 0}
 \]
 with $A \slash I$, we obtain the exact sequence
 \[
 \xymatrix{0 \ar[r] & \mathrm{Tor}^A_1(A \slash I, A \slash J) \ar[r] & J \slash I \cdot J \ar[r] & A \slash I \ar[r] & A \slash I \otimes A \slash J \ar[r] & 0}
 \]
 of $A$-modules. The snake lemma applied to the lower two exact rows in the diagram
 \[
 \xymatrix{& 0 \ar[d] & 0 \ar[d] \\0 \ar[r] & I \cdot J \ar@{=}[d] \ar[r] & I \cap J \ar[d] \ar[r] & I \cap J \slash I \cdot J \ar[d] \ar[r] & 0\\
 0 \ar[r] & I \cdot J \ar[d] \ar[r] & J \ar[d] \ar[r] & J \slash I \cdot J \ar[d] \ar[r] & 0 \\
 0 \ar[r] & 0 \ar[r] & A \slash I \ar@{=}[r] & A \slash I \ar[r] & 0
  }
 \]
 shows that the kernel of $J \slash I \cdot J \rightarrow A \slash I$ is isomorphic to $I \cap J \slash I \cdot J$. Thus we have  $\mathrm{Tor}^A_1(A \slash I, A \slash J) \cong I \cap J \slash I\cdot J$, as claimed.
 
 To compute this Tor group, we use the Koszul resolution $K(t_1, \ldots, t_c)$ of $A \slash I$. Since $A \slash J \otimes K(t_1, \ldots, t_c)$ is the Koszul complex $K(g_1, \ldots, g_c)$ over the poynomial ring $A \slash J \cong R[x_1, \dots, x_n]$, it suffices to check $K(g_1, \ldots, g_c)$ over $R[x_1, \ldots, x_n]$ is exact, which is a local question.
 
 We first consider the case where we localize at a prime $\mathfrak{p}$ which does  not contain one of the $g_i$. In this case, we can permute $g_i$ and $g_c$ so that without loss of generality, $g_c \not \in \mathfrak{p}$ (see \cite[\href{https://stacks.math.columbia.edu/tag/0625}{Lemma 0625}]{stacks-project}). Since $K(g_1, \ldots, g_c)$ can be obtained as cone of multiplication with $g_c$ on $K(g_1, \ldots, g_{c-1})$ (see \cite[\href{https://stacks.math.columbia.edu/tag/0629}{Lemma 0629}]{stacks-project}), we find that $K(g_1, \ldots, g_c)$ is indeed exact when localized at such primes $\mathfrak{p}$.
 
 If $\mathfrak{p} \subseteq R[x_1, \ldots, x_n]$ is a prime containing all the elements $g_1, \ldots, g_c$, then \cite[\href{https://stacks.math.columbia.edu/tag/00SV}{Lemma 00SV}]{stacks-project} is applicable and shows that $g_1, \ldots, g_c$ is a regular sequence. Since regular sequences are in particular Koszul-regular (see \cite[\href{https://stacks.math.columbia.edu/tag/062F}{Lemma 062F}]{stacks-project}), the localization of $K(g_1,\ldots, g_c)$ at such primes is also exact.
 \end{proof}
 
 Recall from Example~\ref{example:weak_milnor_squares} that, given a commutative ring $A$ and two ideals $I, J \subseteq A$, the diagram
 \[
 \xymatrix{ A \slash I \cap J \ar[r] \ar[d] & A \slash J \ar[d] \\ A \slash I \ar[r] & A \slash I+J }
 \]
 is a Milnor square.
 
 \begin{lemma}\label{lemma:SK12_weak_Milnor_excision}
 The functor $\widetilde{SK}_{1,2} \colon \CRing \rightarrow \Set_{\ast}$ satisfies weak Milnor excision.
 \end{lemma}
 
 \begin{proof}
 This follows directly from the two facts that surjective ring homomorphisms $A \rightarrow B$ induce surjective maps $\widetilde{\mathrm{E}}_2(A) \rightarrow \widetilde{\mathrm{E}}_2(B)$ (see Lemma~\ref{lemma:basic_properties_pseudoelementary}) and that $\mathrm{SL}_2(-)$ preserves pullbacks.
 \end{proof}
 
 So far we have used the functor $P_2 \colon \CRing \rightarrow \Set_{\ast}$ and the free patching lemma to deduce various results about the functor $\widetilde{SK}_{1,2}$. Alternatively, we could have used $\mathrm{SL}_2$-torsors, that is, projective $A$-modules of rank $2$ with a chosen trivialization of the determinant bundle $\Lambda^2 P$. We denote the set of isomorphism classes of $\mathrm{SL}_2$-torsors over $A$ by $SP_2(A)$. Since this set can be obtained as set of isomorphism classes $\pi_0 B\mathrm{SL}_2$ of the Adams stack $B\mathrm{SL}_2$, it follows that $SP_2$ satisfies weak analytic excision and weak Milnor excision (see Proposition~\ref{prop:Adams_stack_weak_analytic_excision} and Proposition~\ref{prop:weak_milnor_excision_adams_case}). It is of course possible to prove this much more directly in this particular case.
 
 Since the automorphism group of the trivial $\mathrm{SL}_2$-torsor over $A$ is given by $\mathrm{SL}_2(A)$, we automatically get a factorization as in the free patching lemma into a product of matrices in $\mathrm{SL}_2$ whenever a patching datum yields a trivial torsor. The advantage of working with $SP_2$ is that we can apply this argument even in the case where the Picard group of $A$ (and therefore $P_2(A)$) are non-trivial.
 
 \begin{lemma}\label{lemma:SP2_partially_h_injective}
 The functor $SP_2 \colon \CRing \rightarrow \Set_{\ast}$ is partially $h$-injective, that is, the function
 \[
 SP_2(R) \rightarrow SP_2(R \slash I)
 \]
 has trivial kernel whenever $(R,I)$ is a henselian pair.
 \end{lemma}
 
 \begin{proof}
 If the $\mathrm{SL}_2$-torsor $P$ over $R$ is sent to the trivial torsor over $R \slash I$, then its underlying projective module is free (since $P_2$ is $h$-invariant, see Example~\ref{example:P_r_is_j_injective_and_h_invariant}). The claim now follows from the fact that any two trivializations of the determinant bundle of the free module are isomorphic (since any unit occurs as determinant of an automorphism of the free module).
 \end{proof}
 
 Note that the argument above shows in particular that $SP_2(A) \cong \ast$ whenever $P_2(A) \cong \ast$, a fact which we we will use in the proof of the following lemma. This is the main technical lemma used in the construction of good lifts for basic Nisnevich homomorphisms.
 
 \begin{lemma}\label{lemma:SK12_injective_for_basic_Nisnevich}
 Let $R$ be a finitely generated $\mathbb{Z}$-algebra and let $R \rightarrow S$ be a basic Nisnevich homomorphism along $f \in R$. Assume that $f$ is a nonzerodivisor in $R$ and the following conditions hold:
 \begin{enumerate}
 \item[(1)] The two rings $R$ and $R \slash f$ are both $P_2$-regular and $P_2$-contractible;
 \item[(2)] The two rings $R$ and $R \slash f$ are both $\widetilde{SK}_{1,2}$-regular and $\widetilde{SK}_{1,2}$-contractible;
 \item[(3)] The localization $R_f$ is $\widetilde{SK}_{1,2}$-contractible.
 \end{enumerate}
 Then the function
 \[
 \widetilde{SK}_{1,2}(S) \rightarrow \widetilde{SK}_{1,2}(S_{f})
 \]
 induced by the localization at $f$ has trivial kernel.
 \end{lemma}
 
 \begin{proof}
 Let $S \cong R[x_1, \ldots, x_n] \slash (g_1, \ldots g_n)$ such that $g_i \equiv x_i \mod f$. Consider the polynomial algebra $A \defl R[x_1,\ldots,x_n,t_1,\ldots, t_n]$ with the two ideals $I=(t_i)_{i=1}^n$ and $J=(t_i-g_i)_{i=1}^n$. Since standard smooth ring maps such as $R \rightarrow S$ are relative global complete intersections (see \cite[\href{https://stacks.math.columbia.edu/tag/00T7}{Lemma 00T7}]{stacks-project}), Lemma~\ref{lemma:ideal_intersection_equals_product_relative_gci} implies that $I \cap J=I \cdot J$. The above example of a pullback square thus shows that
 \[
 \xymatrix{ A\slash I \cdot J  \ar[r] \ar[d] & A \slash J \cong R[x_1, \ldots, x_n] \ar[d]  \\ R[x_1, \ldots, x_n] \cong A \slash I \ar[r] & A \slash I+J \cong S}
 \]
 is a Milnor square.
 
 Every $\sigma \in \mathrm{SL}_2(S)$ therefore gives rise to a $\mathrm{SL}_2$-torsor $P$ over $A \slash I \cdot J$ (by patching the trivial torsors over $A \slash I$ and $A \slash J$). From Condition~(2) we know that $\mathrm{SL}_2(R[x_i])=\widetilde{\mathrm{E}}_2(R[x_i])$ holds, so to show that $\sigma$ is pseudoelementary, we only need to check that $P$ is the trivial torsor. Since the localized torsor $P_f$ corresponds to the patching datum $\sigma_f$ (see Remark~\ref{rmk:naturality_of_patching_data_generalized}), it suffices to show that the function
 \[
 SP_2(A \slash I \cdot J) \rightarrow SP_2 \bigl( (A \slash I \cdot J)_f \bigr)
 \]
 has trivial kernel.
 
 Let $B \defl A \slash I \cdot J$ and consider the analytic patching diagram
 \[
 \xymatrix{ B \ar[d] \ar[r] & B_f \ar[d] \\ B^{\wedge}_{(f)} \ar[r] & B^{\wedge}_{(f)}[\frac{1}{f}] }
 \]
 which admits enough good lifts by Proposition~\ref{prop:good_lifts_for_completion}.
 
 We claim that $SP_2(B^{\wedge}_{(f)}) \cong \ast$ and $\widetilde{SK}_{1,2}(B^{\wedge}_{(f)}[\frac{1}{f}]) \cong \ast$ hold. Before establishing these claims, we show that they together imply the desired fact that a $\mathrm{SL}_2$-torsor over $B$ which becomes trivial over $B_f$ has to be trivial. Indeed, by the above two claims, such a torsor corresponds to a patching datum $\bigl((B^{\wedge}_{(f)})^2,(B_f)^2,\sigma\bigr)$ where $\sigma$ is pseudoelementary. Since $B \rightarrow B^{\wedge}_{(f)}$ admits enough good lifts, $\sigma$ factors into a product of matrices in $\mathrm{SL}_2(B^{\wedge}_{(f)})$ and $\mathrm{SL}_2(B_f)$ (see Proposition~\ref{prop:good_lifts_implies_weak_excision}), so this patching datum represents the trivial $\mathrm{SL}_2$-torsor.
 
 Thus it remains to check that $SP_2(B^{\wedge}_{(f)}) \cong \ast$ and $\widetilde{SK}_{1,2}(B^{\wedge}_{(f)}[\frac{1}{f}]) \cong \ast$ hold. To do this, we let $C \defl R[u_1, \ldots, u_n, v_1, \ldots, v_n] \slash (u_i) \cdot (v_j)$ and we consider the ring homomorphism $\varphi \colon C \rightarrow B=A \slash I \cdot J$ given by $\varphi(u_i)=t_i$ and $\varphi(v_j)=t_j-g_j$. We claim that $\varphi$ is an analytic isomorphism along $f$. First note that the assumption that $f \in R$ is a nonzerodivisor and the Milnor square at the beginning of the proof show that $f \in B=A \slash I \cdot J$ is a nonzerodivisor. The analogous Milnor square for the simpler ring $C$ shows that $f \in C$ is also a nonzerodivisor. Thus it suffices to show that $C \slash f \rightarrow B \slash f$ is an isomorphism.
 
 Since quotients of the form $A \slash I \cdot J$ are preserved by passage to the quotient modulo $f$ (as opposed to the case $I \cap J$ for general ideals), this homomorphism is given by
 \[
 \bar{\varphi} \colon R \slash f [u_i,v_j] \slash (u_i) \cdot (v_j) \rightarrow R \slash f [x_i,t_j] \slash (t_i) \cdot (t_j-x_j)
 \]
 where $\bar{\varphi}(u_i)=t_i$ and $\bar{\varphi}(v_j)=t_j - x_j$ (it is here that we use the fact that $g_j \equiv x_j \mod f$ holds). The homomorphism $\bar{\varphi}$ has an inverse given by $x_i \mapsto u_i -v_i$ and $t_i \mapsto u_i$. The resulting isomorphism $B^{\wedge}_{(f)} \cong C^{\wedge}_{(f)}$ preserves $f$ and thus reduces the problem to checking that
 \[
 SP_2\bigl(C^{\wedge}_{(f)} \bigr) \cong \ast \quad \text{and} \quad \widetilde{SK}_{1,2} \bigl(C^{\wedge}_{(f)}[\textstyle \frac{1}{f}] \bigr) \cong \ast
 \]
 hold.
 
 To establish the second of these, we can use the patching diagram
 \[
 \xymatrix{C \ar[d] \ar[r] & C_f \ar[d] \\ C^{\wedge}_{(f)} \ar[r] & C^{\wedge}_{(f)}[\frac{1}{f}]}
 \]
 associated to $C \rightarrow C^{\wedge}_{(f)}$, which reduces the second of the above problems to checking that $\widetilde{SK}_{1,2}\bigl(C^{\wedge}_{(f)}\bigr) \cong \ast$, $\widetilde{SK}_{1,2}(C_f) \cong \ast$, and $SP_2(C) \cong \ast$ hold.
 
 Taking into account the partial $h$-injectivity of $SP_2$ and the $j$-invariance of $\widetilde{SK}_{1,2}$, we have reduced the problem to showing that
 \[
 SP_2(C)\cong \ast, \quad SP_2(C\slash f) \cong \ast, \quad \widetilde{SK}_{1,2}(C_f) \cong \ast, \quad \text{and} \quad \widetilde{SK}_{1,2}(C \slash f) \cong \ast
 \]
 hold.
 
 These follow readily from the assumptions of the lemma and weak Milnor excision, applied to the Milnor squares
 \[
 \xymatrix{D[u_i] \ar[r] \ar[d] & D \ar[d] \\ D[u_i,v_j] \slash (u_i) \cdot (v_j) \ar[r] & D[v_j]}
 \]
 for the rings $D=R$, $D=R_f$, and $D=R \slash f$.
\end{proof}

\begin{prop}\label{prop:good_lifts_for_basic_Nisnevich}
Let $\varphi \colon R \rightarrow S$ be a basic Nisnevich homomorphism along $f$. Then $\varphi \colon R \rightarrow S$ admits enough good lifts.
\end{prop}

\begin{proof}
 We can choose a presentation $S \cong R[x_1, \ldots, x_n] \slash (g_1, \ldots, g_n)$ such that $g_i=x_i +fh_i$ for suitable $h_i \in R[x_1, \ldots, x_n]$. Let $k \in \mathbb{N}$ be the number of nonzero coefficients in all the $h_i$ and consider the homomorphism
 \[
 \psi \colon \mathbb{Z}[t_1, \ldots, t_k, t_{k+1}, \ldots, t_{k+c}][f] \rightarrow R
 \]
 sending $f$ to $f$, each $t_i$ to the corresponding coefficient of one of the polynomials $h_j$ whenever $i \leq k$, and sending $t_i$ to an arbitrary element for $i >k$. Let 
 \[
 S^{\prime} \defl \bigl( \mathbb{Z}[t_i][f][x_1, \ldots, x_n] \slash (x_j+fh_j) \bigr)_{\mathrm{det}(J)} \smash{\rlap{,}}
 \]
 where $J \defl \bigl(\partial (x_j+fh_j) \slash \partial x_i \bigr)$ is the Jacobian matrix. Sending $x_i \in S^{\prime}$ to $x_i \in S$ defines a homomorphism $\psi^{\prime} \colon S^{\prime} \rightarrow S$ making the diagram
 \[
 \xymatrix{ \mathbb{Z}[t_i][f] \ar[d]_{\psi} \ar[r] & S^{\prime} \ar[d]^{\psi^{\prime}} \\ R \ar[r]_-{\varphi} & S}
 \]
 commutative. By choosing the values $\psi(t_i)$, $i >k$ appropriately, we find that any fixed finitely generated subring of $S_f$ is contained in the image of the localization $\psi^{\prime}_f$, so it only remains to show that the above diagram is a good lift for all such $\psi$.
 
 The homomorphism $\mathbb{Z}[t_i][f] \rightarrow S^{\prime}$ is by construction a basic Nisnevich homomorphism along $f$ (to see this, note that $\mathrm{det}(J) \equiv 1 \mod f$, so the desired presentation can be obtained as in Example~\ref{example:localization_basic_Nisnevich}). Moreover, we have $\psi(f)=f$ by construction, so it only remains to check that the function
 \[
 \widetilde{SK}_{1,2}(S^{\prime}) \rightarrow \widetilde{SK}_{1,2}(S^{\prime}_f)
 \]
 has trivial kernel. This follows from Lemma~\ref{lemma:SK12_injective_for_basic_Nisnevich}, applied to the basic Nisnevich homomorphism $\mathbb{Z}[t_i][f] \rightarrow S^{\prime}$ along $f$.
\end{proof}

\begin{cor}\label{cor:good_lifts_henselization}
Let $\varphi \colon A \rightarrow B$ be a ring homomorphism and let $\alpha \in A$. Then the analytic isomorphism $B \rightarrow B \ten{A} A^h_{(\alpha)}$ along $\varphi(\alpha)$ admits enough good lifts.
\end{cor}

\begin{proof}
 From Proposition~\ref{prop:henselization_using_basic_Nisnevich} we know that the henselization $A \rightarrow A^h_{(\alpha)}$ is a filtered colimit of basic Nisnevich homomorphisms along $\alpha$. Since such homomorphisms are stable under base change, the homomorphism $B \rightarrow B \ten{A} A^{h}_{(\alpha)}$ is a filtered colimit of basic Nisnevich homomorphisms along $\varphi(\alpha)$. The conclusion thus follows from Proposition~\ref{prop:good_lifts_for_basic_Nisnevich} and Lemma~\ref{lemma:good_lifts_for_filtered_colimit}.
\end{proof}

\begin{rmk}
Applying Corollary~\ref{cor:good_lifts_henselization} and Part~(a) of Proposition~\ref{prop:good_lifts_implies_weak_excision} to the analytic patching diagram
\[
\xymatrix{A[x,y] \slash (xy-\alpha) \ar[d] \ar[r] & A_{\alpha}[x,x^{-1}] \ar[d] \\ 
A^h_{(\alpha)}[x,y] \slash (xy-\alpha) \ar[r] & A^h_{(\alpha)}[\frac{1}{\alpha}][x,x^{-1}]}
\]
 we find that any patching datum with a pseudoelementary matrix defines a free $A[x,y]\slash (xy-\alpha)$-module. This particular consequence of Corollary~\ref{cor:good_lifts_henselization} can be seen more directly by considering homomorphisms $\psi \colon \mathbb{Z}[t_i][a] \rightarrow A$ sending $a$ to $\alpha$. Indeed, we can then lift the patching datum in question to a patching datum for a module over $\mathbb{Z}[t_i][a][x,y] \slash (xy-a) \cong \mathbb{Z}[t_i][x,y]$, yielding the desired matrix factorization. This does however not show that the matrices appearing in the factorization can be chosen to be pseudoelementary.
\end{rmk}

 We now return to the study of the completability of unimodular rows over rings of the form $A[x,y] \slash (xy-\alpha)$. In favourable cases, we can use the patching diagram
 \[
\xymatrix{A[x,y] \slash (xy-\alpha) \ar[d] \ar[r] & A_{\alpha}[x,x^{-1}] \ar[d] \\ 
A^h_{(\alpha)}[x,y] \slash (xy-\alpha) \ar[r] & A^h_{(\alpha)}[\frac{1}{\alpha}][x,x^{-1}]}
 \]
 to reduce the problem of completing a unimodular row to the question whether a certain $1$-stably elementary matrix over the ring $A^h_{(\alpha)}[\frac{1}{\alpha}][x,x^{-1}]$ is elementary (respectively pseudoelementary if it is a $2 \times 2$-matrix), see Proposition~\ref{prop:elementary_completable_implies_stably_elementary_patching}. 
 
 For a commutative ring $B$ and $r \geq 3$ we introduce the notation 
 \[
 SK_{1,r}^{1-\mathrm{st}}(B) \quad \text{and} \quad \widetilde{SK}_{1,2}^{1-\mathrm{st}}(B)
 \]
  for the subgroup (respectively subset) of $SK_{1,r}(B)$ (respectively $\widetilde{SK}_{1,2}(B)$) consisting of the classes $[\sigma]$ such that $\sigma$ is $1$-stably elementary. This is well-defined since (pseudo)elementary matrices are in particular $1$-stably elementary.
 
 The following lemma allows us to reduce the problem further to a question about the ring $A^{h}_{(\alpha)}[x,x^{-1}]$ (instead of its localization at $\alpha$).
 
 \begin{lemma}\label{lemma:stably_elementary_matrix_factorization}
 Let $\varphi \colon A \rightarrow B$ be an analytic isomorphism along $S$ and let $r \geq 2$. Assume that $SK_{1,r+1}(A) \cong 0$ and $\alpha \in \mathrm{SL}_r(A_S)$ and $\beta \in \mathrm{SL}_r(B)$ are matrices such that the product $(\beta)_{\varphi(S)} \cdot \varphi_S(\alpha) \in \mathrm{SL}_r(B_{\varphi(S)})$ is $1$-stably elementary. Then both $\alpha$ and $\beta$ are $1$-stably elementary.
 \end{lemma}
 
 \begin{proof}
 By assumption, there exists an $\varepsilon \in \mathrm{E}_{r+1}(B_{\varphi(S)})$ such that
 \[
 \left( \begin{smallmatrix} \beta & 0 \\ 0 & 1 \end{smallmatrix} \right)_{\varphi(S)} \cdot \varphi_S \left( \begin{smallmatrix} \alpha & 0 \\ 0 & 1 \end{smallmatrix} \right)=\varepsilon
 \]
 holds. By Vorst's lemma (which is applicable since $r+1 \geq 3$), there exist $\varepsilon_1 \in \mathrm{E}_{r+1}(A_S)$ and $\varepsilon_2 \in \mathrm{E}_{r+1}(B)$ with $\varepsilon=(\varepsilon_2)_{\varphi(S)} \cdot \varphi_S (\varepsilon_1)$. Thus the equality
 \[
 \varphi_S \Bigl( \left( \begin{smallmatrix} \alpha & 0 \\ 0 & 1 \end{smallmatrix} \right) \varepsilon_1^{-1} \Bigr)=\Bigl(  \left( \begin{smallmatrix} \beta^{-1} & 0 \\ 0 & 1 \end{smallmatrix} \right) \varepsilon_2 \Bigr)_{\varphi(S)}
 \]
 holds in $B_S$. Since analytic patching diagrams are pullbacks, it follows that there exists a matrix $\gamma \in \mathrm{SL}_{r+1}(A)=\mathrm{E}_{r+1}(A)$ which is sent to these two matrices by $\lambda_S$ respectively $\varphi$. This shows that $\alpha$ and $\beta^{-1}$, hence $\beta$, are $1$-stably elementary.
 \end{proof}
 
 The following lemma gives the aforementioned reduction step from the ring $A^h_{(\alpha)}[\frac{1}{\alpha}]$ to $A^h_{(\alpha)}$.
 
\begin{lemma}\label{lemma:stably_elementary_factorization_for_henselization_patching_diagram}
 Let $A$ be a commutative ring such that for a fixed $r \geq 2$, there are isomorphisms $P_r(A[x,x^{-1}]) \cong \ast$ and $SK_{r+1}(A[x,x^{-1}]) \cong 0$. Then every $1$-stably elementary matrix over $A^h_{(\alpha)}[\frac{1}{\alpha}][x,x^{-1}]$ factors as a product of a pair of a $1$-stably elementary matrix over $A^h_{(\alpha)}[x,x^{-1}]$ and a $1$-stably elementary matrix over $A_{\alpha}[x,x^{-1}]$.
\end{lemma}

\begin{proof}
 Applying the (argument of the) free patching lemma to the analytic patching diagram
 \[
 \xymatrix{A[x,x^{-1}] \ar[r] \ar[d] & A_{\alpha}[x,x^{-1}] \ar[d] \\ A^h_{(\alpha)}[x,x^{-1}] \ar[r] & A^h_{(\alpha)}[\frac{1}{\alpha}][x,x^{-1}] }
 \]
 associated to $A[x,x^{-1}] \rightarrow A^h_{(\alpha)}[x,x^{-1}]$, we find that every $1$-stably elementary matrix $\gamma \in \mathrm{SL}_r(A^h_{(\alpha)}[\frac{1}{\alpha}][x,x^{-1}])$ factors as a product of a matrix $\gamma_1$ in $\mathrm{SL}_{r}(A_{\alpha}[x,x^{-1}])$ and $\gamma_2 \in \mathrm{SL}_r(A^{h}_{(\alpha)}[x,x^{-1}])$. Since $\gamma$ is $1$-stably elementary, so are $\gamma_1$ and $\gamma_2$ by Lemma~\ref{lemma:stably_elementary_matrix_factorization}.
\end{proof}

 The following proposition establishes a relationship between $1$-stably elementary matrices over $A^h_{(\alpha)}[x,x^{-1}]$ and the corresponding matrices modulo $\alpha$. This is based on the generalization of Hensel's lemma (see Theorem~\ref{thm:generalized_hensel_lemma}). More precisely, we use its Corollary~\ref{cor:localization_of_monic_polynomials_jacobson_radical}, which tells us that $\alpha$ lies in the Jacobson radical of the ring $A^h_{(\alpha)} \langle x \rangle_{1+\alpha A^h_{(\alpha)}[x]}$.
 
  We introduce the notation $\widetilde{SK}_{1,r} \defl SK_{1,r}$ for $r \geq 3$ in order to streamline the statement and the proof of the next proposition. Note that this is in accordance with the definition of pseudoelementary matrices since we have $\mathrm{SL}_r(\mathbb{Z}[t_i])=\mathrm{E}_r(\mathbb{Z}[t_i])$ for all $r \geq 3$. In other words, all pseudoelementary matrices of size $\geq 3$ are elementary.
  
 \begin{prop}\label{prop:stably_elementary_patching_trivial_if_trivial_modulo_alpha}
 Let $A$ be a noetherian commutative ring, let $\alpha \in A$, and let $r \geq 2$ be fixed. Assume that the following conditions hold:
 \begin{enumerate}
 \item[(1)] There are isomorphisms $P_r(A[x,x^{-1}]) \cong \ast$ and $SK_{r+1}(A[x,x^{-1}]) \cong 0$;
 \item[(2)] The local rings of $A$ are unramified regular local rings;
 \item[(3)] We have $\widetilde{SK}_{1,r}^{1-\mathrm{st}}(A \slash \alpha) \cong \ast$ and the image of
 \[
 \widetilde{SK}_{1,r}^{1-\mathrm{st}}(A \slash \alpha[s,s^{-1}] ) \rightarrow \widetilde{SK}_{1,r}^{1-\mathrm{st}}(A \slash \alpha \langle s \rangle)
 \]
 is trivial.
 \end{enumerate}
 Then $\widetilde{SK}_{1,r}^{1-\mathrm{st}}(A_{\alpha}[x,x^{-1}]) \rightarrow \widetilde{SK}_{1,r}^{1-\mathrm{st}}(A^h_{(\alpha)}[\frac{1}{\alpha}][x,x^{-1}])$ is surjective.
 \end{prop}
 
 \begin{proof}
 From Condition~(1) it follows that Lemma~\ref{lemma:stably_elementary_factorization_for_henselization_patching_diagram} is applicable, which reduces the problem to checking that $\widetilde{SK}_{1,r}^{1-\mathrm{st}}(A^h_{(\alpha)}[x,x^{-1}]) \cong \ast$. To see this, we first show that
 \begin{equation}\label{eqn:sk1r_trivial}
 \widetilde{SK}_{1,r}^{1-\mathrm{st}}\bigl(A^h_{(\alpha)}[t]\bigr) \cong \ast \tag{$\star$}
 \end{equation}
 holds.
 
 Note that Condition~(2) implies that the local rings of $A^h_{(\alpha)}$ are unramified regular local rings as well. Indeed, from \cite[Corollary~6.9]{GRECO} it follows that $A^h_{(\alpha)}$ is noetherien, so $A \rightarrow A^h_{(\alpha)}$ is a regular ring map as filtered colimit of smooth ring maps (see \cite[\href{https://stacks.math.columbia.edu/tag/07EP}{Lemma 07EP}]{stacks-project}). Since ``regular'' is a local property (see \cite[\href{https://stacks.math.columbia.edu/tag/07C0}{Lemma 07C0}]{stacks-project}) and regular maps between noetherian rings compose (see \cite[\href{https://stacks.math.columbia.edu/tag/07QI}{Lemma 07QI}]{stacks-project}), Condition~(2) does imply that the local rings of $A^h_{(\alpha)}$ are unramified regular local rings, as claimed.
 
 Thus by Proposition~\ref{prop:SK_1r_regularity_for_dvr_and_unramified} and Theorem~\ref{thm:Vorsts_Theorem_for_2x2_matrices}, evaluation in $0$ induces an isomorphism $\widetilde{SK}_{1,r}^{1-\mathrm{st}}(A^h_{(\alpha)}[t])  \cong \widetilde{SK}_{1,r}^{1-\mathrm{st}}(A^h_{(\alpha)}) $. Using the $j$-invariance of $\widetilde{SK}_{1,r}$ (see Lemma~\ref{lemma:SL_r_j_invariant} and Proposition~\ref{prop:pseudoelementary_j_invariant}), it suffices to check that $\widetilde{SK}_{1,r}^{1-\mathrm{st}}(A^h_{(\alpha)} \slash \alpha) \cong \ast $ holds. This follows from the isomorphism $A^h_{(\alpha)} \slash \alpha \cong A \slash \alpha$ and the first part of Condition~(3). This concludes the proof of Claim~\eqref{eqn:sk1r_trivial}.
 
 Next we fix an element $[\sigma] \in \widetilde{SK}_{1,r}^{1-\mathrm{st}}(A^h_{(\alpha)}[x,x^{-1}])$. We claim that it suffices to check that its image in $\widetilde{SK}_{1,r}^{1-\mathrm{st}}(A^h_{(\alpha)} \langle x^{-1} \rangle) $ is trivial in order to conclude that $[\sigma]=[\id]$. Indeed, if the image of $[\sigma]$ is trivial, then the Zariski patching diagram
 \[
 \xymatrix{A^h_{(\alpha)}[x] \ar[r] \ar[d] & A^h_{(\alpha)}[x,x^{-1}] \ar[d] \\ A^h_{(\alpha)}[x]_{1+x A^h_{(\alpha)}[x]} \ar[r] & A^h_{(\alpha)}\langle x^{-1} \rangle}
 \]
 implies that $[\sigma]$ is the image of some $[\tau] \in \widetilde{SK}_{1,r}^{1-\mathrm{st}}(A^h_{(\alpha)}[x])$ (here we use Theorem~\ref{thm:pseudoelementary_weak_Zariski_excision} in the case $r=2$). Moreover, from the monic inversion principle for $\widetilde{SK}_{1,r+1}=SK_{1,r+1}$ it follows that $\tau$ is $1$-stably elementary. Claim~\eqref{eqn:sk1r_trivial} above thus shows that $[\tau]$, hence $[\sigma]$, is trivial.
 
 To show that the image of $[\sigma]$ in $\widetilde{SK}_{1,r}^{1-\mathrm{st}}(A^h_{(\alpha)} \langle x^{-1} \rangle)$ is trivial, we use yet another Zariski patching diagram. Writing $s=x^{-1}$, the diagram
 \[
 \xymatrix{ A^h_{(\alpha)}[s] \ar[r] \ar[d] & A^h_{(\alpha)} \langle s \rangle \ar[d] \\ 
 A^h_{(\alpha)}[s]_{1+ \alpha A^h_{(\alpha)}[s]}  \ar[r] &A^h_{(\alpha)} \langle s \rangle_{1+\alpha A^h_{(\alpha)}[s]} }
 \]
 is a Zariski patching diagram by the top-bottom lemma (see Lemma~\ref{lemma:top_bottom_lemma}). The top horizontal homomorphism is again given by a localization at monic polynomials, so the same argument as above reduces the problem to showing that the image of $[\sigma]$ in the pointed set $\widetilde{SK}_{1,r}^{1-\mathrm{st}} \bigl(A^h_{(\alpha)} \langle s \rangle_{1+\alpha A^h_{(\alpha)}[s]} \bigr)$ is trivial.
 
 Finally, we use the $j$-invariance of $\widetilde{SK}_{1,r}$ and the fact that $\alpha$ lies in the Jacobson radical of $A^h_{(\alpha)} \langle s \rangle_{1+\alpha A^h_{(\alpha)}[s]}$ (see Corollary~\ref{cor:localization_of_monic_polynomials_jacobson_radical}) to reduce the problem to checking that the image of $[\sigma]$ in $\widetilde{SK}_{1,r}^{1-\mathrm{st}}(A \slash \alpha \langle s \rangle)$ is trivial. This follows from the commutative diagram
 \[
 \xymatrix{ A^h_{(\alpha)}[x, x^{-1}] \ar[r] \ar[d] & A^h_{(\alpha)}\langle x^{-1} \rangle \ar[r] \ar[d] & A^h_{(\alpha)} \langle x^{-1} \rangle_{1+\alpha A^h_{(\alpha)}[x^{-1}]} \ar[ld] \\ A \slash \alpha [x,x^{-1}]  \ar[r] & A\slash \alpha \langle x^{-1} \rangle}
 \]
 and the second part of Condition~(3).
 \end{proof}
\section{Unipotent unimodular rows and symplectic K-theory}\label{section:symplectic}

 We call a unimodular row $(a_1, \ldots, a_r) \in \mathrm{Um}_r(A)$ \emph{unipotent} if $a_1$ is a unipotent element modulo the ideal generated by $(a_2, \ldots, a_r)$. The Hermite ring conjecture implies that all unipotent unimodular rows are completable (see \cite[Proposition~6.(2)]{ROITMAN}). The same is true if $(r-1)!$ is a unit in $A$ by \cite[Proposition~6.(1)]{ROITMAN}. One of the simplest cases where this problem is open is therefore when $r=3$ and when $A$ is an $\mathbb{F}_2$-algebra. In \cite[\S VIII.5]{LAM}, this specific case is called ``Murthy's $(a,b,c)$-problem.'' The goal of this section is to relate this problem to a stable question (concerning the interaction of symplectic and ordinary $K$-theory) and to settle this stable question in some particular cases.
 
  We first note that the problem can be reduced to the following universal example once a ground field $k$ has been fixed: if the row $(1+a,x_1,x_2)$ over the $k$-algebra
  \[
  A=k[a,x_1,y_1,x_2,y_2] \slash (\textstyle\sum_{i=1}^2 x_i y_i - a^{\ell})
  \]
  is completable for all $\ell \geq 2$, then all unipotent unimodular rows $(a,b,c) \in \mathrm{Um}_r(B)$ are completable over all $k$-algebras $B$.
  
  As in \S \ref{section:pseudoelementary}, we write $\widetilde{SK}_{1,r}^{1-\mathrm{st}}(R)$ for $SK_{1,r}^{1-\mathrm{st}}(R)$ and $\widetilde{\mathrm{E}}_r(R)$ for $\mathrm{E}_r(R)$ in the case where $r >2$ to treat all $r \geq 2$ uniformly.
  
 \begin{prop}\label{prop:universal_unipotent_unimodular_row}
 Let $\ell \geq 2$ and let $k$ be a field. Suppose the following holds for some $r \geq 2$: there are isomorphisms
 \[
 \widetilde{SK}_{1,r}^{1-\mathrm{st}} \bigl(k[a,x,y] \slash (xy-a^{\ell}) \bigr) \cong \ast \quad \text{and} \quad
 \widetilde{SK}_{1,r}^{1-\mathrm{st}} \bigl(k\langle s \rangle [a,x,y] \slash (xy-a^{\ell}) \bigr) \cong \ast \smash{\rlap{.}}
 \]
 Then all unimodular rows of length $r+1$ over the ring
 \[
 k[a,x_1,y_1,x_2,y_2] \slash (\textstyle\sum_{i=1}^2 x_i y_i - a^{\ell})
 \]
 are completable.
 \end{prop}
 
 \begin{proof}
 Let $A=k[a,x_1,y_1]$ and $\alpha=a^{\ell}-x_1 y_1$ so that
 \[
 \xymatrix{  k[a,x_1,y_1,x_2,y_2] \slash (\textstyle\sum_{i=1}^2 x_i y_i - a^{\ell}) \ar[r] \ar[d] & A_{\alpha}[x_2,x_2^{-1}] \ar[d] \\ A^h_{(\alpha)}[x_2,y_2] \slash (x_2 y_2 - \alpha) \ar[r] & A^h_{(\alpha)}[\frac{1}{\alpha}][x_2,x_2^{-1}] }
 \]
 is an analytic patching diagram. Since $k$ is $W_{r+1}$-regular, so is $A_{\alpha}[x_2,x_2^{-1}]$ (by the Roitman principle for $W_{r+1}$). The naive $\mathbb{A}^{1}$-contractibility of 
 \[
 k[a,x_i,y_i] \slash \textstyle\sum_{i=1}^2 x_i y_i-a^{\ell}
 \]
 implies that each unimodular row (of length $r+1$) over this ring becomes elementary completable over the localization $A_{\alpha}[x_2,x_2^{-1}]$ at $\alpha$.
 
 We claim that we also have $W_{r+1} \bigl(A^h_{(\alpha)})[x_2,y_2] \slash (x_2y_2 - \alpha) \bigr) \cong \ast$. To see this, we will apply Theorem~\ref{thm:henselian_implies_trivial_kernel_for_overring} and its Corollary~\ref{cor:henselian_implies_trivial_kernel_for_overring}. The functor $W_{r+1}$ satisfies weak analytic and weak Milnor excision (see Propositions~\ref{prop:W_r_weak_analytic_excision} and \ref{prop:weak_milnor_excision_group_case}), it is partially $j$-injective (see Example~\ref{example:partially_j_injective}), and $A[x_2]$ is $W_{r+1}$-contractible. It only remains to check that $A \slash \alpha [x]$ is $W_{r+1}$-contractible.
 
 To see this, we use the isomorphism
 \[
 A \slash \alpha [x] \cong k[s^{\ell},st, t^{\ell},x] \smash{\rlap{,}}
 \]
 where the right hand side denotes the monoid algebra $k[M] \subseteq k[s,t,x]$ generated by $s^{\ell}, st, t^{\ell}, x$. Clearly $M$ is cancellative as a submonoid of the cancellative monoid $\mathbb{N}^{3}$. Thus Gubeladze's result \cite[Theorem~1.1]{GUBELADZE_UNIMODULAR} is applicable and implies that $W_{r+1}(A\slash \alpha [x]) \cong \ast$.
 
  From Proposition~\ref{prop:elementary_completable_implies_stably_elementary_patching} it follows that the patching data corresponding to the projective modules arising as kernels of unimodular rows of length $r+1$ over the ring
  \[
 k[a,x_i,y_i] \slash \textstyle\sum_{i=1}^2 x_i y_i-a^{\ell}
  \]
  are given by $1$-stably elementary matrices $\sigma \in \mathrm{SL}_r(A^h_{(\alpha)}[\frac{1}{\alpha}][x_2,x_2^{-1}])$. We claim that Proposition~\ref{prop:stably_elementary_patching_trivial_if_trivial_modulo_alpha} is applicable in our case. Since $A$ is a polynomial algebra over a field, we have $P_r(A[x_2,x_2^{-1}]) \cong \ast$ and $SK_{r+1}(A[x_2,x_2^{-1}]) \cong 0$. Moreover, the local rings of $A$ are unramified regular local rings since they contain the field $k$. This shows that Conditions~(1) and (2) of Proposition~\ref{prop:stably_elementary_patching_trivial_if_trivial_modulo_alpha} are satisfied.
  
  In order to check Condition~(3), we need the assumption of the current proposition. The first condition tells us that $\widetilde{SK}_{1,r}^{1-\mathrm{st}}(A \slash \alpha) \cong \ast$. To see the second part of Condition~(3), note that we have a factorization
  \[
  A \slash \alpha[s,s^{-1}] \rightarrow k\langle s \rangle[a,x,y] \slash (xy-a^{\ell}) \rightarrow A \slash \alpha \langle s \rangle
  \]
  so the assumption that $\widetilde{SK}_{1,r}^{1-\mathrm{st}} \bigl(k\langle s\rangle[a,x,y]\slash (xy-a^{\ell}) \bigr) \cong \ast$ implies that the image of the induced map on $\widetilde{SK}_{1,r}^{1-\mathrm{st}}$ is trivial.
  
  It follows from Proposition~\ref{prop:stably_elementary_patching_trivial_if_trivial_modulo_alpha} that
  \[
  \widetilde{SK}_{1,r}^{1-\mathrm{st}}(A_{\alpha}[x_2,x_2^{-1}]) \rightarrow \widetilde{SK}_{1,r}^{1-\mathrm{st}}(A^h_{(\alpha)}[\textstyle \frac{1}{\alpha}][x_2,x_2^{-1}])
  \]
  is surjective. Applying this to $\sigma^{-1}$ we find that there exists a $\tau \in \mathrm{SL}_r(A_{\alpha}[x_2,x_2^{-1}])$ and a (pseudo)elementary matrix $\varepsilon \in \widetilde{\mathrm{E}}_r(A^h_{(\alpha)}[\frac{1}{\alpha}][x_2,x_2^{-1}])$ such that $\sigma=\varepsilon \varphi_{\alpha}(\tau)$ holds, where $\varphi$ denotes the homomorphism
  \[
  \varphi \colon A[x_2,y_2] \slash (x_2y_2 -\alpha) \rightarrow A^h_{(\alpha)}[x_2,y_2] \slash (x_2y_2 -\alpha)
  \]
  arising from the henselization of $A$ at $(\alpha)$.
  
  From Vorst's lemma in the case $r \geq 3$ (respectively from Corollary~\ref{cor:good_lifts_henselization} and Part~(a) of Proposition~\ref{prop:good_lifts_implies_weak_excision}) it follows that there exists a factorization 
  \[
  \varepsilon=(\varepsilon_1)_{\varphi(\alpha)} \cdot \varphi_{\alpha}(\varepsilon_2)
  \] 
  for some $\varepsilon_1 \in \widetilde{\mathrm{E}}_r \bigl(A^h_{(\alpha)}[x_2,y_2] \slash (x_2 y_2- \alpha)\bigr)$ and $\varepsilon_2 \in \widetilde{\mathrm{E}}_r(A_{\alpha}[x_2,x_2^{-1}])$. This shows that the patching datum $\sigma$ is isomorphic to $\id$, so the unimodular row corresponding to $\sigma$ is completable. Since $\sigma$ was an arbitrary $1$-stably elementary $r \times r$-matrix, this implies that all unimodular rows of length $r+1$ are completable, as claimed.
 \end{proof}
 
 The above proposition reduces the study of unimodular rows over
 \[
 k[a,x_1,y_1,x_2,y_2] \slash (\textstyle\sum_{i=1}^2 x_i y_i - a^{\ell})
 \]
 to the study of $1$-stably elementary matrices over the ring
 \[
 k[a,x,y] \slash (xy-a^{\ell})
 \]
 (albeit for two different fields $k$). This includes in particular the unimodular row $(1+a,x_1,x_2)$ appearing in Murthy's $(a,b,c)$-problem. In what follows, we will link the study of $1$-stably elementary matrices over the above ring to a stable question. In order to do this, we need to establish one more type of analytic patching diagram where pseudoelementary matrices behave as elementary matrices of size $\geq 3$ do.
 
 \begin{lemma}\label{lemma:good_lift_for_overring_patching_diagram}
 Let $A$ be a commutative ring, let $\alpha \in A$ be a nonzerodivisor, let $S=1+xA[x] \subseteq A[x]$, and let $\varphi \colon A[x] \rightarrow A[x,y] \slash (xy-\alpha)$ denote the canonical homomorphism. For every $\varepsilon \in \widetilde{\mathrm{E}}_2\Bigl( \bigl(A[x,y] \slash(xy-\alpha)\bigr)_S \Bigr)$, there exist pseudoelementary matrices $\varepsilon_1 \in \widetilde{\mathrm{E}}_2 \bigl(A[x,y]\slash(xy-\alpha)\bigr)$ and $\varepsilon_2 \in \widetilde{\mathrm{E}}_2(A[x]_S)$ such that
 \[
 \varepsilon=(\varepsilon_1)_S \cdot \varphi_S(\varepsilon_2)
 \]
 holds.
 
  Moreover, the diagram
  \[
  \xymatrix{
  \widetilde{SK}_{1,2}^{1-\mathrm{st}}(A[x]) \ar[r] \ar[d]_{\varphi} &  \widetilde{SK}_{1,2}^{1-\mathrm{st}}(A[x]_S) \ar[d]^{\varphi_S} \\  \widetilde{SK}_{1,2}^{1-\mathrm{st}} \bigl(A[x,y] \slash (xy-\alpha)\bigr) \ar[r] &  \widetilde{SK}_{1,2}^{1-\mathrm{st}} \Bigl(\bigl(A[x,y]\slash (xy-\alpha)\bigr)_S\Bigr)
  }
  \]
  is a weak pullback diagram.
 \end{lemma}
 
 \begin{proof}
 The second claim is clear from the first and the fact that analytic patching diagrams are pullback squares.
 
 For the first claim, we use a similar argument to the ones involving \emph{good lifts} in \S \ref{section:pseudoelementary}. Specifically, we consider homomorphisms $\psi \colon \mathbb{Z}[t_i,a,x] \rightarrow A[x]$ with $\psi(a)=\alpha$ and $\psi(x)=x$, while the (finitely many) values of $\psi(t_i)$ are arbitrary. Letting $T=1+x \mathbb{Z}[t_i,a,x]$, we get a commutative diagram
 \[
 \xymatrix@R=20pt@C=5pt{\mathbb{Z}[t_i,a,x] \ar[rr] \ar[rd] \ar[dd] && \mathbb{Z}[t_i,a,x]_T \ar[rd] \ar[dd]|!{[ld];[rd]}\hole \\
 & \mathbb{Z}[t_i,a,x,y]\slash (xy-a) \ar[rr] \ar[dd] && \bigl(\mathbb{Z}[t_i,a,x,y] \slash (xy-a)\bigr)_T \ar[dd]  \\
 A[x] \ar[rr]|!{[ru];[rd]}\hole \ar[rd] && A[x]_S \ar[rd] \\ 
 & A[x,y] \slash (xy-\alpha) \ar[rr] && \bigl( A[x,y] \slash (xy-\alpha) \bigr)_S }
 \]
 and an isomorphism $\mathbb{Z}[t_i,a][x,y] \slash (xy-a) \cong \mathbb{Z}[t_i][x,y]$ given by $a \mapsto xy$. By choosing the values of $\psi(t_i)$ appropriately, we find for each $\varepsilon \in \widetilde{\mathrm{E}}_2\Bigl( \bigl(A[x,y] \slash(xy-\alpha)\bigr)_S \Bigr)$ a $\psi$ such that $\varepsilon$ lifts to a pseudoelementary matrix $\varepsilon^{\prime} \in \widetilde{\mathrm{E}}_2\Bigl( \bigl(\mathbb{Z}[t_i,a,x,y] \slash(xy-a)\bigr)_T \Bigr)$. Since $P_2(\mathbb{Z}[t_i,a,x]) \cong \ast$, this $\varepsilon^{\prime}$ factors as a product of matrices $\sigma_1 \in \mathrm{SL}_2(\mathbb{Z}[t_i][x,y])$ and $\sigma_2 \in \mathrm{SL}_2(\mathbb{Z}[t_i,a,x]_T)$ (by the free patching lemma).
 
 By definition, we have $\sigma_1 \in \widetilde{\mathrm{E}}_2$. On the other hand, the element $x$ lies in the Jacobson radical of $\mathbb{Z}[t_i,a,x]_T$, so $\widetilde{SK}_1(\mathbb{Z}[t_i,a,x]) \cong \ast$ since $\widetilde{SK}_1(\mathbb{Z}[t_i,a]) \cong \ast$ and $\widetilde{SK}_1$ is $j$-invariant (see Proposition~\ref{prop:pseudoelementary_j_invariant}). The images of the $\sigma_i$ under the above vertical homomorphisms thus give the desired factorization of $\varepsilon$.
 \end{proof}
 
 In what follows, we let $A \defl k[a]$ and $B \defl k[a,x,y]\slash(xy-a^{\ell})$, where $k$ denotes a fixed ground field and $\ell \geq 2$ a fixed natural number.
 
 \begin{lemma}\label{lemma:localization_induces_bijection_SK_1,r}
 Let $S = 1+xA[x] \subseteq A[x]$. Then the map
 \[
 \widetilde{SK}_{1,r}^{1-\mathrm{st}} (B) \rightarrow  \widetilde{SK}_{1,r}^{1-\mathrm{st}} (B_S)
 \]
 induced by the localization homomorphism $\lambda_S$ is bijective for all $r \geq 2$. 
 \end{lemma}
 
 \begin{proof}
 Let $\varphi \colon A[x] \rightarrow B$ denote the canonical homomorphism. By Lemma~\ref{lemma:analytic_iso_arising_from_overring}, the diagram
 \[
 \xymatrix{A[x] \ar[d]_{\varphi} \ar[r] & A[x]_S \ar[d]^{\varphi_S} \\ B \ar[r] & B_S}
 \]
 is an analytic patching diagram.
 
 Since $P_r(A[x]) \cong \ast$ for all $r \geq 2$, it follows from the free patching lemma that all $\sigma \in \mathrm{SL}_r(B_S)$ factor as a product $\sigma=(\sigma_1)_{S} \cdot \varphi_S(\sigma_2)$ with $\sigma_1 \in \mathrm{SL}_r(B)$ and $\sigma_2 \in \mathrm{SL}_r(A[x]_S)$. If $\sigma$ is $1$-stably elementary, then so are $\sigma_1$ and $\sigma_2$ because $SK_{1,r+1}(A[x]) \cong 0$ (see Lemma~\ref{lemma:stably_elementary_matrix_factorization}). Finally we note that $x$ lies in the Jacobson radical of $A[x]_S$, so $ \widetilde{SK}_{1,r}^{1-\mathrm{st}}(A[x]_S) \cong  \widetilde{SK}_{1,r}^{1-\mathrm{st}}(A) \cong \ast$ for all $r \geq 2$ by $j$-invariance. Thus $\sigma_2$ is pseudoelementary, which shows that
 \[
 \widetilde{SK}_{1,r}^{1-\mathrm{st}} (B) \rightarrow  \widetilde{SK}_{1,r}^{1-\mathrm{st}} (B_S)
 \]
 is surjective.
 
 To see injectivity, we first claim that it suffices to check that $\sigma \in \mathrm{SL}_r(B)$ is (pseudo)elementary if its image in $\mathrm{SL}_r(B_S)$ is. Indeed, if $[\sigma_S]=[\tau_S]$ for some $\sigma, \tau \in \mathrm{SL}_r(B)$, then $\sigma_S=\tau_S \varepsilon$ for some $\varepsilon \in \widetilde{\mathrm{E}}_r(B_S)$. It follows that $(\tau^{-1}\sigma)_S$ is (pseudo)elementary, hence that $\tau^{-1} \sigma \in \widetilde{\mathrm{E}}_r(B)$, which shows that $[\sigma]=[\tau]$ in $\widetilde{SK}_{1,r}^{1-\mathrm{st}} (B)$, as claimed.
 
 Thus let $[\sigma] \in \widetilde{SK}_{1,r}^{1-\mathrm{st}} (B)$ be an element such that $[\sigma_S]=[\id]$. From the weak pullback diagram
 \[
   \xymatrix{
  \widetilde{SK}_{1,2}^{1-\mathrm{st}}(A[x]) \ar[r] \ar[d]_{\varphi} &  \widetilde{SK}_{1,2}^{1-\mathrm{st}}(A[x]_S) \ar[d]^{\varphi_S} \\  \widetilde{SK}_{1,2}^{1-\mathrm{st}} (B) \ar[r] &  \widetilde{SK}_{1,2}^{1-\mathrm{st}} (B_S)
  }
 \]
 (see Lemma~\ref{lemma:good_lift_for_overring_patching_diagram}) it follows that there exists a $\tau \in \mathrm{SL}_r(A[x])$ and $\varepsilon \in \widetilde{\mathrm{E}}_r(B)$ such that $\sigma=\varphi(\tau) \varepsilon$ holds. Since $\widetilde{SK}_{1,r}(A[x]) \cong \ast$, it follows that $\tau$ and therefore $\sigma$ are both (pseudo)elementary (see Corollary~\ref{cor:pseudoelementary_over_several_Laurent_variables} for $r=2$).
\end{proof}

Next we want to relate these groups of $1$-stably elementary matrices to $1$-stably free modules, by considering the matrices as patching data for a suitable Milnor square. To do this, we need some notation. We write $P_r^{1-\mathrm{st}}(R)$ for the set of isomorphism classes of $1$-stably free projective $R$-modules of rank $r$ (that is, $R$-modules $P$ such that $P \oplus R \cong R^{r+1}$ holds). We write $\mathrm{SL}_r^{1-\mathrm{st}}(R)$ for the group of $1$-stably elementary matrices with coefficients in $R$.

 As before, we let $S \defl 1+x k[a,x] \subseteq k[a,x]$ and we let $B \defl k[a,x,y] \slash (xy-a^{\ell})$. Finally, we let $C \defl k[a,x]_S[y,t] \slash \bigl(t^2-t(xy-a^{\ell})\bigr)$.
 
 \begin{prop}\label{prop:stably_elementary_trivial_iff_stably_free_trivial}
 For a fixed natural number $r \geq 2$, we have $P_r^{1-\mathrm{st}}(C) \cong \ast$ if and only if $\widetilde{SK}_{1,r}^{1-\mathrm{st}}(B_S) \cong \ast$ holds.
 \end{prop}
 
 \begin{proof}
 We first note that
 \[
 \xymatrix{C \ar[r] \ar[d] & k[a,x]_S[y] \ar[d]^{\pi} \\ k[a,x]_S[y] \ar[r]_-{\pi} & B_S}
 \]
 is a Milnor patching diagram. To see this, we can consider the two ideals $I=(t)$ and $J=\bigl(t-(xy-a^{\ell})\bigr)$ and note that $I \cap J=I\cdot J$ (for example since $k[a,x]_S[y,t]$ is a unique factorization domain). 
 
 Moreover, we have $P_r(k[a,x]_S[y]) \cong P_r(k[a,x]_S)$ (since $P_r$ satisfies the Roitman principle $(\mathrm{R})$) and $P_r(k[a,x]_S) \cong \ast$ since $P_r(k[a]) \cong \ast$ and $P_r$ is $j$-injective (see Example~\ref{example:P_r_is_j_injective_and_h_invariant}). Similarly we find that
 \[
 \widetilde{SK}_{1,r}^{1-\mathrm{st}}(k[a,x]_S[y]) \cong \widetilde{SK}_{1,r}^{1-\mathrm{st}}(k[a,x]_S) \cong \widetilde{SK}_{1,r}^{1-\mathrm{st}}(k[a]) \cong \ast
 \]
 holds, since $\widetilde{SK}_{1,r}^{1-\mathrm{st}}$ satisfies $(\mathrm{R})$ and is $j$-invariant.
 
 From this, we get the implication ``$\Rightarrow$'': if $\sigma$ is $1$-stably elementary over $B_S$, then the patching datum
 \[
 (k[a,x]_S[y]^r,\; k[a,x]_S[y]^r,\; \sigma)
 \]
 defines a $1$-stably free $C$-module $P$. By assumption, $P$ is free, so from the free patching lemma (see Lemma~\ref{lemma:free_patching}) and the above result that $\widetilde{SK}_{1,r}^{1-\mathrm{st}}(k[a,x]_S[y]) \cong \ast$ it follows that $\sigma$ is (pseudo)elementary.
 
 To see the converse, let $P$ be a projective $C$-module such that $P \oplus C \cong C^{r+1}$. This corresponds to a patching datum
 \[
 (k[a,x]_S[y]^r,\; k[a,x]_S[y]^r,\; \tau)
 \]
 for some $\tau \in \mathrm{GL}_r(B_S)$.
 
 We claim that $\pi \colon k[a,x]_S[y] \rightarrow B_S$ induces a surjection on units. This follows from the analytic patching diagram
 \[
 \xymatrix{k[a,x] \ar[d] \ar[r] & k[a,x]_S \ar[d] \\ B \ar[r] & B_S}
 \]
 and the fact that the Picard group of $k[a,x]$ is trivial, combined with the observation that $B^{\times}=k^{\times}$ ($B=k[a,x,y] \slash(xy-a^{\ell})$ is reduced and graded, with degree $0$ part equal to $k$). This implies that the above patching datum is isomorphic to one where $\tau$ lies in $\mathrm{SL}_r(B_S)$.
 
 Since the patching datum $\bigl( \begin{smallmatrix} \tau & 0 \\ 0 & 1 \end{smallmatrix} \bigr)$ defines $P \oplus C \cong C^{r+1}$, it follows from the free patching lemma that $\bigl( \begin{smallmatrix} \tau & 0 \\ 0 & 1 \end{smallmatrix} \bigr)$ is elementary, so $\tau \in \mathrm{SL}_r^{1-\mathrm{st}}(B_S)$. The assumption therefore implies that $\tau$ is (pseudo)elementary. From this it follows that $P$ is free, as claimed.
 \end{proof}
 
 The following lemma shows that the ring $C$ is ``almost'' 2-dimensional. This result heavily relies on the fact that we have localized at the set $S=1+xk[a,x]$.
 
 \begin{lemma}\label{lemma:Specm(C)_covered_by_2-dimensional_subspaces}
 The noetherian topological space $\Spec_m(C)$ of maximal ideals in $C$ is covered by two subspaces of Krull dimension $2$.
 \end{lemma}
 
 \begin{proof}
 The space $\Spec_m(C)$ is covered by $D(x)=\Spec(C_x)$ and its closed complement $V(x)=\Spec(C \slash x)$. Since the dimension of a subspace is at most the dimension of the total space, it suffices to check that the two rings $C_x$ and $C \slash x$ have Krull dimension $2$.
 
 The quotient $C \slash x \cong k[a,y,t] \slash (t^2+ta^{\ell})$ is integral over $k[a,y]$, hence has Krull dimension $2$. The localization
 \[
C_x \cong k[a,x]_{S,x}[y,t] \slash \bigl(t^2-t(xy-a^{\ell})\bigr) 
 \]
 is integral over $k[a,x]_{S,x}[y]$, so it suffices to check that the latter ring has Krull dimension $2$. Since $k[a,x]_{S,x} \cong k[a]\langle t \rangle$ is a principal ideal domain, hence noetherian of Krull dimension $1$, it follows that $k[a,x]_{S,x}[y]$ has Krull dimension $1+1=2$.
\end{proof}

 It thus follows from \cite[Theorem~V.3.5]{BASS_KTHEORY} that the ring $C$ satisfies Bass' stable range condition $\mathrm{SR}_4$. This, in turn, implies that all unimodular rows over $C$ of length $\geq 4$ are elementary completable. This consequence can also be seen more directly, by mixing Bass' inductive argument with Lam's proof of \cite[Theorem~II.7.3']{LAM}.
 
 \begin{prop}\label{prop:dimension_implies_transitive_elementary_action}
 Let $R$ be a commutative ring such that $\Spec_m(R)$ is noetherian and a union of a finite number of subspaces of Krull dimension $\leq d$. Then $\mathrm{E}_r(R)$ acts transitively on $\mathrm{Um}_{r}(R)$ for all $r \geq d+2$.
 \end{prop}
 
 \begin{proof}
 Assume $\Spec_m(R)=\bigcup_{i=1}^n X_i$ and $\mathrm{dim}(X_i) \leq d$. Fix some $r \geq d+2$ and a unimodular row $(a_1, \ldots, a_r)$ over $R$. Now let $D \subseteq \Spec_m(R)$ be a finite subset which contains (at least) one point of each irreducible component of all the $X_i$.
 
 Let $I=(a_2, \ldots, a_r)$ and note that $(a_1,I) \not \subseteq \mathfrak{m}$ for all $\mathfrak{m} \in D$. It follows from \cite[Lemma~II.7.4]{LAM} that $a_1+I \not \subseteq \bigcup_{\mathfrak{m} \in D} \mathfrak{m}$, that is, there exist $b_2, \ldots, b_r \in R$ such that $a_1^{\prime}=a_1+\sum_{i=2}^r b_i a_i$ does not lie in $\bigcup_{\mathfrak{m}\in D} \mathfrak{m}$. Thus if we let $X^{\prime}=\Spec_m(R \slash a_1^{\prime})$, then every $X_i \cap X^{\prime}$ has Krull dimension $\leq d-1$. By iterating this, we find that there are elementary transformations
 \[
 (a_1, \ldots, a_r) \sim_{\mathrm{E}_r} (a_1^{\prime}, \ldots, a_d^{\prime},a_{d+1}, \ldots, a_r)
 \]
 such that $\Spec_m\bigr(R \slash (a_1^{\prime}, \ldots, a_d^{\prime}) \bigl)$ is covered by a finite number of subspaces of Krull dimension $0$. This implies that $R \slash (a_1^{\prime}, \ldots, a_d^{\prime})$ has only a finite number of maximal ideals, hence is semilocal.
 
 Using the same prime avoidance argument, we can find $c_{d+2}, \ldots, c_r \in R$ such that $a_{d+1}^{\prime}=a_{d+1} + \sum_{i=d+2}^r c_i a_i$ is a unit in $R \slash (a_1^{\prime}, \ldots, a_d^{\prime})$. Since we assumed that $r \geq d+2$, this implies that $(a_{d+1}^{\prime}, a_{d+2}, \ldots, a_r)$ is elementary completable in $R \slash (a_1^{\prime}, \ldots, a_d^{\prime})$. From \cite[Proposition~I.5.6]{LAM} it follows that
 \[
 (a_1^{\prime}, \ldots, a_d^{\prime}, a_{d+1}^{\prime},a_{d+2}, \ldots, a_r)
 \]
 is elementary completable over $R$.
 \end{proof}
 
 \begin{cor}\label{cor:unimodular_rows_length_geq4_completable}
 Let $k$ be any field and let $\ell \geq 2$. Then all unimodular rows over the ring
 \[
R= k[a,x_1,y_1,x_2,y_2] \slash (x_1 y_1 + x_2 y_2 - a^{\ell})
 \]
 of length $\geq 4$ are completable.
 \end{cor}
 
 \begin{proof}
 From the above proposition and Lemma~\ref{lemma:Specm(C)_covered_by_2-dimensional_subspaces} we know that $P_r^{1-\mathrm{st}}(C) \cong \ast$ for all $r \geq 3$. This implies that $SK_{1,r}^{1-\mathrm{st}}(B_S) \cong 0$ for all $r \geq 3$ (see Proposition~\ref{prop:stably_elementary_trivial_iff_stably_free_trivial}). From the isomorphism $SK_{1,r}^{1-\mathrm{st}}(B) \cong SK_{1,r}^{1-\mathrm{st}}(B_S)$ (see Lemma~\ref{lemma:localization_induces_bijection_SK_1,r}) it follows that all $1$-stably elementary matrices of size $\geq 3$ over the ring $B=k[a,x,y] \slash (xy-a^{\ell})$ are elementary.
 
 Since the ground field in the above argument was arbitrary, we can apply Proposition~\ref{prop:universal_unipotent_unimodular_row} to conclude that all unimodular rows of length $\geq 4$ over the ring $R$ are completable.
 \end{proof}
 
 It is well-known that the unimodular row $(1-a,x_1,x_2)$ over the above ring $R$ is completable if $2$ is a unit in $R$, that is, if $\mathrm{char}(k) \neq 2$. The following proposition shows that in fact \emph{all} unimodular rows over $R$ are completable in this case.
 
 \begin{prop}\label{prop:stably_free_trivial_if_char_neq_2}
 If $k$ is a field with $\mathrm{char}(k) \neq 2$ and $\ell \geq 2$ is a natural number, then all stably free modules over the ring
 \[
R= k[a,x_1,y_1,x_2,y_2] \slash (x_1 y_1 + x_2 y_2 - a^{\ell})
 \] 
 are free. 
 \end{prop}
 
 \begin{proof}
 Let $C^{\prime}=k[a,x,y,t] \slash \bigl(t^2-t(xy-a^{\ell})\bigr)$. This ring can be considered as a graded ring by setting the degree of $t$ to be $\ell$, the degree of $x$ to be $\ell-1$, and the degree of $y$ to be $1$ (for example). We first claim that all 1-stably free $C^{\prime}$-modules are free.
 
 To see this, we note that the grading above implies that $C^{\prime}$ is naively $\mathbb{A}^{1}$-contractible over $k$, so it suffices to show that all $1$-stably free $C^{\prime}[t]$-modules are extended from $C^{\prime}$. From Quillen's local-global principle $(\mathrm{Q})$ it follows that it suffices to check this for the local rings $C^{\prime}_{\mathfrak{m}}$ of $C^{\prime}$. Since $C^{\prime}$ is integral over $k[a,x,y]$, it has Krull dimension $3$, so the local rings have Krull dimension at most $3$ and $2$ is a unit in each of these local rings. The claim thus follows from \cite[Theorem~3.1]{RAO_UNIMODULAR}.
 
 From the Milnor patching square
 \[
 \xymatrix{C^{\prime} \ar[r] \ar[d] & k[a,x,y] \ar[d] \\ k[a,x,y] \ar[r] & B}
 \]
 and the free patching lemma it follows that $\widetilde{SK}_{1,2}^{1-\mathrm{st}}(B) \cong \ast$, where 
 \[
 B=k[a,x,y] \slash (xy-a^{\ell}) \smash{\rlap{.}}
 \] 
  We can apply the same reasoning with the field $k \langle s \rangle$ in place of $k$, hence Proposition~\ref{prop:universal_unipotent_unimodular_row} is applicable and shows that all unimodular rows of length $3$ over the ring $R$ are completable. Combined with Corollary~\ref{cor:unimodular_rows_length_geq4_completable}, this shows that all stably free $R$-modules are free.
 \end{proof}
 
 Thus we can focus on fields of characteristic $2$ from now on. Many of the arguments we give are however independent of the characteristic of the ground field, so we will make this assumption only where it is necessary. The remaining case of rows of length $3$ (and ground fields of characteristic $2$) can be linked to a stable question about the ring $C$. In order to state this, we need to recall some basic definitions of symplectic $K$-theory.
 
 A \emph{symplectic $R$-module} is a finitely generated projective $R$-module $P$, equipped with a nonsingular bilinear form $h$ (that is, $h$ induces an isomorphism between $P$ and $P^{\vee}$) which is moreover \emph{alternating}: $h(x,x)=0$ for all $x \in P$. We denote the category of symplectic $R$-modules and morphisms which preserve the alternating forms by $\Alt(R)$. The orthogonal direct sum equips $\Alt(R)$ with the structure of a symmetric monoidal category. The symplectic $K$-group $K_0 \mathrm{Sp}$ is defined to be the Grothendieck group of the monoid of isomorphism classes in $\Alt(R)$ with respect to orthogonal direct sum as addition. Forgetting the bilinear form defines a functor $\Alt(R) \rightarrow \Proj(R)$ which induces a corresponding homomorphism $K_0 \mathrm{Sp}(R) \rightarrow K_0(R)$. Following \cite{BASS_LIBERATION}, we denote its kernel by $W(R)$. In the following proposition, we still use the notation $B \defl k[a,x,y] \slash (xy-a^{\ell})$ and 
 \[
 C\defl k[a,x]_S[y,t] \slash \bigl(t^2-t(xy-a^{\ell})\bigr) \smash{\rlap{,}}
 \]
 where $S$ denotes the multiplicative set $1+xk[a,x]$.
 
 \begin{prop}\label{prop:link_between_stable_and_unstable}
 Let $k$ be a field. The pointed set $\widetilde{SK}_{1,2}^{1-\mathrm{st}}(B)$ is trivial if and only if the kernel $W(C)$ of
 \[
 K_0 \mathrm{Sp}(C) \rightarrow K_0(C) 
 \]
 is trivial. 
 \end{prop}
 
 \begin{proof}
 From Lemma~\ref{lemma:Specm(C)_covered_by_2-dimensional_subspaces} and Proposition~\ref{prop:dimension_implies_transitive_elementary_action} we know that $\mathrm{E}_r(C)$ acts transitively on $\mathrm{Um}_r(C)$ for all $r \geq 4$. Thus we can apply Bass' version of Vaserstein's theorem \cite[Th{\'e}or{\`e}me~4]{BASS_LIBERATION}, which implies that there is a bijection between $W(C)$ and $\mathrm{Um}_3(C) \slash \mathrm{SL}_3(C)$. Since a unimodular row is completable to an invertible matrix if and only if it is completable to a matrix in $\mathrm{SL}_3$, it follows that $W(C) \cong 0$ if and only if $P_2^{1-\mathrm{st}}(C) \cong \ast$. From Proposition~\ref{prop:stably_elementary_trivial_iff_stably_free_trivial} we know that the latter is equivalent to $\widetilde{SK}_{1,2}^{1-\mathrm{st}}(B_S) \cong \ast$. The claim thus follows from the isomorphism
 \[
 \widetilde{SK}_{1,2}^{1-\mathrm{st}}(B) \cong \widetilde{SK}_{1,2}^{1-\mathrm{st}}(B_S)
 \]
 of Lemma~\ref{lemma:localization_induces_bijection_SK_1,r}.
 \end{proof}
 
 Next we want to relate this kernel $W(C)$ to a similar kernel for a simpler ring than $C$ (at the expense of passing to higher symplectic $K$-groups). This requires a short detour on basic properties of symplectic $K$-groups, specifically on long exact sequences relating the two groups $K_0 \mathrm{Sp}$ and $K_1 \mathrm{Sp}$.
 
 We start with a version of Vorst's lemma for the first symplectic $K$-group. Recall that $\chi_r$ is the $2r \times 2r$-matrix
 \[
 \chi_r \defl \begin{pmatrix}
 0 & 1 \\ 
 -1 & 0 \\
 && 0 & 1\\
 && -1 & 0 \\
 &&&& \ddots \\
 &&&&& 0 & 1\\
 &&&&& -1 & 0
 \end{pmatrix}
 \]
 and the group of symplectic matrices is given by
 \[
 \mathrm{Sp}_{2r}(R) \defl \{\;\alpha \in \mathrm{GL}_{2r}(R) \;\vert\;  \alpha^t \chi_r \alpha=\chi_r \; \} \smash{\rlap{.}}
 \]
 To define the elementary subgroup $\mathrm{Ep}_{2r}(R) \subseteq \mathrm{Sp}_{2r}(R)$, we need to introduce some notation. First we let $\sigma \in \Sigma_{2r}$ denote the permutation given by $\sigma(2i)=2i-1$ and $\sigma(2i-1)=2i$. For $k, \ell \in \mathbb{N}$, we write $\varepsilon_{k,\ell}$ for the matrix with entry $1$ in row $k$ and column $\ell$, and $0$ everywhere else. Finally, for $0 \leq i \neq j \leq 2r$, we set
 \[
 se_{i,j}(z)= \begin{cases}
 \id_{2r}+z \varepsilon_{i,j} & \text{if}\ i=\sigma(j) \\
 \id_{2r}+z \varepsilon_{i,j} - (-1)^{i+j} z \varepsilon_{\sigma(j),\sigma(i)} & \text{if}\ i \neq \sigma(j) \ \text{and}\ i<j
 \end{cases}
 \]
 where $z \in R$. The subgroup of $\mathrm{Sp}_{2r}(R)$ generated by the $se_{i,j}(z)$ is denoted by $\mathrm{Ep}_{2r}(R)$. By \cite[Corollary~1.11]{KOPEIKO}, $\mathrm{Ep}_{2r}(R)$ is a normal subgroup of $\mathrm{Sp}_{2r}(R)$ whenever $r \geq 2$. We denote the quotient by $K_1 \mathrm{Sp}_{2r}(R) \defl \mathrm{Sp}_{2r}(R) \slash \mathrm{Ep}_{2r}(R)$. For $r \geq 2$, this defines a functor
 \[
 K_1 \mathrm{Sp}_{2r} \colon \CRing \rightarrow \Grp \smash{\rlap{.}}
 \]
 It follows from the construction that this functor is finitary.
 
 \begin{prop}\label{prop:symplectic_Vorst_lemma}
 let $\varphi \colon A \rightarrow B$ be an analytic isomorphism along $S$. Then for all $r \geq 2$, the following hold:
 \begin{enumerate}
 \item[(a)] Each $\varepsilon \in \mathrm{Ep}_{2r}(B_{\varphi(S)})$ factors as
 \[
 \varepsilon=(\varepsilon_1)_{\varphi(S)} \cdot \varphi_S(\varepsilon_2)
 \]
 for some $\varepsilon_1 \in \mathrm{Ep}_{2r}(B)$ and $\varepsilon_2 \in \mathrm{Ep}_{2r}(A_S)$; 
 \item[(b)]
  The commutative square
  \[
 \xymatrix{K_1 \mathrm{Sp}_{2r}(A) \ar[r] \ar[d]_{\varphi} & K_1 \mathrm{Sp}_{2r}(A_S)  \ar[d]^{\varphi_S} \\
 K_1 \mathrm{Sp}_{2r}(B) \ar[r] & K_1 \mathrm{Sp}_{2r}(B_{\varphi(S)}) }  
  \]
  is a weak pullback square.
\end{enumerate}  
 \end{prop}
 
 \begin{proof}
 This follows from the general result of Stavrova, see Proposition~\ref{prop:unstable_K1_excision_reductive_group}. In this particular case, it can however be seen more directly.
 
 To apply the argument of Vorst's lemma, we need two ingredients: first, we need a version of Suslin's lemma \cite[Lemma~3.3]{SUSLIN_SPECIAL} and second we need that $se_{i,j}(z+w)=se_{i,j}(z)se_{i,j}(w)$ holds for all $z$ and $w \in R$. The latter readily follows from the fact that $\varepsilon_{i,j} \varepsilon_{k,\ell}=\delta_{jk} \varepsilon_{i\ell}$ holds.
 
 For the first ingredient, we need to know that for each $f \in R$, each $p(t) \in R_f[t]$, and each $\alpha \in \mathrm{Ep}_{2r}(R_f)$, there exists some $m \in \mathbb{N}$ and $\sigma(t) \in \mathrm{Ep}_{2r}(R[t])$ such that
 \[
 \alpha \cdot se_{i,j}\bigl(f^m t p(f^m t) \bigr) \cdot \alpha^{-1}=\bigl(\sigma(t)\bigr)_f
 \]
 holds. A proof of this claim can be found in \cite[Theorem~5.2]{STEPANOV} (for a more general class of groups). It can also be proved similarly to \cite[Corollary~1.11]{KOPEIKO}, by replacing the variable $a$ with $tp(t)$.
 
 With these ingredients, we can follow the proof of Vorst's lemma verbatim to establish Part~(a). Part~(b) is a formal consequence of Part~(a).
 \end{proof}
 
 \begin{prop}\label{prop:Horrocks_for_K1Sp}
 For $r \geq 2$, the functors $K_1 \mathrm{Sp}_{2r} \colon \CRing \rightarrow \Grp$ satisfy the Horrocks principle $(\mathrm{H})$.
 \end{prop}
 
 \begin{proof}
 From \cite[Theorem~3.10]{KOPEIKO} it follows that $K_1 \mathrm{Sp}_{2r}$ satisfies $(\mathrm{H}_{\mathrm{loc}})$. Since $K_1 \mathrm{Sp}_{2r}$ is finitary and satisfies weak analytic excision (by the above proposition), it satisfies the Quillen principle $(\mathrm{Q})$. It thus follows from Proposition~\ref{prop:QHloc_implies_H} that the functors $K_1 \mathrm{Sp}_{2r}$ satisfy the Horrocks principle $(\mathrm{H})$.
 \end{proof}
 
 \begin{cor}\label{cor:K1Sp_trivial_for_Laurent_rings}
 Let $k$ be a field and let $m,n \in \mathbb{N}$. Then
 \[
 K_1 \mathrm{Sp}_{2r}(k[x_1^{\pm}, \ldots, x_m^{\pm},t_1,\ldots, t_n]) \cong 0
 \]
 holds for all $r \geq 2$.
 \end{cor}
 
 \begin{proof}
 From \cite[\S 4]{DICKSON} it follows that $K_1 \mathrm{Sp}_{2r}(k) \cong 0$ for all fields $k$ and all $r \geq 2$. Thus $k$ has the strong $K_1 \mathrm{Sp}_{2r}$-extension property over itself (see Definition~\ref{dfn:strong_extension_property}). The claim thus follows from Theorem~\ref{thm:strong_F_extension_for_Laurent}.
 \end{proof}
 
 Another ingredient we need is the $j$-injectivity of $K_1 \mathrm{Sp}_{2r}$. This follows from the following general fact about simply connected Chevalley--Demazure group schemes $G$. Recall that, for any ideal $I \subseteq R$, the group $G(R,I)$ denotes the kernel of the homomorphism $G(R) \rightarrow G(R \slash I)$.
 
 \begin{prop}\label{prop:Chevalley_Demazure_is_j_injective}
 If $G$ is a simply connected Chevalley--Demazure group scheme, then $K_1^G$ is $j$-injective.
 \end{prop}
 
 \begin{proof}
 Let $I \subseteq R$ be an ideal which is contained in the Jacobson radical of $R$. Pick a $\sigma \in G(R)$ whose image in $G(R \slash I)$ lies in the elementary subgroup. Then we can find $\varepsilon \in E(R)$ such that $\bar{\varepsilon}=\bar{\sigma}$ in $G(R \slash I)$. It follows that $\sigma \varepsilon^{-1}$ lies in $G(R,I)$. From \cite[Proposition~2.3]{ABE_SUZUKI} it follows that $G(R,I) \subseteq E(R)$, so $\sigma$ is elementary.
 
 The $j$-injectivity of $K_1^G$ follows from this particular case: if $[\bar{\sigma}]=[\bar{\tau}]$ holds in $K_1^G(R \slash I)$, then $\tau^{-1} \sigma$ is sent to an element of $E(R \slash I)$. From the argument above it follows that $\tau^{-1} \sigma$ lies in $E(R)$, which implies that $[\sigma]=[\tau]$ holds in $K_1^G(R)$.
 \end{proof}
 
 \begin{rmk}\label{rmk:K1Sp_j_injective}
 The above proposition implies that $K_1 \mathrm{Sp}_{2r}$ is $j$-injective for all $r \geq 2$. We will only apply this for the functor
 \[
 K_1 \mathrm{Sp} = \mathrm{colim}_r K_1 \mathrm{Sp}_{2r} \smash{\rlap{,}}
 \]
 in which case one can also use \cite[Corollary~II.7.12]{BASS_UNITARY} or \cite[Lemma~16.3]{SUSLIN_VASERSTEIN}.
 \end{rmk}
 
 Next we need some exact sequences relating $K_1 \mathrm{Sp}$ and $K_0 \mathrm{Sp}$. For the moment, the elementary exact sequences introduced in \cite[\S VII]{BASS_KTHEORY} suffice for our purposes. Recall that for any finitely generated projective $R$-module $Q$, the \emph{hyperbolic module} $H(Q)=(Q \oplus Q^{\vee},h_Q)$ lies in $\Alt(R)$, where $h_Q \bigl( (x,f),(y,g)\bigr)=f(y)-g(x)$. There exist isomorphisms $H(Q \oplus Q^{\prime}) \cong H(Q) \oplus H(Q^{\prime})$ and $(P,h) \oplus (P,-h) \cong H(P)$ for any symplectic $R$-module $(P,h) \in \Alt(R)$ (see \cite[Propositions~I.3.6.(c) and I.3.7]{BASS_UNITARY} for the latter). It follows that each $(P,h)$ is a direct summand of $H(R^n)$ for $n$ sufficiently large.
 
 \begin{prop}\label{prop:Bass_exact_sequences_K*Sp}
 let
 \[
 \xymatrix{A \ar[r] \ar[d] & C \ar[d] \\ B \ar[r] & D}
 \]
 be a Milnor square or an analytic patching diagram. Then there is an exact sequence
\[
\xymatrix@R=5pt@C=15pt{K_1 \mathrm{Sp}(A) \ar[r] & K_1 \mathrm{Sp}(B) \oplus K_1 \mathrm{Sp}(C) \ar[r] & K_1 \mathrm{Sp}(D) \ar[r] & \\ 
\ar[r] & K_0 \mathrm{Sp}(A) \ar[r] & K_0 \mathrm{Sp}(B) \oplus K_0 \mathrm{Sp}(C) \ar[r] & K_0 \mathrm{Sp}(D) }
\]
 of abelian groups.
 \end{prop}
 
 \begin{proof}
 The case of Milnor squares is covered in \cite[Corollary~II.2.3]{BASS_UNITARY}. The proof of the two cases is very similar, so we include a proof of the case of Milnor squares as well.
 
 We prove the claim by showing that the premises of \cite[Theorem~VII.4.3]{BASS_KTHEORY} are satisfied. Thus we need to check that the functors $F_1$ and $F_2$ in the diagram
 \[
 \xymatrix{\Alt(A) \ar[r] \ar[d] & \Alt(C) \ar[d]^{F_2} \\ \Alt(B) \ar[r]_-{F_1} & \Alt(D) }
 \]
 form a cofinal pair, that the square is cartesian, and that it is $E$-surjective (see \cite[Definition~VII.3.3]{BASS_KTHEORY}). The cofinality condition is easiest to check: it follows directly from the fact that each $(P,h) \in \Alt(R)$ is a direct summand of some hyperbolic module $H(R^n)$. That $F_1$ and $F_2$ form a cofinal pair in the sense of \cite[Definition~VII.3.2]{BASS_KTHEORY} follows from the fact that $F_1\bigl(H(B^n)\bigr) \cong H(D^n)$ and similarly $F_2\bigl(H(C^n)\bigr) \cong H(D^n)$.
 
 That the square of categories is cartesian can be proved in the same way for Milnor squares and for analytic patching diagrams. The case of Milnor squares is treated in detail in \cite[Theorem~III.2.2]{BASS_UNITARY}. One can either follow the same argument in the case of analytic patching diagrams, or use the following more abstract reasoning. Namely, let
 \[
 \ca{F} \colon \SymMonCat \rightarrow \Cat
 \]
 be the 2-functor which sends a symmetric monoidal category $(\ca{M},\otimes,I)$ to the category whose objects are pairs $(M,h \colon M \otimes M \rightarrow I)$ (where $I$ denotes the unit object) and morphisms $(M,h) \rightarrow (M^{\prime},h^{\prime})$ given by morphisms $\varphi \colon M \rightarrow M^{\prime}$ such that $h^{\prime} \circ \varphi \otimes \varphi =h$ holds.
 
 This functor is birepresentable, hence it preserves cartesian squares. Indeed, the birepresentation in question is given by the coinserter
 \[
 \xymatrix@R=10pt{ & T_1 \ar[rd]^{F} \\ T_1 \rrtwocell\omit{h} \ar[ru]^{M \otimes M} \ar[rd]_{I} && \ca{M}\\
 & T_1 \ar[ru]_{F} }
 \]
 where $T_1$ denotes the free symmetric monoidal category on a single object $M$.
 
 Since the functor $\Proj(A) \rightarrow \Proj(B) \times \Proj(C)$ detects isomorphisms and the zero morphism, it follows that a pair $(P, h \colon P \otimes P \rightarrow A)$ is nonsingular and alternating if and only if its image in both $\Proj(B)$ and $\Proj(C)$ has these properties (for the second property, we consider the morphism $h \circ x \otimes x$ for $x \colon A \rightarrow P$ arbitrary). These considerations imply that the square
 \[
 \xymatrix{\Alt(A) \ar[r] \ar[d] & \Alt(C) \ar[d] \\ \Alt(B) \ar[r] & \Alt(D) } 
 \]
 is cartesian, as claimed.
 
 It only remains to check $E$-surjectivity. Since each $(P,h) \in \Alt(A)$ is a direct summand of some $H(A^n)$, it suffices to consider $\varepsilon$ in the commutator subgroup $[\mathrm{Sp}_{2n}(D),\mathrm{Sp}_{2n}(D)]$. Since we are allowed to add direct summands, we can apply the fact that $[\mathrm{Sp}_{2n},\mathrm{Sp}_{2n}] \subseteq \mathrm{Ep}_{4n}$ (see \cite[Theorem~II.5.2]{BASS_UNITARY} or \cite[Theorem~1.4]{VASERSTEIN}), so we can assume that $\varepsilon$ is an elementary symplectic matrix. In this case, the conclusion for $E$-surjectivity follows from Vorst's lemma in the case of an analytic patching diagram (see Part~(a) of Proposition~\ref{prop:symplectic_Vorst_lemma}), respectively directly from the surjectivity of one of the ring homomorphisms in the case of a Milnor square.
 
 It remains to show that the groups $K_1$ in \cite[Theorem~VII.4.3]{BASS_KTHEORY} coincide with the groups $K_1 \mathrm{Sp}$. From \cite[Corollary~VII.2.3]{BASS_KTHEORY}, applied to the constant functor $F$ and to the objects $A_i=H(R)$, it follows that 
 \[
 K_1\bigl(\Alt(R)\bigr)=\mathrm{Sp}(R) \slash [\mathrm{Sp}(R),\mathrm{Sp}(R)] \smash{\rlap{.}}
 \]
 The isomorphism $K_1\bigl(\Alt(R)\bigr) \cong K_1 \mathrm{Sp}(R)$ thus follows from $[\mathrm{Sp}(R),\mathrm{Sp}(R)]=\mathrm{Ep}(R)$ (see \cite[Theorem~II.5.2]{BASS_UNITARY} and \cite[Theorem~1.4]{VASERSTEIN}).
 \end{proof}
 
 \begin{rmk}\label{rmk:compatible_and_pfaffian}
 Since the exact sequence of \cite[Theorem~VII.4.3]{BASS_KTHEORY} is natural for functors between two cartesian squares, the forgetful functors $\Alt(R) \rightarrow \Proj(R)$ induce a natural transformation from the symplectic $K$-theory sequence to the ordinary $K$-theory sequence.
 
From the Pfaffian identity
\[
\mathrm{pf}(B^tAB)=\mathrm{det}(B)\mathrm{pf}(A)
\]
valid for any alternating $2r \times 2r$-matrix $A$ and arbitrary $2r \times 2r$-matrix $B$ it follows that symplectic matrices have determinant $1$. Thus the homomorphism
\[
K_1 \mathrm{Sp}(R) \rightarrow K_1(R)
\]
induced by the forgetful functor $\Alt(R) \rightarrow \Proj(R)$ factors through the subgroup $SK_1(R)$.
\end{rmk}

 With these ingredients in place, we can return to the study of the ring $C=k[a,x]_S[y,t] \slash \bigl(t^2-t(xy-a^{\ell})\bigr)$, where $S=1+xk[a,x]$. Recall that we are interested in the vanishing of the group
 \[
 W(C)=\mathrm{ker} \bigl(K_0 \mathrm{Sp}(C) \rightarrow K_0(C)\bigr) \smash{\rlap{.}}
 \]
 The exact sequences of Proposition~\ref{prop:Bass_exact_sequences_K*Sp} make it possible to show that $W(C)$ is isomorphic to the kernel of the forgetful homomorphism
 \[
 K_1 \mathrm{Sp} \bigl(k[a,x,y] \slash (xy-a^{\ell})\bigr) \rightarrow K_1 \bigl(k[a,x,y] \slash (xy-a^{\ell})\bigr) \smash{\rlap{.}}
 \]
 In order to simplify the notation, we let $D \defl k[a,x]_S$. The isomorphism is obtained in two steps, both of which rely on the following lemma.
 
\begin{lemma}\label{lemma:isomorphic_kernels_diagram_chase}
Suppose that in the commutative diagram
\[
\xymatrix@R=15pt@C=15pt{ & 0 \ar[d] & 0 \ar[d] \\ & K_1 \ar[d] \ar[r]^{\alpha} & K_2 \ar[d] \\ 0 \ar[r] & M \ar[d] \ar[r] & N \ar[d] \ar[r] & L \ar[d]^{m_1} \\& M^{\prime} \ar[r]^{m_2} & N^{\prime} \ar[r] & L^{\prime} }
\]
of abelian groups, the row and the two columns whose first entry is zero are exact. If $m_1$ and $m_2$ are both monic, then $\alpha$ is an isomorphism.
\end{lemma}

\begin{proof}
 Clearly $\alpha$ is injective from the exactness assumption, so it only remains to check surjectivity. This follows from a straightforward diagram chase.
\end{proof}

By definition of $C$, the diagram
\[
\xymatrix{C \ar[r] \ar[d] & D[y] \ar[d] \\ D[y] \ar[r] & D[y] \slash (xy-a^{\ell})}
\]
is a Milnor square (where the two arrows with domain $C$ are obtained by quotienting modulo $t$ respectively modulo $t-(xy-a^{\ell})$).

\begin{lemma}\label{lemma:W(C)_isomorphic_to_N1}
 The kernel
 \[
 W(C)=\mathrm{ker} \bigl( K_0 \mathrm{Sp}(C) \rightarrow K_0(C)\bigr)
 \]
 and the kernel
 \[
 N_1 \defl \mathrm{ker} \Bigl( K_1 \mathrm{Sp}\bigl(D[y] \slash (xy-a^{\ell})\bigr) \rightarrow K_1\bigl(D[y] \slash (xy-a^{\ell})\bigr) \Bigr)
 \]
 are isomorphic abelian groups.
\end{lemma}

\begin{proof}
 The Milnor square above induces by Proposition~\ref{prop:Bass_exact_sequences_K*Sp} an exact sequence of symplectic $K$-groups, which is compatible with the corresponding sequence of ordinary $K$-groups by Remark~\ref{rmk:compatible_and_pfaffian}. Thus we have a commutative diagram
 \[
 \xymatrix{& N_1 \ar[r] \ar[d] & W(C) \ar[r] \ar[d] & W(D[y])^{2} \ar[d] \\
 K_1 \mathrm{Sp}(D[y])^{2} \ar[r] \ar[d] & K_1 \mathrm{Sp}\bigl(D[y]\slash (xy-a^{\ell})\bigr) \ar[r] \ar[d] & K_0\mathrm{Sp}(C) \ar[r] \ar[d] & K_0 \mathrm{Sp}(D[y])^{2} \ar[d] \\ 
 K_1(D[y])^{2} \ar[r] & K_1 \bigl(D[y] \slash (xy-a^{\ell})\bigr) \ar[r] & K_0(C) \ar[r] & K_0(D[y])^{2} }
 \]
 whose second and third rows are exact. In order to apply the above lemma, it suffices to check three things: that $W(D[y]) \cong 0$, that $K_1 \mathrm{Sp}(D[y]) \cong 0$, and that $m_2 \colon SK_1\bigl(D[y]\slash (xy-a^{\ell})\bigr) \rightarrow K_0(C)$ is injective (recall from Remark~\ref{rmk:compatible_and_pfaffian} that $K_1 \mathrm{Sp}(R) \rightarrow K_1(R)$ factors through $SK_1(R)$).
 
 Since $K_1$ splits as a direct sum $K_1(R) \cong SK_1(R) \oplus R^{\times}$, the injectivity of $m_2$ follows from the fact that $SK_1(D[y]) \cong SK_1(D)$ is trivial (which is a consequence of the $j$-invariance of $SK_1$, see Lemma~\ref{lemma:SL_r_j_invariant}).
 
 To see that $K_1 \mathrm{Sp}(D[y]) \cong 0$, note that $K_1 \mathrm{Sp}$ satisfies the Roitman principle $(\mathrm{R})$ (by Theorem~\ref{thm:weak_excision_implies_QR}) since it satisfies weak analytic excision (see Proposition~\ref{prop:symplectic_Vorst_lemma}) and it is group valued and finitary. Since $k$ is $K_1 \mathrm{Sp}$-regular, so is $k[a,x]$ and therefore $D=k[a,x]_S$. This shows that $K_1 \mathrm{Sp}(D[y]) \cong K_1 \mathrm{Sp}(D)$, and from the $j$-injectivity of $K_1 \mathrm{Sp}$ (see Remark~\ref{rmk:K1Sp_j_injective}) it follows that $K_1 \mathrm{Sp}(D) \rightarrow K_1 \mathrm{Sp}(k[a])$ is injective. Since $K_1 \mathrm{Sp}(k[a]) \cong 0$, this implies that $K_1 \mathrm{Sp}(D[y]) \cong 0$, as claimed.
 
 To see that $W(D[y]) \cong 0$, it suffices by \cite[Th{\'e}or{\`e}me~4]{BASS_LIBERATION} to check that $W_r(D[y])=\mathrm{Um}_r(D[y]) \slash \mathrm{E}_r(D[y])$ is trivial for all $r \geq 4$ and that furthermore $\mathrm{Um}_3(D[y]) \slash \mathrm{SL}_3(D[y]) \cong \ast$. We will in fact show the slightly stronger claim that $W_r(D[y]) \cong \ast$ for all $r \geq 3$. The argument follows the same reasoning as above, with $W_r$ in place of $K_1 \mathrm{Sp}$. Namely, since $k$ is $W_r$-regular, so is $D$. The fact that $W_r$ is partially $j$-injective (see Example~\ref{example:partially_j_injective}) implies that $W_r(D) \rightarrow W_r(k[a])$ has trivial kernel. The claim thus follows from the isomorphism $W_r(k[a]) \cong \ast$.
\end{proof}

\begin{lemma}\label{lemma:N1_isomorphic_to_N2}
The kernels
\[
 N_1 \defl \mathrm{ker} \Bigl( K_1 \mathrm{Sp}\bigl(D[y] \slash (xy-a^{\ell})\bigr) \rightarrow K_1\bigl(D[y] \slash (xy-a^{\ell})\bigr) \Bigr)
\]
and
\[
 N_2 \defl \mathrm{ker} \Bigl( K_1 \mathrm{Sp}\bigl(k[a,x,y] \slash (xy-a^{\ell})\bigr) \rightarrow K_1\bigl(k[a,x,y] \slash (xy-a^{\ell})\bigr) \Bigr)
\]
are isomorphic.
\end{lemma}

\begin{proof}
Recall that the square
\[
\xymatrix{k[a,x] \ar[r] \ar[d] & D \ar[d] \\ k[a,x,y] \slash (xy-a^{\ell}) \ar[r] & D[y] \slash (xy-a^{\ell})}
\]
is an analytic patching diagram (since $a^{\ell} \in k[a]$ is a nonzerodivisor, see Lemma~\ref{lemma:analytic_iso_arising_from_overring}). Thus Proposition~\ref{prop:Bass_exact_sequences_K*Sp} shows that there exists an exact sequence of symplectic $K$-groups associated to this square, which is compatible with the long exact sequence of ordinary $K$-groups by Remark~\ref{rmk:compatible_and_pfaffian}. In order to simplify the notation, we again write $A=k[a]$ and $B=k[a,x,y] \slash(xy-a^{\ell})$.

 We thus obtain the commutative diagram
 \[
 \xymatrix@C=15pt{& N \ar[r] \ar[d] & N_1 \ar[r] \ar[d] & W(A[x]) \ar[d] \\ 
 K_1 \mathrm{Sp}(A[x]) \ar[r] \ar[d] & K_1 \mathrm{Sp}(D) \oplus K_1 \mathrm{Sp}(B) \ar[r] \ar[d] & K_1 \mathrm{Sp}\bigl(D[y] \slash (xy-a^{\ell}) \bigr) \ar[r] \ar[d] & K_0 \mathrm{Sp}(A[x]) \ar[d] \\ 
  K_1 (A[x]) \ar[r] & K_1(D) \oplus K_1(B) \ar[r] & K_1 \bigl(D[y] \slash (xy-a^{\ell}) \bigr) \ar[r] & K_0 (A[x]) }
 \]
 whose second and third rows are exact and where $N$ denotes the kernel of the vertical homomorphism below.
 
 We first note that $N=N_2$ since $K_1 \mathrm{Sp}(D)$ is trivial (by $j$-injectivity). As in the proof of Lemma~\ref{lemma:W(C)_isomorphic_to_N1} above, it suffices to check that $K_1 \mathrm{Sp}(A[x]) \cong 0$, that $W(A[x]) \cong 0$, and that
 \[
 m_2 \colon SK_1(D) \oplus SK_1(B) \rightarrow SK_1 \bigl( D[y] \slash (xy-a^{\ell})\bigr)
 \]
 is injective (for then we can apply Lemma~\ref{lemma:isomorphic_kernels_diagram_chase}, which gives the desired isomorphism between $N_2$ and $N_1$).
 
 The last of these three claims follows from the fact that $SK_1(A[x]) \cong 0$ and the exactness of the third row. Secondly, we know from Corollary~\ref{cor:K1Sp_trivial_for_Laurent_rings} that $K_1 \mathrm{Sp}(A[x]) \cong 0$. 
 
 Finally, to see that $W(A[x]) \cong 0$, we again use \cite[Th{\'e}or{\`e}me~4]{BASS_LIBERATION}. This result implies that it suffices to check that $W_r(A[x]) \cong \ast$ for all $r \geq 3$. This follows from the fact that fields have the strong $W_r$-extension property, see Example~\ref{example:W_r_strong_extension_property}.
\end{proof}

 With these lemmas in hand, we can finally prove the following theorem stated in the introduction.
 
 \begin{cit}[Theorem~\ref{thm:intermediate_stable_question}]
 If the Hermite ring conjecture holds, then the kernel of
 \[
 K_1 \mathrm{Sp}\bigl(k[a,x,y] \slash (xy-a^{\ell}) \bigr) \rightarrow  K_1 \bigl(k[a,x,y] \slash (xy-a^{\ell}) \bigr)
 \]
 vanishes for all fields $k$ and all natural numbers $\ell \geq 2$.
 
 On the other hand, if this kernel vanishes for $k$ and $k \langle t \rangle$ as ground field, then all stably free modules over the ring
 \[
 R= k[a,x_1, y_1, x_2, y_2] \slash (x_1 y_1 + x_2 y_2 - a^{\ell})
 \]
 are free for all $\ell \geq 2$.
 \end{cit}
 
 \begin{proof}[Proof of Theorem~\ref{thm:intermediate_stable_question}]
 By Lemmas~\ref{lemma:W(C)_isomorphic_to_N1} and \ref{lemma:N1_isomorphic_to_N2}, the kernel in question is isomorphic to the kernel $W(C)=\mathrm{ker} \bigl( K_0 \mathrm{Sp}(C) \rightarrow K_0(C) \bigr)$, where 
 \[
 C=k[a,x]_S[y,t]\slash \bigl(t^2-t(xy-a^{\ell})\bigr)
 \]
 and $S=1+xk[a,x]$. By Proposition~\ref{prop:link_between_stable_and_unstable}, this kernel is trivial if and only if the pointed set
 \[
 \widetilde{SK}_{1,2}^{1-\mathrm{st}}\bigl(k[a,x,y] \slash (xy-a^{\ell}) \bigr)
 \]
 is trivial. Proposition~\ref{prop:universal_unipotent_unimodular_row} at the beginning of this section shows that all unimodular rows of length $3$ over $R$ are completable if the kernel vanishes for both $k$ and $k \langle t \rangle$. From Corollary~\ref{cor:unimodular_rows_length_geq4_completable} it follows that all stably free $R$-modules are free in this case.
 
 To see the first claim, we need to show (by the above equivalent characterisation of the kernel's vanishing) that all $\sigma \in \mathrm{SL}_2^{1-\mathrm{st}}\bigl(k[a,x,y] \slash (xy-a^{\ell})\bigr)$ are pseudoelementary if the Hermite ring conjecture holds.
 
 To see this, let $C^{\prime}=k[a,x,y,t] \slash \bigl(t^2-t(xy-a^{\ell})\bigr)$, so that
 \[
\xymatrix{C^{\prime} \ar[r] \ar[d] & k[a,x,y] \ar[d] \\ k[a,x,y] \ar[r] & k[a,x,y] \slash (xy-a^{\ell})} 
 \]
 is a Milnor square. Each $\sigma \in \mathrm{SL}_2^{1-\mathrm{st}}\bigl(k[a,x,y] \slash (xy-a^{\ell})\bigr)$ gives a patching datum for a $1$-stably free $C^{\prime}$-module of rank $2$, so it suffices to check that all stably free $C^{\prime}$-modules are free (by the free patching lemma applied to $\sigma$). We can consider $C^{\prime}$ as a graded ring by setting $\mathrm{deg}(t)=\ell$, $\mathrm{deg}(a)=1$, and $\mathrm{deg}(x),\mathrm{deg}(y) \geq 1$ so that they sum to $\ell$. The Hermite ring conjecture implies that all stably free modules over $C^{\prime}$ are extended from the degree $0$ part $k$ (by the Swan--Weibel homotopy trick, see \cite[Proposition~V.3.9]{LAM}), hence free.
 \end{proof}

 To proceed further, we need an exact sequence for the higher symplectic $K$-groups $K_{\ast} \mathrm{Sp}$ associated to an analytic patching diagram. In the case of interest, we will be working with rings of characteristic $2$, which in particular means that $2$ is \emph{not} invertible in our rings. A systematic treatment of this case can be found in \cite{CALMES_ET_AL_I, CALMES_ET_AL_II}, so all we need to do is to relate the higher Grothendieck--Witt groups of \cite{CALMES_ET_AL_II} to the classical symplectic $K$-groups. The necessary results for this can be found in \cite{HEBESTREIT_STEIMLE}.
 
 Since both these sources treat a much more general class of objects, we will show how the particular case we are interested in follows from the general theory. Given a form parameter
 \[
 \lambda=(\xymatrix{M_{C_2} \ar[r]^-{\tau} & Q \ar[r]^-{\rho} & M^{C_2}}) 
 \]
 on $(R,M)$ in the sense of \cite[Definition~4.2.26]{CALMES_ET_AL_I}, we have an associated groupoid $\mathrm{Unimod}^{\lambda}(R;M)$ of finitely generated projective $R$-modules equipped with a unimodular (=non-singular) $\lambda$-hermitian form. We are only interested in the case $M=-R$, that is, $R$ equipped with the involution $x \mapsto -x$, and the form parameter
 \[
 e=\bigl(\xymatrix@C=15pt{(-R)_{C_2} \ar[r] & 0 \ar[r] & (-R)^{C_2}} \bigr) \smash{\rlap{.}}
 \]
 Here the $e$-hermitian forms are precisely the alternating forms (see the discussion of \cite[Definition~4.2.26]{CALMES_ET_AL_I}). The extension
 \[
 q^{ge}_{-R} \colon \ca{D}^p(R) \rightarrow \mathbf{Sp}
 \]
 of the functor sending a projective module $P$ to the abelian group of alternating forms on $P$ coincides with the functor $q^{\geq 1}_{-R}$ discussed in \S \ref{section:analytic} (Indeed, by \cite[Proposition~4.2.18]{CALMES_ET_AL_I}, the extension is unique, so the claim follows from \cite[Proposition~4.2.22]{CALMES_ET_AL_I}).
 
 If we write $\mathrm{Iso}(\ca{C})$ for the groupoid core of a category $\ca{C}$, then we have
 \[
 \mathrm{Iso}\bigl(\Alt(R) \bigr)=\mathrm{Unimod}^e(R;-R)
 \]
 using our previous notation $\Alt(R)$ for the category of finitely generated projective modules equipped with a non-singular alternating form. Given a subgroup $c \subseteq K_0(R)$, we write $\mathrm{Unimod}^e_c(R;-R)$ respectively $\Alt^{c}(R)$ for the full subcategories consisting of objects $(P,h)$ with $[P] \in c$.
 
 Recall that $\mathcal{GW}(\ca{C},q)$ denotes the Grothendieck--Witt space of the Poincar{\'e} $\infty$-category $(\ca{C},q)$ (see \cite[Definition~4.1.1]{CALMES_ET_AL_II}). By \cite[Corollary~4.2.3]{CALMES_ET_AL_II}, this is equivalent to the infinite loop space $\Omega^{\infty} \mathrm{GW}(\ca{C},q)$ associated to the spectrum $\mathrm{GW}(\ca{C},q)$ of \cite[Definition~4.2.1]{CALMES_ET_AL_II}.
 
 Finally, we write $\mathrm{Iso}\bigl(\Alt(R)\bigr)^{\mathrm{grp}}$ for the group completion of $\mathrm{Iso}\bigl(\Alt(R)\bigr)$ with respect to orthogonal direct sum and given by Quillen's $S^{-1}S$-construction (for definiteness).
 
 \begin{prop}\label{prop:group_completion_alternating_htpy_pullback}
 Let $\varphi \colon A \rightarrow B$ be an analytic isomorphism along $S$ such that $A \slash 2 \rightarrow B \slash 2 \oplus A_S \slash 2$ is injective. Then the square
 \[
 \xymatrix{ \mathrm{Iso}\bigl(\Alt(A)\bigr)^{\mathrm{grp}} \ar[r] \ar[d] & \mathrm{Iso}\bigl(\Alt^{\mathrm{im}K_0 A}(A_S)\bigr)^{\mathrm{grp}} \ar[d] \\ \mathrm{Iso}\bigl(\Alt(B)\bigr)^{\mathrm{grp}} \ar[r] & \mathrm{Iso}\bigl(\Alt^{\mathrm{im} K_0 B}(B_{\varphi(S)})\bigr)^{\mathrm{grp}} }
 \]
 is a homotopy pullback square. 
 \end{prop}
 
 \begin{proof}
 Throughout this proof, we use the the notation introduced in \cite{HEBESTREIT_STEIMLE}. It follows from \cite[Corollary~8.1.8]{HEBESTREIT_STEIMLE} that
 \[
 \mathrm{Pn}^{\mathrm{ht}} \bigl( \ca{D}^{p}(R), q_{-R}^{\geq 1} \bigr) \rightarrow \mathcal{GW} \bigl( \ca{D}^{p}(R), q_{-R}^{\geq 1} \bigr)
 \]
 is a group completion (since $q_{-R}^{\geq 1} \simeq q_{-R}^{ge}$ by the above discussion). Because the proof of \cite[Corollary~8.1.8]{HEBESTREIT_STEIMLE} proceeds by checking the premises of \cite[Corollary~8.1.5]{HEBESTREIT_STEIMLE}, we also find that
 \[
 \mathrm{Pn}^{\mathrm{ht}} \bigl( \ca{D}^{c}(R), q_{-R}^{\geq 1} \bigr) \rightarrow \mathcal{GW} \bigl( \ca{D}^{c}(R), q_{-R}^{\geq 1} \bigr)
 \]
 is a group completion for any $c \subseteq K_0(R)$.
 
 Moreover, the equivalence
 \[
 \mathrm{Unimod}^e(R;-R) \simeq \mathrm{Pn}^{\mathrm{ht}} \bigl( \ca{D}^{p}(R), q_{-R}^{\geq 1} \bigr)
 \]
 of the proof of \cite[Corollary~8.1.8]{HEBESTREIT_STEIMLE} restricts to an equivalence of groupoids
 \[
 \mathrm{Unimod}_c^e(R;-R) \simeq \mathrm{Pn}^{\mathrm{ht}} \bigl( \ca{D}^{c}(R), q_{-R}^{\geq 1} \bigr)
 \]
 (see \cite[Example~3.1.3(3)]{HEBESTREIT_STEIMLE}). In other words,
 \[
 \mathrm{Iso}\bigl(\Alt^c(R)\bigr) \simeq \mathrm{Pn}^{\mathrm{ht}} \bigl( \ca{D}^{c}(R), q_{-R}^{\geq 1} \bigr)
 \]
 is an equivalence and the resulting composite
 \[
  \mathrm{Iso}\bigl(\Alt^c(R)\bigr) \rightarrow \mathcal{GW} \bigl( \ca{D}^{c}(R), q_{-R}^{\geq 1} \bigr)
 \]
 is a group completion.
 
 From the naturality of the construction above and the uniqueness of group completions it follows that it suffices to check that the square below
 \[
 \xymatrix{ \mathcal{GW} \bigl( \ca{D}^{p}(A), q_{-A}^{\geq 1} \bigr) \ar[r] \ar[d] & \mathcal{GW} \bigl( \ca{D}^{\mathrm{im}K_0 A}(A_S), q_{-A_S}^{\geq 1} \bigr) \ar[d] \\ 
 \mathcal{GW} \bigl( \ca{D}^{p}(B), q_{-B}^{\geq 1} \bigr) \ar[r] & \mathcal{GW} \bigl( \ca{D}^{\mathrm{im}K_0 B}(B_{\varphi(S)}), q_{-B_{\varphi(S)}}^{\geq 1} \bigr) }
 \]
 is a homotopy pullback diagram. Since $\mathcal{GW} \simeq \Omega^{\infty} \mathrm{GW}$ (see \cite[Corollary~4.2.3]{CALMES_ET_AL_II}) and $\Omega^{\infty}$ is right adjoint, it suffices to check that the square
 \[
 \xymatrix{ \mathrm{GW} \bigl( \ca{D}^{p}(A), q_{-A}^{\geq 1} \bigr) \ar[r] \ar[d] & \mathrm{GW} \bigl( \ca{D}^{\mathrm{im}K_0 A}(A_S), q_{-A_S}^{\geq 1} \bigr) \ar[d] \\ 
 \mathrm{GW} \bigl( \ca{D}^{p}(B), q_{-B}^{\geq 1} \bigr) \ar[r] & \mathrm{GW} \bigl( \ca{D}^{\mathrm{im}K_0 B}(B_{\varphi(S)}), q_{-B_{\varphi(S)}}^{\geq 1} \bigr) } 
 \]
 is a homotopy pullback diagram of spectra. This follows from Proposition~\ref{prop:Grothendieck_Witt_weak_excision}, applied to the case $m=1$.
 \end{proof}
 
 Let $c \subseteq K_0(R)$ be a subgroup. In analogy with $\Alt^c(R)$, we write $\Proj^c(R)$ for the category of finitely generated projective $R$-modules $P$ with $[P] \in c$.
 
 \begin{prop}\label{prop:group_completion_projective_htpy_pullback}
 Let $\varphi \colon A \rightarrow B$ be an analytic isomorphism along $S$. Then the square
 \[
 \xymatrix{ \mathrm{Iso}\bigl(\Proj(A)\bigr)^{\mathrm{grp}} \ar[r] \ar[d] & \mathrm{Iso}\bigl(\Proj^{\mathrm{im}K_0 A}(A_S)\bigr)^{\mathrm{grp}} \ar[d] \\ \mathrm{Iso}\bigl(\Proj(B)\bigr)^{\mathrm{grp}} \ar[r] & \mathrm{Iso}\bigl(\Proj^{\mathrm{im}K_0 B}(B_{\varphi(S)})\bigr)^{\mathrm{grp}} }
 \]
 is a homotopy pullback square.
 \end{prop}
 
 \begin{proof}
 This can also be seen as a special case of the machinery developed in \cite{CALMES_ET_AL_II} and \cite{HEBESTREIT_STEIMLE}. Namely, if we let $\mathcal{K}(\ca{C})$ denote the algebraic $K$-theory space of a stable $\infty$-category $\ca{C}$ which has an exhaustive weight structure in the sense of \cite[Definition~3.1.1]{HEBESTREIT_STEIMLE}, then
 \[
 \mathrm{Iso}(\ca{C}^{\mathrm{ht}})^{\mathrm{grp}} \rightarrow \mathcal{K}(\ca{C})
 \]
 is a natural equivalence, where $\ca{C}^{\mathrm{ht}}$ denotes the heart of the weight structure, see \cite[Corollary~8.1.3]{HEBESTREIT_STEIMLE}. By \cite[Example~3.1.3(3)]{HEBESTREIT_STEIMLE}, for all $c \subseteq K_0(R)$, there exists an exhaustive weight structure on $\ca{D}^{c}(R)$ with heart $\Proj^c(R)$. Thus it only remains to show that
 \[
 \xymatrix{ \ca{D}^p(A) \ar[r] \ar[d] & \ca{D}^{\mathrm{im}K_0 A}(A_S) \ar[d] \\ \ca{D}^p(B) \ar[r] & \ca{D}^{\mathrm{im}K_0 B}(B_{\varphi(S)}) }
 \]
 is sent to a homotopy pullback square by the functor $\mathcal{K}$. This follows from a modern form of the localization theorem in $K$-theory. Namely, the square in question is a Verdier square in the sense of \cite[Definition~1.5.1]{CALMES_ET_AL_II} by the proof of \cite[Proposition~4.4.21]{CALMES_ET_AL_II}. That $\mathcal{K}$ sends Verdier squares to homotopy pullback squares (that is, $\ca{K}$ is \emph{Verdier localising} in the sense of \cite[Definition~2.7.1]{CALMES_ET_AL_II}) is shown in the proof of \cite[Corollary~4.4.15]{CALMES_ET_AL_II}. A self-contained account of this localization theorem can also be found in \cite[Theorem~6.1]{HEBESTREIT_LACHMANN_STEIMLE}.
 \end{proof}
 
 The upshot of the above two propositions is that the top and bottom squares in the cube
 \begin{equation*}\label{eqn:Iso_cube}
 \vcenter{ \xymatrix@C=10pt@R=15pt{
\mathrm{Iso}\bigl(\Alt(A)\bigr) \ar[rr] \ar[rd] \ar[dd] && \mathrm{Iso}\bigl(\Alt^{\mathrm{im}K_0A}(A_S)\bigr) \ar[dd] |!{[rd];[ld]}\hole \ar[rd] \\
& \mathrm{Iso}\bigl(\Alt(B)\bigr) \ar[rr] \ar[dd] && \mathrm{Iso}\bigl(\Alt^{\mathrm{im}K_0 B}(B_{\varphi(S)})\bigr) \ar[dd] \\
\mathrm{Iso}\bigl(\Proj(A)\bigr) \ar[rd] \ar[rr] |!{[ru];[rd]}\hole && \mathrm{Iso}\bigl(\Proj^{\mathrm{im}K_0 A}(A_S)\bigr) \ar[rd] \\
& \mathrm{Iso}\bigl(\Proj(B)\bigr) \ar[rr] && \mathrm{Iso}\bigl(\Proj^{\mathrm{im}K_0 B}(B_{\varphi(S)})\bigr)
 }}\tag{$\star \star$}
 \end{equation*}
 are sent to homotopy pullback squares by the $S^{-1}S$-construction whenever $A \rightarrow B$ is an analytic isomorphism along $S$ such that $A \slash 2 \rightarrow B \slash 2 \oplus A_S \slash 2$ is injective. Since the cube commutes up to coherent natural isomorphisms and the $S^{-1}S$-construction
 \[
 \SymMonCat \rightarrow \Cat
 \]
 is 2-functorial for natural isomorphisms, it follows that the resulting cube of spaces commutes up to coherent homotopies. This implies in particular that the vertical functors above induce a natural transformation between the long exact sequences associated to the top and bottom squares.
 
 Recall that an object $A$ in a symmetric monoidal category $(S,\perp,0)$ is called a \emph{basic object} if for all $B \in S$, there exists a $B^{\prime} \in S$ and an $n \in \mathbb{N}$ such that $B \perp B^{\prime} \cong A^{\perp n}$. In this case, we write $G(A^{\infty})$ for the group $\mathrm{colim}_n \mathrm{Aut}(A^{\perp n})$ and we let $E(A^{\infty})=[G(A^{\infty}),G(A^{\infty})]$ be the commutator subgroup of $G(A^{\infty})$.
 
 If $(T,\perp,0)$ is another symmetric monoidal category with a basic object $B$, then we say that a symmetric monoidal functor $F \colon S \rightarrow T$ preserves the basic object if there is a chosen isomorphism $FA \cong B$. This induces group homomorphisms
 \[
\mathrm{Aut}(A^{\perp n}) \rightarrow \mathrm{Aut}(FA^{\perp n}) \cong \mathrm{Aut}(B^{\perp n}) 
 \]
 and thus a group homomorphism $G(A^{\infty}) \rightarrow G(B^{\infty})$, which restricts to a group homomorphism $E(A^{\infty}) \rightarrow E(B^{\infty})$.
 
 \begin{prop}\label{prop:low_dimensional_K_group_natural_identification}
 Let $(S,\perp, 0)$ be a symmetric monoidal groupoid with basic object $A \in S$. Then there are isomorphisms
 \[
\pi_1 B(S^{-1} S) \cong G(A^{\infty}) \slash E(A^{\infty}) \quad \text{and} \quad \pi_2 B(S^{-1} S) \cong H_2\bigl(E(A^{\infty}); \mathbb{Z} \bigr)
 \]
 which are natural for symmetric monoidal functors which preserve basic objects.
 \end{prop}
 
 \begin{proof}
 By \cite[Proposition~3]{WEIBEL_AZUMAYA}, the telescope construction of \cite{GRAYSON_FUNDAMENTAL} exhibits the basepoint component of $B(S^{-1}S)$ as $BG(A^{\infty})^{+}$. The two isomorphisms arise from the long exact sequence associated to the homotopy fiber of $BG(A^{\infty}) \rightarrow B(S^{-1}S)$, which is given by
 \[
 0 \rightarrow \pi_2 B(S^{-1}S) \rightarrow \pi_1 F \rightarrow G(A^{\infty}) \rightarrow G(A^{\infty}) \slash E(A^{\infty}) \cong \pi_1 B(S^{-1}S) \smash{\rlap{.}}
 \]
 Namely, by \cite[Proposition~IV.1.7]{WEIBEL_KBOOK}, $\pi_1 F$ is a universal central extension of $E(A^{\infty})$, so $\pi_2 B(S^{-1}S) \cong H_2\bigl(E(A^{\infty});\mathbb{Z}\bigr)$.
 
 If $G \colon S \rightarrow T$ preserves the basic object, then we get a homotopy
 \[
\xymatrix{BG(A^{\infty}) \ar[r] \ar[d] \ar@{}[rd]|{\simeq_h} & BG(B^{\infty}) \ar[d] \\ B(S^{-1}S) \ar[r] & B(T^{-1}T) } 
 \]
 from the telescope construction of \cite{GRAYSON_FUNDAMENTAL}, which induces a compatible map on homotopy fibers. This, in turn, yields a natural transformation between the resulting long exact sequences of homotopy groups. This gives the desired naturality of the two isomorphisms.
 \end{proof}
 
 \begin{lemma}\label{lemma:surjectivity_on_symbols_K2Sp_to_K2}
 Let $R$ be a commutative ring. Then the homomorphism
 \[
 K_2 \mathrm{Sp}(R) =H_2 \bigl( \mathrm{Ep}(R) ; \mathbb{Z} \bigr) \rightarrow K_2(R) = H_2 \bigl( \mathrm{E}(R) ; \mathbb{Z} \bigr)
 \]
 induced by the inclusion $\mathrm{Ep}(R) \rightarrow \mathrm{E}(R)$ is surjective on Steinberg symbols and on Dennis--Stein symbols.
 \end{lemma}
 
 \begin{proof}
 Let $\mathrm{StSp}(R)$ be the symplectic Steinberg group and let $\mathrm{St}(R)$ be the Steinberg group of $\mathrm{GL}(R)$. There exists a unique homomorphism $F \colon \mathrm{StSp}(R) \rightarrow \mathrm{St}(R)$ making the diagram
 \[
 \xymatrix{\mathrm{StSp}(R) \ar@{-->}[d]_F \ar[r] & \mathrm{Ep}(R) \ar[d] \\ \mathrm{St}(R) \ar[r] & \mathrm{E}(R) }
 \]
 commutative. An explicit description of this homomorphism $F$ can be found in \cite[\S 5.5B]{HAHN_O_MEARA}. The only fact about this homomorphism (besides its existence) that we need is that both $x_{12}(r)$ and $x_{21}(r)$ lie in the image of $F$ for all $r \in R$. With the notation of \cite[\S 5]{HAHN_O_MEARA}, we have $F\bigl(X_{2,2}(r) \bigr)=x_{21(r)}$ and $F\bigl(X_{1,1}(-r)\bigr)=x_{12}(r)$.
 
 Setting $y_{12}(r)=X_{1,1}(-r)$ and $y_{21}(r)=X_{2,2}(r)$, we can define elements $\{s,t\}^{\prime}$ and $\langle x,y \rangle^{\prime} \in \mathrm{StSp}(R)$ for $s,t \in R^{\times}$ and $1-xy \in R^{\times}$ by the same formulas that appear in the definition of the Steinberg and Dennis--Stein symbols for $K_2(R)$.
 
 Since $\mathrm{Ep}(R) \rightarrow \mathrm{E}(R)$ is injective, it follows that $\{s,t\}^{\prime}$ and $\langle x,y \rangle^{\prime}$ lie in the kernel $K_2 \mathrm{Sp}(R)=\mathrm{ker} \bigl(\mathrm{StSp}(R) \rightarrow \mathrm{Ep}(R) \bigr)$, so they give the desired preimages of $\{s,t\}$ and $\langle x, y \rangle \in K_2(R)$.
 \end{proof}
 
 The following theorem summarizes those aspects of the above discussion that we need in the remainder of this section.
 
 \begin{thm}\label{thm:long_exact_sequences_K2Sp}
 Let $k$ be a field, let $A$ and $B$ be $k$-algebras, and let $\varphi \colon A \rightarrow B$ be an analytic isomorphism along $S$. Then there exists a commutative diagram
 \[
 \xymatrix@C=15pt{K_2 \mathrm{Sp}(A) \ar[r] \ar[d] & K_2 \mathrm{Sp}(B) \oplus K_2 \mathrm{Sp}(A_S) \ar[r] \ar[d] & K_2 \mathrm{Sp}(B_{\varphi(S)}) \ar[r] \ar[d] & K_1 \mathrm{Sp}(A) \ar[r] \ar[d] &  \ldots \ar[d] \\ 
 K_2(A) \ar[r]  & K_2(B) \oplus K_2(A_S) \ar[r] & K_2(B_{\varphi(S)}) \ar[r]  & K_1 \mathrm{Sp}(A) \ar[r]  &  \ldots }
 \]
 with exact rows such that the vertical homomorphisms are induced by the inclusion $\mathrm{Sp}(R) \rightarrow \mathrm{GL}(R)$. In particular, the first three vertical homomorphisms are surjective on Steinberg symbols and on Dennis--Stein symbols.
 \end{thm}
 
 \begin{proof}
 Since $A$ is an algebra over a field, we either have $2 \in A^{\times}$ or $2=0$ in $A$. In both cases, the homomorphism $A \slash 2 \rightarrow B \slash 2 \oplus A_S \slash 2$ is injective (either we have $A \slash 2 \cong 0$, or this map coincides with $A \rightarrow B \oplus A_S$, which is injective since analytic patching diagrams are pullback squares). Propositions~\ref{prop:group_completion_alternating_htpy_pullback} and \ref{prop:group_completion_projective_htpy_pullback} imply that the top and bottom face of the cube~\eqref{eqn:Iso_cube} become homotopy pullback squares after group completion. Since the resulting cube of spaces commutes up to coherent homotopies, we get a natural transformation between the two long exact sequences induced by the vertical functors.
 
 If we pick the basic object given by the hyperbolic module $H(R)$ in $\mathrm{Iso}\bigl(\Alt(R)\bigr)$ and $R^2$ in $\mathrm{Iso} \bigl(\Proj(R) \bigr)$, then all functors in the cube~\eqref{eqn:Iso_cube} preserve the basic object. Combined with the natural isomorphisms of Proposition~\ref{prop:low_dimensional_K_group_natural_identification}, we get the desired horizontal exact sequences, connected by commuting vertical homomorphisms induced by the inclusion $\mathrm{Sp}(R) \rightarrow \mathrm{GL}(R)$. The final claim thus follows from Lemma~\ref{lemma:surjectivity_on_symbols_K2Sp_to_K2}.
 \end{proof}
 
 Note that $k[a,x,y] \slash (xy-a^{\ell}) \rightarrow k \llbracket a \rrbracket [x,y] \slash (xy-a^{\ell})$ is an analytic isomorphism along $a^{\ell}$ since $k[a] \slash a^{\ell} \rightarrow k \llbracket a \rrbracket \slash a^{\ell}$ is an isomorphism and both rings are domains (by Lemma~\ref{lemma:analytic_iso_arising_from_overring}, the element $x$ is a nonzerodivisor in both rings, so the rings are subrings of $k[a,x,x^{-1}]$ and $k\llbracket a \rrbracket [x,x^{-1}]$ respectively). If we apply the above theorem to this analytic isomorphism along $a^{\ell}$, we obtain the diagram
 \begin{equation*}\label{eqn:four_lemma}
\vcenter{ \xymatrix{
 K_2 \mathrm{Sp}(B) \oplus K_2 \mathrm{Sp}(k[a^{\pm},x^{\pm}]) \ar[r]^-{\varphi_1} \ar[d] & K_2 (B) \oplus K_2 (k[a^{\pm},x^{\pm}]) \ar[d] \\
 K_2 \mathrm{Sp}\bigl(k(\!(a)\!)[x,x^{-1}] \bigr) \ar[r]^-{\varphi_2} \ar[d] & K_2 \bigl(k(\!(a)\!)[x,x^{-1}] \bigr) \ar[d] \\
 K_1 \mathrm{Sp} \bigl(k[a,x,y] \slash (xy-a^{\ell}) \bigr) \ar[r]^-{\varphi_3} \ar[d] & K_1 \bigl(k[a,x,y] \slash (xy-a^{\ell}) \bigr) \ar[d] \\
  K_1 \mathrm{Sp}(B) \oplus K_1 \mathrm{Sp}(k[a^{\pm},x^{\pm}]) \ar[r]^-{\varphi_4} & K_1 (B) \oplus K_1 (k[a^{\pm},x^{\pm}])
 }} \tag{$\square$}
 \end{equation*}
 with exact columns, where we have used the notation $B \defl k \llbracket a \rrbracket [x,y] \slash (xy-a^{\ell})$ to keep the diagram legible. In the remainder of this section, we will show that the premises of the four lemma are satisfied (under some assumptions on the ground field $k$): the homomorphisms $\varphi_2$ and $\varphi_4$ are monomorphisms and $\varphi_1$ is an epimorphism, from which it follows that $\varphi_3$ is a monomorphism.
 
 We already know from Corollary~\ref{cor:K1Sp_trivial_for_Laurent_rings} that $K_1 \mathrm{Sp}(k[a^{\pm},x^{\pm}]) \cong 0$. The following lemma shows that $K_1 \mathrm{Sp} \bigl( k \llbracket a \rrbracket [x,y] \slash (xy-a^{\ell}) \bigr) \cong 0$ as well, so $\varphi_4$ is indeed a monomorphism.
 
 \begin{lemma}\label{lemma:K1Sp_of_power_series_overring_vanishes}
 For all fields $k$ and all $\ell \geq 2$ we have $K_1 \mathrm{Sp} \bigl( k \llbracket a \rrbracket [x,y] \slash (xy-a^{\ell}) \bigr) \cong 0$.
 \end{lemma}
 
 \begin{proof}
 Since $\bigl(k \llbracket a \rrbracket, (a^{\ell}) \bigr)$ is a henselian pair, we can apply Theorem~\ref{thm:henselian_implies_trivial_kernel_for_overring} and its Corollary~\ref{cor:henselian_implies_trivial_kernel_for_overring}. The functor $F=K_1 \mathrm{Sp}$ satisfies weak analytic excision by Part~(b) of Proposition~\ref{prop:symplectic_Vorst_lemma} and it satisfies weak Milnor excision as a consequence of the exact sequence of Proposition~\ref{prop:Bass_exact_sequences_K*Sp}. By Remark~\ref{rmk:K1Sp_j_injective}, the functor $K_1 \mathrm{Sp}$ is $j$-injective. This shows that Conditions~(i), (iii), and (iv) of Theorem~\ref{thm:henselian_implies_trivial_kernel_for_overring} and Corollary~\ref{cor:henselian_implies_trivial_kernel_for_overring} are satisfied.
 
 It remains to check that $K_1 \mathrm{Sp}(k \llbracket a \rrbracket [x]) \cong 0$ and $K_1 \mathrm{Sp}(k \llbracket a \rrbracket \slash a^{\ell} [x]) \cong 0$ hold. The first follows from a general result of Grunewald, Mennicke, and Vaserstein, which implies that $K_1 \mathrm{Sp}(k \llbracket a \rrbracket [x]) \cong K_1 \mathrm{Sp}(k \llbracket a \rrbracket)$ (see \cite[Corollary~1.4]{GRUNEWALD_MENNICKE_VASERSTEIN}). From the $j$-injectivity and the fact that $K_1 \mathrm{Sp}(k) \cong 0$ we get the first claim.
 
 For the second claim, we can again use $j$-injectivity to reduce the problem to checking that $K_1 \mathrm{Sp}(k[x]) \cong 0$ holds, which we know to be true from Corollary~\ref{cor:K1Sp_trivial_for_Laurent_rings}.
 \end{proof}
 
 Next we show that $\varphi_1$ is surjective, again for arbitrary ground fields $k$.
 
 \begin{lemma}\label{lemma:surjectivity_of_second_half_of_varphi_1}
 The group $K_2 (k[a^{\pm},x^{\pm}])$ is generated by Steinberg symbols, so
 \[
 K_2 \mathrm{Sp} (k[a^{\pm},x^{\pm}]) \rightarrow K_2 (k[a^{\pm},x^{\pm}])
 \]
 is surjective.
 \end{lemma}
 
 \begin{proof}
 The second claim follows from the first and Lemma~\ref{lemma:surjectivity_on_symbols_K2Sp_to_K2}. The first claim is a consequence of the fundamental theorem of $K$-theory. Namely, we have 
 \[
 K_2(k[a^{\pm},x^{\pm}]) \cong K_2(k[a^{\pm}])\oplus K_1 (k[a^{\pm}]) \smash{\rlap{,}}
 \]
 where the second factor is given by the inclusion
 \[
 \{x,-\} \colon K_1(k[a^{\pm}]) \cong (k[a^{\pm}])^{\times} \rightarrow K_2(k[a^{\pm},x^{\pm}]) \smash{\rlap{.}}
 \]
 It only remains to check that $K_2(k[a^{\pm}])$ is generated by Steinberg symbols. A second application of the fundamental theorem shows that $K_2(k[a^{\pm}]) \cong K_2(k) \oplus K_1(k)$, where the second factor is the image of $\{a,-\} \colon k^{\times} \rightarrow K_2(k[a^{\pm}])$. From Matsumoto's theorem we know that $K_2(k)$ is generated by Steinberg symbols (see \cite[Theorem~III.6.1]{WEIBEL_KBOOK}).
 \end{proof}
 
 To show that
 \[
 K_2 \mathrm{Sp}\bigl( k \llbracket a \rrbracket [x,y] \slash (xy-a^{\ell}) \bigr) \rightarrow  K_2 \bigl( k \llbracket a \rrbracket [x,y] \slash (xy-a^{\ell}) \bigr) 
 \]
 is surjective, we repeatedly use the following fact.
 
 \begin{lemma}\label{lemma:generated_by_symbols_radical_ideal}
 Let $R$ be a commutative ring and let $I \subseteq R$ be an ideal which is contained in the Jacobson radical of $R$. If $K_2(R \slash I)$ is generated by Steinberg and Dennis--Stein symbols, then so is $K_2(R)$.
 \end{lemma}
 
 \begin{proof}
 First note that $K_2(R) \rightarrow K_2(R \slash I)$ is surjective on symbols. For example, if $\{\bar{r},\bar{s}\} \in K_2(R \slash I)$ where $\bar{s}, \bar{t} \in (R \slash I)^{\times}$, we can pick lifts $r$ of $\bar{r}$ and $s$ of $\bar{s}$ in $R$. Since $I$ is contained in the Jacobson radical, we have $r,s \in R^{\times}$, so $\{r,s\} \in K_2(R)$ gives the desired preimage of $\{\bar{r},\bar{s}\} \in K_2(R \slash I)$.
 
 The claim for Dennis--Stein symbols $\langle \bar{x},\bar{y} \rangle$ follows similarly from the observation that arbitrary lifts $x$ of $\bar{x}$ and $y$ of $\bar{y}$ satisfy $1-xy \in R^{\times}$ (since $1-\bar{x}\bar{y} \in (R \slash I)^{\times}$ by assumption). Thus $\langle x, y \rangle$ is defined and gives the desired preimage of $\langle \bar{x}, \bar{y} \rangle$.
 
 The claim now follows from the exactness of
 \[
 K_2(R,I) \rightarrow K_2(R) \rightarrow K_2(R \slash I)
 \]
 (see \cite[Theorem~III.5.7.1]{WEIBEL_KBOOK}) and the fact that $K_2(R,I)$ is generated by Dennis--Stein symbols (see \cite[Theorem~III.5.11.1(b)]{WEIBEL_KBOOK}).
 \end{proof}
 
 \begin{lemma}\label{lemma:surjectivity_of_first_half_of_varphi_1}
 For each field $k$ and each $\ell \geq 2$, the homomorphism
 \[
 K_2 \mathrm{Sp}\bigl( k \llbracket a \rrbracket [x,y] \slash (xy-a^{\ell}) \bigr) \rightarrow  K_2 \bigl( k \llbracket a \rrbracket [x,y] \slash (xy-a^{\ell}) \bigr) 
 \]
 is surjective.
 \end{lemma}
 
 \begin{proof}
 We use the notation $A \defl k \llbracket a \rrbracket$ and $B \defl A[x,y] \slash (xy-a^{\ell})$. Recall that the inclusion $A[y] \rightarrow B$ is an analytic isomorphism along $S \defl 1+yA[y] \subseteq A[y]$ (see Lemma~\ref{lemma:analytic_iso_arising_from_overring}). From Theorem~\ref{thm:long_exact_sequences_K2Sp} we thus get a commutative diagram
 \[
 \xymatrix{K_2 \mathrm{Sp}(A[y]) \ar[r] \ar[d] & K_2 \mathrm{Sp} (B) \oplus K_2 \mathrm{Sp} (A[y]_S) \ar[r] \ar[d] & K_2 \mathrm{Sp} (B_S) \ar[r] \ar[d] & K_1 \mathrm{Sp} (A[y]) \ar[d] \\ 
 K_2 (A[y]) \ar[r] & K_2(B) \oplus K_2(A[y]_S) \ar[r]  & K_2(B_S) \ar[r] & K_1(A[y])}
 \]
 with exact rows. By Lemma~\ref{lemma:generated_by_symbols_radical_ideal}, $K_2(A[y]) \cong K_2(A)$ is generated by Steinberg and Dennis--Stein symbols, so the leftmost vertical homomorphism is an epimorphism (see Lemma~\ref{lemma:surjectivity_on_symbols_K2Sp_to_K2}). On the other hand, from \cite[Corollary~1.4]{GRUNEWALD_MENNICKE_VASERSTEIN}  and $j$-injectivity of $K_1 \mathrm{Sp}$ we know that $K_1 \mathrm{Sp}(A[y]) \cong K_1 \mathrm{Sp}(A) \cong 0$, so the rightmost vertical map is a monomorphism. To establish the surjectivity of $K_2 \mathrm{Sp} (B) \rightarrow K_2(B)$, it suffices by the four lemma to check that
 \[
 K_2 \mathrm{Sp} (B_{1+yA[y]}) \rightarrow K_2(B_{1+yA[y]})
 \]
 is surjective.
 
 Since $a^{\ell}$ lies in the Jacobson radical of $A=k\llbracket a \rrbracket$, Proposition~\ref{prop:composite_patching_for_overrings} implies that the composite
 \[
 A[x] \rightarrow B \rightarrow B_{1+yA[y]}
 \]
 is an analytic isomorphism along $1=xA[x]$, so
 \[
 \xymatrix{A[x] \ar[r] \ar[d] & A[x]_{1+xA[x]} \ar[dd] \\ B \ar[d] \\ B_{1+yA[y]} \ar[r] & B_{1+yA[y],1+xA[x]}}
 \]
 is an analytic patching diagram. Applying the same reasoning as above to the diagram resulting from Theorem~\ref{thm:long_exact_sequences_K2Sp} we find that it suffices to show that the homomorphism
 \[
 K_2 \mathrm{Sp} (B_{1+yA[y],1+xA[x]}) \rightarrow K_2(B_{1+yA[y],1+xA[x]})
 \]
 is surjective.
 
 From Lemma~\ref{lemma:surjectivity_on_symbols_K2Sp_to_K2} we know that it suffices to show that the target is generated by Steinberg and Dennis--Stein symbols. To see this, we first note that $K_1(A[x]) \cong (A[x])^{\times}$ since $SK_1(A[x]) \cong SK_1(A) \cong 0$ by $j$-invariance of $SK_1$. It follows that the two homomorphisms
 \[
 K_1(A[x]) \rightarrow K_1(A[x]_{1+xA[x]}) \quad \text{and} \quad K_1(A[x]) \rightarrow K_1(B_{1+yA[y]})
 \]
 are jointly monomorphic (since analytic patching diagrams are pullback squares). From the exact sequence in $K$-theory induced by the above patching diagram it follows that
 \[
 K_2(B_{1+yA[y]}) \oplus K_2(A[x]_{1+xA[x]}) \rightarrow K_2(B_{1+yA[y],1+xA[x]})
 \]
 is surjective. From Lemma~\ref{lemma:generated_by_symbols_radical_ideal} we conclude that $K_2(A[x]_{1+xA[x]})$ is generated by symbols, so it only remains to check that the image of
 \[
 K_2(B_{1+yA[y]}) \rightarrow K_2(B_{1+yA[y],1+xA[x]})
 \]
 is generated by symbols.
 
 The same reasoning as above, applied to the analytic patching diagram
 \[
 \xymatrix{A[y] \ar[r] \ar[d] & A[y]_{1+yA[y]} \ar[d] \\ B \ar[r] & B_{1+yA[y]}}
 \]
 shows that
 \[
 K_2(B) \oplus K_2(A[y]_{1+yA[y]}) \rightarrow K_2(B_{1+yA[y]})
 \]
 is surjective. This reduces the problem to checking that the image of
 \[
  K_2(B) \rightarrow K_2(B_{1+yA[y],1+xA[x]})
 \]
 is generated by symbols.
 
 Since $A=k\llbracket a \rrbracket$ is $a^{\ell}$-adically complete, the pair $\bigl(A,(a^{\ell})\bigr)$ is henselian. From Proposition~\ref{prop:overrings_in_henselian_case} it follows that $B \rightarrow B_{1+yA[y],1+xA[x]}$ factors through the localization $B \rightarrow B_{1+a^{\ell}B}$, so it suffices to show that $K_2(B_{1+a^{\ell}B})$ is generated by symbols.
 
 By construction, $a^{\ell}$ (and therefore $a$) lies in the Jacobson radical of $B_{1+a^{\ell} B}$. Thus Lemma~\ref{lemma:generated_by_symbols_radical_ideal} reduces the problem to checking that
 \[
 K_2(B \slash a) \cong K_2\bigl(k[x,y] \slash (xy)\bigr)
 \]
 is generated by Steinberg and Dennis--Stein symbols.
 
 By \cite[Proposition~2.4]{DENNIS_KRUSEMEYER}, we have $K_2\bigl(k[x,y] \slash (xy) \bigr) \cong K_2(k) \oplus L_1 \oplus L_2 \oplus K$, where $K$ is a subgroup generated by Dennis--Stein symbols (see \cite[Corollary~2.7]{DENNIS_KRUSEMEYER}) and $L_i$ is in this case isomorphic to the kernel of $K_2(k[x]) \rightarrow K_2(k)$. Thus $L_i \cong 0$ since $k$ is a regular ring. The claim follows from the fact that $K_2(k)$ is generated by Steinberg symbols by Matsumoto's theorem (see \cite[Theorem~III.6.1]{WEIBEL_KBOOK}).
 \end{proof}
 
 With this final ingredient in place, we can prove the following theorem of the introduction.
 
 \begin{cit}[Theorem~\ref{thm:injectivity_for_perfect_field}]
 If $k$ is a perfect field of characteristic $2$ and $\ell \geq 2$, then
 \[
 K_1 \mathrm{Sp}\bigl(k[a,x,y] \slash (xy-a^{\ell}) \bigr) \rightarrow  K_1 \bigl(k[a,x,y] \slash (xy-a^{\ell}) \bigr)
 \]
 is injective
 \end{cit}
 
 \begin{proof}[Proof of Theorem~\ref{thm:injectivity_for_perfect_field}]
 From Corollary~\ref{cor:K1Sp_trivial_for_Laurent_rings} and Lemma~\ref{lemma:K1Sp_of_power_series_overring_vanishes} we know that the map $\varphi_4$ in Diagram~\eqref{eqn:four_lemma} is injective. Lemmas~\ref{lemma:surjectivity_of_second_half_of_varphi_1} and \ref{lemma:surjectivity_of_first_half_of_varphi_1} together imply that the homomorphism $\varphi_1$ of Diagram~\eqref{eqn:four_lemma} is surjective. By the four lemma, applied to Diagram~\eqref{eqn:four_lemma}, it only remains to show that
 \[
 \varphi_2 \colon K_2 \mathrm{Sp} \bigl(k(\!(a)\!)[x,x^{-1}]\bigr) \rightarrow K_2\bigl(k(\!(a)\!)[x,x^{-1}]\bigr)
 \]
 is injective to establish the claim of the theorem.
 
 From \cite[Theorem~A]{MORITA_REHMANN} we know that the kernel of $\varphi_2$ is given by $I^3\bigl(k(\!(a)\!)\bigr) \oplus I^2\bigl(k(\!(a)\!)\bigr)$, where $I^n(F)$ denotes the $n$-th power of the fundamental ideal of the Witt ring $W_{\mathrm{sym}}(F)$ of symmetric bilinear forms over a field $F$.
 
 Let $F=k(\!(a)\!)$ and let $S \subseteq F$ denote the subfield of squares in $F$ (in other words, $S$ is the image of the Frobenius endomorphism). Since $k$ is perfect, the set $S$ consists of all those formal Laurent series $f=\sum_{i=-k}^{\infty} f_i a^i$ such that $f_i=0$ for all odd $i \in \mathbb{Z}$. A general formal Laurent series can thus be written as
 \[
 f=\textstyle \sum_{i\;\text{even}} f_i a^i+a(\textstyle \sum_{i\; \text{odd}} f_ia^{i-1}) \smash{\rlap{,}}
 \]
 which shows that $S \subseteq F$ is a field extension of degree $2$. It follows from \cite[Theorem~5]{MILNOR_WITT} that $I(F) \neq 0$ but $I^2(F)= 0$ (and thus also $I^3(F)=0$). This shows that $\varphi_2$ is indeed injective when the field $k$ is perfect.
 \end{proof}

\bibliographystyle{amsalpha}
\bibliography{henselian}

\end{document}